\documentclass[a4paper,10pt]{article}
\usepackage{comment,nicefrac,amsthm,amsmath,amssymb,bbm,epsfig,enumerate,geometry}
\newcommand{\smallsum}{\textstyle\sum}

\newcommand{\calF}{\mathcal F}

\newcommand{\R}{\mathbb{R}}
\newcommand{\N}{\mathbb{N}}
\newcommand{\E}{\mathbb{E}}
\renewcommand{\P}{\mathbb{P}}
\newcommand{\tr}{\operatorname{trace}}
\newcommand{\Hess}{\operatorname{Hess}}

\newtheorem{lemma}{Lemma}[section]

\newtheorem{theorem}[lemma]{Theorem}
\newtheorem{definition}[lemma]{Definition}
\newtheorem{prop}[lemma]{Proposition}
\newtheorem{corollary}[lemma]{Corollary}

\providecommand{\eps}{{\ensuremath{\varepsilon}}}

\providecommand{\N}{{\ensuremath{\mathbbm{N}}}}
\providecommand{\R}{{\ensuremath{\mathbbm{R}}}}
\providecommand{\E}{{\ensuremath{\mathbb{E}}}}
\renewcommand{\P}{{\ensuremath{\mathbb{P}}}}
\providecommand{\1}{{\ensuremath{\mathbbm{1}}}}

\providecommand{\N}{{\ensuremath{\mathbbm{N}}}}
\providecommand{\R}{{\ensuremath{\mathbbm{R}}}}

\providecommand{\E}{{\ensuremath{\mathbb{E}}}}
\renewcommand{\P}{{\ensuremath{\mathbb{P}}}}
\providecommand{\1}{{\ensuremath{\mathbbm{1}}}}
\providecommand{\HS}{{\ensuremath{\textup{HS}}}}

\begin{document}
\title{Exponential integrability properties of numerical approximation processes
for nonlinear stochastic differential equations}

\author{Martin Hutzenthaler, Arnulf Jentzen,
and Xiaojie Wang\footnote{corresponding author; e-mail: x.j.wang7@csu.edu.cn}}

\maketitle
\makeatletter
\let\@makefnmark\relax
\let\@thefnmark\relax
\@footnotetext{\emph{AMS 2010 subject classification:} 60H35, 65C30}
\@footnotetext{\emph{Key words and phrases:}
  Exponential moments,
  numerical approximation,
  stochastic differential equation,
  Euler scheme,
  Euler-Maruyama,
  implicit Euler scheme,
  tamed Euler scheme,
  strong convergence rate
  }
\makeatother

\begin{abstract}
  Exponential integrability properties of numerical approximations
  are a key tool for establishing positive rates of strong and numerically weak convergence
  for a large class of nonlinear stochastic differential equations.
  It turns out that well-known numerical approximation processes
  such as Euler-Maruyama approximations,
  linear-implicit Euler approximations,
  and some tamed Euler approximations from the literature
  rarely preserve
  exponential integrability properties of the exact solution.
  The main contribution of this article is to identify
  a class of stopped increment-tamed Euler approximations
  which preserve
  exponential integrability properties of the exact solution
  under
  minor additional
  assumptions on the involved functions.
\end{abstract}

%

\section{Introduction}

Let
$ T \in ( 0, \infty ) $,
$ d, m \in \N = \{ 1, 2, \ldots \} $,
let
$ \mu \colon \R^d \to \R^d $
and
$ \sigma \colon \R^d \to \R^{ d \times m } $
be locally Lipschitz continuous functions,
let
$ 
  ( \Omega, \mathcal{F}, \P )
$
be a probability space with a normal filtration
$
  ( \mathcal{F}_t )_{ t \in [0,T] }
$,
let $ W \colon [0,T] \times \Omega \to \R^m $
be a standard $ ( \mathcal{F}_t )_{ t \in [0,T] } $-Brownian motion
with continuous sample paths,
and let
$
  X \colon [0,T] \times \Omega \to \R^d
$
be an 
$ ( \mathcal{F}_t )_{ t \in [0,T] } $-adapted
stochastic process
with continuous sample paths satisfying
that for all $ t \in [0,T] $
it holds $ \P $-a.s.\ that
\begin{equation}
\label{eq:SDE.intro}
  X_t = X_0 + \int_0^t
  \mu( X_s ) \, ds +
  \int_0^t \sigma( X_s ) \, dW_s
  .
\end{equation}
The stochastic process $ X $ 
is thus a solution process of 
the stochastic differential equation (SDE)~\eqref{eq:SDE.intro}.

The goal of this paper is to identify numerical approximations
$
  Y^N \colon[0,T]\times\Omega\to\R^d
$, $N\in\N$,
that converge
in the strong sense
to the exact solution
of the SDE~\eqref{eq:SDE.intro}
and
that \emph{preserve exponential integrability properties}
in the sense that
for all suffciently regular functions
$
  U \colon \R^d \to [0,\infty)
$
with 
$
  \sup_{t\in[0,T]}\E\big[\exp(U(X_t))\big] < \infty
$
it holds that
$
  \sup_{ N \in \N }\sup_{t\in[0,T]}
  \E\big[
    \exp\!\big(
      U( Y_t^N )
    \big)
  \big] < \infty
$.
Our main motivation for this is that such exponential integrability properties
are a key tool for establishing
rates of strong and numerically
weak convergence
for a large class of nonlinear SDEs.
To be more specific, strong convergence rates of 
approximations of a multi-dimensional SDE
have, except of in D\"{o}rsek~\cite{Doersek2012}
and except of in \cite{HutzenthalerJentzen2014},
only been established if the coefficients of the
SDE are globally monotone
(see, e.g., (H2) in Pr{\'e}v{\^o}t \&
R\"{o}ckner~\cite{PrevotRoeckner2007}
for the
global monotonicity assumption).
Unfortunately, most of the nonlinear SDEs from
the literature fail to satisfy the global monotonicity
assumption (see, e.g., Section~4
in Cox et al.~\cite{CoxHutzenthalerJentzen2014}
for a list of examples).
In Corollary 3.2 in D\"{o}rsek~\cite{Doersek2012}, strong convergence
rates for spatial spectral Galerkin approximations
of the solution of the vorticity formulation
of two-dimensional stochastic Navier-Stokes equations
have been established by exploiting
\emph{exponential integrability properties}.
Moreover, 
the perturbation
estimate in Theorem~1.2
in \cite{HutzenthalerJentzen2014}
implies in a general setting that
suitable \emph{exponential integrability properties} of a family
of approximation processes are sufficient
to establish \emph{strong convergence rates}.
This conditional result together 
with the
\emph{exponential integrability properties}
established
in this article then yields
strong convergence rates for
the numerical scheme
proposed in this article
(see \eqref{eq:proposed.method} below)
for several SDEs with non-globally monotone
coefficients
(see Theorem 1.3 in~\cite{HutzenthalerJentzen2014} for details).
In particular, to the best of our knowledge,
the numerical scheme
proposed in this article
(see \eqref{eq:proposed.method} below)
is the first approximation method
for which temporal strong convergence rates
have been proved
(see Theorem 1.3
in \cite{HutzenthalerJentzen2014})
for at least one
multi-dimensional SDE with non-globally monotone coefficients
(see Section 3.1
in \cite{HutzenthalerJentzen2014}
for a list of example SDEs
for which temporal strong convergence rates
for the numerical method~\eqref{eq:proposed.method} below have been proved).
In addition, exponential integrability properties of numerical approximations
are necessary for approximating expectations of exponentially growing test functions
of the exact solution.

There are a number of SDEs
in the literature
that admit exponential integrability properties.
We focus on
Corollary~2.4 in Cox et al.~\cite{CoxHutzenthalerJentzen2014}
(see, for example, also Lemma 2.3 in Zhang~\cite{Zhang2010}).
Let
$
  \rho \in [0,\infty)
$,
$
  U \in
  C^3( \R^d, [0,\infty) )
$,
and
$
  \bar{U} \in C( \R^d, \R)
$
satisfy for all $ x \in \R^d $ 
that
$
  \E\big[
    e^{ U( X_0 ) }
  \big]
  <
  \infty
$,
$
  \inf_{ y \in \R^d }
  \bar{U}( y ) > - \infty
$,
and 
\begin{equation} \label{eq:Lyapunov.condition.ExpMoments}
  U'(x)\mu(x)+\tfrac{1}{2}\text{tr}\Big(\sigma(x)\sigma(x)^{*}(\Hess U)(x)\Big)
  +\tfrac{1}{2}\left\|\sigma(x)^{*}(\nabla U)(x)\right\|^2
  +\bar{U}(x)\leq \rho \, U(x)
  .
\end{equation}
Then
Corollary~2.4 in Cox et al.~\cite{CoxHutzenthalerJentzen2014}
yields that
for all $ t \in [0,T] $
it holds that
\begin{equation} 
\label{eq:ExpMoments.inequality.solution}
  \E\left[
    \exp\!\left(
      \tfrac{U(X_t)}{e^{\rho t}}+\int_0^t \tfrac{\bar{U}(X_s)}{e^{\rho s}}\,ds
    \right)
  \right]
      \leq
      \E\!\left[
        e^{ U( X_0 ) }
      \right]
    \in(0,\infty)
    .
\end{equation}
Section~\ref{sec:examples} lists a selection of SDEs from the literature
which satisfy condition~\eqref{eq:Lyapunov.condition.ExpMoments}.
Further instructive exponential integrability results
for solutions of SDEs can be found, e.g.,
in~\cite{
BouRabeeHairer2013,
EsSarhirStannat2010,
FangImkellerZhang2007,
HairerMattingly2006,
HieberStannat2013,
Li1994,
Zhang2010}.
In the light of inequality~\eqref{eq:ExpMoments.inequality.solution},
the goal of this paper is, in particular, to identify numerical approximations
that converge
in the strong sense
to the exact solution
of the SDE~\eqref{eq:SDE.intro}
and that
\emph{preserve inequality~\eqref{eq:ExpMoments.inequality.solution}}
in a suitable sense; see inequality~\eqref{eq:thm.inequality} below.

It turns out that many well-known numerical methods
for SDEs
fail to
preserve exponential integrability properties.
For instance, in the special case where
$ d = m = 1 $,
where $X_0 = \xi \in \R $,
and where for all $ x \in \R $
it holds that
$ \mu( x ) = - x ^3 $
and
$ \sigma( x ) = 1 $,
the solution process $ X $ of the SDE~\eqref{eq:SDE.intro}
satisfies that for all $ t \in [0,T] $
it holds $ \P $-a.s.\ that 
\begin{equation} \label{eq:SDE.x3}
   X_t
 =
   X_0
   -\int_0^t ( X_s )^3 \, ds + W_t
   .
\end{equation}
In that case,
inequality~\eqref{eq:Lyapunov.condition.ExpMoments}
holds with 
$ \rho = 0 
$, 
$
  \eps \in ( 0, \tfrac{ 1 }{ 2 } )
$,
$ 
  U = 
  \big( 
    \R \ni x \mapsto \eps |x|^4 \in [0,\infty)
  \big)
$,
and
$
  \bar{U}
  =
  \big(
  \R \ni x \mapsto 4 \eps \left( 1 - 2 \eps \right) x^6 - 6 \eps x^2
  \in \R
  \big)
$;
see Subsection~\ref{ssec:exampleSDE} below.
Thus,
Corollary~2.4 in Cox et al.~\cite{CoxHutzenthalerJentzen2014}
implies for all $ \eps \in[0,\tfrac{ 1 }{ 2 }]$ that
  \begin{equation}  \begin{split} \label{eq:exampleSDE.Lyapunov.intro}
    \sup_{ t \in [0,T] }
    \E\!\left[
      \exp\!\left(
        \eps \left| X_t \right|^4
        +
        \int_0^t
        4 \eps
        \left( 1 - 2 \eps \right)
        \left| X_s \right|^6
        -
        6 \eps \left| X_s \right|^2
        ds
      \right)
    \right]
  \leq
    \E\!\left[
      e^{
        \eps |X_0|^4
      }
    \right]
  < \infty
  \end{split}     \end{equation}
and, in particular,
for all $ \eps \in  [0, \frac{ 1 }{ 2 }) $
that
$
  \sup_{ t \in [0,T] }
    \E\!\left[
      \exp\!\left(
        \eps | X_t |^4
      \right)
    \right]
  < \infty
$.
If $ Y^N \colon [0,T] \times \Omega \to \R^d $, $N\in\N$,
are the classical Euler-Maruyama approximations
as defined in~\eqref{eq:stopped.Euler} below
with $D_t=\R$, $t\in(0,T]$,
then moments are
finite but
unbounded
in the sense that 
for all $ N \in \N $, $ p \in (0, \infty ) $
it holds that
$
  \E\big[
    | Y_T^N |^p
  \big] < \infty
$
and
$
  \lim_{ M \to \infty }
  \E\big[ |Y_T^M|^p \big] = \infty
$
(see Theorem 2.1 in~\cite{hjk11} for the case $ p \in [1,\infty) $
and Theorem 2.1 in~\cite{HutzenthalerJentzenKloeden2013})
whereas approximations of
$
  \E\!\left[
    \exp\!\left(
      \delta | X_t |^4
    \right)
  \right] 
$, 
$
  \delta \in (0,\eps) 
$,
$ t \in (0,T] 
$, 
are infinite in the sense that
for all 
$ N \in \N $, $ p \in (0,\infty) $, $ q \in (2,\infty) $
it holds that
$
  \inf_{ t \in (0,T] } 
  \E\big[\exp(p|Y_t^N|^q)\big] = \infty
$;
see Lemma~\ref{l:stopped.Euler.no.exp} below.
Next, if $ Y^N \colon [0,T] \times \Omega \to \R^d $, $ N \in \N $,
are the linear-implicit Euler approximations
as defined in~\eqref{eq:stopped.linear.implicit.Euler} below
with $ D_t = \R $, $ t \in (0,T] $,
then strong convergence holds
in the sense that for all $ p \in (0,\infty) $
it holds that
$
  \lim_{ N \to \infty }
  \E\big[
    | X_T - Y_T^N |^p
  \big] = 0
$
whereas approximations of
$
  \E\!\left[
    \exp\!\left(
      \delta | X_t |^4
    \right)
  \right]
$, 
$
  \delta \in (0,\eps) 
$,
$ t \in (0,T] $, 
are infinite in the sense that
for all $ N \in \N $, $ p \in (0,\infty) $, $ q \in ( 2, \infty) $
it holds that
$
  \inf_{t\in(0,T]}\E\big[\exp(p|Y_t^N|^q)\big]=\infty
$;
see Lemma~\ref{l:linear.implicit.Euler.no.exp} below.
Moreover, if $ Y^N \colon [0,T] \times \Omega \to \R^d $, $N\in\N$,
are tamed Euler approximations
as defined in~\eqref{eq:increment-tamed.Euler1}
or as in~\eqref{eq:increment-tamed.Euler2} below
with $ D_t = \R $, $ t \in (0,T] $,
then  
approximations of
$
  \E\!\left[
    \exp\!\left(
      \delta | X_t |^4 
    \right)
  \right]
$, 
$ \delta \in (0,\eps) $,
$ t \in (0,T] $, 
are finite but unbounded in the sense that
for all $ N \in \N $,
$ p \in (0,\infty) $, $ q \in (3,\infty) $
it holds that
$
  \sup_{t\in[0,T]}\E\big[\exp(\eps|Y_t^N|^4)\big]<\infty
$ 
and 
$
  \lim_{ M \to \infty }
  \inf_{ t \in (0,T] }
  \E\big[ \exp( p |Y_t^M|^q ) \big] = \infty
$;
see Corollary~\ref{c:increment.tamed.no.exp1}
and Corollary~\ref{c:increment.tamed.no.exp2} below.
In the above sense,
Euler-Maruyama approximations, linear-implicit Euler approximations,
and tamed Euler approximations
as defined in~\eqref{eq:increment-tamed.Euler1}
or as in~\eqref{eq:increment-tamed.Euler2}
are not suitable for approximating
$
  \E\!\left[\exp\!\left(\delta|X_t|^4\right)\right]
$,
$ \delta \in (0,\eps) $,
$ t \in (0,T] $,
in the numerical weak sense.
Lemma~\ref{l:unbounded.expmoments} below also indicates that
further numerical one-step approximation methods whose one-step increment function
grows sufficiently fast as  the discretization step size decreases
are not  suitable for approximating
expectations of exponential functionals in the generality of
Theorem~\ref{thm:main.result} below.

There are many results in the literature
which prove uniform boundedness of polynomial moments
of numerical approximations of certain nonlinear SDEs
with superlinearly growing coefficients;
see, e.g.,
\cite{
BouRabeeHairer2013,
BrzezniakCarelliProhl2013,
DereichNeuenkirchSzpruch2012,
GyoengyMillet2005,
Halidias2013,
hms02,
HighamMaoSzpruch2013,
h96,
hj11,
HutzenthalerJentzen2014Memoires,
hjk12,
MaoSzpruch2013NoRate,
NeuenkirchSzpruch2014,
Sabanis2013ECP,
Sabanis2016AAP,
Schurz2003,
hmps11,
TretyakovZhang2013,
WangGan2013}. 
To the best of our knowledge,
the only reference on exponential integrability properties
of appropriate numerical approximations for nonlinear SDEs is
Bou-Rabee \& Hairer~\cite{BouRabeeHairer2013}.
More precisely,
Lemma 3.6
in Bou-Rabee \& Hairer~\cite{BouRabeeHairer2013}
implies that there exists
$ \theta \in (0,\beta) $
such that
$
  \sup_{ h \in (0,1] }
  \E\big[
    \exp(
      \theta
      U( \bar{X}_{ \lfloor 1 / h \rfloor }^h )
    )
  \big]
  < \infty
$
where
$
  \bar{X}^h \colon \{ 0, 1, 2, \dots \}
  \times \Omega \to \R^d
$,
$
  h \in (0,1]
$,
is a 'patched' version of the
Metropolis-Adjusted Langevin Algorithm (MALA)
for the overdamped Langevin dynamics
(see Subsection~\ref{sec:overdamped_Langevin} below)
where the
potential energy function
$U\in C^4(\R^d,\R)$
satisfies certain assumptions; see \cite{BouRabeeHairer2013}
for the details.
In addition, Proposition 5.2
in Bou-Rabee \& Hairer~\cite{BouRabeeHairer2013}
provides an exponential one-step estimate
for MALA.

In this article, we propose the following method to approximate
the solution of the SDE~\eqref{eq:SDE.intro}
and to
preserve inequality~\eqref{eq:ExpMoments.inequality.solution}
in a suitable sense.
Let $ Y^N \colon [0,T] \times \Omega \to \R^d $,
$ N \in \N $,
be the mappings
which satisfy that
for all
$ N \in \N $,
$ n \in \{0,1,\ldots,N-1\} $,
$ t \in (\tfrac{nT}{N},\tfrac{(n+1)T}{N} ] $
that
$ Y^N_0 = X_0 $
and
\begin{equation}
\label{eq:proposed.method}
  Y^N_t =
    Y_{\frac{nT}{N}}^N
    +
    \1_{
      \!
      \left\{
        \|
          Y_{ n T / N }^N
        \|
        \leq
        \exp\left(
          | \ln( N / T ) |^{ 1 / 2 }
        \right)
      \right\}
    }
       \! \left[
       \tfrac{
         \mu(
           Y_{ n T / N }^N
         )
         (
           t - \frac{ n T }{ N }
         )
         +
         \sigma(
           Y_{ n T / N }^N
         )
         (
           W_t
           -
           W_{ n T / N }
         )
       }{
         1 +
         \|
         \mu(
           Y_{ n T / N }^N
         )
         (
           t - \frac{ n T }{ N }
         )
         +
         \sigma(
           Y_{ n T / N }^N
         )
         (
           W_t
           -
           W_{ n T / N }
         )
         \|^2
       }
            \right]
  .
\end{equation}
  This method differs from the classical Euler-Maruyama scheme in
  two aspects. First, the Euler-Maruyama increment is divided through
  by one plus the squared norm of the Euler-Maruyama increment.
  This ensures that the increments of the numerical method~\eqref{eq:proposed.method}
  are uniformly bounded.
  Second, the approximation paths with $N\in\N$ time discretizations
  are stopped after leaving the
  set
  $
    \{
      x \in \R^d\colon
      \| x \|
      \leq
      \exp\big(
          | \ln( N/T ) |^{1/2}
      \big)
    \}
  $
  where we choose the stopping levels mainly such that
  for all $ p \in (0,\infty) $
  it holds that
  $
    \lim_{ N \to \infty }
    \exp\!\big(
        | \ln( N/T ) |^{1/2}
    \big) \, N^{ - p } = 0
  $.
  These a priori bounds
  give us control on certain rare events.
  In addition, observe that the numerical approximations
  $
    \{0,1,\ldots,N\}\times\Omega\ni(n,\omega)\mapsto
    Y_{ n T / N }^N(\omega)\in \R^d
  $, $N\in\N$, can be easily
  implemented recursively.
  In fact, this implementation requires only a few additional
  arithmetical operations in each recursion step compared to
  the classical Euler-Maruyama approximations.
  Theorem~\ref{thm:main.result} below
  shows that
  the numerical approximations~\eqref{eq:proposed.method}
  preserve inequality~\eqref{eq:ExpMoments.inequality.solution}
  in a suitable sense
  under slightly stronger assumptions on $\mu$, $\sigma$, $U$, and $\bar{U}$.
%
%
\begin{theorem} \label{thm:main.result}
  Assume the above setting,
  let $ p, c \in [1,\infty) $,
  let
  $\tau_N\colon\Omega\to[0,T]$, $N\in\N$,
  be mappings
  satisfying
  for all $ N \in \N $
  that
$
  \tau_N =
  \inf\!\big(
    \big\{
    t \in \{ 0, \frac{ T }{ N }, \frac{ 2 T }{ N }, \dots, T \}
    \colon
    \| Y^N_t \|
    >
    \exp(
      | \ln( N / T ) |^{ 1 / 2 }
    )
  \big\}
  \cup\{T\}
  \big)
$,
  and assume 
  for all
  $ x, y \in\R^d $,
  $ i \in \{ 1, 2, 3 \} $
  that
  \begin{align}
      \| \mu(x) \|
      +
      \| \sigma(x) \|_{\HS(\R^m,\R^d)}
    &
  \leq
    c
    \left(
      1 +
      \| x \|^{ c }
    \right)
    ,
  \\
      | \bar{U}(x) - \bar{U}(y) |
    &
    \leq
    c
    \left(
      1 + \| x \|^{ c } + \| y \|^{ c }
    \right)
      \| x - y \|
    ,
  \\
    \| x \|^{ 1 / c }
    &
    \leq
    c
    \left(
      1 +
      U( x )
    \right)
    ,
  \\
    \|
      U^{(i)}( x )
    \|_{L^{(i)}(\R^d,\R)}
    &\leq
    c
    \left( 1 + U(x) \right)^{ \max\{ 1 - i / p , 0 \} }
    .
  \end{align}
  Then it holds for all
$ r \in (0,\infty) $ that
$
  \lim_{ N \to \infty }
  \big(
  \sup_{ t \in [0,T] }
  \E\big[
    \| X_t - Y^N_t \|^r
  \big]
  \big)
  = 0
$,
it holds that
\begin{equation}
\label{eq:thm.inequality}
  \limsup_{ N \to \infty }
  \sup_{ t \in [0,T] }
  \E\!\left[
       \exp\!\left(
         \tfrac{
           U( Y^N_t )
         }{
           e^{ \rho t }
         }
         +
         \smallint_0^{ t \wedge \tau_N }
           \tfrac{
             \bar{U}( Y^N_s )
           }{
             e^{ \rho s }
           }
         \, ds
       \right)
     \right]
  \leq
  \E\!\left[
        e^{
          U( X_0 )
        }
  \right]
  < \infty
  ,
\end{equation}
and it holds that
$
  \sup_{
    N \in \N
  }
  \sup_{ t \in [0,T] }
  \E\big[
       \exp\!\big(
           e^{ - \rho t }
           \,
           U( Y^N_t )
         +
         \smallint_0^{ t \wedge \tau_N }
             e^{ - \rho s }
             \,
             \bar{U}( Y^N_s )
         \, ds
       \big)
     \big]
  <
  \infty
$.
\end{theorem}
Theorem~\ref{thm:main.result} is a special case of
our main result,
Corollary~\ref{cor:for_examples} below,
in which the state space of the exact solution of the SDE
under consideration is an open subset of $ \R^d $.
Corollary~\ref{cor:for_examples}, in turn,
follows from our general result on exponential integrability properties
of stopped increment-tamed Euler-Maruyama schemes,
Theorem~\ref{thm:stopped.Euler.bounded.increments} below,
and from  convergence in probability
of stopped increment-tamed Euler-Maruyama schemes,
Corollary~\ref{cor:convergence_increment_tamed} below
(see also Subsection~\ref{sec:outline_proof} below for a brief outline 
of the proof of Corollary~\ref{cor:for_examples}).
To the best of our knowledge,
Theorem~\ref{thm:main.result}
and its generalization in
Corollary~\ref{cor:for_examples} below
respectively
are the first results in the literature
which imply
exponential integrability properties 
for
numerical approximations of
the stochastic Ginzburg-Landau equation
in
Subsection~\ref{ssec:Stochastic Ginzburg-Landau equation},
for numerical approximations of
the stochastic Lorenz equation with additive noise
in Subsection~\ref{ssec:stochastic.Lorenz.equation},
for numerical approximations of
the stochastic van der Pol oscillator
in Subsection~\ref{ssec:stochastic.van.der.Pol.oscillator},
for numerical approximations of
the stochastic Duffing-van der Pol oscillator
in Subsection~\ref{ssec:stochastic.Duffing.van.der.Pol.oscillator},
for numerical approximations of
the model from experimental psychology
in Subsection~\ref{ssec:experimental.psychology},
for numerical approximations of
the stochastic SIR model
in Subsection~\ref{ssec:stochastic.SIR.model},
or -- under additional assumptions
on the model --
for numerical approximations of
the Langevin dynamics
in Subsection~\ref{ssec:Langevin.dynamics}.

\subsection{A brief outline of the proof of the main result of this article}
\label{sec:outline_proof}

In this section we give a brief and rough outline of the proof of Theorem~\ref{thm:main.result}.
Theorem~\ref{thm:main.result} is a special case of Corollary~\ref{cor:for_examples},
which is the main result of this article.
For our outline of the proof of Theorem~\ref{thm:main.result},
suppose the assumptions of Theorem~\ref{thm:main.result},
let $ D_t \subseteq \R^d $, $ t \in (0,\infty) $,
be the sets with the property that for all 
$ t \in (0,\infty) $
it holds that
$
  D_t = 
  \big\{ 
    x \in \R^d 
    \colon 
    \left\| x \right\| 
    \leq 
    \exp\!\big(
      | \ln( t ) |^{ 1 / 2 }
    \big)
  \big\}
$,
let 
$
  \mathcal{G}^U_{ \mu, \sigma } \colon \R^d \to \R
$
be the function with the property that
for all $ x \in \R^d $
it holds that
$
  \mathcal{G}^U_{ \mu, \sigma }( x ) = 
  U'(x) \mu( x ) 
  +
  \frac{ 1 }{ 2 }
  \text{tr}\big(
    \sigma( x ) \sigma( x )^*
    ( \operatorname{Hess} U)( x )
  \big)
$
(cf.\ \eqref{eq:Lyapunov.condition.ExpMoments} above),
let 
$
  \Phi_h \colon \R^d \times [0,h] \times \R^m \to \R^d
$,
$ h \in [0,T] $,
be the functions with the property that
for all 
$ h \in [0,T] $,
$ (x,t,y) \in \R^d \times [0,h] \times \R^m $ it holds that
\begin{equation}
\label{eq:Phi_def}
  \Phi_h(x,t,y)
  =
  x 
  +
    \tfrac{
      \mu( x ) t + \sigma( x ) y
    }{
      1 +
      \left\|
        \mu( x ) t + \sigma( x ) y
      \right\|^2
    }
  ,
\end{equation}
and let 
$ Z^{ s, x, h } \colon [ 0, h ] \times \Omega \to \R^d $,
$ h \in [0,T-s] $,
$ s \in [0,T] $,
$ x \in \R^d $,
be the stochastic processes with the property that
for all $ s \in [0,T] $, $ h \in [0,T-s] $, $ x \in \R^d $, $ t \in [0, h ] $
it holds that
$
  Z^{ s, x, h }_t 
  = 
  \Phi_h( x, t, W_{ s + t } - W_s )
$.
Observe that
for all $ N \in \N $, $ t \in [0, \frac{ T }{ N } ] $,
$ \omega \in \Omega $
with 
$
  X_0( \omega ) \in D_{ T / N }
$
it holds that
$
  Y^N_t( \omega )
  = 
  Z^{ 0, X_0( \omega ), T / N }_t( \omega )
$.
A key step in the proof of Theorem~\ref{thm:main.result} is to show that
there exists a real number 
$ c \in (0,\infty) $
such that
for all $ s \in [0,T] $, $ h \in [0, \min\{ T - s , 1 \} ] $, 
$ x \in D_h $, $ t \in [0,h] $
it holds that
\begin{equation}
\label{eq:to_show}
  \E\!\left[ 
    \exp\!\left(
      \tfrac{
        U( Z^{ s, x, h }_t )
      }{
        e^{ \rho t }
      }
      +
      \int_0^t
      \tfrac{
        \bar{U}( Z^{ s, x, h }_r )
      }{
        e^{ \rho r }
      }
      \, dr
    \right)
  \right]
  \leq
  \exp\!\left(
    c \, t^{ 1 + 1 / c } + U(x) 
  \right)
\end{equation}
(cf.\ \eqref{eq:UbarU_estimate} in the proof 
of Theorem~\ref{thm:stopped.Euler.bounded.increments}
in Subsection~\ref{sec:exponential_moments} below
and cf.\ also Lemma~\ref{l:exp.mom.abstract.one-step} in Subsection~\ref{sec:a_one_step_estimate} below).
We prove \eqref{eq:to_show} by exploiting the assumption that
$
  \forall \, x \in \R^d \colon
  \mathcal{G}^U_{ \mu, \sigma }( x ) 
  + 
  \frac{ 1 }{ 2 }
  \left\| \sigma( x )^* ( \nabla U )( x ) \right\|^2 + \bar{U}( x )
  \leq 
  \rho \, U( x )
$
(see \eqref{eq:Lyapunov.condition.ExpMoments} above)
and by applying the It\^{o} formula and the 
fundamental theorem of calculus respectively.
More formally, 
the fundamental theorem of calculus 
and the assumption that
$
  \forall \, x \in \R^d \colon
  \mathcal{G}^U_{ \mu, \sigma }( x ) 
  + 
  \frac{ 1 }{ 2 }
  \left\| \sigma( x )^* ( \nabla U )( x ) \right\|^2 + \bar{U}( x )
  \leq 
  \rho \, U( x )
$
prove that
for all 
$ s \in [0,T] $, $ h \in [0,T-s] $, $ x \in D_h $, $ t \in [0,h] $
it holds that
\begin{equation}
\label{eq:outline_first}
\begin{split}
&
  \E\!\left[ 
    \exp\!\left(
      \tfrac{
        U( Z^{ s, x, h }_t )
      }{
        e^{ \rho t }
      }
      +
      \int_0^t
      \tfrac{
        \mathbbm{1}_{ D_h }\!( x ) 
        \,
        \bar{U}( Z^{ s, x, h }_r )
      }{
        e^{ \rho r }
      }
      \, dr
    \right)
  \right]
  -
  e^{ U( x ) }
=
  \E\!\left[ 
    \exp\!\left(
      \tfrac{
        U( Z^{ s, x, h }_t )
      }{
        e^{ \rho t }
      }
      +
      \int_0^t
      \tfrac{
        \bar{U}( Z^{ s, x, h }_r )
      }{
        e^{ \rho r }
      }
      \, dr
    \right)
  \right]
  -
  e^{ U( x ) }
\\ & \leq  
  \smallint\limits_0^t
  \Big|
  \tfrac{ \partial }{ \partial u }
  \E\!\left[ 
    \exp\!\left(
      \tfrac{
        U( Z^{ s, x, h }_u )
      }{
        e^{ \rho u }
      }
      +
      \int_0^u
      \tfrac{
        \bar{U}( Z^{ s, x, h }_r )
      }{
        e^{ \rho r }
      }
      \, dr
    \right)
  \right]
  -
  \left(
  \mathcal{G}^U_{ \mu, \sigma }( x ) 
  + 
  \tfrac{ 1 }{ 2 }
  \left\| \sigma( x )^* ( \nabla U )( x ) \right\|^2 + \bar{U}( x )
  -
  \rho \, U( x )
  \right)
  \Big|
  \, 
  du
  .
\end{split}
\end{equation}
In the next step we apply It\^{o}'s formula,
we exploit the fact that
$
  \forall \, h \in (0, 1 ] , u \in (0,h] \colon
  D_h 
  = 
  \big\{ 
    x \in \R^d 
    \colon 
    \left\| x \right\| 
    \leq 
    \exp\!\big(
      | \ln( h ) |^{ 1 / 2 }
    \big)
  \big\}
  \subseteq
  \big\{ 
    x \in \R^d 
    \colon 
    \left\| x \right\| 
    \leq 
    \exp\!\big(
      | \ln( u ) |^{ 1 / 2 }
    \big)
  \big\}
$,
we exploit \eqref{eq:Phi_def},
and we use a number of elementary
estimates (see the proof 
of Lemma~\ref{l:exp.mom.abstract.one-step} in Subsection~\ref{sec:a_one_step_estimate} for details)
to obtain that
there exists a real number 
$ c \in (0,\infty) $
such that
for all 
$ s \in [0,T] $,
$ h \in [ 0, \min\{ T - s , 1 \} ] $,
$ x \in D_h $,
$ u \in (0,h] $
it holds that
\begin{equation}
\label{eq:outline_second}
\begin{split}
&
  \Big|
  \tfrac{ \partial }{ \partial u }
  \E\!\left[ 
    \exp\!\left(
      \tfrac{
        U( Z^{ s, x, h }_u )
      }{
        e^{ \rho u }
      }
      +
      \int_0^u
      \tfrac{
        \bar{U}( Z^{ s, x, h }_r )
      }{
        e^{ \rho r }
      }
      \, dr
    \right)
  \right]
  -
  \left(
  \mathcal{G}^U_{ \mu, \sigma }( x ) 
  + 
  \tfrac{ 1 }{ 2 }
  \left\| \sigma( x )^* ( \nabla U )( x ) \right\|^2 + \bar{U}( x )
  -
  \rho \, U( x )
  \right)
  \Big|
\\ & \leq
    c \, u^{ 1 / c }  
    \,
    e^{ U( x ) }.
\end{split}
\end{equation}
Putting \eqref{eq:outline_second}
into \eqref{eq:outline_first}
then results in \eqref{eq:to_show}.
Using \eqref{eq:to_show} iteratively, in turn, 
will allow us to prove that there exists a real number 
$ c \in (0,\infty) $
such that
for all $ N \in \N \cap [ T , \infty ) $
it holds that
\begin{equation}  
\label{eq:sketch_iteratively}
\begin{split}
    \sup_{ t \in [0,T] }
    \E\!\left[
      \exp\!\left( 
        \tfrac{
          U( Y^N_t ) 
        }{
          e^{ \rho t } 
        }
        +
        \smallint_0^t
          \tfrac{
            \1_{
              D_{ T / N } 
            }(
              Y^N_{
                \lfloor r \rfloor_{ \theta }
              }
            )
            \, 
            \bar{U}( Y^N_r ) 
          }{
            e^{ \rho r } 
          }
          \, dr
        \right)
      \right]
&   \leq
      \exp\!\left(
        c N \left[ \tfrac{ T }{ N } \right]^{ 1 + 1 / c } 
      \right)
      \,
      \E\!\left[
        e^{
          U( Y^N_0 ) 
        }
      \right]
\\ &   
  =
      \exp\!\left(
        \frac{ c T^{ 1 + 1 / c } }{ N^{ 1 / c } }
      \right)
      \,
      \E\!\left[
        e^{
          U( X_0 ) 
        }
      \right]
\end{split}
\end{equation}
(see Corollary~\ref{Cor:exp.mom.abstract}
and 
\eqref{eq:first.exp.mom.bound}
and 
\eqref{eq:first.exp.mom.bound.before.final} 
in the proof of Theorem~\ref{thm:stopped.Euler.bounded.increments}
for details).
Clearly, \eqref{eq:sketch_iteratively}
implies 
\begin{equation}
\label{eq:outline_limsup}
  \limsup_{ N \to \infty }
  \sup_{ t \in [0,T] }
  \E\!\left[
       \exp\!\left(
         \tfrac{
           U( Y^N_t )
         }{
           e^{ \rho t }
         }
         +
         \smallint_0^{ t \wedge \tau_N }
           \tfrac{
             \bar{U}( Y^N_s )
           }{
             e^{ \rho s }
           }
         \, ds
       \right)
     \right]
  \leq
  \E\!\left[
        e^{
          U( X_0 )
        }
  \right]
  < \infty
\end{equation}
and 
\begin{equation}
\label{eq:outline_sup}
  \sup_{
    N \in \N \cap [T,\infty)
  }
  \sup_{ t \in [0,T] }
  \E\!\left[
       \exp\!\big(
           e^{ - \rho t }
           \,
           U( Y^N_t )
         +
         \smallint_0^{ t \wedge \tau_N }
             e^{ - \rho s }
             \,
             \bar{U}( Y^N_s )
         \, ds
       \big)
     \right]
  <
  \infty
  .
\end{equation}
Display~\eqref{eq:outline_limsup}
shows display~\eqref{eq:thm.inequality}
in Theorem~\ref{thm:main.result}
and the inequality
$
  \sup_{
    N \in \N
  }
  \sup_{ t \in [0,T] }
  \E\big[
       \exp\!\big(
           e^{ - \rho t }
           \,
           U( Y^N_t )
         +
         \smallint_0^{ t \wedge \tau_N }
             e^{ - \rho s }
             \,
             \bar{U}( Y^N_s )
         \, ds
       \big)
     \big]
  <
  \infty
$
in Theorem~\ref{thm:main.result}
follows immediately from 
estimate~\eqref{eq:outline_sup}
(cf.\ \eqref{eq:bound_Y_smallN} in the proof 
of Corollary~\ref{cor:for_examples} below).
Moreover, extensions of the notions and the results 
in Sections~3.2--3.4
in \cite{HutzenthalerJentzen2014Memoires}
will allow us to prove that
for all
$ r \in (0,\infty) $ 
it holds that
$
  \lim_{ N \to \infty }
  \big(
  \sup_{ t \in [0,T] }
  \E\big[
    \| X_t - Y^N_t \|^r
  \big]
  \big)
  = 0
$
(see Section~\ref{sec:consistency} below for details).
This completes this sketch of 
the proof of Theorem~\ref{thm:main.result}.

\subsection{Notation}

Throughout this article the following notation is used.
By $ \N = \{ 1, 2, 3, \dots \} $ we denote
the set of natural numbers and by
$ \N_0 = \{ 0, 1, 2, \dots \} = \N \cup \{ 0 \} $
we denote the union of the set of natural numbers and zero.
Additionally, for a natural number $ d \in \N $
and a set $ D \subseteq \R^d $ we denote by
$
  \mathring{D} 
$
the interior of $ D $, that is, the set 
given by
\begin{equation}
  \mathring{D} =
  \left\{ 
    x \in D \colon 
    \left(
      \exists \, \varepsilon \in (0,\infty) \colon
      \left\{ 
        y \in \R^d \colon 
        \left\| x - y \right\| < \varepsilon 
      \right\}
      \subseteq D
    \right)
  \right\}
  .
\end{equation}
Furthermore, let 
$
  \left\| \cdot \right\| \colon
  \left(
    \cup_{ n \in \N }
    \R^n
  \right)
  \to 
  [0,\infty)
$
and
$
  \langle \cdot, \cdot \rangle \colon
  \left(
    \cup_{ n \in \N }
    (
    \R^n \times \R^n
    )
  \right)
  \to
  [0,\infty)
$
be the functions with the property that
for all $ n \in \N $, $ v = ( v_1, \dots, v_n ) $, $ w = ( w_1, \dots, w_n ) \in \R^n $
it holds that
$
  \left\| v \right\|
  =
  \big[ 
    \left| v_1 \right|^2 +
    \ldots +
    \left| v_n \right|^2
  \big]^{ 1 / 2 }
$
and
$
  \langle
    v, w
  \rangle
    =
    v_1 w_1 + 
    \ldots +
    v_n w_n
$.
Moreover, for natural numbers
$ d, m \in \N $
and a
$ d \times m $-matrix
$ A \in \R^{ d \times m } $
we denote by
$ A^{*} \in \R^{ m \times d } $
the transpose of the matrix $ A $
and by $ \| A \|_{ \HS( \R^m, \R^d) } $ 
the Hilbert-Schmidt norm of the matrix $ A $.
Furthermore, for natural numbers 
$ k, d, m \in \N $
we denote by
$
  L^{ (k) }( \R^d, \R^m ) 
$
the set of all $ k $-linear mappings from 
$ ( \R^d )^k = \R^d \times \R^d \times \dots \times \R^d $
to 
$ \R^m $
and we denote by
$
  \left\| \cdot \right\|_{ L^{ (k) }( \R^d, \R^m ) } 
  \colon
  L^{ (k) }( \R^d, \R^m )
  \to 
  [0,\infty)
$
the mapping with the property that for all 
$
  A \colon \R^d \times \R^d \times \dots \times \R^d \to \R^m
  \in L^{ (k) }( \R^d, \R^m )
$
it holds that
\begin{equation}
  \left\| A \right\|_{ L^{ (k) }( \R^d, \R^m ) }
=
  \sup_{
    v_1, v_2, \dots, v_k \in \R^d \backslash \{ 0 \}
  }
  \left(
  \frac{
    \left\| A( v_1, v_2, \dots, v_k ) \right\|
  }{
    \left\| v_1 \right\|
    \cdot 
    \left\| v_2 \right\|
    \cdot 
    \ldots 
    \cdot 
    \left\| v_k \right\|
  }
  \right)
  \in [0,\infty)
  .
\end{equation}
Additionally, for natural numbers
$ k, d, m \in \N $,
an open set $ U \subseteq \R^d $,
and a $ k $-times continuously differentiable 
function $ f \in C^k( U, \R^m ) $
we denote by 
$
  f^{ (k) } \colon U \to L^{ (k) }( \R^d, \R^m )
$
the $ k $-th derivative of $ f $.
Observe that for all 
$ k, d, m \in \N $,
$ v^{ (1) } = ( v^{ (1) }_1, \dots, v^{ (1) }_d ) $, 
$ \dots $, 
$ v^{ (k) } = ( v^{ (k) }_1, \dots, v^{ (k) }_d ) \in \R^d $,
all open sets $ U \subseteq \R^d $,
all $ k $-times continuously differentiable 
functions $ f \in C^k( U, \R^m ) $,
and all $ x = ( x_1, \dots, x_d ) \in U $
it holds that
\begin{equation}
  f^{ (k) }( x )( v_1, \dots, v_k )
  =
  \sum_{ l_1, \dots, l_k = 1 }^d
  \big(
    \tfrac{ \partial^k }{
      \partial x_{ l_1 } \dots \partial x_{ l_k }
    }
    f
  \big)( x )
  \cdot v^{ (1) }_{ l_1 } \cdot v^{ (2) }_{ l_2 } \cdot \ldots \cdot v^{ (k) }_{ l_k }
  .
\end{equation}
Moreover, for sets $ A $ and $ B $ 
we denote by $ \mathbbm{M}( A , B ) $ the 
set of all mappings from $ A $ to $ B $.
In addition,
for natural numbers
$ d, m \in \mathbb{N} $
and arbitrary functions
$
  \mu
  \colon \mathbb{R}^d
  \rightarrow \mathbb{R}^d
$
and
$
  \sigma
  \colon \mathbb{R}^d
  \rightarrow \mathbb{R}^{ d \times m }
$
we denote
by
$
  \mathcal{G}_{ \mu, \sigma }
  \colon
  C^2( \mathbb{R}^d, \mathbb{R} )
  \rightarrow
  \mathbbm{M}( \R^d, \R )
$
the formal generator associated to $ \mu $ and $ \sigma $,
that is,
we denote by
$
  \mathcal{G}_{ \mu, \sigma }
  \colon
  C^2( \mathbb{R}^d, \mathbb{R} )
  \rightarrow
  \mathbbm{M}( \R^d, \R )
$
the mapping with the property that
for all 
$
  \varphi \in 
  C^2( \mathbb{R}^d, \mathbb{R} )
$,
$ x \in \mathbb{R}^d $
it holds that
\begin{equation}
\label{eq:generator}
\begin{split}
  ( \mathcal{G}_{ \mu, \sigma } \varphi )
  (x)
  =
  \left<
    \mu(x),
    (\nabla \varphi)(x)
  \right>
+
    \tfrac{ 1 }{ 2 }
  \tr\!\big(
    \sigma(x)
    \sigma(x)^{ * }
    (\text{Hess } \varphi)(x)
  \big)
  .
\end{split}
\end{equation}
Furthermore,
for a natural number $ d \in \mathbb{N} $
and a Borel
measurable set
$ A \in \mathcal{B}( \mathbb{R}^d ) $
we denote by
$
  \lambda_{ A }
  \colon
  \mathcal{B}( A )
  \rightarrow [0,\infty]
$
the Lebesgue-Borel
measure on
$ A \subseteq \mathbb{R}^d $.
Moreover, for measurable spaces
$ ( A, \mathcal{A} ) $
and
$ ( B, \mathcal{B} ) $
we denote by
$ \mathcal{M}( \mathcal{A} , \mathcal{B} ) $
the set of all $ \mathcal{A} $/$ \mathcal{B} $-measurable mappings.
In addition,
for numbers
$ n, d \in \mathbb{N} $,
$
  p, c \in (0,\infty)
$,
a set $ B \subseteq \R $,
and an open and convex set $ A \subseteq \R^d $
we denote by
$ C^n_{ p, c }( A, B ) $
(cf.\ (1.12) in \cite{HutzenthalerJentzen2014Memoires})
the set given by
\begin{equation}
\begin{split}
  C^n_{p,c}( A, B )
 & 
 =
  \left\{
    f \in C^{ n - 1 }( A, B )
    \colon
    \begin{array}{c}
      \forall \, x, y \in A, i \in \N_0 \cap [0,n) \colon
      \| f^{ (i) }( x ) - f^{ (i) }( y ) \|_{
        L^{ (i) }( \R^d, \R )
      }
      \leq
    \\
      c \left\| x - y \right\|
      \big[
        1 +
        \sup_{ r \in [0,1] }
        | f( r x + (1-r) y ) |
      \big]^{ 
        \max\{ 1 - \nicefrac{ ( i + 1 ) }{ p } , 0 \}
      }
    \end{array}
  \right\}.
\end{split}
\end{equation}
Next let
$
  ( \cdot ) \vee ( \cdot )
  \colon \R^2 \to \R
$
and
$
  ( \cdot ) \wedge ( \cdot )
  \colon \R^2 \to \R
$
be the mappings with the property that
for all $ x, y \in \R $
it holds that
$ x \vee y = \max\{ x, y \} $ and
$ x \wedge y = \min\{ x , y \} $.
In addition, for a set $ \Omega $
we denote by $ 2^{ \Omega } $ the power set of 
$ \Omega $ (the set of all subsets of $ \Omega $)
and for a set $ \Omega $ we denote by
$
  \#_{ \Omega } \colon 2^{ \Omega } \to [0,\infty] 
$
the counting measure on $ \Omega $.
Furthermore, for a real number
$ T \in [0,\infty) $ we denote by
$ \mathcal{P}_T $
the set given by
$
  \mathcal{P}_T
  =
  \{
    A \subseteq [0,T] \colon
    \#_{ \R }( A ) < \infty
    \text{ and }
    \{ 0, T \} \subseteq A
  \}
$
(the set of all partitions of the interval $ [0,T] $).
Moreover, 
let 
$ 
  \ell \colon 2^{ \R } \to (-\infty,\infty]
$
be the mapping with the property that for all 
$ \theta \subseteq \R $
it holds that
$
  \ell( \theta ) = \#_{ \R }( \theta ) - 1
$.
In addition,
for a real number $ T \in [0,\infty) $
we denote by
$
  \left| \cdot \right|_T \colon \mathcal{P}_T \to [0,T]
$
the mapping with the property that
for all $ \theta \in \mathcal{P}_T $
it holds that
\begin{equation}
  \left| \theta \right|_T
  =
  \max\!\Big(
  \{ 0 \} \cup
  \Big\{
    x \in (0,\infty)
    \colon
    \big(
      \exists \, a, b \in \theta
      \colon
      \big[
        x = b - a
      \text{ and }
        \theta \cap ( a, b ) = \emptyset
      \big]
    \big)
  \Big\}
  \Big)
  \in [0,T]
  .
\end{equation}
Note for every $ T \in [0,\infty) $
and every $ \theta \in \mathcal{P}_T $
that $ \left| \theta \right|_T \in [0,T] $ is the maximum
step size of the partition $ \theta $.
Finally,
let
$ 
  \lfloor \cdot \rfloor_{ \theta } 
  \colon [0,\infty) \to [0,\infty) 
$,
$ \theta \in [ (0,\infty) \cup ( \cup_{ T \in [0,\infty) } \mathcal{P}_T ) ]
$,
and 
$ 
  \llcorner \cdot \lrcorner_{ \theta } \colon 
  [0,\infty) 
  \to 
  [0,\infty)
$,
$ \theta \in ( \cup_{ T \in [0,\infty) } \mathcal{P}_T ) $,
be the mappings which satisfy
for all 
$ 
  \theta \in 
  (
    \cup_{ T \in [0,\infty) }
    \mathcal{P}_T 
  )
$, 
$ h, t \in (0,\infty) $
that
$
  \lfloor t \rfloor_h
=
  \max\!\left(
    \{ 0, h, 2h, 3h, \dots \} \cap [0,t]
  \right)
$,
$
  \lfloor t \rfloor_{ \theta }
=
  \max\!\left(
    \theta \cap [0,t]
  \right)
$,
$
  \llcorner 
    t 
  \lrcorner_{ \theta }
=
  \max\!\left(
    \theta \cap [0,t)
  \right)
$,
and 
$
  \lfloor
    0 
  \rfloor_h
  =
  \lfloor
    0 
  \rfloor_{ \theta }
  =
  \llcorner 
    0 
  \lrcorner_{ \theta }
  = 0
$.

%
\section{Exponential moments for numerical approximation processes}
\label{sec:exponential}
%

In this section we establish exponential integrability properties for 
certain numerical approximation processes of stochastic differential equations.
In Subsection~\ref{sec:from_one_step_to_aprior}
we show how exponential moment bounds 
for the numerical approximation processes
can be derived from 
suitable one-step estimates.
We first prove a general Lyapunov-type estimate in 
Proposition~\ref{prop:stability2} 
in Subsection~\ref{sec:from_one_step_to_aprior}.
Thereafter, we establish appropriate Lyapunov-type estimates
for exponentially growing Lyapunov-type functions 
in Lemma~\ref{l:stability3} and Corollary~\ref{Cor:exp.mom.abstract} 
in Subsection~\ref{sec:from_one_step_to_aprior}.
The results in Subsection~\ref{sec:from_one_step_to_aprior}
are elementary extensions of known results
in the literature 
(cf., e.g., 
Section~2.1.1 in \cite{HutzenthalerJentzen2014Memoires}
and
Section~3.1 in Schurz~\cite{Schurz2005}).
Proposition~\ref{prop:stability2},
Lemma~\ref{l:stability3},
and Corollary~\ref{Cor:exp.mom.abstract}
all assume suitable one-step estimates 
for the considered numerical approximation processes;
see \eqref{eq:semi.V.stable2} in the case of Proposition~\ref{prop:stability2},
see \eqref{l:stability3} in the case of Lemma~\ref{l:stability3},
and see \eqref{eq:exp.mom.abstract.assumption} in the case of Corollary~\ref{Cor:exp.mom.abstract}.
The purpose of Subsection~\ref{sec:a_one_step_estimate} is to prove that an appropriate class 
of stopped numerical approximation schemes fulfills the one-step estimate~\eqref{eq:exp.mom.abstract.assumption}
(see Lemma~\ref{l:exp.mom.abstract.one-step} in Subsection~\ref{sec:a_one_step_estimate}, 
which is the key result of this article)
so that Corollary~\ref{Cor:exp.mom.abstract} in Subsection~\ref{sec:from_one_step_to_aprior}
can be applied.
In Subsection~\ref{sec:exponential_moments}
we then combine 
Corollary~\ref{Cor:exp.mom.abstract}
in Subsection~\ref{sec:from_one_step_to_aprior}
and 
Lemma~\ref{l:exp.mom.abstract.one-step} in Subsection~\ref{sec:a_one_step_estimate}
to finally obtain exponential integrability properties for a class of stopped
increment-tamed Euler-Maruyama schemes;
see Theorem~\ref{thm:stopped.Euler.bounded.increments}
and Corollary~\ref{cor:stopped.Euler.bounded.increments}
in Subsection~\ref{sec:exponential_moments}.

\subsection{From one-step estimates to exponential moments}
\label{sec:from_one_step_to_aprior}

The following proposition, Proposition \ref{prop:basis},
is an extended and generalized version of Corollary 2.2 in \cite{HutzenthalerJentzen2014Memoires}.
The proof of Proposition~\ref{prop:basis} is similar to the proof of
Proposition~2.1 in \cite{HutzenthalerJentzen2014Memoires}.

\begin{prop}
\label{prop:basis}
\label{prop:stability2}
Let
$
  \left(
    \Omega, \mathcal{F},
    \mathbb{P}
  \right)
$
be a probability
space,
let
$
  \left( E, \mathcal{E} \right)
$
be a measurable space, 
let 
$ V \in \mathcal{M}( \mathcal{E}, \mathcal{B}( [0, \infty] ) ) $, 
let 
$
  Z \colon \mathbb{N}_0 \times
  \Omega
  \rightarrow E
$
be a stochastic process,
let
$ \gamma_n \!\in\! [0, \infty)$, $n \!\in\! \mathbb{N}_0 $,
$\delta_n \!\in\! (0, \infty] $, $n \in \mathbb{N}_0$,
and $\Omega_n \in \mathcal{F}$, $n \in \mathbb{N}_0$,
be sequences 
which satisfy
for all $ n \in \N_0 $
that
$ 
  \Omega_0 = \Omega
$,
that
$
  \Omega_n \backslash \Omega_{n+1} \subseteq \{ V(Z_n) > \delta_n \}
$,
and that
\begin{equation}
\label{eq:semi.V.stable2}
   \mathbb{E} \Big[\mathbbm{1}_{
  \Omega_{n+1}
  }
    V( Z_{ n + 1 } )\Big]
\leq
  \gamma_n
  \cdot
  \mathbb{E} \Big[\mathbbm{1}_{
  \Omega_{n}
  }
    V( Z_{ n } )\Big]
  .
\end{equation}
Then it holds for all
$ n \in \mathbb{N}_0 $,
$ p \in [1,\infty] $,
$ 
  \bar{V} \in \mathcal{M}( \mathcal{E} , \mathcal{B}( [0,\infty] ) ) 
$
with $ \bar{V} \leq V $
that
\begin{equation}
\label{eq:EV.inequality}
    \mathbb{E}\big[
      \mathbbm{1}_{ \Omega_n }
      V(
        {Z}_n
      )
    \big]
   \leq
     \!\left( \prod_{k = 0}^{n-1} \gamma_k \right)
    \cdot
    \mathbb{E}\big[
      V(
        {Z}_0
      )
    \big], \quad
  \;
  \mathbb{P}\!\left[
    (
      \Omega_n
    )^{ c }
  \right]
  \leq
  \left(
  \sum_{ k = 0 }^{ n - 1 }
  \frac{
      \prod_{l=0}^{k-1} \gamma_l
  }{
    \delta_k
  }
  \right)
  \mathbb{E}\big[
    V(
      {Z}_0
    )
  \big] ,
\end{equation}
\begin{equation}
\label{eq:EVbar}
\begin{split}
  \mathbb{E}\big[
    \bar{V}({Z}_n)
  \big]
&
  \leq \!
      \left( \prod_{k = 0}^{n-1} \gamma_k \right)
    \cdot
    \mathbb{E}\big[
      V(
        {Z}_0
      )
    \big]
  +
  \|
    \bar{V}( Z_n )
  \|_{
    L^p( \Omega; \mathbb{R} )
  }
  \bigg[
\bigg(
  \sum_{ k = 0 }^{ n - 1 }
  \frac{
      \prod_{l=0}^{k-1} \gamma_l
  }{
    \delta_k
  }
  \bigg)
  \mathbb{E}\big[
    V(
      {Z}_0
    )
  \big]
  \bigg]^{ \! ( 1 - \tfrac{1}{p} ) }.
\end{split}
\end{equation}
\end{prop}

\begin{proof}[Proof
of Proposition~\ref{prop:stability2}]
The first inequality in \eqref{eq:EV.inequality} 
is an immediate consequence of \eqref{eq:semi.V.stable2}. To arrive at the second estimate in \eqref{eq:EV.inequality}, note first for all
$
  n \in \mathbb{N}
$
that
\begin{equation}
\label{eq:sets}
  \left(
    \Omega_n
  \right)^c
  =
  \big(
    \Omega_{
      n-1
    }
    \backslash
    \Omega_n
  \big)
  \uplus
  \big(
  \left(
    \Omega_{
      n-1
    }
  \right)^{ c }
  \backslash \Omega_n
  \big)
  \subseteq
  \big(
    \Omega_{
      n-1
    }
    \backslash
    \Omega_{n}
  \big)
  \cup
  \big(
  \left(
    \Omega_{
      n-1
    }
  \right)^{ c }
  \big).
\end{equation}
Iterating inclusion~\eqref{eq:sets} and
using
$ \Omega_0 = \Omega $
shows for all
$
  n \in \mathbb{N}_0
$ that
\begin{equation}
\begin{split}
  \left(
    \Omega_n
  \right)^c
& \subseteq
  \bigg(
  \bigcup_{ k = 0 }^{ n-1 }
  \left(
    \Omega_{
      k
    }
    \!
    \backslash
    \Omega_{
      { k + 1 }
    }
  \right)
  \bigg)
  \;\bigcup\;
  \big(
  \left(
    \Omega_{
      0
    }
  \right)^{ c }
  \big)
  =
  \bigcup_{ k = 0 }^{ n-1 }
  \left(
    \Omega_{
      k
    }
    \!
    \backslash
    \Omega_{
      { k + 1 }
    }
  \right) =  \bigcup_{ k = 0 }^{ n-1 }
  \Big( \Omega_{k} \cap \big(
    \Omega_{
      k
    }
    \backslash
    \Omega_{
      { k + 1 }
    }\big)
  \Big)
 \\ & \subseteq
  \bigcup_{ k = 0 }^{ n-1 }
  \left(
    \Omega_{
      k
    }
    \cap
    \left\{
      V(
        {Z}_{
          k
        }
      )
      >
      \delta_k
    \right\}
  \right)
  =
  \bigcup_{ k = 0 }^{ n-1 }
    \left\{
      \mathbbm{1}_{
        \Omega_{
          k
        }
      }
      V(
        {Z}_{
          k
        }
      )
      >
      \delta_k
    \right\}.
\end{split}
\end{equation}
Additivity of
the probability measure
$ \mathbb{P} $,
Markov's inequality
and
the first inequality in \eqref{eq:EV.inequality}
therefore
imply for all
$
  n \in \mathbb{N}_0
$
that
\begin{equation}
\begin{split}
  \mathbb{P}\big[
  (
    \Omega_n
  )^c
  \big]
& \leq
  \sum_{ k = 0 }^{ n-1 }
  \mathbb{P}\Big[ \,
    \mathbbm{1}_{
      \Omega_{
        k
      }
    }
      V(
        {Z}_{
          k
        }
      )
      >
      \delta_k
  \Big]
 \leq
  \sum_{ k = 0 }^{ n-1 }
  \!\left[
    \frac{
      \mathbb{E}\big[
      \mathbbm{1}_{
        \Omega_{
          k
        }
      }
        V(
          Z_{
            k
          }
        )
      \big]
    }{
      \delta_k
    }
  \right]
\\ & \leq
  \sum_{ k = 0 }^{ n - 1 }
  \!\left[
  \frac{
      \left(\prod_{l=0}^{k-1} \gamma_l\right)\!\cdot
    \mathbb{E}[
      V(
        {Z}_0
      )
    ]
  }{
    \delta_k
  }
  \right].
\end{split}
\end{equation}
This is the second inequality in \eqref{eq:EV.inequality}
and the proof of
\eqref{eq:EV.inequality}
is thus completed.
Next observe that H\"{o}lder's
inequality implies for all
$
  \tilde{\Omega} \in \mathcal{F}
$,
$ p \in [1,\infty] $
and all
$ \mathcal{F} $/$
  \mathcal{B}( [0,\infty] )
$-measurable
mappings
$
  X \colon \Omega
  \rightarrow [0,\infty]
$
that
\begin{equation}
\label{eq:simpleest}
  \mathbb{E}[
    X
  ]
\leq
  \mathbb{E}\!\left[
    \mathbbm{1}_{ \tilde{ \Omega } }
    X
  \right]
  +
  \big(
    \mathbb{P}\big[
      ( \tilde{\Omega} )^c
    \big]
  \big)^{
    \! ( 1 - 1/p )
  }
  \|
    X
  \|_{
    L^{ p }( \Omega; \mathbb{R} )
  }.
\end{equation}
Combining \eqref{eq:EV.inequality}
and \eqref{eq:simpleest}
finally shows that
for all
$ n \in \mathbb{N}_0 $,
$ p \in [1,\infty] $
and all
$ \mathcal{E} $/$ \mathcal{B}( [0,\infty] )
$-measurable functions
$ \bar{V} \colon E \rightarrow [0,\infty] $
with $ \forall \, x \in E \colon \bar{V}(x) \leq V(x) $
it holds that
\begin{equation}  \begin{split}
  \mathbb{E}\big[
    \bar{V}({Z}_n)
  \big]
  & \leq
  \mathbb{E}\big[
    \mathbbm{1}_{\Omega_n}
    V( { Z }_{ n } )
  \big]
  +
  \|
    \bar{V}( Z_n )
  \|_{
    L^p( \Omega; \mathbb{R} )
  }
  \left(
    \mathbb{P}\big[
      ( \Omega_n )^c
    \big]
  \right)^{
    ( 1 - 1/p )
  }
\\
  &\leq
   \left( \prod_{k = 0}^{n-1} \gamma_k \right)\!
    \cdot
    \mathbb{E}\big[
      V(
        {Z}_0
      )
    \big]
  +
  \|
    \bar{V}( Z_n )
  \|_{
    L^p( \Omega; \mathbb{R} )
  }
  \bigg[
\bigg(
  \sum_{ k = 0 }^{ n - 1 }
  \frac{
      \prod_{l=0}^{k-1} \gamma_l
  }{
    \delta_k
  }
  \bigg)
  \mathbb{E}\big[
    V(
      {Z}_0
    )
  \big]
  \bigg]^{ \! ( 1 - \tfrac{1}{p} ) }
  .
\end{split}
\end{equation}
The proof of
Proposition~\ref{prop:stability2}
is thus completed.
\end{proof}

The next elementary lemma
(Lemma~\ref{l:stability3})
establishes an a priori bound based on
a specific class of path dependent
Lyapunov-type functions
(see \eqref{eq:NR} below for details
and cf., e.g., also Section~3.1 in Schurz~\cite{Schurz2005}).
For completeness the proof
of Lemma~\ref{l:stability3}
is given below.

\begin{lemma}
\label{l:stability3}
Let  $ ( \Omega,\mathcal{F},\P)$ be a probability space,
let $ ( E, \mathcal{E})$ be a measurable space,
let
$ T, \rho \in [0, \infty) $,
$ \theta \in \mathcal{P}_T $,
$ 
  c \in \mathbb{R}
$,  
$
  U, \bar{U} 
  \in \mathcal{M}( \mathcal{E}, \mathcal{B}( \R ) )
$,
$ A \in \mathcal{E} $,
and let
$ Y \colon [0, T] \times \Omega \rightarrow E $
be a product measurable stochastic process 
which satisfies
that for all $ t \in [0,T] $
it holds $ \P $-a.s.\ that
$
  \smallint_0^t
          \1_{A}(
            Y_{ \lfloor r \rfloor_{ \theta } }
          )
          \,
          | \bar{U}(Y_r) |
          \, dr < \infty
$
and
\begin{equation}
    \E\!\left[
      \exp\!\left(
        - c \, t
        +
        \tfrac{
          U( Y_t )
        }{
          e^{ \rho t }
        }
        \,
        +
        \smallint_0^t
        \tfrac{
          \1_{ A }(
            Y_{ \lfloor r \rfloor_{ \theta } }
          )
          \,
          \bar{U}( Y_r )
        }{
          e^{ \rho r }
        }
          \, dr
                 \right)
         \Big| \,
         ( Y_r )_{ 
           r \in 
           \left[ 
             0, 
             \llcorner t \lrcorner_{ \theta } 
           \right]
         }
      \right]
    \leq
     \exp\!\left(
       - c \,
       \llcorner t \lrcorner_{ \theta } 
       +
       \tfrac{
         U( 
           Y_{ 
             \llcorner t \lrcorner_{ \theta } 
           } 
         )
       }{
         e^{ 
           \rho 
           \llcorner t \lrcorner_{ \theta } 
         }
       }
       +
       \smallint_0^{ 
         \llcorner t \lrcorner_{ \theta } 
       }
       \tfrac{
         \1_{ A }(
           Y_{
             \lfloor r \rfloor_{ \theta }
           }
         )
         \,
         \bar{U}( Y_r )
       }{
         e^{ \rho r }
       }
       \, dr
     \right)
   .
\label{eq:NR}
\end{equation}
Then it holds for all $t \in [0, T]$ that
\begin{equation}\label{eq:stability.assertion}
    \E\!\left[
      \exp\!\left(
        \tfrac{
          U( Y_t )
        }{
          e^{ \rho t }
        }
        +
        \smallint_0^t
        \tfrac{
          \mathbbm{1}_{A}(
            Y_{ \lfloor r \rfloor_{ \theta } }
          )
          \,
          \bar{U}(Y_r)
        }{
          e^{ \rho r }
        }
        \, dr
      \right)
    \right]
    \leq
      e^{ct}\, \E \!\left[e^{U (Y_0)}\right]
   .
\end{equation}

\end{lemma}

\begin{proof}[Proof of Lemma~\ref{l:stability3}]
 Assumption \eqref{eq:NR} implies for all $t \in [0, T]$ that
 \begin{equation} \begin{split}
 &
 \E\!\left[
      \exp\!\left( - ct +
        \tfrac{
          U( Y_t )
        }{
          e^{ \rho t }
        }
        +
        \smallint_0^t
        \tfrac{
          \mathbbm{1}_{A}(
            Y_{ \lfloor r \rfloor_{ \theta } }
          )
          \,
          \bar{U}(Y_r)
        }{
          e^{ \rho r }
        }
        \, dr
      \right)
    \right]
 \\ &
 =
 \E \! \left[ \E \! \left[  \exp\!\left( - ct +
        \tfrac{
          U( Y_t )
        }{
          e^{ \rho t }
        }
        +
        \smallint_0^t
        \tfrac{
          \mathbbm{1}_{A}(
            Y_{ \lfloor r \rfloor_{ \theta } }
          )
          \,
          \bar{U}(Y_r)
        }{
          e^{ \rho r }
        }
        \, dr
      \right)
      \Big|
      \left( Y_s \right)_{ 
        s \in 
        \left[ 
          0, 
          \llcorner t \lrcorner_{ \theta } 
        \right]
      }
    \right]
    \right]
 \\ &
 \leq
 \E\!\left[ 
   \exp\!\left(
       - c \,
       \llcorner t \lrcorner_{ \theta } 
       +
       \tfrac{
         U(  
           Y_{ 
             \llcorner t \lrcorner_{ \theta } 
           } 
         )
       }{
         e^{ 
           \rho \,
           \llcorner t \lrcorner_{ \theta } 
         }
       }
       +
       \smallint_0^{ 
         \llcorner t \lrcorner_{ \theta } 
       }
       \tfrac{
         \1_A(
           Y_{
             \lfloor r \rfloor_{ \theta }
           }
         )
         \,
         \bar{U}( Y_r )
       }{
         e^{ \rho r }
       }
       \, dr
     \right) \right]
 \\ &
 \leq \ldots \leq
  \E\!\left[
    \exp\!\left(
      \tfrac{ U( Y_0 )
      }{
        e^{ \rho \cdot 0 }
      }
    \right)
  \right]
 = \E\! \left[ e^{ U( Y_0 ) }
  \right] .
 \end{split} \end{equation}
 This completes the proof of Lemma~\ref{l:stability3}.
\end{proof}

The next corollary, Corollary~\ref{Cor:exp.mom.abstract},
specialises Lemma~\ref{l:stability3} to the
case where the product measurable stochastic process
appearing in \eqref{eq:NR} and \eqref{eq:stability.assertion} is an appropriate
one-step approximation
process for an SDE driven by a standard
Brownian motion;
see \eqref{eq:process_dynamic} below for details.

\begin{corollary} 
\label{Cor:exp.mom.abstract}
  Let 
  $
    T, \rho, c
    \in [0,\infty)
  $,
  $ \theta \in \mathcal{P}_T $,
  $ d, m \in \mathbb{N} $,
  $
    \Phi \in 
    \mathcal{M}( 
      \mathcal{B}( \mathbb{R}^d \times [0, T] \times \mathbb{R}^m ), 
      \mathcal{B}( \mathbb{R}^d ) 
    ) 
  $,
  $
    A \in \mathcal{B}( \R^d )
  $,
  $ 
    U \in \mathcal{M}( \mathcal{B}( \mathbb{R}^d ), \mathcal{B}( [0, \infty) ) ) 
  $,
  $
    \bar{U} \in \mathcal{M}( \mathcal{B}( \mathbb{R}^d ), \mathcal{B}( \R ) )
  $, 
  let
  $
    (\Omega,\mathcal{F},\P,(\mathcal{F}_t)_{t\in[0,T]})
  $
  be a filtered probability space,
  let $W\colon[0,T]\times\Omega\to\R^m$ be a standard
  $(\mathcal{F}_t)_{t\in[0,T]}$-Brownian motion
  with continuous sample paths, let $Y\colon[0,T]\times\Omega\to\R^d$ be an 
  $ ( \mathcal{F}_t )_{ t \in [0,T] } $-adapted
  stochastic process which satisfies
  for all $ t \in [0, T ] $
  that
  \begin{equation}
  \label{eq:process_dynamic}
    Y_t =
    \1_{ \R^d \backslash A }( 
      Y_{ \llcorner t \lrcorner_{ \theta } }
    )
    \cdot 
    Y_{ \llcorner t \lrcorner_{ \theta } }
    +
    \1_{A}(
      Y_{ \llcorner t \lrcorner_{ \theta } }
    )
    \cdot
    \Phi\!\left(
      Y_{ \llcorner t \lrcorner_{ \theta } }
      ,
      t - 
      \llcorner t \lrcorner_{ \theta } 
      ,
      W_t - 
      W_{ 
        \llcorner t \lrcorner_{ \theta } 
      }
    \right)
    ,
  \end{equation}
  assume that
  for all $ x \in A $
  it holds $ \P $-a.s.\ that
  $
     \smallint_0^T
       \1_{A}(
         Y_{ \lfloor r \rfloor_{ \theta } }
       )
       \,
       | \bar{U}(Y_r) |
     \, 
     dr 
     +
     \smallint_0^{ \left| \theta \right|_T }
       |
         \bar{U}( \Phi( x, r, W_r ) )
       |
     \, dr
     < \infty
 $,
 and assume 
 for all $ ( t, x ) \in \big( 0, | \theta |_T \big] \times A $
 that
 \begin{equation}  \begin{split} \label{eq:exp.mom.abstract.assumption}
    \E\!\left[
      \exp\!\left(
        \tfrac{
          U( \Phi( x, t, W_t ) )
        }{
          e^{ \rho t }
        }
        +
        \smallint_0^t
        \tfrac{
          \bar{U}( \Phi( x, s, W_s ) )
        }{
          e^{ \rho s }
        }
        \, ds
      \right)
    \right]
    \leq e^{c t+U(x)}
    .
\end{split}     
\end{equation}
Then it holds for all $ t \in [0,T] $ that
\begin{equation}
\label{eq:exp.mom.abstract}
    \E\!\left[
      \exp\!\left(
        \tfrac{
          U( Y_t )
        }{
          e^{ \rho t }
        }
        +
        \smallint_0^t
        \tfrac{
          \mathbbm{1}_{A}(
            Y_{ \lfloor r \rfloor_{\theta} }
          )
          \,
          \bar{U}(Y_r)
        }{
          e^{ \rho r }
        }
        \, dr
      \right)
    \right]
    \leq
      e^{ct} \, \E\!\left[e^{{U}(Y_0)}\right]
    .
\end{equation}
\end{corollary}
\begin{proof}[Proof of Corollary~\ref{Cor:exp.mom.abstract}]
We prove Corollary~\ref{Cor:exp.mom.abstract}
through an application of Lemma~\ref{l:stability3}.
  For this we observe that
  assumption~\eqref{eq:exp.mom.abstract.assumption}
  implies
  for all
  $
    (t,x) \in \big( 0 , \left|\theta \right|_T \big] \times \R^d
  $
  that
  \begin{equation}  \begin{split}
  \label{eq:exp.mom.abstract.assumption2}
    \E\!\left[
      \exp\!\left(
        \tfrac{
          \1_A(x) \, U( \Phi( x, t, W_t ) )
        }{
          e^{ \rho t }
        }
        +
        \smallint_0^t
        \tfrac{
          \1_{A}(x)
          \,
          \bar{U}( \Phi( x, s, W_s ) )
        }{
          e^{ \rho s }
        }
        \, ds
      \right)
    \right]
    \leq
    e^{
      c t
      +
      \1_A( x )
      \,
      U(x)
    }
    .
  \end{split}
  \end{equation}
  Next note that
  equation~\eqref{eq:process_dynamic},
  Jensen's
  inequality, and
  inequality~\eqref{eq:exp.mom.abstract.assumption2}
  imply that for all $ t \in [0,T] $
  it holds $ \P $-a.s.\ that
  \begin{align}
   \nonumber &
     \E\!\left[
       \exp\!\left(
         \tfrac{
           \1_A(
             Y_{ 
               \llcorner t \lrcorner_{ \theta } 
             }
           )
           \,
           U(
             Y_t
           )
         }{
           e^{ \rho t }
         }
    +
    \smallint_{ 
      \llcorner t \lrcorner_{ \theta } 
    }^t
    \tfrac{
      \1_A(
        Y_{ 
          \llcorner t \lrcorner_{ \theta } 
        }
      )
      \,
      \bar{U}(
        Y_s
      )
    }{
      e^{ \rho s }
    }
    \, ds
    \right)
    \Big| 
      \; 
      ( Y_s )_{ 
        s \in 
        [ 
          0, 
          \llcorner t \lrcorner_{ \theta } 
        ] 
      }
    \right]
  \\ & =
   \nonumber
     \E\!\left[
       \exp\!\left(
         \tfrac{
           \1_A(
             Y_{ 
               \llcorner t \lrcorner_{ \theta } 
             }
           )
           \,
           U(
             \Phi(
               Y_{ 
                 \llcorner t \lrcorner_{ \theta } 
               },
               t - 
               \llcorner t \lrcorner_{ \theta } 
               ,
               W_t 
               - 
               W_{ 
                 \llcorner t \lrcorner_{ \theta } 
               }
             )
           )
         }{
           e^{ \rho t }
         }
    +
    \smallint_{ 
      \llcorner t \lrcorner_{ \theta } 
    }^t
    \tfrac{
      \1_A(
        Y_{ 
          \llcorner t \lrcorner_{ \theta } 
        }
      )
      \,
      \bar{U}(
        \Phi( 
            Y_{ 
              \llcorner t \lrcorner_{ \theta } 
            }
          ,  
            s - 
            \llcorner t \lrcorner_{ \theta } 
          , 
            W_s - 
            W_{ 
              \llcorner t \lrcorner_{ \theta } 
            } 
        )
      )
    }{
      e^{ \rho s }
    }
    \, ds
    \right)
    \Big| \; 
      ( 
        Y_s
      )_{ 
        s \in 
        [
          0 , 
          \llcorner t \lrcorner_{ \theta } 
        ] 
      }
    \right]
  \\ & =
   \nonumber
     \E\!\Bigg[
       \left|
       \exp\!\left(
         \tfrac{
           \1_A(
             Y_{ 
               \llcorner t \lrcorner_{ \theta } 
             }
           )
           \,
           U(
             \Phi(
               Y_{ 
                 \llcorner t \lrcorner_{ \theta } 
               }
               ,
               t - 
               \llcorner t \lrcorner_{ \theta } 
               ,
               W_t - 
               W_{ 
                 \llcorner t \lrcorner_{ \theta } 
               }
             )
           )
         }{
           e^{ 
             \rho 
             ( 
               t - 
               \llcorner t \lrcorner_{ \theta } 
             ) 
           }
         }
      +
      \smallint_0^{ 
        t - 
        \llcorner t \lrcorner_{ \theta } 
      }
      \tfrac{
        \1_A(
          Y_{ 
            \llcorner t \lrcorner_{ \theta } 
          }
        )
        \,
        \bar{U}(
          \Phi( 
            Y_{ 
              \llcorner t \lrcorner_{ \theta } 
            }, s, 
            W_{ 
              \llcorner t \lrcorner_{ \theta } 
              + s 
            } 
            - 
            W_{ 
              \llcorner t \lrcorner_{ \theta } 
            } 
          )
        )
      }{
        e^{ \rho s }
      }
      \, ds
    \right)
    \right|^{
      \exp( 
        - \rho \,
        \llcorner t \lrcorner_{ \theta } 
      )
    }
\\ &
\qquad
    \Big| 
      \; 
      ( Y_s )_{ 
        s \in 
        [ 
          0, 
          \llcorner t \lrcorner_{ \theta } 
        ] 
      }
    \Bigg]
  \\ & \leq
\nonumber
     \Bigg|
     \E\!\Bigg[
       \exp\!\left(
         \tfrac{
           \1_A(
             Y_{ 
               \llcorner t \lrcorner_{ \theta } 
             }
           )
           \,
           U(
             \Phi(
               Y_{ 
                 \llcorner t \lrcorner_{ \theta } 
               }
               ,
               t - 
               \llcorner t \lrcorner_{ \theta } 
               ,
               W_t - 
               W_{ 
                 \llcorner t \lrcorner_{ \theta } 
               }
             )
           )
         }{
           e^{ 
             \rho 
             ( 
               t - 
               \llcorner t \lrcorner_{ \theta } 
             ) 
           }
         }
    +
    \smallint_0^{ 
      t - 
      \llcorner t \lrcorner_{ \theta } 
    }
    \tfrac{
      \1_A(
        Y_{ 
          \llcorner t \lrcorner_{ \theta } 
        }
      )
      \,
      \bar{U}( 
        \Phi(
          Y_{ 
            \llcorner t \lrcorner_{ \theta } 
          }, s, 
          W_{ 
            \llcorner t \lrcorner_{ \theta } 
            + s 
          } 
          - 
          W_{ 
            \llcorner t \lrcorner_{ \theta } 
          }
        ) 
      )
    }{
      e^{ \rho s }
    }
    \, ds
    \right)
\\ &
\qquad
    \Big| 
      \; 
      (
        Y_s
      )_{ 
        s \in 
        [
          0, 
          \llcorner t \lrcorner_{ \theta } 
        ] 
      }
    \Bigg]
    \Bigg|^{
      \exp( 
        - \rho \,
        \llcorner t \lrcorner_{ \theta } 
      )
    }
   \nonumber
  \\ & \leq
   \nonumber
     \left|
       e^{
         c 
         \left( 
           t - 
           \llcorner t \lrcorner_{ \theta } 
         \right)
         +
           \1_A(
             Y_{ 
               \llcorner t \lrcorner_{ \theta } 
             }
           )
           \,
           U(
             Y_{ 
               \llcorner t \lrcorner_{ \theta } 
             }
           )
       }
     \right|^{
      \exp( 
        - \rho \,
        \llcorner t \lrcorner_{ \theta } 
      )
    }
  \leq
       \exp\!\left(
         c 
         \left( 
           t 
           - 
           \llcorner t \lrcorner_{ \theta } 
         \right)
         +
         \tfrac{
           \1_A(
             Y_{ 
               \llcorner t \lrcorner_{ \theta } 
             }
           )
           \,
           U(
             Y_{ 
               \llcorner t \lrcorner_{ \theta } 
             }
           )
         }{
           e^{ 
             \rho \,
             \llcorner t \lrcorner_{ \theta } 
           }
         }
    \right)
    .
  \end{align}
Combining this with \eqref{eq:process_dynamic} shows that
for all $ t \in [0,T] $
it holds $ \P $-a.s.\ that
\begin{equation} \begin{split}
    &
    \E\!\left[
      \exp\!\left(
        - c t
        +
        \tfrac{
          U( Y_t )
        }{
          e^{ \rho t }
        }
        +
        \smallint_0^t
        \tfrac{
          \1_A( 
            Y_{ 
              \lfloor r \rfloor_{ \theta } 
            } 
          )
          \,
          \bar{U}( Y_r )
        }{
          e^{ \rho r }
        }
        \,
        dr
      \right)
      \Big|
      \,
      ( Y_r )_{ 
        r \in 
        [
          0, 
          \llcorner t \lrcorner_{ \theta } 
        ] 
      }
    \right]
  \\ &
  =
    \E\!\left[
      \exp\!\left(
        \tfrac{
          \1_A( 
            Y_{ 
             \llcorner t \lrcorner_{ \theta } 
            } 
          )
          \,
          U( Y_t )
        }{
          e^{ \rho t }
        }
        +
        \tfrac{
          \1_{ \R^d \backslash A }( 
            Y_{ 
              \llcorner t \lrcorner_{ \theta } 
            } 
          )
          \,
          U( Y_t )
        }{
          e^{ \rho t }
        }
        +
        \smallint_{ 
          \llcorner t \lrcorner_{ \theta } 
        }^t
        \tfrac{
          \1_A( 
            Y_{ 
              \llcorner t \lrcorner_{ \theta } 
            } 
          )
          \,
          \bar{U}(Y_r)
        }{
          e^{ \rho r }
        }
        \,
        dr
      \right)
      \Big|
      \,
      ( Y_r )_{ 
        r \in 
        [ 
          0, 
          \llcorner t \lrcorner_{ \theta } 
        ] 
      }
    \right]
  \\ & \quad \cdot
   \exp\!\left(
     - c t
     +
     \smallint_0^{ 
       \llcorner t \lrcorner_{ \theta } 
     }
     \tfrac{
       \1_A( 
         Y_{ 
           \lfloor r \rfloor_{ \theta } 
         } 
       )
       \,
       \bar{U}( Y_r )
     }{
       e^{ \rho r }
     }
     \, dr
   \right)
  \\ & =
    \E\!\left[
      \exp\!\left(
        \tfrac{
          \1_A( 
            Y_{ 
              \llcorner t \lrcorner_{ \theta } 
            } 
          )
          \,
          U( Y_t )
        }{
          e^{ \rho t }
        }
        +
        \tfrac{
          \1_{ \R^d \backslash A }( 
            Y_{ 
              \llcorner t \lrcorner_{ \theta } 
            } 
          )
          \,
          U( 
            Y_{ 
              \llcorner t \lrcorner_{ \theta } 
            } 
          )
        }{
          e^{ \rho t }
        }
        +
        \smallint_{ 
          \llcorner t \lrcorner_{ \theta } 
        }^t
        \tfrac{
          \1_A( 
            Y_{ 
              \llcorner t \lrcorner_{ \theta } 
            } 
          )
          \,
          \bar{U}( Y_r )
        }{
          e^{ \rho r }
        }
        \,
        dr
      \right)
      \Big|
      \,
      ( Y_r )_{ 
        r \in 
        [
          0, 
          \llcorner t \lrcorner_{ \theta } 
        ] 
      }
    \right]
  \\ & \quad \cdot
   \exp\!\left(
     - c t
     +
     \smallint_0^{ 
       \llcorner t \lrcorner_{ \theta } 
     }
     \tfrac{
       \1_A( 
         Y_{ 
           \lfloor r \rfloor_{ \theta } 
         } 
       )
       \,
       \bar{U}( Y_r )
     }{
       e^{ \rho r }
     }
     \, dr
   \right)
 \\& \leq
      \exp\!\left(
        c
        \left( 
          t 
          - 
          \llcorner t \lrcorner_{ \theta } 
        \right)
        +
        \tfrac{
          \1_A( 
            Y_{ 
              \llcorner t \lrcorner_{ \theta } 
            } 
          )
          \,
          U( 
            Y_{ 
              \llcorner t \lrcorner_{ \theta } 
            } 
          )
        }{
          e^{ 
            \rho \,
            \llcorner t \lrcorner_{ \theta } 
          }
        }
     - c t
        +
        \tfrac{
          \1_{ \R^d \backslash A }( 
            Y_{ 
              \llcorner t \lrcorner_{ \theta } 
            } 
          )
          \,
          U( 
            Y_{ 
              \llcorner t \lrcorner_{ \theta } 
            } 
          )
        }{
          e^{ \rho t }
        }
     +
     \smallint_0^{ 
       \llcorner t \lrcorner_{ \theta } 
     }
     \tfrac{
       \1_A( 
         Y_{ 
           \lfloor r \rfloor_{ \theta } 
         } 
       )
       \,
       \bar{U}( Y_r )
     }{
       e^{ \rho r }
     }
     \, dr
   \right)
 \\ & \leq
      \exp\!\left(
        - 
        c \,
        \llcorner t \lrcorner_{ \theta } 
        +
        \tfrac{
          U( 
            Y_{ 
              \llcorner t \lrcorner_{ \theta } 
            } 
          )
        }{
          e^{ 
            \rho \,
            \llcorner t \lrcorner_{ \theta } 
          }
        }
     +
     \smallint_0^{ 
       \llcorner t \lrcorner_{ \theta } 
     }
     \tfrac{
       \1_A( 
         Y_{ 
           \lfloor r \rfloor_{ \theta } 
         } 
       )
       \,
       \bar{U}( Y_r )
     }{
       e^{ \rho r }
     }
     \, dr
   \right)
   .
\end{split} 
\end{equation}
  Combining this with Lemma~\ref{l:stability3}
  yields for all $ t \in [0,T] $
  that
\begin{equation*}
\begin{split}
    &
    \E\!\left[
      \exp\!\left(
        \tfrac{
          U(Y_{t})
        }{
          e^{ \rho t }
        }
        +
        \smallint_0^t
        \tfrac{
          \1_{A}(
            Y_{ \lfloor r \rfloor_{\theta} }
          )
          \,
          \bar{U}(Y_r)
        }{
          e^{ \rho r }
        }
        \, dr
      \right)
    \right]
    \leq e^{ c t } \,
      \E\!\left[e^{U(Y_0)}\right]
    .
  \end{split}
 \end{equation*}
 This finishes the proof of Corollary~\ref{Cor:exp.mom.abstract}.
\end{proof}
%

\subsection{A one-step estimate for exponential moments}
\label{sec:a_one_step_estimate}

In Lemma~\ref{l:exp.mom.abstract.one-step}
below a one-step estimate for exponential moments
(see \eqref{eq:exp.mom.abstract.assumption} in Corollary~\ref{Cor:exp.mom.abstract} above)
is proved for a general class of stopped
one-step numerical approximation schemes.
The proof of Lemma~\ref{l:exp.mom.abstract.one-step}
uses the elementary estimate in Lemma~\ref{l:exp.Gauss} below.
Moreover, the proof of Lemma~\ref{l:exp.Gauss} exploits
the following
well-known result, Lemma~\ref{l:exp.series}.
For completeness the proof of Lemma~\ref{l:exp.series}
is given below.

\begin{lemma}
\label{l:exp.series}
It holds for all $x \in \R$ that
\begin{equation}
e^x = 2 \!\left( \sum_{n=0}^{\infty} \frac{x^{2n}}{(2n)!} \right) - \frac{1}{e^x} \leq
2 \left( \sum_{n=0}^{\infty} \frac{x^{2n}}{(2n)!} \right).
\end{equation}
\end{lemma}
\begin{proof}[Proof of Lemma~\ref{l:exp.series}]
Note for all $ x \in \R $ that
\begin{equation}
  e^{ - x }
=
  \sum_{ n = 0 }^{ \infty }
  \frac{ ( - x )^n }{ n! }
=
  \left(
    \sum_{ n = 0 }^{ \infty } \frac{x^{2n}}{(2n)!}
  \right)
  -
  \left(
    \sum_{ n =0 }^{ \infty }
    \frac{x^{2n+1}}{(2n+1)!}
  \right)
  .
\end{equation}
This implies for all $x \in \R$ that
\begin{equation}
\begin{split}
  e^x &
=
  \sum_{ n = 0 }^{ \infty }
  \frac{ x^{ 2 n } }{ (2 n)! }
  +
  \sum_{ n = 0 }^{ \infty }
  \frac{ x^{ 2 n + 1 } }{ ( 2 n + 1 )! }
=
  \sum_{ n = 0 }^{ \infty }
  \frac{ x^{ 2 n } }{ ( 2 n ) ! }
  +
  \left(
    \left[
      \sum_{ n = 0 }^{ \infty } \frac{x^{2n}}{(2n)!}
    \right]
    - e^{ - x }
  \right)
\\ & =
  2 \left( \sum_{n=0}^{\infty} \frac{x^{2n}}{(2n)!}\right) - e^{-x} \leq 2 \! \left( \sum_{n=0}^{\infty} \frac{x^{2n}}{(2n)!}\right).
\end{split}
\end{equation}
The proof of Lemma~\ref{l:exp.series} is thus completed.
\end{proof}

\begin{lemma}\label{l:exp.Gauss}
Let
$ T \in [0, \infty ) $,
$ d, m \in \N$,
$A \in \R^{d \times m}$, let
  $
    ( \Omega, \mathcal{F}, \P ) $
  be a probability space,
  and let
  $
    W \colon [0,T] \times \Omega \to \R^m
  $
  be a standard
  Brownian motion. Then
  it holds for all $ t \in [0, T] $ that
$
  \E\!\left[ e^{ \| A W_t \|} \right]
  \leq 2
  \exp\!\big(
    \frac{ t }{ 2 } \| A \|^2_{ \HS( \R^m, \R^d ) }
  \big) .
$
\end{lemma}

\begin{proof}[Proof of Lemma~\ref{l:exp.Gauss}]
Throughout this proof let 
$ f_n \colon \R^m \to [0,\infty) $,
$ n \in \N_0 $,
be the functions with the property that
for all 
$ x \in \R^m $
and all
$
  n \in \N_0
$
it holds that
$
  f_n(x) = \| A x \|^{ 2 n }
$.
Then note for all $ x \in \R^m $ and all $ n \in \N $ that
\begin{equation}
\begin{split}
  \tr\!\left(
    ( \operatorname{Hess} f_n )( x )
  \right)
& =
  \tr\!\left(
    2 n
    \,
    \| A x \|^{ ( 2 n - 2 ) } A^* A
    +
    \1_{ \{ x \neq 0 \} }
    \,
    2 n \left( 2 n - 2 \right)
    \| A x \|^{ ( 2 n - 4 ) }
    \left( A^* A x \right)
    \left( A^* A x \right)^*
  \right)
\\ & =
  2 n
  \left\| A x \right\|^{ ( 2 n - 2 ) }
  \| A \|_{
    \HS( \R^m, \R^d )
  }^2
  +
  \1_{ \{ x \neq 0 \} } \, 2 n
  \left( 2 n - 2 \right)
  \left\|
    A x
  \right\|^{ ( 2 n - 4 ) }
  \left\|
    A^* A x
  \right\|^2
\\ & \leq
  2 n
  \left\| A x \right\|^{ ( 2 n - 2 ) }
  \left\| A \right\|_{
    \HS( \R^m, \R^d )
  }^2
  +
  2 n \left( 2 n - 2 \right)
  \left\| A x \right\|^{ ( 2 n - 2 ) }
  \left\| A \right\|_{
    \HS( \R^m, \R^d)
  }^2
\\ & =
  2 n \left( 2 n - 1 \right)
  \left\| A \right\|_{
    \HS( \R^m, \R^d )
  }^2
  \,
  f_{ n - 1 }( x )
  \, .
\end{split}
\end{equation}
It\^o's formula hence shows for all
$
  s_0 \in [0, T]
$
and all
$ n \in \N $
that
\begin{equation}
\begin{split}
\label{eq:BM.Moments}
  \E\!\left[
    \left\|
      A W_{ s_0 }
    \right\|^{ 2 n }
  \right]
& =
  \E\big[
    f_n( W_{ s_0 } )
  \big]
=
  \frac{ 1 }{ 2 }
  \int_0^{ s_0 }
  \E\big[
    \tr\!\left(
      (\operatorname{Hess} f_n)( W_{ s_1 } )
    \right)
  \big]
  \,
  ds_1
\\ & \leq
  \frac{ 1 }{ 2 }
  \int_0^{ s_0 }
  2n \left( 2 n - 1 \right)
  \| A \|_{
    \HS( \R^m, \R^d )
  }^2
  \,
  \E\big[
    f_{ n - 1 }(
      W_{ s_1 }
    )
  \big]
  \,
  d s_1
\\ & \leq
  \ldots
\leq
  \frac{
    ( 2 n )!
  }{
    2^n
  }
  \left\|
    A
  \right\|_{
    \HS( \R^m, \R^d )
  }^{ 2 n }
  \smallint_0^{ s_0 }
  \smallint_0^{s_1}
  \cdots
  \smallint_0^{ s_{ n - 1 } }
  \E\big[
    f_0( W_{ s_n } )
  \big]
  \, ds_n \cdots ds_2 \, ds_1
\\ & =
  \frac{ ( 2 n )! }{
    2^n n!
  }
  \left\| A
  \right\|_{
    \HS( \R^m, \R^d)
  }^{ 2 n }
  ( s_0 )^n
  \, .
\end{split}\end{equation}
Combining this with Lemma~\ref{l:exp.series} implies
for all $ t \in [0, T] $
that
\begin{equation}
  \E\!\left[
    e^{ \| A W_t \| }
  \right]
\leq
  2
  \left(
    \sum_{ n = 0 }^{ \infty }
    \frac{
      \E\!\left[
        \| A W_t \|^{ 2 n }
      \right]
    }{
      ( 2 n )!
    }
  \right)
\leq
  2
  \left(
    \sum_{ n = 0 }^{ \infty }
    \frac{
      t^n
      \| A \|_{\HS(\R^m, \R^d)}^{2n}
    }{
      2^n n!
    }
  \right)
  = 2 e^{\frac{t}{2}\|A\|_{\HS(\R^m, \R^d)}^{2} }.
\end{equation}
This finishes the proof of Lemma~\ref{l:exp.Gauss}.
\end{proof}

Beside Lemma~\ref{l:exp.Gauss},
the proof of Lemma~\ref{l:exp.mom.abstract.one-step} also
uses the following two lemmas 
(Lemma~\ref{lemma:equivalent0} 
and Lemma~\ref{l:function.space.estimate}).
The proof of Lemma~\ref{l:function.space.estimate} uses
inequality~(2.56) in Lemma 2.11 
in~\cite{HutzenthalerJentzen2014Memoires}.

\begin{lemma}
\label{lemma:equivalent0}
Let $ d \in \N $, 
$ n \in \N_0 $,
$ c, p \in (0,\infty) $,
$ x, y \in \R^d $,
$ V \in C^{ n + 1 }_{ p, c}( \R^d , [0, \infty) ) $.
Then
\begin{enumerate}[(i)]
\item \label{eq:bounded_derivative} it holds
for all 
$ 
  t \in \{ s \in [0, 1] \colon \R \ni u \mapsto V(x+uy) \in \R \, \text{is differentiable at } s \} 
$ 
that
$ 
  | \tfrac{\partial}{\partial t} V(x + ty) |
  \leq 
  c \, \| y \| \,
  [ 1 + V( x + t y ) ]^{ 
    \max\{ 1 - \nicefrac{ 1 }{ p } , 0 \} 
  } 
$
and
\item 
\label{eq:bounded_derivatives} 
it holds for all
$ i \in \N \cap [0,n] $,
$ z_1, \ldots, z_i \in \R^d $,
$ 
  t \in 
  \{ 
    s \in [0, 1] \colon 
    \R \ni u \mapsto V^{(i)}( x + u y )( z_1, \ldots, z_i) \in \R 
    \text{ is differentiable at } s
  \} 
$  
that
\begin{equation}
\big|
\tfrac{\partial}{\partial t}
\big(
  V^{(i)} (x+ty) (z_1,\ldots, z_i)
\big)
\big|
\leq
c \, \| z_1 \| \cdots \| z_i \| \, \| y \| \,
[ 1 + V(x + t y ) ]^{ 
  \max\{ 1 - \nicefrac{ ( i + 1 ) }{ p } , 0 \}
}
  .
\end{equation}
\end{enumerate}
\end{lemma}
\begin{proof}[Proof of Lemma~\ref{lemma:equivalent0}]
First of all,
note that
the assumption that $ V \in C_{ p, c }^{ n + 1 }( \R^d , [0,\infty) ) $
ensures that for all 
$ t \in [0, 1] $, 
$ h \in \R $
it holds
that
\begin{equation}
\label{eq:basic_estimate}
\begin{split}
|
  V(x+ty) - V(x+(t+h)y)
|
&
\leq
c
\,
|h| 
\, \|y\|
\, 
\big[ 
  1 + 
  \sup\nolimits_{ r \in [0,1] } 
  V\big( 
    x + (t + (1 - r) h) y
  \big) 
\big]^{
  \max\{ 1 - \nicefrac{ 1 }{ p } , 0 \}
}
\\
&
=
c
\,
|h| 
\, \| y \| 
\,
  \big[ 
    1 + 
    \sup\nolimits_{ r \in [0,1] }
    V( x + (t + r h) y)
  \big]^{
    \max\{ 1 - \nicefrac{ 1 }{ p } , 0 \}
  }
  .
\end{split}
\end{equation}
Next observe that again
the assumption that $ V \in C_{ p, c }^{ n + 1 }( \R^d , [0,\infty) ) $
ensures that 
$ V $ is locally Lipschitz continuous.
Hence, we obtain for all 
$ t \in [0, 1] $ 
that 
\begin{equation}
\label{eq:lim_estimate}
  \limsup\nolimits_{ 
    ( \R \backslash \{0\} ) \ni h \to 0 
  } 
  | 
    \sup\nolimits_{ r \in [0,1] } V( x + ( t + r h ) y ) - V( x + t y ) 
  | 
  = 0 .
\end{equation}
Combining this with~\eqref{eq:basic_estimate} proves~\eqref{eq:bounded_derivative}.
In the next step we note that again 
the assumption that $ V \in C_{ p, c }^{ n + 1 }( \R^d , [0,\infty) ) $
shows that for all 
$ i \in \N \cap [0, n] $,
$ z_1,\ldots, z_i\in \R^d \backslash \{0\} $,
$ t \in [0, 1] $, $ h \in \R $  
it holds that
\begin{equation}
\begin{split}
&
\tfrac{
  |
    V^{(i)}(x+ty) (z_1,\ldots, z_i)
    -
    V^{(i)}(x+(t+h)y) (z_1,\ldots, z_i) 
  |
}{
  \|z_1\| \cdots \|z_i\|
}
\leq
  \| 
    V^{(i)}( x + t y )
    -
    V^{(i)}( x + (t+h)y ) 
  \|_{L^{(i)}(\R^d,\R)}
\\
&
\leq
c
|h| \| y\|
[ 1 + \sup\nolimits_{r\in [0,1]}V( x + (t+ (1-r) h) y) ]^{ 
  \max\{ 1 - \nicefrac{ (i + 1) }{ p } , 0 \} 
}
\\
&
=
c
|h| \| y\|
[ 1 + \sup\nolimits_{r\in [0,1]}V( x + (t+ r h) y) ]^{ 
  \max\{ 1 - \nicefrac{ (i + 1) }{ p } , 0 \} 
}
.
\end{split}
\end{equation}
This 
and~\eqref{eq:lim_estimate} 
establish~\eqref{eq:bounded_derivatives}.
The proof of Lemma~\ref{lemma:equivalent0} is thus completed.
\end{proof}

\begin{lemma}
\label{l:function.space.estimate}
Let $ c, p \in [1,\infty) $, $ d \in \N $, $ x, y \in \R^d $, $ V \in C^1_{ p, c }( \R^d, [0,\infty) ) $. 
Then it holds that 
$ 
1 + V(x+y) \leq c^p 2^{ p - 1 } ( 1 + V(x) + \| y \|^p ) 
$.
\end{lemma}

\begin{proof}[Proof of Lemma~\ref{l:function.space.estimate}]
Throughout this proof let $ f \colon \R \to \R $
be the function with the property
that for all $ t \in \R $ it holds that
$ f(t) =  V(x + t y) $.
Next observe that the fact that $ V $ is locally Lipschitz continuous ensures that
$ f $ is globally Lipschitz continuous.
Moreover, note that 
Item~\eqref{eq:bounded_derivative}
in Lemma~\ref{lemma:equivalent0}
implies  
for all
$ t \in \{ s \in [0,1] \colon 
 	\R \ni u \mapsto f(u) \in \R \text{ is differentiable at }s
 	\} $ that
\begin{equation}
\begin{split}
  \tfrac{\partial}{\partial t}
  ( 1+ f(t) )
\leq
\left|
  \tfrac{\partial}{\partial t}
  ( 1+ f(t) )
\right|
\leq
  c \left\| y \right\|
  \left[ 
    1 + f(t) 
  \right]^{
    ( 1 - \nicefrac{ 1 }{ p } )
  }
=
  c \left\| y \right\|
  \left|
    1 + f(t) 
  \right|^{
    ( 1 - \nicefrac{ 1 }{ p } )
  }
. 
\end{split}
\end{equation}
The fact that $ f $ is globally Lipschitz continuous
and inequality~(2.56) in 
Lemma~2.11 in Hutzenthaler \& Jentzen~\cite{HutzenthalerJentzen2014Memoires}
(with $ T = 1 $, $ c = c \, \| y \| $, $ p = p $,
$ y = ( [0,1] \ni t \mapsto 1 + f(t) \in \R ) $ in the notation of Lemma~2.11 in
Hutzenthaler \& Jentzen~\cite{HutzenthalerJentzen2014Memoires})
hence prove for all $ t \in [0,1] $ that
\begin{equation}
\begin{split}
\label{eq:function_estimate}
1 + f(t)
\leq
  2^{ p - 1 }
  \bigg[
    1 + f(0) 
    +
  \left|
    \tfrac{ c \| y \| t }{ p }
  \right|^p
\bigg]
=
  2^{ p - 1 }
\bigg[
 1 + V(x)
+  
    \tfrac{ c^p \| y \|^p t^p }{ p^p }
\bigg]
\leq
  2^{ p - 1 }
  \big[
    1 + V(x)
    +   
    c^p \| y \|^p 
  \big]
  .
\end{split}
\end{equation}
This implies that 
$
1 + V(x+y)
\leq
2^{p-1}
\left[
1 + V(x)
+ 
c^p 
\| y\|^p
\right]
\leq
c^p 
\,
2^{p-1}
\left[
1 + V(x)
+ 
\| y\|^p
\right]
.
$
The proof of Lemma~\ref{l:function.space.estimate} is thus completed.
\end{proof}

\begin{lemma}
\label{l:exp.mom.abstract.one-step}
  Let $ \alpha, h \in (0,\infty) $,
  $ d, m \in \mathbb{N} $,
  $
    c, p \in [1,\infty)
  $,
  $
    \gamma_0, \gamma_1, \dots, \gamma_6, \gamma_7, \rho
    \in [0,\infty)
  $,
  $
    \mu \in \mathcal{M}( \mathcal{B}( \mathbb{R}^d ), \mathcal{B}( \mathbb{R}^d ) )
  $,
  $
    \sigma
    \in 
    \mathcal{M}( 
      \mathcal{B}( \mathbb{R}^d ) ,
      \mathcal{B}( \mathbb{R}^{ d \times m } ) 
    )
  $,
  $
    \bar{U}
    \in C(\R^d, \R)
  $,
  $
    U \in C^3_{p,c}( \mathbb{R}^d, [0,\infty) )
  $,
  $ 
    D \in
    2^{
      \{
        x \in \R^d \colon 
        U(x) \leq c h^{ - \alpha }  
      \}
    }
  $,
  let 
  $
    \Phi \in
    {C}^{ 0, 1, 2 }(
      D \times [0,h] \times \R^m, \R^d
    )
  $,
  let
  $
    ( \Omega, \mathcal{F}, \P ) $
  be a probability space,
  let
  $
    W \colon [0,h] \times \Omega \to \R^m
  $
  be a standard
  Brownian motion with continuous sample paths, assume
  for all $ x \in \R^d $ that
  \begin{equation}
  \left\|
      \mu(x)
    \right\|
    \leq
    c
    \left(1 + \left|U(x)\right|^{\gamma_0}\right)
    ,
    \qquad
    \|
      \sigma(x)
    \|_{
      \HS( \mathbb{R}^m, \mathbb{R}^d )
    }
    \leq
    c
    \left( 1 + \left| U( x ) \right|^{ \gamma_1 } \right),
    \label{eq:mu_sigma_Ugrowth_assumption}
  \end{equation}
  \begin{equation}
  \label{eq:generator.exponential}
    ( \mathcal{G}_{ \mu, \sigma } U)(x)
    +
    \tfrac{1}{2}\left\|\sigma(x)^{*}(\nabla U)(x)\right\|^2
    +
    \bar{U}(x)
  \leq
    \rho \cdot U(x)
    ,
  \end{equation}
  assume
  for all 
  $
    r \in [1, \infty)
  $,
  $
    x \in D
  $, 
  $ 
    s \in (0, h]
  $
  that
  $
    \Phi( x, 0, 0 ) = x 
  $
  and
  \begin{align}
    \left\|
      \big(
        \tfrac{\partial}{\partial s} \Phi
      \big)( x , s , W_s ) -
      \mu(x)
    \right\|_{
      L^4( \Omega; \R^d )
    }
  &
    \leq c s^{ \gamma_2 } ,
 \label{eq:assumption.on.phi.mue}
    \\
    \big\|
      \big(
        \tfrac{ \partial }{ \partial y } \Phi
      \big)( x, s, W_s ) - \sigma(x)
    \big\|_{
      L^8(
        \Omega; \HS( \R^m, \R^d )
      )
    }
    &
    \leq c s^{ \gamma_3 } ,
  \label{eq:assumption.on.phi.sigma}
  \\
    \big\|
      (
        \triangle_y \Phi
      )( x, s, W_s )
    \big\|_{
      L^4( \Omega; \R^d )
    }
   &
     \leq c s^{ \gamma_4 } ,
     \label{eq:assumption.on.phi.sigma2}
  \\
    \left\|
      \Phi(x,s,W_s) - x
    \right\|_{L^r(\Omega; \R^d)}
     \leq
     c
     \min\!\big(
       r ,
       1 + \left|U(x) \right|^{\gamma_5},
       (1 + \left|U(x) \right|^{\gamma_5})
     &
       \left\|
         \mu(x) s +
         \sigma(x) W_s
       \right\|_{
         L^r( \Omega; \R^d )
       }
     \big)
     \label{eq:assumption.on.phi.increment}
     ,
  \end{align}
  and assume
  for all
  $ x, y \in \R^d $
  that
  $
    |
      \bar{U}(x)
    |
  \leq
    c
    \left( 1 +
    \left| U( x ) \right|^{ \gamma_6 } \right)
  $
  and
  $
    |
      \bar{U}(x) - \bar{U}( y )
    |
  \leq
    c
    \left( 1 +
      | U(x) |^{ \gamma_7 } +
      | U(y) |^{ \gamma_7 }
    \right)
    \| x - y \|
  $.
  Then it holds for all
  $
    (t,x) \in (0, h ] \times D
  $
  that
  \begin{align} 
  \label{eq:exp.mom.abstract.one-step}
  &
    \E\!\left[
      \exp\!\left(
        \tfrac{U(
          \Phi(x,t,W_t)
         )}{e^{ \rho t }}
        +
        \smallint_0^t
        \tfrac{
          \bar{U}(
          \Phi(x, s, W_s)
        )}{e^{ \rho s }}\,
        ds
      \right)
    \right]
  \\ &
  \nonumber
    \leq
    e^{ U(x) }
    \bigg[
      1 +
      \smallint_0^t
      \exp\!\left(
	\tfrac{
	  s
	  \,
	  \left[ 2 c \right]^{
	    4 p
	    ( \gamma_6 + 2 )
	    \max( \gamma_0, \gamma_1, \gamma_5, 2 )
	  }
	}{
	  \left[
	    \min(s, 1)
	  \right]^{
	    \alpha
	    \left[
	      ( p \gamma_5 + 1 ) ( \gamma_6 + 2 )
	      + \gamma_0 + 2 \gamma_1
	    \right]
	  }
	}
    \right)
    \tfrac{
      \max( \rho, 1 )
      \,
      \left[
	2 p c \max( s, 1 )
      \right]^{
	6 p
	\left( \gamma_7 + 3 \right)
	\max( 1, \gamma_0, \gamma_1, \dots, \gamma_5 )
      }
    }{
      \left[
	\min( s, 1)
      \right]^{
	\left[
	  \alpha
	  \left(
	    2 \gamma_0 + 4 \gamma_1 + 2 \gamma_5 + (p \gamma_5 + 1 ) \gamma_7 + 2
	  \right)
	  -
	  \min( 1 / 2 , \gamma_2, \gamma_3, \gamma_4 )
	\right]
      }
    }
  \,
      ds
    \bigg]
    .
  \end{align}
\end{lemma}
\begin{proof}[Proof of Lemma~\ref{l:exp.mom.abstract.one-step}]
  Throughout this proof 
  let $ Y^x \colon [0, h] \times \Omega \to \R^d $, $ x \in D $, be 
  the stochastic processes with the property that 
  for all $ s \in [0,h] $, $ x \in D $
  it holds that
  $ Y_s^x = \Phi(x, s, W_s ) $
  and let
  $
    \tau_n \colon\Omega\to[0,h]
  $, $ n \in \N $,
  be the functions with the property that
  for all $ n \in \N $ it holds that
  $
    \tau_n = 
    \inf\!\big(
      \{ s \in [0,h] \colon \| W_s \| > n \} \cup \{ h \} 
    \big)
  $. 
  Next observe that It\^o's formula implies that
  for all $ x \in D $, $ t \in [0,h] $
  it holds $ \P $-a.s.\ that
  \begin{equation}  
  \begin{split} 
  \label{eq:after.Ito}
    &
    \exp\!\left(
      e^{ - \rho t } 
      U( Y^x_t ) 
      + 
      \int_0^t 
      e^{ - \rho r }
      \bar{U}( Y^x_r ) \, dr
    \right)
    - e^{ U(x) }
    \\
   =&
    \int_0^t
    \exp\!\left(
      e^{-\rho s}U(Y^x_s)+\int_0^s e^{-\rho r}\bar{U}(Y^x_r)\,dr
    \right)
      e^{ - \rho s }
      \,
      U'( Y^x_s )\big(
        \tfrac{ \partial }{ \partial y }
      \Phi\big)(x,s,W_s)
      \, dW_s
    \\&
    +
    \int_0^t
    \exp\!\left(e^{-\rho s}U(Y^x_s)+\int_0^s e^{-\rho r}\bar{U}(Y^x_r)\,dr\right)
    e^{-\rho s}
    \bigg(
      \bar{U}(Y_s^x)
      -\rho U(Y^x_s)
      +U'(Y^x_s)\big(\tfrac{\partial}{\partial s}\Phi\big)
         (x,s,W_s)
    \\&\qquad\quad\;
      +\tfrac{1}{2}\tr\!\left(
       \big(\tfrac{\partial}{\partial y}\Phi\big)(x,s,W_s)
       \big[\big(\tfrac{\partial}{\partial y}\Phi\big)(x,s,W_s)\big]^{*}
       (\operatorname{Hess}U)(Y^x_s)
       \right)
    \\& \qquad\quad\;+
      \tfrac{ 1 }{ 2 }
      \,
      e^{ - \rho s }
      \,
       \big\|
         \big[\big(\tfrac{\partial}{\partial y}\Phi\big)(x,s,W_s)\big]^{*}
         (\nabla U)(Y^x_s)
       \big\|^2
      +  \tfrac{ 1 }{ 2 } 
      \smallsum_{ i = 1 }^m 
      U'( Y^x_s ) 
      \big( 
        \tfrac{ \partial^2 }{ \partial y_i^2 } \Phi 
      \big)( x, s, W_s )
    \bigg) \, ds
    .
  \end{split}     \end{equation}
This shows
  for all
  $
    t \in [0, h ]
  $,
  $
    x 
    \in D
  $,
  $
    n \in \N
  $
  that
  \begin{equation}  \begin{split}
    &
    \E\!\left[
      \exp\!\left(
        e^{ - \rho ( t \wedge \tau_n ) }
        U( Y^x_{ t \wedge \tau_n } )
        +
        \int_0^{ t \wedge \tau_n }
        e^{ - \rho r }
        \bar{U}( Y_r^x )
        \, dr
      \right)
    \right]
    -
    e^{ U(x) }
  \\ & =
   \E\bigg[
    \int_0^{t\wedge\tau_n}
    \exp\!\left(e^{-\rho s}U(Y^x_s)
                   +\int_0^{s}e^{-\rho r}\bar{U}(Y_r^x)\,dr
             \right)
    e^{-\rho s}
    \cdot\bigg(
      \bar{U}(Y_s^x)
      -\rho U(Y^x_s)
    \\&\qquad\quad\;
      +U'(Y^x_s)\big(\tfrac{\partial}{\partial s}\Phi\big)
         (x,s,W_s)
      +\tfrac{1}{2}\tr\left(
       \big(\tfrac{\partial}{\partial y}\Phi\big)(x, s,W_s)
       \big[\big(\tfrac{\partial}{\partial y}\Phi\big)(x,s,W_s)\big]^{*}
       (\operatorname{Hess}U)(Y^x_s)
       \right)
    \\&\qquad\quad\;
      +\tfrac{1}{2}e^{-\rho s}\left\|
         \big[\big(\tfrac{\partial}{\partial y}\Phi\big)(x,s,W_s)\big]^{*}
         (\nabla U)(Y^x_s)
       \right\|^2
      +  \tfrac{ 1 }{ 2 } U'( Y^x_s )
      \, ( \triangle_y \Phi )( x, s, W_s )
    \bigg)
    \,ds
    \bigg]
    .
  \end{split}     \end{equation}
  Assumption~\eqref{eq:generator.exponential}
  therefore yields for all $ t \in [0, h ] $, $ x \in D $, $ n \in \N $ 
  that
  \begin{equation}  \begin{split}
    &\E\!\left[\exp\!\left(e^{-\rho (t\wedge\tau_n)}U(Y^x_{t\wedge\tau_n})
                   +\int_0^{t\wedge\tau_n}e^{-\rho r}\bar{U}(Y_r^x)\,dr
    \right)\right]
    -e^{U(x)}
 \\ & =
   \E\bigg[
    \int_0^{t\wedge\tau_n}
    \exp\!\left(e^{-\rho s}U(Y^x_s)
                   +\int_0^{s}e^{-\rho r}\bar{U}(Y_r^x)\,dr
    \right)
    e^{-\rho s}
    \\  &\quad
    \cdot\bigg(
      -\rho U(x)+(\mathcal{G}_{\mu,\sigma}U)(x)
      +\tfrac{1}{2}e^{-\rho s}
      \left\|\sigma(x)^{*}(\nabla U)(x)\right\|^2
      +\bar{U}(x)
    \\& \quad
      + \bar{U}(Y^x_s)-\bar{U}(x)
      -\rho( U(Y^x_s)-U(x))
      +U'(Y^x_s)
      \big(
        \tfrac{ \partial }{ \partial s } \Phi
      \big)(x,s,W_s)
      -
      U'(x) \mu(x)
    \\& \quad
      +\tfrac{1}{2}\tr\left(
       \big(\tfrac{\partial}{\partial y}\Phi\big)(x,s,W_s)
       \big[\big(\tfrac{\partial}{\partial y}\Phi\big)(x,s,W_s)\big]^{*}
       (\operatorname{Hess}U)(Y^x_s)
       \right)
       -\tfrac{1}{2}\tr\left(\sigma(x)\sigma(x)^{*}
       (\operatorname{Hess}U)(x)\right)
    \\& \quad
      +\tfrac{1}{2}e^{-\rho s}\left\|
         \big[\big(\tfrac{\partial}{\partial y}\Phi\big)(x,s,W_s)\big]^{*}
         (\nabla U)(Y^x_s)
       \right\|^2
      -\tfrac{1}{2}e^{-\rho s}\left\|\sigma(x)^{*}(\nabla U)(x)\right\|^2
      +
      \tfrac{
        U'( Y^x_s )
        \, ( \triangle_y \Phi )( x, s, W_s )
      }{ 2 }
    \bigg)
    \, ds
    \bigg]
\\ & \leq
   \E\bigg[
    \int_0^{t\wedge\tau_n}
    \exp\!\left(e^{-\rho s}U(Y^x_s)
                   +\int_0^{s}e^{-\rho r}\bar{U}(Y_r^x)\,dr
        \right)
    \\& \quad
    \cdot\bigg(
      \left|\bar{U}(Y^x_s)- \bar{U}(x)\right|
      +\rho\;| U(Y^x_s)-U(x)|
      +\left|U'(Y^x_s)\big(\tfrac{\partial}{\partial s}\Phi\big) (x,s,W_s)
            -U'(x)\mu(x)\right|
    \\& \quad
      +\tfrac{1}{2}\left|
       \tr\left(
       \big(\tfrac{\partial}{\partial y}\Phi\big)(x,s,W_s)
       \big[\big(\tfrac{\partial}{\partial y}\Phi\big)(x,s,W_s)\big]^{*}
       (\operatorname{Hess}U)(Y^x_s)
       \right)
       -\tr\left(\sigma(x)\sigma(x)^{*}(\operatorname{Hess}U)(x)\right)
       \right|
    \\& \quad
      +\tfrac{1}{2}e^{-\rho s}
      \left|
        \left\|
         \big[\big(\tfrac{\partial}{\partial y}\Phi\big)(x,s,W_s)\big]^{*}
         (\nabla U)(Y^x_s)
       \right\|^2
      -\left\|\sigma(x)^{*}(\nabla U)(x)\right\|^2
      \right|
      +
      \tfrac{
        \left|
         U'( Y^x_s )
         \, ( \triangle_y \Phi )( x, s, W_s )
        \right|
      }{ 2 }
    \bigg)
    \,ds
    \bigg].
  \end{split}     \end{equation}
  Hence, Fatou's lemma, Fubini's theorem, and H\"older's inequality
  imply for all $t\in[0,h]$, $ x \in D $ that
  \begin{align}
    &\E\!\left[\exp\!\left(e^{-\rho t}U(Y^x_{t})
                   +\int_0^{t}e^{-\rho r}\bar{U}(Y_r^x)\,dr
    \right)\right]
      -e^{U(x)}
    \nonumber \\&
    \leq
    \liminf_{n\to\infty}\E\!\left[\exp\!\left(e^{-\rho (t\wedge\tau_n)}
      U(Y^x_{t\wedge\tau_n})
                   +\int_0^{t\wedge\tau_n}e^{-\rho r}\bar{U}(Y_r^x)\,dr
    \right)\right]
    -e^{U(x)}
   \nonumber \\&
   \leq
   \E\bigg[
    \int_0^{t}
    \exp\!\left(U(Y^x_s)
                   +\int_0^{s}e^{-\rho r}\bar{U}(Y_r^x)\,dr
    \right)
    \nonumber \\&\qquad\;
    \cdot\bigg(
      \left|\bar{U}(Y^x_s)- \bar{U}(x)\right|
      +\rho| U(Y^x_s)-U(x)|
      +\left|U'(Y^x_s)\left(\tfrac{\partial}{\partial s}\Phi\right) (x,s,W_s)
            -U'(x)\mu(x)\right|
    \nonumber \\&\qquad\;
      +\tfrac{1}{2}\left|
       \tr\left(
       \big(\tfrac{\partial}{\partial y}\Phi\big)(x,s,W_s)
       \big[\big(\tfrac{\partial}{\partial y}\Phi\big)(x,s,W_s)\big]^{*}
       (\operatorname{Hess}U)(Y^x_s)
       \right)
       -\tr\left(\sigma(x)\sigma(x)^{*}(\operatorname{Hess}U)(x)\right)
       \right|
   \nonumber \\&\qquad\;
      +\tfrac{1}{2}e^{-\rho s}
      \left|
        \left\|
         \big[\big(\tfrac{\partial}{\partial y}\Phi\big)(x,s,W_s)\big]^{*}
         (\nabla U)(Y^x_s)
       \right\|^2
      -\left\|\sigma(x)^{*}(\nabla U)(x)\right\|^2
      \right|
      +
      \tfrac{
        \left|
         U'( Y^x_s )
         \, ( \triangle_y \Phi )( x, s, W_s )
        \right|
      }{ 2 }
    \bigg)
    \,ds
    \bigg]
  \nonumber
  \\ &
   \leq
    \int_0^{t}
    \left\|
    \exp\!\left(U(Y^x_s)
                   +\int_0^{s}e^{-\rho r}\bar{U}(Y_r^x)\,dr
    \right)
    \right\|_{L^2(\Omega;\R)}
   \nonumber \\&\qquad\;
    \cdot
    \bigg[
      \rho\| U(Y^x_s)-U(x)\|_{L^2(\Omega;\R)}
      +\left\|U'(Y^x_s)\big(\tfrac{\partial}{\partial s}\Phi\big) (x,s,W_s)
            -U'(x)\mu(x)\right\|_{L^2(\Omega;\R)}
   \nonumber \\&\qquad\;
      +\tfrac{1}{2}\!\left\|
       \tr \!\left(\!
       \big(\tfrac{\partial}{\partial y}\Phi\big)(x,s,W_s)
       \big[\big(\tfrac{\partial}{\partial y}\Phi\big)(x,s,W_s)\big]^{*}
       (\operatorname{Hess}U)(Y^x_s)
       -\sigma(x)\sigma(x)^{*}(\operatorname{Hess}U)(x) \!\right)
       \right\|_{L^2(\Omega;\R)}
   \nonumber \\&\qquad\;
      +\tfrac{1}{2}e^{-\rho s}
      \left\|
        \big\|
         \big[\big(\tfrac{\partial}{\partial y}\Phi\big)(x,s,W_s)\big]^{*}
         (\nabla U)(Y^x_s)
       \big\|^2
      -\left\|\sigma(x)^{*}(\nabla U)(x)\right\|^2
      \right\|_{L^2(\Omega;\R)}
  \nonumber  \\&\qquad\;
      +
      \tfrac{
        \left\|
           U'( Y^x_s )
           \, ( \triangle_y \Phi )( x, s, W_s )
        \right\|_{
          L^2( \Omega; \R )
        }
      }{ 2 }
      +\left\|\bar{U}(Y^x_s)- \bar{U}(x)\right\|_{L^2(\Omega;\R)}
    \bigg]
    \, ds
    \, .
\label{eq:after.Fubini}
\end{align}
  Next we estimate the $ L^2 $-norms
  on the right-hand side separately.
  Combining the assumption that
  $U \in C_{p,c}^3( \mathbb{R}^d, [0,\infty) ) \subseteq C_{ p, c }^1( \R^d, [0,\infty) ) $ 
  with
  Lemma~\ref{l:function.space.estimate} 
  (with $ c = c $, $ p = p $, $ V = U $, $ x = x $,
  $ y = r ( y - x ) $
  for $ r \in [0,1] $, $ x, y \in \R^d $
  in the notation of Lemma~\ref{l:function.space.estimate})
  implies for all $ x, y \in \mathbb{R}^d $,
  $ i \in \{ 0, 1, 2 \} $ that
  \begin{equation}
  \label{eq:derivativeU1}
  \begin{split}
    \big\|
      U^{(i)}(y) -
      U^{(i)}(x)
    \big\|_{
      L^{(i)}( \mathbb{R}^d,
        \mathbb{R}
      )
    }
 & \leq
   c 
   \,
    \| y - x \|
    \left[ 1 +
      \sup\nolimits_{ r \in [0,1] }
      U\big( r y + ( 1 - r ) x ) \big)
    \right]^{
      \max\{ 1 - \nicefrac{ ( i + 1 ) }{ p } , 0 \}  
    }
  \\ & =
   c 
      \sup\nolimits_{ r \in [0,1] }
    \left[ 
      1 +
      U\big( x + r ( y - x ) ) \big)
    \right]^{
      \frac{ \max(p - i - 1, 0) }{ p }
    }
    \| y - x \|
  \\ & \leq
   c
      \sup\nolimits_{ r \in [0,1] }
    \Big[ 
      c^p 
      \, 
      2^{ p - 1 } 
      \big(
        1 + U(x)
        +
        \| r ( y - x ) \|^p
      \big)
    \Big]^{
      \frac{ \max(p - i - 1, 0) }{ p }
    }
    \| y - x \|
  \\ & \leq
    \frac{
      c
      \left( 2 c \right)^{
        \max\{ p - i - 1, 0 \}
      }
    }{
      2^{ \max\{ p - i - 1, 0 \} / p }
    }
    \left(
    \left[ 1 +
      U( x )
    \right]^{
      \frac{ \max(p - i - 1,0) }{ p }
    }
    +
    \left\| y - x
    \right\|^{
      \max\left( p - i-1,0 \right)
    }
    \right)
    \left\| y - x \right\|
  \\ & \leq
    \tfrac{
      \left( 2 c \right)^{
        p
      }
    }{
      2
    }
    \left(1 +
    \left|
      U( x )
    \right|^{
      \frac{ \max(p - i - 1,0) }{ p }
    }
    +
    \left\| y - x
    \right\|^{
      \max\left( p - i-1,0 \right)
    }
    \right)
    \left\| y - x \right\|
    .
  \end{split}
  \end{equation}
  This, in particular, shows for all $ x , y \in \R^d $ that
  \begin{equation}
  \label{eq:derivativeU1_2}
    \left|
      U(y) -
      U(x)
    \right|
  \leq
    \tfrac{
      \left( 2 c \right)^{
        p
      }
    }{
      2
    }
    \left( 1 +
      \left|U( x ) \right|^{\frac{p-1}{p}}
    +
    \left\| y - x
    \right\|^{
      \left( p - 1 \right)
    }
    \right)
    \left\| y - x \right\|.
  \end{equation}
  Combining this with
  H\"older's inequality
  and the fact that
  $
    \forall \, x \in \R^d \colon
    | \bar{U}(x) |
  \leq
    c \left( 1 + \left| U(x) \right|^{ \gamma_6 } \right)
  $
  yields for all $ s \in (0, h] $, $ x \in \R^d $ that
\begin{align}
    &
    \E\!\left[\exp\!\left(2U(Y_s^x)
                   +2\int_0^{s}e^{-\rho r}\bar{U}(Y_r^x)\,dr
          -2U(x)\right)\right]
    \leq \E\!\left[\exp\!\left(
        2 \left| U( Y_s^x ) - U(x) \right|
            + 2 \smallint_0^s \left| \bar{U}( Y_r^x ) \right| dr
          \right)\right]
    \nonumber \\
    &\leq
    \E\!\left[
      \exp\!\left(
        ( 2 c )^p
        \big[ 1 +
          \left|U(x)\right|^{\frac{p-1}{p}}
          +
          \left\|
            Y_s^x - x
          \right\|^{ ( p - 1 ) }
        \big]
        \left\|
          Y_s^x - x
        \right\|
        +
        2 c \smallint_0^s \left( 1 + |U( Y_r^x )|^{\gamma_6} \right) dr
     \right)
   \right]
   \nonumber
   \\
    &\leq
    \left\|
      \exp\!\left(
        ( 2 c )^p
        \big[ 1 +
          \left|U(x)\right|^{\frac{p-1}{p}}
          +
          \left\|
            Y_s^x - x
          \right\|^{ ( p - 1 ) }
        \big]
        \left\|
          Y_s^x - x
        \right\| \right) \right\|_{L^1(\Omega; \R)}
        \left\| \exp\!\left( 2 c \!\smallint_0^s \left(
        1 + |U( Y_r^x )|^{\gamma_6} \right) dr \right)\right\|_{L^{\infty}(\Omega; \R)}
    \nonumber \\
    &\leq
    \E \! \left[
      \exp\!\left(
        ( 2 c )^p
        \big[ 1 +
          \left|U(x)\right|^{\frac{p-1}{p}}
          +
          \left\|
            Y_s^x - x
          \right\|^{ ( p - 1 ) }\!
        \big]
        \left\|
          Y_s^x - x
        \right\| \right)
    \right]
        \exp\!\left(
          2 c \smallint_0^s \big( 1 + \|U( Y_r^x )\|_{L^{\infty}(\Omega; \R)}^{\gamma_6} \big)
        \, dr \right).
        \label{eq:estimate.fist.term}
\end{align}
Next we estimate the two factors
  on the right-hand side of \eqref{eq:estimate.fist.term} separately.
  Using H\"older's inequality and
  assumption~\eqref{eq:assumption.on.phi.increment}
  shows
  for all
  $
    s \in (0,h]
  $,
  $
    x \in D
  $
  that
\begin{align}
  &
  \E\!\left[
      \exp\!\left(
        ( 2 c )^p
        \big[ 1 +
          \left|U(x)\right|^{\frac{p-1}{p}}
          +
          \left\|
            Y_s^x - x
          \right\|^{ ( p - 1 ) }\!
        \big]
        \left\|
          Y_s^x - x
        \right\| \right)
  \right]
  \nonumber \\ & =
  \E\!\left[
    \sum_{n=0}^{\infty} \tfrac{( 2 c )^{pn}}{n !}
        \big[ 1 +
          \left|U(x)\right|^{\frac{p-1}{p}}
          +
          \left\|
            \Phi(x, s, W_s ) - x
          \right\|^{ ( p - 1 ) }\!
        \big]^n
        \left\|
          \Phi(x,s,W_s) - x
        \right\|^n
  \right]
  \nonumber \\ & =
  \sum_{ n = 0 }^{ \infty }
  \left\| \tfrac{( 2 c )^{pn}}{n !}
        \big[ 1 +
          \left|U(x)\right|^{\frac{p-1}{p}}
          +
          \left\|
            \Phi(x, s, W_s ) - x
          \right\|^{ ( p - 1 ) }\!
        \big]^n
        \left\|
          \Phi(x,s,W_s) - x
        \right\|^n
        \right\|_{L^1(\Omega; \R)}
  \nonumber \\ & \leq
  \sum_{ n = 0 }^{ \infty }
  \tfrac{ ( 2 c )^{ p n }
  }{ n! }
  \left[ 1 +
    \left|U(x)\right|^{\frac{p-1}{p}}
    +
    \| \Phi(x,s,W_s) - x \|^{(p-1)}_{L^{\infty}(\Omega; \R^d)}
  \right]^n
  \left\|
    \Phi(x,s,W_s) - x
  \right\|^n_{L^n(\Omega; \R^d)}
  \nonumber \\ & \leq
  \sum_{n=0}^{\infty} \tfrac{( 2 c )^{pn}}{n !}
  \left[ 1 +
    \left|U(x)\right|^{\frac{p-1}{p}}
    +
    c^{ (p - 1 ) }
    ( 1 + \left| U(x) \right|^{\gamma_5})^{(p-1)}
  \right]^n
  \left[
    c \left(1+ \left| U(x) \right|^{\gamma_5}\right)
    \left\|
      \mu(x) s + \sigma(x) W_s
    \right\|_{L^n(\Omega; \R^d)}
  \right]^n
  \nonumber \\ & =
  \sum_{ n = 0 }^{ \infty }
  \tfrac{ ( 2 c )^{ p n } }{ n! }
  \left[
    c \big( 1 + \left|U(x)\right|^{\frac{p-1}{p}} \big) \big(1+ \left| U(x) \right|^{\gamma_5}\big)
    +
    c^p
    ( 1 + \left| U(x) \right|^{\gamma_5})^{p}
  \right]^n
  \E\big[
    \left\| \mu(x) s + \sigma(x) W_s
    \right\|^n
  \big]
  \nonumber \\ & =
  \E\!\left[
    \exp\!\left(
      \left[ 2c \right]^p
      \left[
        c \big( 1 + \left|U(x)\right|^{\frac{p-1}{p}} \big) \big(1+ \left| U(x) \right|^{\gamma_5}\big)
    +
    c^p
    ( 1 + \left| U(x) \right|^{\gamma_5})^{p}
      \right]
      \left\|
        \mu(x) s
        +
        \sigma(x) W_s
      \right\|
    \right)
  \right]
  .
\end{align}
Hence,
  assumption~\eqref{eq:mu_sigma_Ugrowth_assumption}
  and Lemma~\ref{l:exp.Gauss} yield
  for all
  $
    s \in (0,h]
  $,
  $
    x \in D
  $
  that
\begin{equation}
\begin{split}
  &
  \E\!\left[
      \exp\!\left(
        \left[ 2 c \right]^p
        \big[
          1 +
          \left|U(x)\right|^{\frac{p-1}{p}}
          +
          \left\|
            Y_s^x - x
          \right\|^{ ( p - 1 ) }\!
        \big]
        \left\|
          Y_s^x - x
        \right\| \right)
  \right]
  \\ & \leq
  \E\!\left[
    \exp\!\left(
      c \left[2c\right]^{p}
      \left[ 1 +
        \left|U(x)\right|^{\frac{p-1}{p}}
        + \left| U(x) \right|^{\gamma_5 }
        + \left| U(x) \right|^{(\gamma_5 + \frac{p-1}{p})}
    +
    \left[2c\right]^{(p-1)}
    ( 1 + \left| U(x) \right|^{p\gamma_5})
      \right]
      \left\|
        \mu(x) s
        +
        \sigma(x) W_s
      \right\|
    \right)
  \right]
  \\ & \leq
  \E\!\left[
    \exp\!\left(
      2^p c^{(p+1)}
      \left[ 1 +
        \frac{ c }{ s^{ \alpha(p-1)/p} }
        + \frac{ c^{ \gamma_5 } }{ s^{ \alpha \gamma_5} }
        + \frac{ c^{ ( \gamma_5 + 1 ) } }{
          s^{ \alpha(\gamma_5 +(p-1)/p)}
        }
    +
    \left[2c\right]^{(p-1)}
    \left(
      1 + \frac{ c^{ p \gamma_5 } }{ s^{ \alpha p\gamma_5} }
    \right)
      \right]
      \left\|
        \mu(x) s
        +
        \sigma(x) W_s
      \right\|
    \right)
  \right]
  \\ & \leq
  \E\!\left[
    \exp\!\left(
      2^p c^{(p+1)} \left[ \min(s,1)
      \right]^{- \alpha (  p \gamma_5 + 1 )
      }
      \left[
        4 c^{ ( \gamma_5 + 1 ) }
        +
        2^p
        c^{ ( p + p \gamma_5 - 1 ) }
      \right]
      \left\|
        \mu(x) s
        +
        \sigma(x) W_s
      \right\|
    \right)
  \right]
  \\ & \leq
  \E\!\left[
    \exp\!\left(
      2^{ ( 2 p + 2 ) }
      c^{ p ( \gamma_5 + 3 ) }
      \left[ \min(s,1)
      \right]^{- \alpha (  p \gamma_5 + 1 )
      }
      \left(
        \| \mu(x) \| s + \| \sigma(x) W_s \|
      \right)
    \right)
  \right]
  \\ & =
  \exp\!\left(
    2^{ ( 2 p + 2 ) }
    c^{ p ( \gamma_5 + 3 ) }
    \left[ \min(s,1)
      \right]^{- \alpha (  p \gamma_5 + 1 )
      }
    \|\mu(x)\| s
  \right)
  \E\!\left[
    \exp\!\left(
      2^{ ( 2 p + 2 ) }
      c^{ p ( \gamma_5 + 3 ) }
      \left[ \min(s,1)
      \right]^{- \alpha (  p \gamma_5 + 1 )
      }
      \left\| \sigma(x) W_s \right\|
    \right)
  \right]
  \\ & \leq
  \exp\!\left(
    2^{ ( 2 p + 2 ) }
    c^{ p ( \gamma_5 + 3 ) }
    \left[ \min(s,1)
      \right]^{- \alpha (  p \gamma_5 + 1 )
      }
    \|\mu(x)\| s
  \right)
  2
  \exp\!\left(
    \tfrac{
      s
      2^{  ( 4p + 4 ) }
      c^{ 2 p ( \gamma_5 + 3 ) }
      \,
      \left\| \sigma(x) \right\|^2_{\HS(\R^m, \R^d)}
    }{
      2
      \,
      \left[ \min(s,1)
      \right]^{ 2\alpha (  p \gamma_5 + 1 )
      }
    }
  \right)
  \\ & \leq
  2
  \exp\!\left(
    s \, 2^{ ( 4 p + 3 ) }
    c^{ 2 p ( \gamma_5 + 3 ) }
    \left[ \min(s,1)
      \right]^{- 2\alpha (  p \gamma_5 + 1 )
      }
    \left[
      \| \mu(x) \|
      +
      \| \sigma(x) \|^2_{ \HS( \R^m, \R^d ) }
    \right]
  \right)
  \\ & \leq
  2 \exp\!\left(
      s \,
      2^{ ( 4 p + 3 ) }
      \,
      c^{ 2 p ( \gamma_5 + 3 ) }
      \left[ \min(s,1)
      \right]^{- 2\alpha (  p \gamma_5 + 1 )
      }
      \left[
        c \left( 1 +
        \left|
          U(x)
        \right|^{ \gamma_0 } \right)
        +
        c^2
        \left( 1 + \left| U(x)
        \right|^{ \gamma_1 } \right)^2
      \right]
    \right)
  \\ & \leq
  2 \exp\!\left(
      s \,
      2^{ ( 4 p + 3 ) }
      \,
      c^{ 2 p ( \gamma_5 + 3 ) }
      \left[ \min(s,1)
      \right]^{- 2\alpha (  p \gamma_5 + 1 )
      }
      \left[
        c  +
        c^{ ( 1 + \gamma_0 ) }
        s^{ - \alpha \gamma_0 }
        +
        2c^2
         +  2 c^{ 2 ( 1 + \gamma_1 ) } s^{ - 2 \alpha \gamma_1 }
      \right]
    \right)
  \\ & \leq
  2
    \exp\!\left(
      s
      \,
      2^{ ( 4 p + 3 ) }
      \,
      c^{ 2 p ( \gamma_5 + 3 ) }
      \left[
        \min(s,1)
      \right]^{
        - \alpha (
          2 p \gamma_5 + 2
          +
          \gamma_0 + 2 \gamma_1
        )
      }
      \left[ 2 c^{ ( 1 + \gamma_0 ) } + 4 c^{ 2 ( 1 + \gamma_1 ) } \right]
      \right)
  \\ & \leq
  2
    \exp\!\left(
      s
      \,
      2^{ ( 4 p + 6 ) }
      \,
      c^{
        2 p ( \max( \gamma_0 / 2 , \gamma_1 ) + \gamma_5 + 4 )
      }
      \left[
        \min(s,1)
      \right]^{
        - \alpha (
          2 p \gamma_5 + 2
          +
          \gamma_0 + 2 \gamma_1
        )
      }
      \right).
  \label{eq:estimate.fist.term1}
\end{split}
\end{equation}
  Next we combine
  \eqref{eq:derivativeU1_2}
  with
  assumption~\eqref{eq:assumption.on.phi.increment}
  to obtain
  for all
  $
    r \in (0,h]
  $,
  $
    x \in D
  $
  that
\begin{equation}
\begin{split}
\label{eq:before.estimate.fist.term2}
&
  \|
    U( Y_r^x )
  \|_{ L^{ \infty }( \Omega; \R )
  }
\leq
  U(x) + \| U( Y_r^x ) - U(x) \|_{ L^{ \infty }( \Omega; \R ) }
\\ &
  \leq
      U(x) + \tfrac{ ( 2 c )^p }{ 2 }
      \left\|
        \left( 1 +
         \left| U(x) \right|^{(p-1)/p}
         +
         \| Y_r^x - x \|^{ ( p - 1 ) }
        \right)
        \left\| Y_r^x - x
        \right\|
      \right\|_{
        L^{ \infty }( \Omega; \R ) }
\\ & \leq
      U(x) + \tfrac{ ( 2 c )^p }{ 2 }
      \left[
        c \big( 1 +
         \left| U(x) \right|^{(p-1)/p} \big)\left( 1 + \left| U(x) \right|^{\gamma_5} \right)
        +
        c^p
        \left( 1 + \left| U(x) \right|^{\gamma_5} \right)^p
      \right]
 \\ &
  \leq
    U(x) + \tfrac{ (2c)^p }{2} \big[ 2c \max(1, U(x)) \cdot 2 \max(1, \left| U(x) \right|^{\gamma_5}) + \left[ 2c \right]^p \max(1, |U(x)|^{p\gamma_5}) \big]
  \\ &
   \leq
   U(x) + \tfrac{ (2c)^p }{2} \big[4c \max(1, \left| U(x) \right|^{(\gamma_5 +1)}) + \left[ 2c \right]^p \max(1, |U(x)|^{p\gamma_5}) \big]
   \\ &
   \leq
   U(x) + \tfrac{ (2c)^p }{2} \max(1, |U(x)|^{(p\gamma_5 +1)})\left[ 4c + (2c)^p \right]
   \\ &
   \leq U(x) + \tfrac{3}{2} \left[2c\right]^{2p} \max(1, |U(x)|^{(p\gamma_5 +1)})
   \\ &
   \leq 2^{(2p +1)} c^{2p} \max(1, |U(x)|^{(p\gamma_5 +1)}).
\end{split}
\end{equation}
Therefore, it holds for all
  $
    s \in (0,h]
  $,
  $ x \in D $
  that
\begin{equation}  \begin{split}
\label{eq:estimate.fist.term2}
2 c \smallint_0^s \big( 1 + \|U( Y_r^x )\|_{L^{\infty}(\Omega; \R)}^{\gamma_6}\big) \, dr
 & \leq
2 s c + 2 s c \left( 2^{(2p +1)} c^{2p} \max \!\big(1, |U(x)|^{(p\gamma_5 +1)}\big) \right)^{\gamma_6}
\\ & \leq
2 s c + 2 s c \,2^{(2p +1)\gamma_6} \, c^{2p\gamma_6} \max\!\big(1, |U(x)|^{(p\gamma_5 +1)\gamma_6}\big)
\\ & \leq
2 s c +  s \, 2^{[(2p +1)\gamma_6 +1]}
\,
c^{
  ( 2 p \gamma_6 + 1 + ( p \gamma_5 + 1 ) \gamma_6 )
}
\max\!\big(1, s^{- \alpha(p\gamma_5 +1)\gamma_6}\big)
\\ & \leq
s \, 2^{[(2p +1)\gamma_6 +2]}
\,
c^{
  [ ( p ( \gamma_5 + 2 ) + 1 ) \gamma_6 + 1 ]
}
\left[\min(s, 1)\right]^{-\alpha (p \gamma_5 + 1)\gamma_6}.
\end{split}\end{equation}
Inserting \eqref{eq:estimate.fist.term1} and \eqref{eq:estimate.fist.term2} into \eqref{eq:estimate.fist.term} then shows for all
  $
    s \in (0,h]
  $,
  $
    x \in D
  $
  that
  \begin{equation}
  \begin{split}
    &
    \E\!\left[\exp\!\left(2U(Y_s^x)
                   +2\smallint_0^{s}e^{-\rho r}\bar{U}(Y_r^x)\,dr
          -2U(x)\right)\right]
\\ &
  \leq
    2
    \exp\!\left(
      s
      \,
      2^{ ( 4 p + 6 ) }
      \,
      c^{
        2 p ( \max( \gamma_0 / 2 , \gamma_1 ) + \gamma_5 + 4 )
      }
      \left[
        \min(s, 1)
      \right]^{
        - \alpha(
          2 p \gamma_5  + 2
          +
          \gamma_0 + 2 \gamma_1
        )
      }
    \right)
  \\ & \quad
    \cdot
    \exp\!\left(
      s \, 2^{[( 2p+1 )\gamma_6 + 2]}\,
      c^{
        [ ( p ( \gamma_5 + 2 ) + 1 ) \gamma_6 + 1 ]
      }
      \left[\min(s, 1)\right]^{-\alpha (p \gamma_5 + 1)\gamma_6}
    \right)
  \\ & \leq
    2
    \exp\!\left(
      \left[
        2^{ (4 p + 6 ) }
        \,
        c^{
          2 p ( \max( \gamma_0 / 2 , \gamma_1 ) + \gamma_5 + 4 )
        }
        +
        2^{[(2p +1)\gamma_6 + 2 ]}
        \,
        c^{
          [ p \gamma_6 ( \gamma_5 + 3 ) + 1 ]
        }
      \right]
      \tfrac{ s }{
      \left[ \min(s, 1) \right]^{ \alpha
        \left[
          ( p \gamma_5 + 1 ) ( \gamma_6 + 2 )
          +  \gamma_0 + 2 \gamma_1
        \right]
      }
      }
    \right)
  \\ & \leq
    2
    \exp\!\left(
      2^{
        [
          1 + (2 p +3) (\gamma_6 + 2 )
        ]
      }
      \,
      c^{
        p
        \left[
          \max( \gamma_0, 2 \gamma_1 )
          +
          ( \gamma_5 + 4 )
          ( \gamma_6 + 2 )
        \right]
      }
      \left[ \min(s, 1) \right]^{ - \alpha
        \left[
          ( p \gamma_5 + 1 ) ( \gamma_6 + 2 )
          +  \gamma_0 + 2 \gamma_1
        \right]
      }
      s
    \right).
  \end{split}
  \end{equation}
  Therefore,
  we obtain for all
  $
    s \in (0,h]
  $,
  $
    x \in D
  $
  that
  \begin{equation}  \begin{split}
  \label{eq:L2.expU}
    \left\|
      \exp\!\left(
        U( Y_s^x )
        +
        \smallint_0^s
        \tfrac{
          \bar{U}( Y_r^x )
        }{
          e^{\rho r}
        }
        \, dr
      \right)
    \right\|_{L^2(\Omega;\R)}\!
  \leq
    \sqrt{ 2 }
    \exp\!\left(
      \tfrac{
        2^{
          (2 p + 3 ) (\gamma_6 + 2 )
        }
        \,
        c^{
          p
          \left[
            \max( \gamma_0, 2 \gamma_1 )
            +
            ( \gamma_5 + 4 )
            ( \gamma_6 + 2 )
          \right]
        }
        \,
        s
      }{
        \left[
          \min(s, 1)
        \right]^{
          \alpha
          \left[
            ( p \gamma_5 + 1 ) ( \gamma_6 + 2 )
            + \gamma_0 + 2 \gamma_1
          \right]
        }
      }
  \right)
  e^{ U(x) }
  \,
  .
  \end{split}
  \end{equation}
  Moreover,
  the fact that
  $
  \forall \, r \in [2,\infty), s \in [0,h], x \in \R^d
  \colon
  \|
    \sigma(x) W_s
  \|_{
    L^r( \Omega; \R^m )
  }
  \leq
  \sqrt{ s r ( r - 1 ) / 2 }
  \left\|
    \sigma(x)
  \right\|_{
    \HS( \R^m, \R^d )
  }
  $
  (see, e.g., Lemma 7.7 in Da Prato \& Zabczyk~\cite{dz92}),
  assumption~\eqref{eq:mu_sigma_Ugrowth_assumption},
  and
  assumption~\eqref{eq:assumption.on.phi.increment}
  imply for all
  $
    r \in [2,\infty)
  $,
  $
    s \in (0,h]
  $,
  $
    x \in D
  $
  that
\begin{equation} \begin{split}
  \label{eq:increment.Y}
    &\left\|Y_s^x-x\right\|_{L^r(\Omega;\R^d)}
    =
    \left\|\Phi(x,s,W_s)-x\right\|_{L^r(\Omega;\R^d)}
    \leq
    c \left( 1 + \left| U(x) \right|^{\gamma_5} \right)
    \left\|
      \mu(x) s + \sigma(x) W_s
    \right\|_{
      L^r( \Omega; \R^d )
    }
  \\ & \leq
    c
    \left(
      1
      +
      c^{ \gamma_5 }
      s^{ -\alpha \gamma_5 }
    \right)
    \left(
      \left\| \mu(x) \right\| s
      +\!
      \sqrt{s r (r - 1 )/2}
      \left\|
        \sigma(x)
      \right\|_{
        \HS( \R^m, \R^d )
      }
    \right)
   \\ &
  \leq
    c
    \left(
      1
      +
      c^{ \gamma_5 }
      s^{ -\alpha \gamma_5 }
    \right)
    \left(
      c s
      \left(
        1 +
        c^{ \gamma_0 }
        s^{ -\alpha \gamma_0 }
      \right)
      +
      c
      \sqrt{ s r ( r - 1 )/2 }
      \left(
        1 +
        c^{ \gamma_1 }
        s^{ - \alpha \gamma_1 }
      \right)
    \right)
    \\&
    \leq
    2
    c^{
      ( 2 + \max( \gamma_0, \gamma_1 ) + \gamma_5 )
    }
    \left[
      \min(s, 1)
    \right]^{ - \alpha \gamma_5 }
    \left(
      1 + \sqrt{ r(r-1)/2 }
    \right)
    \sqrt{s}
    \max\!\left(
      \sqrt{s}
      \left(
        1 +
        s^{ -\alpha \gamma_0 }
      \right)
      ,
      1 + s^{ - \alpha \gamma_1 }
    \right)
 \\ & \leq
    2
    c^{
      ( 2 + \max( \gamma_0, \gamma_1 ) + \gamma_5 )
    }
    \left[ \min(s, 1)\right]^{-\alpha \gamma_5}  r \sqrt{s} \left[\max(s, 1)\right]^{1/2}
    \max\!\left(
      1 +
      s^{ -\alpha \gamma_0 }
      ,
      1 + s^{ - \alpha \gamma_1 }
    \right)
\\ & \leq
    4 r
    c^{
      ( 2 + \max( \gamma_0, \gamma_1 ) + \gamma_5 )
    }
    s^{1/2}
    \left[
      \max(s, 1)
    \right]^{1/2}
    \left[
      \min(s,1)
    \right]^{ -\alpha(\gamma_0 + \gamma_1+\gamma_5) }
\\ & =
    4 r
    c^{
      ( 2 + \max( \gamma_0, \gamma_1 ) + \gamma_5 )
    }
    \max(s, 1)
    \left[
      \min(s,1)
    \right]^{
      [
        1 / 2 - \alpha ( \gamma_0 + \gamma_1 + \gamma_5 )
      ]
    }
    .
\end{split} \end{equation}
  Combining \eqref{eq:assumption.on.phi.increment} and \eqref{eq:increment.Y} with H\"older's inequality
  and inequality~\eqref{eq:derivativeU1} yields for all
  $
    r \in [2,\infty)
  $,
  $
    i \in \{ 0, 1, 2 \}
  $,
  $
    s \in (0, h ]
  $,
  $
    x \in D
  $
  that
  \begin{equation}  \begin{split}  \label{eq:increment.Ui}
    &
      \big\|
        U^{ (i) }( Y_s^x ) - U^{ (i) }( x )
      \big\|_{
        L^r( \Omega; L^{ (i) }( \R^d, \R) )
      }
    \\ &
    \leq
    \left\|
      \tfrac{ ( 2 c )^p
      }{ 2 }
    \left( 1 +
    \left|
      U( x )
    \right|^{
      \frac{ \max(p - i - 1,0) }{ p }
    }
    +
    \left\| Y_s^x - x
    \right\|^{
      \max\left( p - i-1,0 \right)
    }
    \right)
    \left\| Y_s^x - x \right\|
     \right\|_{L^{r}(\Omega;\R)}
    \\&
    \leq
    \tfrac{(2c)^p}{2}
    \left( \left\| Y_s^x - x
    \right\|_{L^{r}(\Omega; \R^d)} +
      \left|
      U( x )
    \right|^{
      \frac{ \max(p - i - 1,0) }{ p }
    }  \left\| Y_s^x - x
    \right\|_{L^{r}(\Omega; \R^d)}
    +
    \left\| Y_s^x - x
     \right\|_{L^{r\cdot\max(p-i, 1)}(\Omega; \R^d)}^{
      \max\left( p - i,1 \right)
    }
    \right)
  \\&
    \leq
    \tfrac{(2c)^p}{2}
    \left(
      1 +
      \tfrac{
        c
      }{
        s^{
          \alpha \max(p - i - 1,0) / p
        }
      }
      +
    \left\|
      Y_s^x - x
     \right\|^{\max(p-i-1, 0)}_{L^{r\cdot\max(p-i, 1)}(\Omega; \R^d)}
    \right)
    \left\| Y_s^x - x
     \right\|_{L^{r\cdot\max(p-i, 1)}(\Omega; \R^d)}
 \\ & \leq
     \tfrac{ ( 2 c )^p }{ 2 }
     \left[
       \tfrac{
         2 c
       }{
         \left[ \min(s, 1) \right]^{ \alpha }
       }
       +
       [crp]^{ \max(p-i-1,0) }
     \right]
     4 r p \,
     c^{
       ( 2 + \max( \gamma_0, \gamma_1 ) + \gamma_5 )
     }
     \max(s, 1)
     \left[ \min(s,1)
     \right]^{
       [
         1 / 2 - \alpha ( \gamma_0 + \gamma_1 + \gamma_5 )
       ]
     }
  \\ & \leq
    2^{ ( p + 1 ) }
    \,
     c^{
       ( p + 2 + \max( \gamma_0, \gamma_1 ) + \gamma_5 )
     }
    \left[
      2 c r p
      +
      \left[ c r p \right]^p
    \right]
     \max(s, 1)
     \left[ \min(s,1)
     \right]^{
       [
         1 / 2 - \alpha ( \gamma_0 + \gamma_1 + \gamma_5 + 1 )
       ]
     }
  \\ & \leq
    6
    \,
     c^{
       ( 2 p + 2 + \max( \gamma_0, \gamma_1 ) + \gamma_5 )
     }
    \left[ 2 r p \right]^p
    \max(s, 1)
    \left[
      \min(s,1)
    \right]^{
      [
        1 / 2 - \alpha ( \gamma_0 + \gamma_1 + \gamma_5 + 1 )
      ]
    }
   \, .
  \end{split}
  \end{equation}
  This, the assumption that
  $
    U \in C_{ p, c }^3( \mathbb{R}^d, [0,\infty) )
  $,
  Lemma~\ref{lemma:equivalent0},
  H\"{o}lder's inequality, 
  assumption~\eqref{eq:mu_sigma_Ugrowth_assumption}, 
  and assumption~\eqref{eq:assumption.on.phi.mue}
  show for all
  $ s \in (0, h ] $,
  $
    x \in D
  $
  that
  \begin{equation}
  \begin{split}
  \label{eq:second.summand}
      &
      \left\|
        U'( Y_s^x )
        \left(
          \tfrac{ \partial }{ \partial s } \Phi
        \right)\!( x, s, W_s ) -
        U'(x) \, \mu(x)
      \right\|_{
        L^2( \Omega; \R )
      }
    \\ & \leq
      \left\|
        \left\|
          U'( Y^x_s )
        \right\|_{L(\R^d,\R)}
        \left\|
          \left(
            \tfrac{ \partial }{ \partial s } \Phi
          \right)\!( x, s, W_s ) -
          \mu(x)
        \right\|
      \right\|_{
        L^2( \Omega; \R )
      }
      +
      \left\|
        U'( Y_s^x ) - U'(x)
      \right\|_{
        L^2( \Omega; L(\R^d,\R) )
      }
      \left\| \mu(x) \right\|
    \\ & \leq
      \left\|
        U'(x)
      \right\|_{L(\R^d,\R)}
      \left\|
        \left(
          \tfrac{ \partial }{ \partial s } \Phi
        \right)\!( x, s, W_s ) -
        \mu(x)
      \right\|_{
        L^2( \Omega; \R^d )
      }
    \\ & \quad
        +
        \left\|
          U'( Y_s^x ) - U'( x )
        \right\|_{
          L^4( \Omega; L( \R^d, \R ) )
        }
        \left[
          \left\| \mu(x) \right\|
          +
          \left\|
            \left(
              \tfrac{ \partial }{ \partial s } \Phi
            \right)\!(x, s, W_s )
            -
            \mu(x)
          \right\|_{
            L^4( \Omega; \R^d )
          }
        \right]
  \\ &
      \leq
        c
        \left[1 + U(x)\right]^{ \frac{ ( p - 1 ) } { p } }
          c s^{\gamma_2}
  \\ & \quad
    +
    6
    \,
     c^{
       ( 2 p + 2 + \max( \gamma_0, \gamma_1 ) + \gamma_5 )
     }
    \left[ 8 p \right]^p
    \max(s, 1)
    \left[
      \min(s,1)
    \right]^{
      [
        1 / 2 - \alpha ( \gamma_0 + \gamma_1 + \gamma_5 + 1 )
      ]
    }
    \left[
      c \left( 1 + c^{ \gamma_0 } s^{ - \alpha \gamma_0 } \right)
      + c s^{ \gamma_2 }
    \right]
  \\ & \leq
    c^2
    s^{ \gamma_2 }
    \left[ 1 + c s^{ - \alpha } \right]^{
      ( p - 1 ) / p
    }
\\ &
    +
    6
    \,
     c^{
       ( 2 p + 3 + \max( \gamma_0, \gamma_1 ) + \gamma_0 + \gamma_5 )
     }
    \left[ 8 p \right]^p
    \max(s, 1)
    \left[
      \min(s,1)
    \right]^{
      [
        1 / 2 - \alpha ( 2 \gamma_0 + \gamma_1 + \gamma_5 + 1 )
      ]
    }
    \left[
      2
      + s^{ \gamma_2 }
    \right]
  \\ & \leq
     2
     c^3
     s^{ \gamma_2 }
     \left[ \min(s, 1) \right]^{ - \alpha }
   +
    18
    \,
    c^{
      ( 2 p + 3 + \max( \gamma_0, \gamma_1 ) + \gamma_0 + \gamma_5 )
    }
    \left[ 8 p \right]^p
    \left[
      \max(s,1)
    \right]^{ ( 1 + \gamma_2 ) }
    \left[
      \min(s,1)
    \right]^{
      [
        1 / 2 - \alpha ( 2 \gamma_0 + \gamma_1 + \gamma_5 + 1 )
      ]
    }
  \\ & =
     2
     c^3
     \left[
       \max( s, 1 )
     \right]^{ \gamma_2 }
     \left[ \min(s, 1) \right]^{ ( \gamma_2 - \alpha ) }
 \\ &
   +
    18
    \,
    c^{
      ( 2 p + 3 + \max( \gamma_0, \gamma_1 ) + \gamma_0 + \gamma_5 )
    }
    \left[ 8 p \right]^p
    \left[
      \max(s,1)
    \right]^{ ( 1 + \gamma_2 ) }
    \left[
      \min(s,1)
    \right]^{
      [
        1 / 2 - \alpha ( 2 \gamma_0 + \gamma_1 + \gamma_5 + 1 )
      ]
    }
 \\ & \leq
    20
    \left[ 8 p \right]^p
    c^{
      ( 2 p + 3 + \max( \gamma_0, \gamma_1 ) + \gamma_0 + \gamma_5 )
    }
    \left[
      \max(s,1)
    \right]^{ ( 1 + \gamma_2 ) }
    \left[
      \min(s,1)
    \right]^{
      [
        \min( \gamma_2, 1 / 2 ) - \alpha ( 2 \gamma_0 + \gamma_1 + \gamma_5 + 1 )
      ]
    }
    .
  \end{split}
  \end{equation}
  Analogously, \eqref{eq:increment.Ui}, the assumption that
  $ U \in C_{p,c}^3( \mathbb{R}^d, [0,\infty) ) $,
  Lemma~\ref{lemma:equivalent0},
  H\"{o}lder's inequality, 
  assumption~\eqref{eq:mu_sigma_Ugrowth_assumption},
  and assumption~\eqref{eq:assumption.on.phi.sigma}
  show for all
  $ s \in (0, h ] $,
  $
    x \in D
  $
  that
  \begin{equation}
  \label{eq:UprimePhiy}
  \begin{split}
    &
    \big\|
      U'( Y_s^x )
      \,
      \big(
        \tfrac{ \partial }{ \partial y } \Phi
      \big)( x, s, W_s )
      -
      U'(x)
      \, \sigma(x)
    \big\|_{
      L^4( \Omega; L(\R^m, \R) )
    }
  \\ &
  \leq
    \left\|
      \left\|
        U'( Y^x_s )
      \right\|_{
        L( \R^d, \R )
      }
      \big\|
        \big(
          \tfrac{ \partial }{ \partial y } \Phi
        \big)( x, s, W_s )
        -
        \sigma( x )
      \big\|_{
        L( \R^m, \R^d )
      }
    \right\|_{
      L^4( \Omega; \R )
    }
  \\ & \quad
    +
    \left\|
      U'( Y_s^x ) - U'(x)
    \right\|_{
      L^4( \Omega; L(\R^d,\R) )
    }
    \left\|
      \sigma(x)
    \right\|_{
      L( \R^m, \R^d )
    }
  \\ &
  \leq
    \left\|
      U'(x)
    \right\|_{
      L( \R^d, \R )
    }
    \big\|
      \big(
        \tfrac{ \partial }{ \partial y } \Phi
      \big)( x, s, W_s )
      -
      \sigma( x )
    \big\|_{
      L^4( \Omega; L(\R^m, \R^d ) )
    }
  \\ & \quad
    +
    \left\|
      U'( Y_s^x ) - U'(x)
    \right\|_{
      L^8( \Omega; L(\R^d,\R) )
    }
    \left[
      \left\|
        \sigma(x)
      \right\|_{
        L( \R^m, \R^d )
      }
      +
        \big\|
          \big(
            \tfrac{ \partial }{ \partial y } \Phi
          \big)( x, s, W_s ) -
          \sigma(x)
        \big\|_{
          L^8( \Omega; L( \R^m, \R^d ) )
        }
      \right]
  \\ & \leq
        c
        \left[
          1 +
          U(x)
        \right]^{ ( p - 1 ) / p }
        c s^{\gamma_3}
  \\ & \quad
   +
     6
     \,
     c^{
       ( 2 p + 2 + \max( \gamma_0, \gamma_1 ) + \gamma_5 )
     }
    \left[ 16 p \right]^p
    \max(s, 1)
    \left[
      \min(s,1)
    \right]^{
      [
        1 / 2 - \alpha ( \gamma_0 + \gamma_1 + \gamma_5 + 1 )
      ]
    }
          \left[
            c
            \left(1 + c^{ \gamma_1 } s^{ - \alpha\gamma_1}\right)
            +
            c s^{ \gamma_3 }
          \right]
  \\ & \leq
        c^2
        s^{ \gamma_3 }
        \left[
          1 + c s^{ - \alpha }
        \right]^{
          ( p - 1 ) / p
        }
  \\ & \quad
   +
     6
     \,
     c^{
       ( 2 p + 3 + \max( \gamma_0, \gamma_1 ) + \gamma_1 + \gamma_5 )
     }
    \left[ 16 p \right]^p
    \max(s, 1)
    \left[
      \min(s,1)
    \right]^{
      [
        1 / 2 - \alpha ( \gamma_0 + 2 \gamma_1 + \gamma_5 + 1 )
      ]
    }
          \left[
            2
            +
            s^{ \gamma_3 }
          \right]
  \\ & \leq
       2 c^3 s^{ \gamma_3 }
       \left[
         \min(s,1)
       \right]^{ - \alpha }
       +
       18
       \,
       c^{
         ( 2 p + 3 + \max( \gamma_0, \gamma_1 ) + \gamma_1 + \gamma_5 )
       }
       \left[ 16 p \right]^p
       \left[
         \max(s,1)
       \right]^{ ( 1 + \gamma_3 ) }
       \left[
         \min(s,1)
       \right]^{
         [
           1 / 2 - \alpha ( \gamma_0 + 2 \gamma_1 + \gamma_5 + 1 )
         ]
       }
  \\ & \leq
       20
       \left[ 16 p \right]^p
       c^{
         ( 2 p + 3 + \max( \gamma_0, \gamma_1 ) + \gamma_1 + \gamma_5 )
       }
       \left[
         \max(s,1)
       \right]^{ ( 1 + \gamma_3 ) }
       \left[
         \min(s,1)
       \right]^{
         [
           \min( \gamma_3, 1 / 2 ) - \alpha ( \gamma_0 + 2 \gamma_1 + \gamma_5 + 1 )
         ]
       }
       .
  \end{split}
  \end{equation}
In the next step we note
  for all
  $
    A_1, A_2 \in \R^{ d \times m }
  $,
  $
    B_1, B_2 \in \R^{ d \times d }
  $
  that
  \begin{align}
    \nonumber
    &
    \left|
    \tr\left(
      A_1 A_1^{ * } B_1 -
      A_2 A_2^{ * } B_2
    \right)
    \right|
    =
    \left|
    \tr\left(A_1^{*}B_1A_1-A_2^{*}B_2A_2\right)
    \right|
    =
    \big|
    \left\langle A_1,B_1A_1\right\rangle_{\HS(\R^m,\R^d)}
    -
    \left\langle A_2,B_2A_2\right\rangle_{\HS(\R^m,\R^d)}
    \big|
  \\ &
  \nonumber
    =
    \left|
    \left\langle A_2,(B_1-B_2)A_2\right\rangle_{\HS(\R^m,\R^d)}
    +
    \left\langle A_1-A_2,B_1A_1\right\rangle_{\HS(\R^m,\R^d)}
    +
    \left\langle A_2,B_1(A_1-A_2)\right\rangle_{\HS(\R^m,\R^d)}
    \right|
  \\ &
  \nonumber
  \leq
    \left\|
      A_2
    \right\|_{
      \HS( \R^m, \R^d )
    }
    \left\|
      ( B_1 - B_2 ) A_2
    \right\|_{
      \HS( \R^m, \R^d )
    }
    +
    \left\|
      A_1 - A_2
    \right\|_{
      \HS( \R^m, \R^d )
    }
    \left\|
      B_1 A_1
    \right\|_{
      \HS( \R^m, \R^d )
    }
  \\ &
  \quad
    +
    \left\|
      A_2
    \right\|_{\HS(\R^m,\R^d)}
    \left\|
      B_1
      ( A_1 - A_2 )
    \right\|_{\HS(\R^m,\R^d)}
  \nonumber
  \\ &
  \leq
      \left\|
        B_1 - B_2
      \right\|_{
        L( \R^d )
      }
      \left\|
        A_2
      \right\|_{
        \HS( \R^m, \R^d )
      }^2
    +
    \left\|
      A_1 - A_2
    \right\|_{
      \HS( \R^m, \R^d )
    }
      \left\|
        B_1
      \right\|_{
        L( \R^d )
      }
    \left[
      \left\|
        A_1
      \right\|_{
        \HS( \R^m, \R^d )
      }
      +
      \left\|
        A_2
      \right\|_{\HS(\R^m,\R^d)}
    \right]
  \\ &
  \nonumber
  \leq
      \left\|
        B_1 - B_2
      \right\|_{
        L( \R^d )
      }
      \left\|
        A_2
      \right\|_{
        \HS( \R^m, \R^d )
      }^2
  \\ & \nonumber
  \quad
    +
    \left[
      \left\|
        A_1 - A_2
      \right\|_{
        \HS( \R^m, \R^d )
      }^2
      +
      2
      \left\|
        A_1 - A_2
      \right\|_{
        \HS( \R^m, \R^d )
      }
      \left\|
        A_2
      \right\|_{
        \HS(\R^m,\R^d)
      }
    \right]
    \left[
      \left\|
        B_1 - B_2
      \right\|_{
        L( \R^d )
      }
      +
      \left\|
        B_2
      \right\|_{
        L( \R^d )
      }
    \right]
    .
  \end{align}
  Next we apply this inequality
  with
  $
    A_1
    =
    (
      \tfrac{ \partial }{ \partial y } \Phi
    )(x,s,W_s)
  $,
  $
    A_2 = \sigma(x)
  $,
  $
    B_1
    =
    ( \Hess U )( Y_s^x )
  $,
  and
  $
    B_2 = (\Hess U)( x )
  $
  for
  $ s \in [0,h] $,
  we take expectations, we apply H\"older's inequality, we use the assumption that
  $U \in C_{p,c}^3( \mathbb{R}^d, [0,\infty) ) $,
  we use Lemma~\ref{lemma:equivalent0}~\eqref{eq:bounded_derivatives}
  (with $d = d, n = 2, c =c, p = p, x = x, y = w, V = U, i = 1, z_1 = v, t = 0$
  for $x \in D, v, w \in \{ u \in \R^d \colon \| u \| \leq 1 \}$ 
  in the notation of Lemma~\ref{lemma:equivalent0}~\eqref{eq:bounded_derivatives}),
  and we apply inequalities~\eqref{eq:mu_sigma_Ugrowth_assumption},
  \eqref{eq:assumption.on.phi.sigma},
  and \eqref{eq:increment.Ui} to obtain
  for all
  $
    s \in (0,h]
  $,
  $
    x \in D
  $
  that
  \begin{align}
    \nonumber
    & \left\|
      \tr\!\Big(
       \big(
         \tfrac{ \partial }{ \partial y } \Phi
       \big)( x, s, W_s )
       \big[
         \big(\tfrac{\partial}{\partial y}\Phi\big)(x,s,W_s)
       \big]^*
       (\operatorname{Hess}U)(
         Y_s^x
       )
       -
      \sigma(x) \,
      \sigma(x)^*
      \,
      ( \operatorname{Hess} U )(x)
     \Big)
     \right\|_{L^{2}(\Omega;\R)}
  \\
  \nonumber
  & \leq
    \left\|
      ( \Hess U )( Y_s^x )
      -
      ( \Hess U )( x )
    \right\|_{
      L^2( \Omega; L( \R^d )
      )
    }
    \left\|
      \sigma(x)
    \right\|_{
      \HS( \R^m, \R^d )
    }^2
  \\ &
  \nonumber
   \quad
  +
    \Big\|
      \big\|
        (
          \tfrac{ \partial }{ \partial y } \Phi
        )( x, s, W_s )
        -
        \sigma(x)
      \big\|_{
        \HS( \R^m, \R^d )
      }^2
      +
      2
      \,
      \big\|
        (
          \tfrac{ \partial }{ \partial y }
        \Phi)(x,s,W_s)
        -
        \sigma(x)
      \big\|_{\HS(\R^m,\R^d)}
      \left\|
        \sigma(x)
      \right\|_{
        \HS(\R^m,\R^d)
      }
    \Big\|_{
      L^4( \Omega; \R )
    }
  \\&
  \nonumber
  \quad \cdot
    \left[
      \left\|
        ( \Hess U )( Y_s^x )
        -
        ( \Hess U )( x )
      \right\|_{
        L^4( \Omega; L( \R^d )
        )
      }
      +
      \left\|
        ( \Hess U )( x )
      \right\|_{
        L( \R^d )
      }
    \right]
 \\  \nonumber
 & \leq
    6
    \,
     c^{
       ( 2 p + 2 + \max( \gamma_0, \gamma_1 ) + \gamma_5 )
     }
    \left[ 4 p \right]^p
    \max(s, 1)
    \left[
      \min(s,1)
    \right]^{
      [
        1 / 2 - \alpha ( \gamma_0 + \gamma_1 + \gamma_5 + 1 )
      ]
    }
    2 c^2
    \left(
      1 +
      c^{ 2 \gamma_1 }
      s^{
        - 2 \alpha \gamma_1
      }
    \right)
  \\ & \quad
  \nonumber
  +
  \left[
    c^2
    s^{2\gamma_3}
    +
    2
    c
    s^{\gamma_3}
    c \,
    \left( 1 + c^{ \gamma_1 } s^{ - \alpha \gamma_1 } \right)
  \right]
  \\ &
  \nonumber
   \quad
  \cdot
  \left[
    6
    \,
     c^{
       ( 2 p + 2 + \max( \gamma_0, \gamma_1 ) + \gamma_5 )
     }
    \left[ 8 p \right]^p
    \max(s, 1)
    \left[
      \min(s,1)
    \right]^{
      [
        1 / 2 - \alpha ( \gamma_0 + \gamma_1 + \gamma_5 + 1 )
      ]
    }
    +
    c
    \left[
      1 + U(x)
    \right]^{
      \max( p - 2 , 0 ) / p
    }
  \right]
\\ & \leq
\label{eq:fourth.summand}
    24
    \left[ 4 p \right]^p
     c^{
       ( 2 p + 4 + \max( \gamma_0, \gamma_1 ) + 2 \gamma_1 + \gamma_5 )
     }
    \max(s, 1)
    \left[
      \min(s,1)
    \right]^{
      [
        1 / 2 - \alpha ( \gamma_0 + 3 \gamma_1 + \gamma_5 + 1 )
      ]
    }
  \\ & \quad
  \nonumber
  +
  c^{ ( 2 + \gamma_1 ) }
  \,
  s^{ \gamma_3 }
  \left[
    s^{ \gamma_3 }
    +
    2
    +
    2 s^{ - \alpha \gamma_1 }
  \right]
  \\ &
  \nonumber
   \quad
  \cdot
  \left[
    6
    \,
     c^{
       ( 2 p + 2 + \max( \gamma_0, \gamma_1 ) + \gamma_5 )
     }
    \left[ 8 p \right]^p
    \max(s, 1)
    \left[
      \min(s,1)
    \right]^{
      [
        1 / 2 - \alpha ( \gamma_0 + \gamma_1 + \gamma_5 + 1 )
      ]
    }
    +
    2 c^2
    \left[
      \min( s, 1 )
    \right]^{ - \alpha }
  \right]
\\ & \leq
\nonumber
    24
    \left[ 4 p \right]^p
     c^{
       ( 2 p + 4 + \max( \gamma_0, \gamma_1 ) + 2 \gamma_1 + \gamma_5 )
     }
    \max(s, 1)
    \left[
      \min(s,1)
    \right]^{
      [
        1 / 2 - \alpha ( \gamma_0 + 3 \gamma_1 + \gamma_5 + 1 )
      ]
    }
  \\ & \quad
  \nonumber
  +
    5 \,
     c^{
       ( 2 p + 4 + \max( \gamma_0, \gamma_1 ) + \gamma_1 + \gamma_5 )
     }
  \left[
    \max( s, 1 )
  \right]^{
    ( 1 + 2 \gamma_3 )
  }
    \left[
      \min(s,1)
    \right]^{
      [
        \gamma_3 - \alpha ( \gamma_0 + 2 \gamma_1 + \gamma_5 + 1 )
      ]
    }
  \left[
    6
    \left[ 8 p \right]^p
    +
    2
  \right]
\\ & \leq
\nonumber
    24
    \left[ 4 p \right]^p
     c^{
       ( 2 p + 4 + \max( \gamma_0, \gamma_1 ) + 2 \gamma_1 + \gamma_5 )
     }
    \max(s, 1)
    \left[
      \min(s,1)
    \right]^{
      [
        1 / 2 - \alpha ( \gamma_0 + 3 \gamma_1 + \gamma_5 + 1 )
      ]
    }
  \\ & \quad
  \nonumber
  +
    35
    \left[ 8 p \right]^p
     c^{
       ( 2 p + 4 + \max( \gamma_0, \gamma_1 ) + \gamma_1 + \gamma_5 )
     }
  \left[
    \max( s, 1 )
  \right]^{
    ( 1 + 2 \gamma_3 )
  }
    \left[
      \min(s,1)
    \right]^{
      [
        \gamma_3 - \alpha ( \gamma_0 + 2 \gamma_1 + \gamma_5 + 1 )
      ]
    }
\\ & \leq
\nonumber
    47
    \left[ 8 p \right]^p
     c^{
       ( 2 p + 4 + \max( \gamma_0, \gamma_1 ) + 2 \gamma_1 + \gamma_5 )
     }
  \left[
    \max( s, 1 )
  \right]^{
    ( 2 \gamma_3 + 1 )
  }
    \left[
      \min(s,1)
    \right]^{
      [
        \min( \gamma_3 , 1 / 2 )
        - \alpha ( \gamma_0 + 3 \gamma_1 + \gamma_5 + 1 )
      ]
    }
  .
  \end{align}
  Furthermore, we observe that the fact that
  $
    \forall \, a , b \in \R^m \colon
    \left|
      \| a \|^2 - \| b \|^2
    \right|
    \leq
    \| a - b \|
    \left(
      \left\| a - b
      \right\|
      +
      2 \left\| b \right\|
    \right)
  $,
  H\"older's inequality, 
  inequality~\eqref{eq:UprimePhiy},
  and inequality~\eqref{eq:mu_sigma_Ugrowth_assumption}
  prove 
  for all
  $
    s \in (0,h]
  $,
  $
    x \in D
  $
  that
  \begin{equation}  \begin{split}
  \label{eq:Uprime.Phiy.square}
    &
      \left\|
        \big\|
          \big[
            \big(
              \tfrac{ \partial }{ \partial y } \Phi
            \big)( x,  s, W_s)
          \big]^{ * }
          ( \nabla U )( Y^x_s )
        \big\|^2
        -
        \left\|
          \sigma(x)^*
          ( \nabla U )( x )
        \right\|^2
      \right\|_{
        L^2( \Omega; \R )
      }
  \\ & \leq
    \big\|
      \big[
        \big(
          \tfrac{ \partial }{ \partial y } \Phi
        \big)( x, s, W_s )
      \big]^*
      ( \nabla U )( Y^x_s )
      -
      \sigma(x)^*
      (\nabla U)(x)
    \big\|_{
      L^4( \Omega; \R^m )
    }
  \\ & \quad
  \cdot
    \left\|
      \big\|
        \big[
          \big(
            \tfrac{ \partial }{ \partial y } \Phi
          \big)(x, s,W_s)
        \big]^*
        ( \nabla U )( Y^x_s )
        -
        \sigma(x)^*
        ( \nabla U )( x )
      \big\|
      +
      2
      \left\|
        \sigma(x)^*
      \right\|_{
        L( \R^d, \R^m)
      }
      \left\|
        ( \nabla U )( x )
      \right\|
    \right\|_{
      L^4( \Omega; \R )
    }
  \\ & \leq
       20
       \left[ 16 p \right]^p
       c^{
         ( 2 p + 3 + \max( \gamma_0, \gamma_1 ) + \gamma_1 + \gamma_5 )
       }
       \left[
         \max(s,1)
       \right]^{ ( 1 + \gamma_3 ) }
       \left[
         \min(s,1)
       \right]^{
         [
           \min( \gamma_3, 1 / 2 ) - \alpha ( \gamma_0 + 2 \gamma_1 + \gamma_5 + 1 )
         ]
       }
  \\ & \quad
  \cdot
   \left[
     \tfrac{
       20
       \left[ 16 p \right]^p
       c^{
         ( 2 p + 3 + \max( \gamma_0, \gamma_1 ) + \gamma_1 + \gamma_5 )
       }
       \left[
         \max(s,1)
       \right]^{ ( 1 + \gamma_3 ) }
     }{
       \left[
         \min(s,1)
       \right]^{
         -
         [
           \min( \gamma_3, 1 / 2 ) - \alpha ( \gamma_0 + 2 \gamma_1 + \gamma_5 + 1 )
         ]
       }
     }
       +
     2 c
     \left( 1 + c^{ \gamma_1 } s^{ - \alpha \gamma_1 } \right)
     c
     \left[
       1 + c s^{ - \alpha }
     \right]^{ ( p - 1 ) / p }
   \right]
  \\ & \leq
      20
       \left[ 16 p \right]^p
       c^{
         2 ( 2 p + 3 + \max( \gamma_0, \gamma_1 ) + \gamma_1 + \gamma_5 )
       }
       \left[
         \max(s,1)
       \right]^{ ( 1 + \gamma_3 ) }
       \left[
         \min(s,1)
       \right]^{
         [
           \min( \gamma_3, 1 / 2 ) - 2 \alpha ( \gamma_0 + 2 \gamma_1 + \gamma_5 + 1 )
         ]
       }
  \\ & \quad
  \cdot
   \left[
       20
       \left[ 16 p \right]^p
       \left[
         \max(s,1)
       \right]^{ ( 1 + \gamma_3 ) }
     +
     8
   \right]
 \\ & \leq
       \left[ 2^9 p \right]^{ 2 p }
       c^{
         2 ( 2 p + 3 + \max( \gamma_0, \gamma_1 ) + \gamma_1 + \gamma_5 )
       }
       \left[
         \max(s,1)
       \right]^{ ( 2 + 2 \gamma_3 ) }
       \left[
         \min(s,1)
       \right]^{
         [
           \min( \gamma_3, 1 / 2 ) - 2 \alpha ( \gamma_0 + 2 \gamma_1 + \gamma_5 + 1 )
         ]
       }
             .
  \end{split}     \end{equation}
  In addition, we note that
  H\"older's inequality, the assumption that
  $
    U \in C_{ p, c }^3( \R^d, [0,\infty) )
  $,
  Lemma~\ref{lemma:equivalent0},
  inequality~\eqref{eq:assumption.on.phi.sigma2},
  and 
  inequality~\eqref{eq:increment.Ui} imply
  for all
  $
    s \in (0,h]
  $,
  $
    x \in D
  $
  that
  \begin{equation}  \begin{split}  \label{eq:third.summand}
  &
    \left\|
      U'(Y_s^x)
      \,
      (
        \triangle_y
        \Phi
      )(x,s,W_s)
    \right\|_{
      L^2( \Omega; \R )
    }
  \leq
    \left\|
         U'(Y_s^x)
    \right\|_{L^{4}(\Omega;L(\R^d,\R))}
    \left\|
      (
        \triangle_y
        \Phi
      )(x,s,W_s)
    \right\|_{L^{4}(\Omega;\R^d)}
  \\ & \leq
    \left(
      \left\|
        U'(Y_s^x)-U'(x)
      \right\|_{
        L^4( \Omega; L(\R^d, \R) )
      }
      +
      \left\|
        U'(x)
      \right\|_{
        L( \R^d, \R )
      }
    \right)
    \left\|
      (
        \triangle_y
        \Phi
      )(x,s,W_s)
    \right\|_{ L^4( \Omega; \R^d ) }
  \\ & \leq
    \left(
    6
    \left[ 8 p \right]^p
     c^{
       ( 2 p + 2 + \max( \gamma_0, \gamma_1 ) + \gamma_5 )
     }
    \max(s, 1)
    \left[
      \min(s,1)
    \right]^{
      [
        1 / 2 - \alpha ( \gamma_0 + \gamma_1 + \gamma_5 + 1 )
      ]
    }
      +
      c
      \left[
        1 +
        U( x )
      \right]^{ ( p - 1 ) / p }
    \right)
     c s^{ \gamma_4 }
  \\ &
  \leq
   \left(
    6
    \left[ 8 p \right]^p
     c^{
       ( 2 p + 2 + \max( \gamma_0, \gamma_1 ) + \gamma_5 )
     }
    \max(s, 1)
    \left[
      \min(s,1)
    \right]^{
      [
        1 / 2 - \alpha ( \gamma_0 + \gamma_1 + \gamma_5 + 1 )
      ]
    }
      +
      2 c^2
      \left[
        \min(s, 1)
      \right]^{ - \alpha }
    \right)
     c s^{ \gamma_4 }
  \\ &
  \leq
   \left(
    6
    \left[ 8 p \right]^p
      +
      2
    \right)
     c^{
       ( 2 p + 3 + \max( \gamma_0, \gamma_1 ) + \gamma_5 )
     }
    \left[
      \max(s, 1)
    \right]^{
      ( \gamma_4 + 1 )
    }
    \left[
      \min(s,1)
    \right]^{
      [
        \gamma_4
        - \alpha ( \gamma_0 + \gamma_1 + \gamma_5 + 1 )
      ]
    }
  \\ &
  \leq
    7
    \left[ 8 p \right]^p
     c^{
       ( 2 p + 3 + \max( \gamma_0, \gamma_1 ) + \gamma_5 )
     }
    \left[
      \max(s, 1)
    \right]^{
      ( \gamma_4 + 1 )
    }
    \left[
      \min(s,1)
    \right]^{
      [
        \gamma_4
        - \alpha ( \gamma_0 + \gamma_1 + \gamma_5 + 1 )
      ]
    }
      .
  \end{split}     \end{equation}
  Moreover, 
  the fact that
  $
    \forall \, x, y \in \R^d 
    \colon
    \left|
      \bar{U}( x ) -
      \bar{U}( y )
    \right|
  \leq
    c
    \left( 1 +
      | U(x) |^{ \gamma_7 }
      +
      | U(y) |^{ \gamma_7 }
    \right)
    \| x - y \|
  $,
  inequality \eqref{eq:before.estimate.fist.term2},
  and inequality~\eqref{eq:increment.Y}
  show
  for all
  $
    s \in (0,h]
  $,
  $
    x \in D
  $
  that
\begin{equation}  \begin{split}
\label{eq:estimate.Ubar}
   &
   \left\|
     \bar{U}( Y_s^x ) - \bar{U}(x)
   \right\|_{L^2(\Omega;\R)}
 \leq
   \left\|
     c
     \left( 1 +
       \left| U(x) \right|^{ \gamma_7 }
       +
       \left| U(Y_s^x) \right|^{\gamma_7}
     \right)
     \left\|
       Y_s^x - x
     \right\|
   \right\|_{
     L^2( \Omega; \R)
   }
   \\ & \leq
   c
   \left[ 1 +
     |U(x)|^{\gamma_7}
     +
     \left\| U( Y_s^x) \right\|^{ \gamma_7 }_{
       L^{ \infty }( \Omega; \R )
     }
   \right]
   \| Y_s^x - x
   \|_{L^2(\Omega;\R^d)}
   \\&
    \leq
    c
    \Big[
      1 +
      | U(x) |^{ \gamma_7 }
      +
      \left[
        2^{ ( 2 p + 1 ) }
        c^{ 2 p }
        \max(
          1, |U(x)|^{ ( p \gamma_5 + 1 ) }
        )
      \right]^{ \gamma_7 }
    \Big]
    \tfrac{
      8
      \,
      c^{
        ( 2 + \max( \gamma_0, \gamma_1 ) + \gamma_5 )
      }
      \max(s, 1)
    }{
      \left[
        \min(s,1)
      \right]^{
        - [ 1 / 2 - \alpha ( \gamma_0 + \gamma_1 + \gamma_5 ) ]
      }
    }
 \\ & \leq
    \Big[
      1 +
      c^{ \gamma_7 }
      s^{ - \alpha \gamma_7 }
      +
      2^{ ( 2 p + 1 ) \gamma_7 }
      \,
      c^{ ( 2 p + p \gamma_5 + 1 ) \gamma_7 }
      \left[
        \min( s, 1 )
      \right]^{
        - \alpha \gamma_7 ( p \gamma_5 + 1 )
      }
    \Big]
    \tfrac{
      8
      \,
      c^{
        ( 3 + \max( \gamma_0, \gamma_1 ) + \gamma_5 )
      }
      \max(s, 1)
    }{
      \left[
        \min(s,1)
      \right]^{
        - [ 1 / 2 - \alpha ( \gamma_0 + \gamma_1 + \gamma_5 ) ]
      }
    }
 \\ & \leq
    \left[
      2
      +
      2^{ ( 2 p + 1 ) \gamma_7 }
    \right]
      8
      \,
      c^{
        [ 3 + \max( \gamma_0, \gamma_1 ) + \gamma_5 + ( p ( \gamma_5 + 2 ) + 1 ) \gamma_7 ]
      }
      \max(s, 1)
      \left[
        \min(s,1)
      \right]^{
        [
          1 / 2
          - \alpha
          (
            \gamma_0 + \gamma_1 + \gamma_5
            +
            \gamma_7 ( p \gamma_5 + 1 )
          )
        ]
      }
 \\ & \leq
      24
      \cdot
      2^{ ( 2 p + 1 ) \gamma_7 }
      \,
      c^{
        [ 3 + \max( \gamma_0, \gamma_1 ) + \gamma_5 + ( p ( \gamma_5 + 2 ) + 1 ) \gamma_7 ]
      }
      \max(s, 1)
      \left[
        \min(s,1)
      \right]^{
        [
          1 / 2
          - \alpha
          (
            \gamma_0 + \gamma_1 + \gamma_5
            +
            ( p \gamma_5 + 1 ) \gamma_7
          )
        ]
      }
      .
  \end{split}     \end{equation}
In the next step
we insert~\eqref{eq:L2.expU},
  \eqref{eq:increment.Ui},
  \eqref{eq:second.summand},
  \eqref{eq:fourth.summand},
  \eqref{eq:Uprime.Phiy.square},
  \eqref{eq:third.summand},
  and
  \eqref{eq:estimate.Ubar}
  into~\eqref{eq:after.Fubini}
  to obtain for all
  $ t \in (0,h] $, $ x \in D $
  that
  \begin{equation}  \begin{split}
    &
    \E\!\left[
      \exp\!\left(
        e^{ - \rho t }
        U( Y^x_t )
        +
        \int_0^t
          e^{ - \rho r }
          \bar{U}( Y_r^x )
        \, dr
      \right)
    \right]
    - e^{ U(x) }
   \\&
   \leq
    \int_0^t
    \sqrt{ 2 }
    \exp\!\left(
      \tfrac{
        2^{
          (2 p + 3 ) (\gamma_6 + 2 )
        }
        \,
        c^{
          p
          \left[
            \max( \gamma_0, 2 \gamma_1 )
            +
            ( \gamma_5 + 4 )
            ( \gamma_6 + 2 )
          \right]
        }
        \,
        s
      }{
        \left[
          \min(s, 1)
        \right]^{
          \alpha
          \left[
            ( p \gamma_5 + 1 ) ( \gamma_6 + 2 )
            + \gamma_0 + 2 \gamma_1
          \right]
        }
      }
  \right)
  e^{ U(x) }
  \\ &
    \quad
    \cdot
  \bigg[
    6 \rho
    \,
     c^{
       ( 2 p + 2 + \max( \gamma_0, \gamma_1 ) + \gamma_5 )
     }
    \left[ 4 p \right]^p
    \max(s, 1)
    \left[
      \min(s,1)
    \right]^{
      [
        1 / 2 - \alpha ( \gamma_0 + \gamma_1 + \gamma_5 + 1 )
      ]
    }
  \\ & \qquad
    +
    20
    \left[ 8 p \right]^p
    c^{
      ( 2 p + 3 + \max( \gamma_0, \gamma_1 ) + \gamma_0 + \gamma_5 )
    }
    \left[
      \max(s,1)
    \right]^{ ( 1 + \gamma_2 ) }
    \left[
      \min(s,1)
    \right]^{
      [
        \min( \gamma_2, 1 / 2 ) - \alpha ( 2 \gamma_0 + \gamma_1 + \gamma_5 + 1 )
      ]
    }
    \\ & \qquad
      +
    \tfrac{ 47 }{ 2 }
    \left[ 8 p \right]^p
     c^{
       ( 2 p + 4 + \max( \gamma_0, \gamma_1 ) + 2 \gamma_1 + \gamma_5 )
     }
  \left[
    \max( s, 1 )
  \right]^{
    ( 2 \gamma_3 + 1 )
  }
    \left[
      \min(s,1)
    \right]^{
      [
        \min( \gamma_3 , 1 / 2 )
        - \alpha ( \gamma_0 + 3 \gamma_1 + \gamma_5 + 1 )
      ]
    }
    \\& \qquad
      +
       \tfrac{ 1 }{ 2 }
       \left[ 2^9 p \right]^{ 2 p }
       c^{
         2 ( 2 p + 3 + \max( \gamma_0, \gamma_1 ) + \gamma_1 + \gamma_5 )
       }
       \left[
         \max(s,1)
       \right]^{ ( 2 + 2 \gamma_3 ) }
       \left[
         \min(s,1)
       \right]^{
         [
           \min( \gamma_3, 1 / 2 ) - 2 \alpha ( \gamma_0 + 2 \gamma_1 + \gamma_5 + 1 )
         ]
       }
    \\ & \qquad +
    4
    \left[ 8 p \right]^p
     c^{
       ( 2 p + 3 + \max( \gamma_0, \gamma_1 ) + \gamma_5 )
     }
    \left[
      \max(s, 1)
    \right]^{
      ( \gamma_4 + 1 )
    }
    \left[
      \min(s,1)
    \right]^{
      [
        \gamma_4
        - \alpha ( \gamma_0 + \gamma_1 + \gamma_5 + 1 )
      ]
    }
  \\ & \qquad
    +
      24
      \cdot
      2^{ ( 2 p + 1 ) \gamma_7 }
      \,
      c^{
        [ 3 + \max( \gamma_0, \gamma_1 ) + \gamma_5 + ( p ( \gamma_5 + 2 ) + 1 ) \gamma_7 ]
      }
      \max(s, 1)
      \left[
        \min(s,1)
      \right]^{
        [
          1 / 2
          - \alpha
          (
            \gamma_0 + \gamma_1 + \gamma_5
            +
            ( p \gamma_5 + 1 ) \gamma_7
          )
        ]
      }
  \bigg]
    \, ds
  \\ & \leq
    e^{ U(x) }
    \int_0^t
    \sqrt{ 2 }
    \exp\!\left(
      \tfrac{
        2^{
          (2 p + 3 ) (\gamma_6 + 2 )
        }
        \,
        c^{
          p
          \left[
            \max( \gamma_0, 2 \gamma_1 )
            +
            ( \gamma_5 + 4 )
            ( \gamma_6 + 2 )
          \right]
        }
        \,
        s
      }{
        \left[
          \min(s, 1)
        \right]^{
          \alpha
          \left[
            ( p \gamma_5 + 1 ) ( \gamma_6 + 2 )
            + \gamma_0 + 2 \gamma_1
          \right]
        }
      }
  \right)
  \\ & \qquad
  \cdot
      c^{
        [
          6
          +
          4 p
          +
          6 \max( \gamma_0, \gamma_1, \gamma_5 )
          + p \gamma_7 ( \gamma_5 + 3 )
        ]
      }
    \left[
      6 \rho
      \left[ 4 p \right]^{ p }
      +
      48
      \left[ 8 p \right]^p
      +
      \tfrac{ 1 }{ 2 }
      \left[ 2^9 p \right]^{ 2 p }
      +
      2^{ ( 3 p \gamma_7 + 5 ) }
    \right]
  \\ & \qquad
  \cdot
    \left[
      \max(s,1)
    \right]^{ \max( 1 + \gamma_2 , 2 + 2 \gamma_3, 1 + \gamma_4 )  }
    \left[
      \min( s, 1)
    \right]^{
      \left[
        \min( 1 / 2 , \gamma_2, \gamma_3, \gamma_4 )
        -
        \alpha
        \left(
          2 \gamma_0 + 4 \gamma_1 + 2 \gamma_5 + (p \gamma_5 + 1 ) \gamma_7 + 2
        \right)
      \right]
    }
  \, ds
    \, .
  \end{split}
  \end{equation}
This implies for all
  $ t \in (0,h] $, $ x \in D $ that
  \begin{equation*}  \begin{split}
  \label{eq:estimate.last.term1}
&
    \E\!\left[
      \exp\!\left(
        \tfrac{
        U(Y^x_{t})}{e^{  \rho t }}
        +
        \smallint_0^t
        \tfrac{
        \bar{U}(Y_r^x)}{e^{ \rho r }}
        \, dr
      \right)
    \right]
  \\ & \leq
    e^{ U(x) }
    +
    e^{ U(x) }
    \int_0^t
    \max( \rho, 1 )
    \left[ 2^9 p \right]^{ 2 p }
    2^{ 3 p \gamma_7 }
    \exp\!\left(
      \tfrac{
        2^{
          (2 p + 3 ) (\gamma_6 + 2 )
        }
        \,
        c^{
          p
          \left[
            \max( \gamma_0, 2 \gamma_1 )
            +
            ( \gamma_5 + 4 )
            ( \gamma_6 + 2 )
          \right]
        }
        \,
        s
      }{
        \left[
          \min(s, 1)
        \right]^{
          \alpha
          \left[
            ( p \gamma_5 + 1 ) ( \gamma_6 + 2 )
            + \gamma_0 + 2 \gamma_1
          \right]
        }
      }
  \right)
  \\ & \qquad
  \cdot
  \frac{
      c^{
        [
          6
          +
          4 p
          +
          6 \max( \gamma_0, \gamma_1, \gamma_5 )
          + p \gamma_7 ( \gamma_5 + 3 )
        ]
      }
    \left[
      \max(s,1)
    \right]^{
      [ \max( \gamma_2 , 1 + 2 \gamma_3, \gamma_4 ) + 1 ]
    }
  }{
    \left[
      \min( s, 1)
    \right]^{
      \left[
        \alpha
        \left(
          2 \gamma_0 + 4 \gamma_1 + 2 \gamma_5 + (p \gamma_5 + 1 ) \gamma_7 + 2
        \right)
        -
        \min( 1 / 2 , \gamma_2, \gamma_3, \gamma_4 )
      \right]
    }
  }
  \, ds
  \\ &
   \leq
    e^{ U(x) }
    \bigg[
      1 +
      \smallint_0^t
    \exp\!\left(
      \tfrac{
        2^{
          (2 p + 3 ) (\gamma_6 + 2 )
        }
        \,
        c^{
          4 p
          ( \gamma_6 + 2 )
          \max( \gamma_0, \gamma_1, \gamma_5, 2 )
        }
        \,
        s
      }{
        \left[
          \min(s, 1)
        \right]^{
          \alpha
          \left[
            ( p \gamma_5 + 1 ) ( \gamma_6 + 2 )
            + \gamma_0 + 2 \gamma_1
          \right]
        }
      }
  \right)
 \\ & \qquad
   \cdot
  \tfrac{
    \max( \rho, 1 )
    \left[
      2 p c
    \right]^{
      6 p
      \left( \gamma_7 + 3 \right)
      \max( 1, \gamma_0, \gamma_1, \gamma_5 )
    }
      \left[
        \max(s,1)
      \right]^{
        [ \max( \gamma_2 , 1 + 2 \gamma_3, \gamma_4 ) + 1 ]
      }
  }{
    \left[
      \min( s, 1)
    \right]^{
      \left[
        \alpha
        \left(
          2 \gamma_0 + 4 \gamma_1 + 2 \gamma_5 + (p \gamma_5 + 1 ) \gamma_7 + 2
        \right)
        -
        \min( 1 / 2 , \gamma_2, \gamma_3, \gamma_4 )
      \right]
    }
  }
  \,
      ds
    \bigg]
    \\ &
    \leq
    e^{ U(x) }
    \bigg[
      1 +
      \smallint_0^t
    \exp\!\left(
      \tfrac{
        \left[ 2 c \right]^{
          4 p
          ( \gamma_6 + 2 )
          \max( \gamma_0, \gamma_1, \gamma_5, 2 )
        }
        \,
        s
      }{
        \left[
          \min(s, 1)
        \right]^{
          \alpha
          \left[
            ( p \gamma_5 + 1 ) ( \gamma_6 + 2 )
            + \gamma_0 + 2 \gamma_1
          \right]
        }
      }
  \right)
  \tfrac{
    \max( \rho, 1 )
    \left[
      2 p c
    \right]^{
      6 p
      \left( \gamma_7 + 3 \right)
      \max( 1, \gamma_0, \gamma_1, \gamma_5 )
    }
      \left[
        \max(s,1)
      \right]^{
        [ \max( \gamma_2 , 1 + 2 \gamma_3, \gamma_4 ) + 1 ]
      }
  }{
    \left[
      \min( s, 1)
    \right]^{
      \left[
        \alpha
        \left(
          2 \gamma_0 + 4 \gamma_1 + 2 \gamma_5 + (p \gamma_5 + 1 ) \gamma_7 + 2
        \right)
        -
        \min( 1 / 2 , \gamma_2, \gamma_3, \gamma_4 )
      \right]
    }
  }
  \,
      ds
    \bigg]
    .
  \end{split}
  \end{equation*}
  This proves \eqref{eq:exp.mom.abstract.one-step} 
  and thereby finishes the proof of
  Lemma~\ref{l:exp.mom.abstract.one-step}.
\end{proof}

%

\subsection{Exponential moments for stopped increment-tamed Euler-Maruyama schemes}
\label{sec:exponential_moments}

Using
Corollary~\ref{Cor:exp.mom.abstract}
and
Lemma~\ref{l:exp.mom.abstract.one-step}
above,
we are now ready to establish
exponential moment bounds for a class of
stopped increment-tamed Euler-Maruyama schemes
in the next theorem.

\begin{theorem}
\label{thm:stopped.Euler.bounded.increments}
  Let $ \gamma, \rho \in [0,\infty) $,
  $ T \in (0,\infty) $,
  $ d, m \in \mathbb{N} $,
  $ p, c \in [1, \infty) $,
  $ q \in (1,\infty) $,
  $
    \mu \in \mathcal{M}( \mathcal{B}( \R^d ), \mathcal{B}( \R^d ) )
  $,
  $
    \sigma \in
    \mathcal{M}( \mathcal{B}( \R^d ), \mathcal{B}( \R^{ d \times m } ) )
  $,
  $
    U \in C_{p,c}^3( \mathbb{R}^d, [0,\infty) )
  $,
  $
    \bar{U}\in C(\R^d, \R)
  $,
  $
    \alpha \in \big( 0, \frac{ 1 }{ 2 } \min\{ \frac{ 1 }{ 7 \gamma + 2 }, \frac{ q - 1 }{ (q + 8) \gamma + 2 } \} \big)
  $,
  let
  $
    D_h \in 
    \mathcal{B}( \{ x \in \R^d \colon U(x) \leq c h^{ - \alpha } \} )
  $,
  $ h \in (0, T] $,
  be a non-increasing family of sets,
  assume for all $ h \in (0,T] $
  that
  $
    \mu|_{ D_h } \in C( D_h, \R^d )
  $
  and 
  $
    \sigma|_{ D_h } \in C( D_h, \R^{ d \times m } )
  $,
  let
  $
    (
      \Omega, \mathcal{F}, \P
    )
  $ 
  be a probability space with a normal filtration 
  $
    (
      \mathcal{F}_t
    )_{
      t \in [0, T ]
    }
  $,
  let
  $
    W \colon [0,T] \times \Omega \to \R^m
  $
  be a standard
  $
    ( \mathcal{F}_t )_{ t \in [0,T] }
  $-Brownian motion with
  continuous sample paths,
  let
  $
    Y^{ \theta }
    \colon
    [0,T] \times\Omega\to\R^d
  $,
  $ \theta \in \mathcal{P}_T $,
  be 
  $ ( \mathcal{F}_t )_{ t \in [0,T] } $-adapted  
  stochastic processes with continuous sample 
  paths which satisfy
  for all
  $ t \in [ 0, T ] $,
  $
    \theta
    \in \mathcal{P}_T
  $
  that
  \begin{equation} \label{eq:thm:Y}
    Y_t^{ \theta }
  =
    Y_{
      \lfloor t \rfloor_{ \theta }
    }^{ \theta }
    +
    \1_{ D_{ | \theta |_T } }\!(
      Y_{ \lfloor t \rfloor_{ \theta } }^{ \theta }
    ) \!
    \left[
      \frac{
        \mu(
          Y_{ \lfloor t \rfloor_{ \theta } }^{ \theta }
        ) \,
        (
          t - \lfloor t \rfloor_{ \theta }
        ) +
        \sigma( Y_{ \lfloor t \rfloor_{ \theta } }^{ \theta }
        )
        \left( W_t - W_{ \lfloor t \rfloor_{ \theta } } \right)
      }{
        1 +
        \big\|
          \mu( Y_{ \lfloor t \rfloor_{ \theta } }^{ \theta } ) \,
          ( t - \lfloor t \rfloor_{ \theta } )
          +
          \sigma( Y_{ \lfloor t \rfloor_{ \theta } }^{ \theta } )
          \left( W_t - W_{ \lfloor t \rfloor_{ \theta } } \right)
          \!
        \big\|^q
      }
    \right]
    ,
  \end{equation}
  and assume
  for all
  $ x, y \in\R^d $ with $ x \neq y $
  that
  \begin{align}
  \label{eq:mu.sigma.quadratic.exponent.corollary}
    \|\mu(x)\|
    +\|\sigma(x)\|_{\HS(\R^m,\R^d)}
    +|\bar{U}(x)|
    \leq c \left( 1 + \left|U(x)\right|^{\gamma} \right),
    \quad
    \tfrac{
      | \bar{U}(x) - \bar{U}(y) |
    }{
      \| x - y \|
    }
    &
    \leq
    c \left( 1 + |U(x)|^{\gamma} + |U(y)|^{\gamma}\right)
    ,
  \\
  \label{eq:generator.exponential.bounded.increments.quadratic.corollary}
    (\mathcal{G}_{\mu,\sigma}U)(x)
     +\tfrac{1}{2}\left\|\sigma(x)^{*}(\nabla U)(x)\right\|^2
     +\bar{U}(x)
  & \leq
    \rho \cdot U(x)
    .
  \end{align}
  Then it holds for all $ t \in [0,T] $, $ \theta \in \mathcal{P}_T $ that
\begin{equation}  \begin{split}  \label{eq:exp.mom.quadratic.exponent}
    \limsup_{
      \left| \vartheta \right|_T
      \searrow 0
    }
    \sup_{ u \in [0,T] }
    \E\!\left[
      \exp\!\left(
        \tfrac{
          U( Y_u^{ \vartheta } )
        }{
          e^{ \rho u }
        }
        +
        \smallint_0^u
          \tfrac{
            \1_{ D_{ \left| \vartheta \right|_T } }\!(
              Y_{ \lfloor s \rfloor_{ \vartheta } }^{ \vartheta }
            )
            \,
            \bar{U}( Y_s^{ \vartheta } )
          }{
            e^{ \rho s }
          }
        \, ds
      \right)
    \right]
    \leq
      \limsup_{
        \left| 
          \vartheta
        \right|_T
        \searrow 0
      }
      \E\!\left[
        e^{
          U( Y_0^{ \vartheta } )
        }
      \right]
      \qquad  
      \text{and}
  \end{split}
  \end{equation}
\begin{equation*}  \begin{split}
&
  \E\!\left[
       \exp\!\left(
         \tfrac{
           U( Y^{ \theta }_t )
         }{ e^{ \rho t } }
         +
         \smallint_0^t
           \tfrac{
             \1_{ D_{ | \theta |_T } }\!(
               Y^{ \theta }_{\lfloor s \rfloor_{ \theta } }
             )
             \,
             \bar{U}( Y^{ \theta }_s )
           }{
             e^{ \rho s }
           }
         \, ds
       \right)
          \right]
  \leq
  \exp\!\left(
      \tfrac{
        \max( \rho, 1 )
        \,
        \left[
          \min( | \theta |_T, 1)
        \right]^{
          \left[
            \min( 1 / 2, (q-1)/2 - \alpha (q +1) \gamma )
            -
            \alpha
            \left(
              7 \gamma + 2
            \right)
          \right]
        }
      }{
        \exp\left(
          -
	    \left[
	      5
	      c
	      q \max( T, 1 )
	    \right]^{
	      9 p
	      ( q + 1 )
	      \max( \gamma, 1 )
	      \max( \gamma, q, 2 )
	      ( \gamma + 2 )
            }
        \right)
      }
  \right)
    \E\!\left[
      e^{ U( Y_0^{ \theta } ) }
    \right]
    .
\end{split}     \end{equation*}
\end{theorem}

\begin{proof}[Proof of Theorem~\ref{thm:stopped.Euler.bounded.increments}]
  Throughout this proof let 
  $ \gamma_0, \gamma_1, \dots, \gamma_7, \hat{c} \in \R $
  be the real numbers given by
  $ 
    \gamma_0 = \gamma_1 = \gamma_6 = \gamma_7 = \gamma 
  $,
  $
    \gamma_2 = \gamma_3 = \nicefrac{ q }{ 2 } 
    - \alpha \gamma \left( q + 1 \right) 
  $,
  $
    \gamma_4 = 
    \nicefrac{ (q - 1) }{ 2 } 
    - \alpha \gamma \left( q + 1 \right) 
  $,
  $
    \gamma_5 = 0
  $,
  and
  $
    \hat{c} 
    =
    \left[ 16 c^{ ( 1 + \gamma ) } q \max(T, 1)\right]^{(q+1)}
  $,
  let
  $
    \Psi_h \colon \R^d \times [0,h] \times \R^m \to \R^d 
  $, 
  $ h \in (0,T] $,
  be the functions 
  which satisfy for all
  $ h \in (0,T] $,
  $ s \in [0,h] $,
  $ y \in \R^m $,
  $ x \in \R^d $
  that
  \begin{equation}
    \Psi_h( x, s, y ) 
    = 
    x + 
    \1_{ D_h }( x ) 
    \!\left[
      \tfrac{
        \mu( x ) s + \sigma( x ) y 
      }{
        1 + 
        \left\|
          \mu( x ) s + \sigma( x ) y 
        \right\|^q 
      }
    \right] 
  \end{equation}
  and let
  $
    \Phi_h \colon D_h \times [0,h] \times \R^m \to \R^d 
  $, 
  $ h \in (0,T] $,
  be the functions which satisfy
  for all
  $ h \in (0,T] $,
  $ s \in [0,h] $,
  $ y \in \R^m $,
  $ x \in D_h $
  that
  \begin{equation}
    \Phi_h( x, s, y ) 
    = 
    \Psi_h( x, s, y )
    .
  \end{equation}
  We now verify step by step
  all assumptions of
  Lemma~\ref{l:exp.mom.abstract.one-step}.
  First of all, note 
  for all
  $ h \in ( 0, T] $,
  $ x \in D_h $
  that
  $
    \Phi_h(x,0,0) = x
  $.
  Moreover,
  observe that
  \eqref{eq:mu.sigma.quadratic.exponent.corollary}
  ensures that
  \eqref{eq:mu_sigma_Ugrowth_assumption} in Lemma \ref{l:exp.mom.abstract.one-step} 
  is fulfilled.
  Furthermore, note
  for all 
  $ h \in (0,T] $, 
  $ (x, s, y) \in D_h \times (0, h] \times \R^m $ 
  that
  \begin{equation}  
  \begin{split}
  &
    \left(
      \tfrac{ \partial }{ \partial s } \Phi_h
    \right)\!(x,s,y)
  \\ &
    =
      \frac{
        \mu(x) 
        \left( 
          1 + \left\| \mu(x) s + \sigma(x) y \right\|^q 
        \right)
        -
        \left(
          \mu(x) s + \sigma(x) y 
        \right) q 
        \left\|
          \mu(x) s + \sigma(x) y 
        \right\|^{ ( q - 2 ) }
        \left\langle 
          \mu(x) s + \sigma(x) y , \mu(x) 
        \right\rangle
      }{
        \left(
          1 + 
          \left\| \mu(x) s + \sigma(x) y \right\|^q 
        \right)^2
      }
  \\ &
    =  
      \mu(x) -
      \frac{
         \mu(x) 
         \left\|
           \mu(x) s + \sigma(x) y
         \right\|^q
      }{
        1 + \left\| \mu(x) s + \sigma(x) y \right\|^q }
      -
    \frac{
      q
      \left(
        \mu(x) s + \sigma(x) y
      \right) 
      \left\|
        \mu(x) s + \sigma(x) y 
      \right\|^{
        (q-2)
      }
      \left\langle \mu(x) s + \sigma(x) y, \mu(x) \right\rangle
    }{
      \left( 
        1 + \left\| \mu(x) s + \sigma(x) y \right\|^q 
      \right)^2
    } 
  .
  \end{split}     
  \end{equation}
  This implies for all $h \in (0, T]$, $s \in (0, h]$, $x \in D_h $, $y \in \R^m$ that
  \begin{equation}  \begin{split}
  \label{eq:phi.mu.diff.before}
    \left\|
      \left(\tfrac{\partial}{\partial s}\Phi_h\right) \! (x,s,y)-\mu(x)
    \right\|
    \leq
       (q+1)\left\|\mu(x)s+\sigma(x)y\right\|^q\left\|\mu(x)\right\|.
  \end{split}     \end{equation}
  Moreover,
  the inequality 
  $
  \forall \, r \in [2, \infty) , s \in [0, T], x \in \R^d
  \colon
  \|
    \sigma(x) W_s
  \|_{
    L^r( \Omega; \R^d )
  }
  \leq
  \sqrt{ s r ( r - 1 ) / 2 }
  \left\|
    \sigma(x)
  \right\|_{
    \HS( \R^m, \R^d )
  }
  $ 
  shows for all
  $
    r \in [2,\infty)
  $,
  $
   h \in (0, T], s \in (0,h]
  $,
  $
    x \in D_h \subseteq D_s
  $
  that
  \begin{equation}
  \begin{split}
  \label{eq:Lr.increment}
    \left\|\mu(x)s+\sigma(x)W_s\right\|_{L^r(\Omega;\R^d)}
    &
    \leq c s \left( 1 + |U(x)|^{\gamma} \right)+ c \sqrt{sr(r-1)/2} \left( 1 + |U(x)|^{\gamma} \right)
    \\ &
    \leq
    c \left( s + c^{ \gamma } s^{ ( 1 - \alpha \gamma ) }
    \right)
    +
    c \sqrt{ r ( r - 1 ) / 2 }
    \left( s^{1/2} + c^{ \gamma } s^{(1/2 - \alpha\gamma)} \right)
    \\&
    \leq c^{ ( 1 + \gamma ) } \max(T, 1)\, s^{(1/2-\alpha\gamma)} \left(2+ 2\sqrt{r(r-1)/2}\right)
\\ & \leq
    2 c^{ ( 1 + \gamma ) } r \max(T, 1) \, s^{(1/2-\alpha\gamma)}.
  \end{split}
  \end{equation}
  This together with \eqref{eq:phi.mu.diff.before} and the fact
  $
    \alpha \gamma < 1
  $
  implies for all
  $
    h \in (0,T]
  $,
  $
    s \in (0,h]
  $,
  $
    x \in D_h \subseteq D_s
  $
  that
  \begin{equation}  \begin{split}
    \left\|
      \left(\tfrac{\partial}{\partial s}\Phi_h\right)\!(x,s,W_s)-\mu(x)
    \right\|_{L^4(\Omega;\R^d)}
    &\leq
    (q+1)
    \left\|\mu(x)s+\sigma(x)W_s\right\|_{ L^{ 4 q }( \Omega; \R^d ) }^q
    \|\mu(x)\|
    \\&
    \leq (q+1)\left[8 c^{ (1 + \gamma) } q \max(T, 1)\right]^q s^{q(1/2-\gamma\alpha)}\,
    c \left( 1 + \left|U(x)\right|^{\gamma} \right)
    \\&
    \leq c\left(q+1\right)\left[8 c^{ (1 + \gamma) } q \max(T, 1)\right]^q s^{q(1/2-\gamma\alpha)}
    \left( 1 + c^{ \gamma } s^{ - \alpha \gamma} \right)
    \\&
    \leq
    c^{ ( 1 + \gamma ) }
    \left(q+1\right)\left[8 c^{ (1 + \gamma) } q \max(T, 1)\right]^q
    s^{(q/2-\alpha(q+1)\gamma)
      }
    \left( s^{ \alpha \gamma} + 1 \right)
  \\ &
  \leq
    2 q
    c^{ ( 1 + \gamma ) }
    \left[
      8 c^{ (1 + \gamma) } q \max(T, 1)
    \right]^q
    s^{
      ( q / 2 - \alpha (q + 1) \gamma )
    }
    \, 2 \left[ \max(T, 1) \right]
\\ &
    \leq 
    \left[8 c^{ (1 + \gamma) } q \max(T, 1)\right]^{(q + 1)}
    s^{
      (q/2-\alpha(q+1)\gamma)
    }
    \leq \hat{c}
     s^{\gamma_2
      }.
  \end{split}     \end{equation}
  This proves that \eqref{eq:assumption.on.phi.mue} in Lemma \ref{l:exp.mom.abstract.one-step}
  is fulfilled.
  Similarly, it holds for all 
  $ h \in (0,T] $, $ (x, s, y) \in D_h \times (0, h] \times \R^m $ that
  \begin{equation}  
  \begin{split}
  &
    \big(\tfrac{\partial}{\partial y}\Phi_h\big)(x,s,y)
  \\ &
    = 
    \frac{\sigma(x)
           \left(1+\left\|\mu(x)s+\sigma(x)y\right\|^q\right)
            -\left(\mu(x)s+\sigma(x)y\right)q\left\|
            \mu(x)s+\sigma(x)y\right\|^{(q-2)}
              \left( \mu(x)s+\sigma(x)y\right)^{*}\sigma(x)
          }{\left(1+\left\|\mu(x)s+\sigma(x)y\right\|^q \right)^2}
    \\&
    = 
    \sigma(x)-
    \frac{
           \sigma(x) \left\|\mu(x)s+\sigma(x)y\right\|^q
          }{1+\left\|\mu(x)s+\sigma(x)y\right\|^q}
    -
    \frac{ q\left(\mu(x)s+\sigma(x)y\right)\left\| \mu(x)s+\sigma(x)y\right\|^{(q-2)}
              \left( \mu(x)s+\sigma(x)y\right)^{*}\sigma(x)
          }{\left(1+\left\|\mu(x)s+\sigma(x)y\right\|^q\right)^2} 
          .
  \end{split}     
  \end{equation}
  This implies for all $ h \in (0, T]$, $ s \in (0, h] $, 
  $ x \in D_h $, $ y \in \R^m $ that
  \begin{equation}  \begin{split}
    \big\|
      \big(\tfrac{\partial}{\partial y}\Phi_h\big)(x,s,y)-\sigma(x)
    \big\|_{\HS(\R^m,\R^d)}
    \leq
       (q+1)\left\|\mu(x)s+\sigma(x)y\right\|^q \left\|\sigma(x)\right\|_{\HS(\R^m,\R^d)}.
  \end{split}     \end{equation}
  This together with~\eqref{eq:Lr.increment}
  implies for all $ h \in (0,T] $, 
  $ s \in (0, h] $, 
  $ x \in D_h \subseteq D_s $ that
  \begin{equation}  \begin{split}
    \left\|\big(\tfrac{\partial}{\partial y}\Phi_h\big)(x,s,W_s)-\sigma(x)
    \right\|_{L^8(\Omega;\HS(\R^m,\R^d))}
    &
    \leq (q+1)
    \left\|\mu(x)s+\sigma(x)W_s\right\|_{L^{8q}(\Omega;\R^d)}^q\|\sigma(x)\|_{\HS(\R^m,\R^d)}
    \\&
    \leq (q+1)\left[ 16 c^{ ( 1 + \gamma ) } q \max(T, 1) \right]^q s^{q(1/2-\gamma\alpha)}
    c \left( 1 + c^{ \gamma } s^{ - \alpha \gamma}\right)
  \\&
  \leq
    c^{ ( 1 + \gamma ) } \left( q+1 \right)\left[ 16 c^{ ( 1 + \gamma ) } q \max(T, 1) \right]^q s^{[ q/2-\alpha(q+1)\gamma
                           ]
                         } \left( s^{ \alpha \gamma} + 1 \right)
  \\&
  \leq
    2
    c^{ ( 1 + \gamma ) }
    q \left[ 16 c^{ ( 1 + \gamma ) } q \max(T, 1) \right]^q s^{[ q/2-\alpha(q+1)\gamma
                           ]
                         }\, 2 \max(T, 1)
    \\&
    \leq \left[ 16 c^{ ( 1 + \gamma ) } q \max(T, 1) \right]^{(q + 1)} s^{ [ q/2-\alpha(q+1)\gamma
                           ]
                         } = \hat{c} s^{\gamma_3}
                         .
  \end{split}     \end{equation}
  This shows that \eqref{eq:assumption.on.phi.sigma}
  in Lemma~\ref{l:exp.mom.abstract.one-step}
  is fulfilled.
  In the next step let
  $
    \psi \colon \R^d \to \R^d
  $
  and
  $
    \sigma_i \colon \R^d \to \R^d
  $,
  $ i \in \{ 1, \dots, m \} $,
  be the functions
  with the property that
  for all $ z \in \R^d $
  it holds that
  $
    \psi(z) = \frac{ z }{ 1 + \|z\|^q }
  $
  and
  $
    \sigma(z) =
    (
      \sigma_1(z), \sigma_2(z), \ldots, \sigma_m(z)
    )
  $.
  Observe that 
  $ \psi \in C^2( \R^d, \R^d ) $
  and that for all
  $
    z = ( z_1, z_2 , \dots, z_d ) 
  $,
  $
    u = ( u_1, u_2, \dots, u_d ) 
  $, 
  $ 
    v = ( v_1, v_2 , \dots , v_d ) \in \R^d
  $
  it holds that
  \begin{equation} \begin{split}
  &
    \psi'(z) u
    =
    \sum_{ k = 1 }^d
    \big( \tfrac{ \partial }{ \partial z_k } \psi \big)( z ) \cdot u_k
    =
    \begin{cases}
      u
      &
      \colon
      z = 0
    \\
      \frac{u}{1+\|z\|^q}
      -\frac{
        q z \|z\|^{(q-2)}\langle z,u\rangle}
            {\left(1+\|z\|^q\right)^2}
      &
      \colon
      z \neq 0
    \end{cases}
    \\
    &
    \text{and}
    \qquad
    \psi''(z)( u, v )
    =
    \sum_{ k, l = 1 }^d
    \big( \tfrac{ \partial^2 }{ \partial z_k \partial z_l } \psi \big)( z ) \cdot u_k \cdot u_l
\\ & 
    \quad\qquad
    =
    \begin{cases}
      0
      &
      \colon
      z = 0
      \\
      -
      \frac{
        q \|z\|^{ (q - 2 ) }
        \left[
          u
          \langle z, v \rangle
          +
          v
          \langle z, u \rangle
          +
          z
          \langle u, v \rangle
        \right]
      }{
        \left( 1 + \|z\|^q \right)^2
      }
      -
      \frac{
        q
        \left( q - 2 \right)
        \|z\|^{ ( q - 4 ) }
        z
        \langle z, v \rangle
        \langle z, u \rangle
      }{
        \left(1+\|z\|^q\right)^2
      }
      +
      \frac{
        2 q^2 \|z\|^{ 2 (q - 2 ) }
        z \langle z, u \rangle  \langle z,v\rangle
      }{
        \left(1+\|z\|^q\right)^3
      }
     &
     \colon
     z \neq 0
    \end{cases}
    .
  \end{split} \end{equation}
  This implies
  for all $ z, u \in \R^d $
  that
  \begin{equation} \begin{split}
  &
    \psi''(z) (u, u)
    =
    \begin{cases}
      0
      &
      \colon
      z = 0
      \\
      \frac{
        2 q^2 \|z\|^{ 2 (q - 2 ) }
        z
        \left|
          \langle z, u \rangle
        \right|^2
      }{
        \left(1+\|z\|^q\right)^3
      }
      -
      \frac{
        q \|z\|^{ (q - 2 ) }
        \left[
          2 u
          \langle z, u \rangle
          +
          z
          \left\| u \right\|^2
        \right]
      }{
        \left( 1 + \|z\|^q \right)^2
      }
      -
      \frac{
        q
        \left( q - 2 \right)
        \|z\|^{ ( q - 4 ) }
        z
        \left|
          \langle z, u \rangle
        \right|^2
      }{
        \left(1+\|z\|^q\right)^2
      }
     &
     \colon
     z \neq 0
    \end{cases}
    .
  \end{split} \end{equation}
  Hence, we obtain
  for all $ i \in \{1, 2, \ldots, m \}$,
  $
    (x, s, y) \in \R^d \times (0, h] \times \R^m
  $
  that
  \begin{equation}
  \begin{split}
    &
    \frac{ \partial^2 }{ \partial y_i^2 }
    \!\Big( \psi \big(\mu(x) s + \sigma(x) y \big)\Big)
    = \frac{\partial}{\partial y_i} \!\Big( \psi ' \big(\mu(x) s + \sigma(x) y \big) \big( \sigma_i( x ) \big) \Big)
    = \psi '' \big(\mu(x) s + \sigma(x) y \big) \big( \sigma_i( x ), \sigma_i( x ) \big)
    \\ & =
    \frac{
      \mathbbm{1}_{
        \R^d \backslash \{ 0 \}
      }(
        \mu(x) s + \sigma( x ) y
      )
      \,
      2 q^2
      \left\|
        \mu(x) s + \sigma(x) y
      \right\|^{ 2 ( q - 2 ) }
      \left(
        \mu(x) s + \sigma(x) y
      \right)
      \left|
        \langle \mu(x) s + \sigma(x) y, \sigma_i(x)
        \rangle
      \right|^2
    }{
      \left(
        1 +
        \left\|
          \mu(x) s + \sigma(x) y
        \right\|^q
      \right)^3
    }
  \\ &
    -
    \tfrac{
      \mathbbm{1}_{
        \R^d \backslash \{ 0 \}
      }(
        \mu(x) s + \sigma( x ) y
      )
      \,
      q
      \,
      \left\|
        \mu(x) s + \sigma(x) y
      \right\|^{ ( q - 2 ) }
      \big[
        2 \sigma_i(x)
        \langle
          \mu(x) s + \sigma(x) y ,
          \sigma_i(x)
        \rangle
        +
        \left( \mu(x) s + \sigma(x) y \right)
        \left\|
          \sigma_i(x)
        \right\|^2
      \big]
   }{
     \left(1+\|\mu(x) s + \sigma(x) y\|^q\right)^2
   }
   \\ &
      -
      \frac{
        \mathbbm{1}_{
          \R^d \backslash \{ 0 \}
        }(
          \mu(x) s + \sigma( x ) y
        )
        \,
        q \left( q - 2 \right)
        \left\|\mu(x) s + \sigma(x) y \right\|^{ ( q - 4 ) }
        \left( \mu(x) s + \sigma(x) y \right)
           |\langle \mu(x) s + \sigma(x) y, \sigma_i(x) \rangle |^2
      }{
        \left(1+\|\mu(x) s + \sigma(x) \|^q\right)^2
      }
  .
  \end{split}     \end{equation}
  This and the Cauchy-Schwarz inequality show for all
  $
    (x, s, y) \in \R^d \times (0, h] \times \R^m
  $
  that
  \begin{equation}
  \begin{split}
    &
      \sum_{ i = 1 }^m
    \left\|
      \frac{ \partial^2 }{ \partial y_i^2 }
      \!\Big( \psi \big(\mu(x) s + \sigma(x) y \big)\Big)
    \right\|
    \\ & \leq
    \frac{
      2 q^2
      \left\|
        \mu(x) s + \sigma(x) y
      \right\|^{ ( 2 q - 1 ) }
      \left\|
        \sigma( x )
      \right\|^2_{ \HS( \R^m, \R^d ) }
    }{
      \left(
        1 +
        \left\|
          \mu(x) s + \sigma(x) y
        \right\|^q
      \right)^3
    }
    +
    \frac{
      3 q
      \left\|
        \mu(x) s + \sigma(x) y
      \right\|^{ ( q - 1 ) }
      \left\|
        \sigma( x )
      \right\|^2_{ \HS( \R^m, \R^d ) }
   }{
     \left(1+\|\mu(x) s + \sigma(x) y\|^q\right)^2
   }
   \\ &
      +
      \frac{
        q \left| q - 2 \right|
        \left\|\mu(x) s + \sigma(x) y \right\|^{ ( q - 1 ) }
        \left\|
          \sigma( x )
        \right\|^2_{ \HS( \R^m, \R^d ) }
      }{
        \left(1+\|\mu(x) s + \sigma(x) \|^q\right)^2
      }
\\ & \leq
    \left[
        2 q^2
        +
        3 q
        +
        q \left| q - 2 \right|
    \right]
        \left\|\mu(x) s + \sigma(x) y \right\|^{ ( q - 1 ) }
        \left\|
          \sigma( x )
        \right\|^2_{ \HS( \R^m, \R^d ) }
  .
  \end{split}     
  \end{equation}
  Consequently, it follows for all
  $ h \in (0,T] $, $ s \in (0, h] $, 
  $ x \in D_h \subseteq D_s $, $ y \in \R^m $
  that
  \begin{equation}\begin{split}
  \big\| (\triangle_y \Phi_h)(x, s, y) \big\| &\leq
      \sum_{i=1}^m
    \big\|
      \big(
        \tfrac{ \partial^2 }{ \partial y_i^2 } \Phi_h
      \big)( x, s, y )\big\|
       \leq
       \sum_{i=1}^m \left[ \left( 2 q^2 + 3q + q^2 \right) \left\|\mu(x)s+\sigma(x)y\right\|^{(q-1)} \|\sigma_i(x)\|^2 \right]
    \\ & = 3q\left(q+1\right)\left\|\mu(x)s+\sigma(x)y\right\|^{(q-1)}
    \|\sigma(x)\|_{\HS(\R^m,\R^d)}^2.
  \end{split}\end{equation}
  This together with~\eqref{eq:Lr.increment} and the fact $2 \alpha \gamma < 1$
  yields for all $h \in (0, T]$,
  $ s \in(0,h] $, $ x \in D_h \subseteq D_s $ that
  \begin{equation}  \begin{split}
    &
    \big\| (\triangle_y \Phi_h)(x, s, W_s) \big\|_{L^4(\Omega;\R^d)}
    \leq 3q\left(q+1\right)\left\|\mu(x)s+\sigma(x)W_s\right\|_{L^{4(q-1)}(\Omega;\R^d)}^{(q-1)}
      \|\sigma(x)\|_{\HS(\R^m,\R^d)}^2
    \\&
    \leq
    3q\left(q+1\right)\left\|\mu(x)s+\sigma(x)W_s\right\|_{L^{4q}(\Omega;\R^d)}^{(q-1)}
      \|\sigma(x)\|_{\HS(\R^m,\R^d)}^2
    \\&
    \leq
    3q\left(q+1\right)\left[8 c^{ ( 1 + \gamma ) } q \max(T, 1) \right]^{(q-1)}
      s^{(q-1)(1/2-\gamma\alpha)}\,
    c^2 \left( 1 + \left|U(x)\right|^{\gamma}\right)^2
    \\ &
    \leq 6q\left(q+1\right)c^2 \left[8 c^{ ( 1 + \gamma ) } q \max(T, 1) \right]^{(q-1)}
      s^{(q-1)(1/2-\gamma\alpha)}
      \left( 1 + c^{ 2 \gamma } s^{ - 2\alpha \gamma}\right)
    \\&
    \leq 12
    c^{ ( 2 + 2 \gamma ) }
    q^2 \left[8 c^{ ( 1 + \gamma ) } q \max(T, 1) \right]^{(q-1)}
    s^{[(q-1)/2 -\alpha(q+1)\gamma]
      } \, 2 \left[ \max(T, 1)\right]
    \\&
    \leq \left[ 8 c^{ ( 1 + \gamma ) } q \max(T, 1) \right]^{(q+1)}
    s^{ [(q-1)/2 -\alpha(q+1)\gamma]
      }
    \leq \hat{c}
    s^{\gamma_4
      }.
  \end{split}     \end{equation}
  This proves that
  \eqref{eq:assumption.on.phi.sigma2} in Lemma \ref{l:exp.mom.abstract.one-step} is fulfilled.
  Next observe for all $h\in(0,T]$,
   $s \in (0, h]$, $ x \in D_h $, $r \in [1, \infty)$ that
  \begin{equation}  \begin{split}
    \left\|
      \Phi_h( x, s, W_s ) - x 
    \right\|_{
      L^r(\Omega; \R^d)
    }
  & 
    \leq 
    \min\!\left\{ 
      1, 
      \left\|
        \mu(x) s + \sigma(x) W_s 
      \right\|_{
        L^r( \Omega ; \R^d ) 
      } 
    \right\}
  \\ & 
    \leq 
    \hat{c} 
    \min\!\left\{ 
      r , 
      1 + | U(x) |^{ \gamma_5 } , 
      \left( 1 + |U(x)|^{ \gamma_5 } \right) 
      \left\|
        \mu(x) s + \sigma(x) W_s
      \right\|_{
        L^r( \Omega; \R^d ) 
      } 
    \right)
    .
  \end{split}     
  \end{equation}
  This shows that \eqref{eq:assumption.on.phi.increment}
  in Lemma \ref{l:exp.mom.abstract.one-step} is fulfilled.
  Thus all assumptions of
  Lemma~\ref{l:exp.mom.abstract.one-step} are satisfied.
  Next let
  $
    \varrho_h
    \in (0,\infty)
  $,
  $ h \in (0,T] $,
  be the real numbers with the property that
  for all $ h \in (0,T] $
  it holds that
  \begin{equation} 
  \begin{split}
    \varrho_h
    &
  =
    \exp\!\left(
	\tfrac{
	  h
	  \,
	  \left[ 2 \hat{c} \right]^{
	    4 p
	    ( \gamma_6 + 2 )
	    \max( \gamma_0, \gamma_1, \gamma_5, 2 )
	  }
	}{
	  \left[
	    \min(h, 1)
	  \right]^{
	    \alpha
	    \left[
	      ( p \gamma_5 + 1 ) ( \gamma_6 + 2 )
	      + \gamma_0 + 2 \gamma_1
	    \right]
	  }
	}
    \right)
    \tfrac{
      \max( \rho, 1 )
      \,
      \left[
	2 p \hat{c} \max( h, 1 )
      \right]^{
	6 p
	\left( \gamma_7 + 3 \right)
	\max( 1, \gamma_0, \gamma_1, \dots, \gamma_5 )
      }
    }{
      \left[
	\min( h, 1)
      \right]^{
	\left[
	  \alpha
	  \left(
	    2 \gamma_0 + 4 \gamma_1 + 2 \gamma_5 + (p \gamma_5 + 1 ) \gamma_7 + 2
	  \right)
	  -
	  \min( 1 / 2 , \gamma_2, \gamma_3, \gamma_4 )
	\right]
      }
    }
  \\ & =
      \exp\!\left(
        \tfrac{
          h \,
	  \left[ 2 \hat{c} \right]^{
	    4 p
	    ( \gamma + 2 )
	    \max( \gamma, 2 )
	  }
        }{
          \left[
            \min(h, 1)
          \right]^{
            \alpha
            \left[
              4 \gamma
              + 2
            \right]
          }
        }
      \right)
      \tfrac{
        \max( \rho, 1 )
        \,
        \left[
  	  2 p \hat{c} \max( h, 1 )
        \right]^{
  	  3 p
	  \left( \gamma + 3 \right)
          \max \{ 2, 2 \gamma, q - 2 \alpha \gamma( q + 1) \}
        }
      }{
        \left[
          \min( h, 1)
        \right]^{
          \left[
            \alpha
            \left[
              7 \gamma + 2
            \right]
            -
            \min( 1 / 2, (q-1)/2 - \alpha (q +1) \gamma )
          \right]
        }
      }
      .
  \end{split} 
  \end{equation}
  Note that
  the estimates
  $
    \alpha \left[ 4 \gamma + 2 \right] - 1 < 0
  $
  and
  $
    \alpha
          \left[ 7 \gamma + 2 \right]
          -
          \min( 1 / 2 , (q-1) / 2 - \alpha (q + 1 ) \gamma )
    < 0
  $
  ensure that the function
  $
    (0,T] \ni h \mapsto \varrho_h \in (0,\infty)
  $
  is non-decreasing
  and that
  $
    \lim_{ h \searrow 0 } \varrho_h = 0
  $.
  Combining
  Lemma~\ref{l:exp.mom.abstract.one-step}
  with the fact that
  $
    (0,T] \ni h \mapsto \varrho_h \in (0,\infty)
  $
  is non-decreasing
  implies
  for all 
  $ h \in (0,T] $, 
  $ (t,x) \in (0,h] \times D_h $ 
  that
  \begin{equation}  
  \label{eq:UbarU_estimate}
  \begin{split}
    &
    \E\!\left[
      \exp\!\left(
        \tfrac{
          U(
            \Phi_h(x,t,W_t)
          )
        }{
          e^{\rho t}
        }
        +
        \smallint_0^t
        \tfrac{ 
          \bar{U}( \Phi_h( x, s, W_s ) )
        }{
          e^{ \rho s }
        }
        \, ds
      \right)
    \right]
    \leq
    \left(
      1 + \int_0^t \varrho_s \, ds
    \right)
    e^{ U( x ) }
    \leq
    ( 1 + \varrho_h t)
    \, e^{ U(x) }
    .
  \end{split}     
  \end{equation}
  Clearly, this implies for all
  $
    \theta \in \mathcal{P}_T
  $,
  $
    (t,x) 
    \in 
    \big( 0 , | \theta |_T \big] 
    \times 
    D_{ | \theta |_T } 
  $
  that
  \begin{equation}  \begin{split}
    &
    \E\!\left[
    \exp\!\left(
      \tfrac{
        U(
          \Psi_{ | \theta |_T }( x, t, W_t )
        )
      }{
        e^{ \rho t }
      }
        +
        \smallint_0^t 
        \tfrac{ 
          \bar{U}(
            \Psi_{ | \theta |_T }( x, s, W_s ) 
          )
        }{
          e^{ \rho s } 
        } \, ds
      \right)
    \right]
    \leq 
    e^{
      \varrho_{ | \theta |_T } t + U(x)
    }
    .
  \end{split}     \end{equation}
  Corollary~\ref{Cor:exp.mom.abstract}
  hence
  yields for all  $t\in[0,T]$, $\theta \in \mathcal{P}_T$ that
  \begin{equation}  \begin{split}
  \label{eq:first.exp.mom.bound}
    &
    \E\!\left[\exp\!\left( \tfrac{U(Y^{\theta}_{t})}{e^{\rho t}}
          +\smallint_0^{t}
                         \tfrac{
                           \1_{
                             D_{ \left| \theta \right|_T } 
                           }(
                             Y^{ \theta }_{
                               \lfloor r \rfloor_{ \theta }
                             }
                           )
                         \, \bar{U}(Y^{\theta}_r)}{e^{\rho r}}\,dr
                 \right)
          \right]
    \leq
      e^{\varrho_{ \left| \theta \right|_T } t }
      \,
      \E \! \left[e^{{U}(Y^{\theta}_0)}\right].
  \end{split}     \end{equation}
  This implies for all $\theta \in \mathcal{P}_T$ that
  \begin{equation}  \begin{split}
  \label{eq:first.exp.mom.bound.before.final}
    &
    \sup_{t\in[0,T]} \E\!\left[\exp\!\left( \tfrac{U(Y^{\theta}_{t})}{e^{\rho t}}
          +\smallint_0^{t}
                         \tfrac{\1_{D_{ \left| \theta \right|_T }}
                         (Y^{\theta}_{\lfloor s\rfloor_{\theta}}) \, \bar{U}(Y^{\theta}_s)}{e^{\rho s}}\,ds
                 \right)
          \right]
    \leq
      e^{\varrho_{ \left| \theta \right|_T } T }
      \,
      \E \! \left[e^{{U}(Y^{\theta}_0)}\right].
  \end{split}     \end{equation}
  This and the fact
  $
    \lim_{h \searrow 0}\varrho_h=0
  $
  then show \eqref{eq:exp.mom.quadratic.exponent}.
  Next observe that the estimate
  $
    \forall \,
    x \in [ 5^{ 72 }, \infty ) 
    \colon
    x \leq \exp\!\left( x^{ 1 / 20 } \right)
  $
  shows for all $\theta \in \mathcal{P}_T$
  that
  \begin{equation} \begin{split}
  \label{eq:vartheta.h.estimate}
  &
    \varrho_{ \left| \theta \right|_T } T
  \\ & =
      \exp\!\left(
        \tfrac{
           \left| \theta \right|_T  \,
	  \left[
	    2
	    \left[
	      16
	      c^{ ( \gamma + 1 ) }
	      q \max( T, 1 )
	    \right]^{ ( q + 1 ) }
	  \right]^{
	    4 p
	    ( \gamma + 2 )
	    \max( \gamma, 2 )
	  }
        }{
          \left[
            \min( \left| \theta \right|_T , 1)
          \right]^{
            \alpha
            \left[
              4 \gamma
              + 2
            \right]
          }
        }
      \right)
\\ & 
  \cdot 
      \tfrac{
        \max( \rho, 1 )
        \,
        T
        \,
        \left[
  	  2 p \max(  \left| \theta \right|_T , 1 )
	    \left[
	      16
	      c^{ ( \gamma + 1 ) }
	      q \max( T, 1 )
	    \right]^{ ( q + 1 ) }
        \right]^{
  	  3 p
	  \left( \gamma + 3 \right)
          \max \{ 2, 2 \gamma, q - 2 \alpha \gamma( q + 1) \}
        }
      }{
        \left[
          \min(  \left| \theta \right|_T , 1)
        \right]^{
          \left[
            \alpha
            \left[
              7 \gamma + 2
            \right]
            -
            \min( 1 / 2, (q-1)/2 - \alpha (q +1) \gamma )
          \right]
        }
      }
  \\ & \leq
      \tfrac{
        \max( \rho, 1 )
        \,
        \exp\left(
	    \left[
	      5
	      c
	      q \max( T, 1 )
	    \right]^{
	      8 p \max( \gamma, 1 ) ( q + 1 )
	      ( \gamma + 2 )
	      \max( \gamma, 2 )
            }
        \right)
        \,
        \left[
	  5
	  c
	  p
	  q
          \max( T, 1 )
        \right]^{
  	  6 p ( q + 1 )
  	  \max( \gamma, 1 )
	  \left( \gamma + 3 \right)
          \max( 2, 2 \gamma, q )
        }
      }{
        \left[
          \min(  \left| \theta \right|_T , 1)
        \right]^{
          \left[
            \alpha
            \left[
              7 \gamma + 2
            \right]
            -
            \min( 1 / 2, (q-1)/2 - \alpha (q +1) \gamma )
          \right]
        }
      }
  \\ & \leq
      \tfrac{
        \max( \rho, 1 )
        \,
        \exp\left(
	    \left[
	      5
	      c
	      q \max( T, 1 )
	    \right]^{
	      8 p
	      ( q + 1 )
	      \max( \gamma, 1 )
	      \max( \gamma, 2 )
	      ( \gamma + 2 )
            }
            +
        \left[
	  5
	  c
	  p
	  q
          \max( T, 1 )
        \right]^{
          3 / 10
  	  p ( q + 1 )
  	  \max( \gamma, 1 )
          \max( 2, 2 \gamma, q )
	  \left( \gamma + 3 \right)
        }
        \right)
      }{
        \left[
          \min(  \left| \theta \right|_T , 1)
        \right]^{
          \left[
            \alpha
            \left[
              7 \gamma + 2
            \right]
            -
            \min( 1 / 2, (q-1)/2 - \alpha (q +1) \gamma )
          \right]
        }
      }
  \\ & \leq
      \tfrac{
        \max( \rho, 1 )
        \,
        \exp\left(
          2
          \,
	    \left[
	      5
	      c
	      q \max( T, 1 )
	    \right]^{
	      8 p
	      ( q + 1 )
	      \max( \gamma, 1 )
	      \max( \gamma, q, 2 )
	      ( \gamma + 2 )
            }
        \right)
      }{
        \left[
          \min(  \left| \theta \right|_T , 1)
        \right]^{
          \left[
            \alpha
            \left[
              7 \gamma + 2
            \right]
            -
            \min( 1 / 2, (q-1)/2 - \alpha (q +1) \gamma )
          \right]
        }
      }
      .
  \end{split}\end{equation}
  Combining \eqref{eq:first.exp.mom.bound.before.final} with \eqref{eq:vartheta.h.estimate} completes
  the proof of  Theorem~\ref{thm:stopped.Euler.bounded.increments}.
\end{proof}

The next corollary of
Theorem~\ref{thm:stopped.Euler.bounded.increments}
considers the case in which there exists 
a Borel measurable set $ D \subseteq \R^d $
such that the sets 
$ D_h \in \mathcal{B}( \R^d ) $,
$ h \in ( 0, T] $, satisfy
for all $ h \in (0,T] $
that
$
  D_h
  \subseteq
    \{
      x \in D \colon
      U( x ) \leq
      c
      \exp\!\left(
        c \,
        |
          \ln( h )
        |^{ 1 / 2 }
      \right)
    \}
$
(see Corollary~\ref{cor:stopped.Euler.bounded.increments}
below for details).

\begin{corollary}
\label{cor:stopped.Euler.bounded.increments}
  Let
  $ d, m \in \N $,
  $
    \rho
    \in [0,\infty)
  $,
  $ T \in (0,\infty) $,
  $ c, q \in (1,\infty) $,
  $
    \bar{U} \in C(\R^d, \R)
  $,
  $
    D \in \mathcal{B}( \R^d )
  $,
  $
    U \in \cup_{ p \in [1,\infty) } C_{ p, c }^3( \mathbb{R}^d, [0,\infty) )
  $,
  $
    \mu \in \mathcal{M}( \mathcal{B}( \R^d ), \mathcal{B}( \R^d ) )
  $,
  $
    \sigma \in
    \mathcal{M}( \mathcal{B}( \R^d ), \mathcal{B}( \R^{ d \times m } ) )
  $,
  let
  $
    (
      \Omega, \mathcal{F}, \P
    )
  $ 
  be a probability space with a normal filtration 
  $
    (
      \mathcal{F}_t
    )_{
      t \in [0, T ]
    }
  $,
  let
  $
    D_{ h } \in 
    \mathcal{B}(
      \{
        x \in D \colon
        U( x ) \leq
        c 
        \exp(
          c \,
          |
            \ln( h )
          |^{ 1 / 2 }
        )
      \}
    )
  $,
  $ h \in (0, T] $,
  be a non-increasing family of sets
  such that for all $ h \in (0,T] $
  it holds that
  $
    \mu|_{ D_h } \in C( D_h, \R^d )
  $
  and 
  $
    \sigma|_{ D_h } \in C( D_h, \R^{ d \times m } )
  $,
  let
  $
    W\colon[0,T]\times\Omega\to\R^m
  $ be a standard
  $
    (\mathcal{F}_t)_{t\in[0,T]}
  $-Brownian motion with continuous sample paths,
  let
  $
    Y^{ \theta }
    \colon
    [0,T] \times\Omega\to\R^d
  $,
  $ \theta \in \mathcal{P}_T $,
  be 
  $ ( \mathcal{F}_t )_{ t \in [0,T] } $-adapted
  stochastic processes with continuous sample paths
  which satisfy
  $
    \sup_{ \theta \in \mathcal{P}_T }
    \E\big[
      e^{ U( Y^{ \theta }_0 ) }
    \big]
    < \infty
  $
  and which satisfy
  for all
  $ t \in [0,T] $,
  $ \theta \in \mathcal{P}_T $
  that
  \begin{equation} \label{eq:cor:Y}
    Y_t^{ \theta }
  =
    Y_{
      \lfloor t \rfloor_{ \theta }
    }^{ \theta }
    +
    \1_{ D_{ | \theta |_T } }\!(
      Y_{ \lfloor t \rfloor_{ \theta } }^{ \theta }
    )
    \left[
      \tfrac{
        \mu(
          Y_{ \lfloor t \rfloor_{ \theta } }^{ \theta }
        ) \,
        (
          t - \lfloor t \rfloor_{ \theta }
        ) +
        \sigma( Y_{ \lfloor t \rfloor_{ \theta } }^{ \theta }
        )
        \left( W_t - W_{ \lfloor t \rfloor_{ \theta } } \right)
      }{
        1 +
        \|
          \mu( Y_{ \lfloor t \rfloor_{ \theta } }^{ \theta } ) \,
          ( t - \lfloor t \rfloor_{ \theta } )
          +
          \sigma( Y_{ \lfloor t \rfloor_{ \theta } }^{ \theta } )
          \left( W_t - W_{ \lfloor t \rfloor_{ \theta } } \right)
        \|^q
      }
    \right]
    ,
  \end{equation}
  and assume
  for all
  $ x, y \in\R^d $
  that
  \begin{align}
   \label{eq:mu.sigma.quadratic.exponent.corollary2}
    &
      \| \mu(x) \|
      +
      \| \sigma(x) \|_{\HS(\R^m,\R^d)}
  \leq
    c
    \left(
      1 +
      \| x \|^{ c }
    \right)
    ,
    \qquad
      | \bar{U}(x) - \bar{U}(y) |
    \leq
    c
    \left(
      1 + \| x \|^{ c } + \| y \|^{ c }
    \right)
      \| x - y \|
    ,
  \\
  \label{eq:generator.exponential.bounded.increments.quadratic.corollary2}
    &
    (\mathcal{G}_{\mu,\sigma}U)(x)
     +\tfrac{1}{2}\left\|\sigma(x)^{*}(\nabla U)(x)\right\|^2
     +\bar{U}(x)
  \leq
    \rho \cdot U(x) ,
    \qquad
    \| x \|^{ 1 / c }
    \leq
    c
    \left(
      1 +
      U( x )
    \right)
    .
  \end{align}
  Then it holds that
  $
  \sup_{
    \theta \in \mathcal{P}_T
  }
  \sup_{ t \in [0,T] }
  \E\big[
       \exp(
         e^{ - \rho t }
         \,
         U( Y^{ \theta }_t )
         +
         \smallint_0^t
           e^{ - \rho s }
           \,
           \1_{ D_{ | \theta |_T } }(
             Y^{ \theta }_{\lfloor s \rfloor_{ \theta } }
           )
           \,
           \bar{U}( Y^{ \theta }_s )
         \, ds
       )
     \big]
  <
  \infty
  $
and it holds that
  $
    \limsup_{
      | \theta |_T \searrow 0
    }
    \sup_{ t \in [0,T] }
    \E\big[
      \exp(
        e^{ - \rho t }
        \,
        U( Y_t^{ \theta } )
        +
        \smallint_0^t
            e^{ - \rho s }
            \,
            \1_{ D_{ | \theta |_T } }(
              Y_{ \lfloor s \rfloor_{ \theta } }^{ \theta }
            )
            \,
            \bar{U}( Y_s^{ \theta } )
        \, ds
      )
    \big]
    \leq
    \limsup_{
      | \theta |_T \searrow 0
    }
      \E\big[
        e^{
          U( Y_0^{ \theta } )
        }
      \big]
$.
\end{corollary}

\begin{proof}[Proof
of Corollary~\ref{cor:stopped.Euler.bounded.increments}]
We show Corollary~\ref{cor:stopped.Euler.bounded.increments} through
an application of Theorem~\ref{thm:stopped.Euler.bounded.increments}.
For this let $ \gamma, \alpha \in \R $
be the real numbers with the property that
$ \gamma = c \left( c + 1 \right) $
and
$
  \alpha =
  \frac{ 1 }{ 4 } \min\{ \frac{ 1 }{ 7 \gamma + 2 }, \frac{ q - 1 }{ (q + 8) \gamma + 2 } \}
$
and observe
that
\eqref{eq:mu.sigma.quadratic.exponent.corollary2},
\eqref{eq:generator.exponential.bounded.increments.quadratic.corollary2},
and the assumption that
$
  U \in \cup_{ p \in [1,\infty) } C_{p,c}^3( \R^d, [0,\infty) )
$
ensure that there exist real numbers
$ p \in [1,\infty) $
and 
$ \tilde{c} \in [c,\infty) $
such that
$
  U \in C_{ p, \tilde{c} }^3( \R^d, [0,\infty) )
$,
such that
for all
$ h \in (0,T] $
it holds that
$
  c \exp\!\left( c \, | \ln( h ) |^{ 1 / 2 } \right)
  \leq
  \tilde{c}
  \,
  h^{ - \alpha }
$,
and such that
for all $ x, y \in \R^d $ with $ x \neq y $
it holds that
\begin{equation}
    \|\mu(x)\|
    +\|\sigma(x)\|_{\HS(\R^m,\R^d)}
    +|\bar{U}(x)|
    \leq \tilde{c} \left( 1 + \left| U(x) \right|^{\gamma} \right),
    \;\;
      | \bar{U}(x) - \bar{U}(y) |
    \leq
    \tilde{c} \left( 1 + |U(x)|^{\gamma} + |U(y)|^{\gamma}\right)
      \| x - y \|
    .
\end{equation}
An application of
Theorem~\ref{thm:stopped.Euler.bounded.increments}
thus completes the proof
of Corollary~\ref{cor:stopped.Euler.bounded.increments}.
\end{proof}

Theorem~\ref{thm:stopped.Euler.bounded.increments}
and
Corollary~\ref{cor:stopped.Euler.bounded.increments}
above establish exponential integrability
properties for a family of stopped increment-tamed
Euler-Maruyama approximation schemes.
Another interesting class of approximation
schemes which might admit exponential
integrability properties are certain
{\it rejection- or reflection-type methods}.
More formally, let $ d, m \in \N $, let
$ ( \Omega, \mathcal{F}, \P ) $
be a probability space,
let $ D_t \in \mathcal{B}( \R^d ) $,
$ t \in (0,T] $,
be an appropriate non-increasing family of
sets,
let $ W \colon [0,T] \times \Omega \to \R^m $
be a standard Brownian motion,
and let
$
  Y^N \colon \{ 0, 1, \dots, N \} \times
  \Omega \to \R^d
$,
$ N \in \N $,
be stochastic processes which satisfy 
that for all $ N \in \N $,
$ n \in \{ 0, 1, \dots, N - 1 \} $
that
\begin{equation}
  Y_{n+1}^N=
     Y_n^N+
     \1_{
       \{
         Y_n^N + \mu( Y_n^N )
         \frac{ T }{ N } +
         \sigma( Y_n^N )
         (
           W_{ ( n + 1 ) T / N }
           - W_{ n T / N }
         )
         \in D_{ T / N }
        \}
     }
     \left[
       \mu( Y_n^N )
       \tfrac{ T }{ N }
       +
       \sigma( Y_n^N )
       (
         W_{ \frac{ ( n + 1 ) T }{ N } }
         - W_{ \frac{ n T }{ N } }
       )
    \right]
  .
\end{equation}
Under suitable additional assumptions,
we suspect that the stochastic processes
$ Y^N $, $ N \in \N $, also admit
exponential integrability properties.
In the setting of the Langevin equation,
a similar class of approximation methods has been
considered in
Bou-Rabee \&\
Hairer~\cite{BouRabeeHairer2013}.
Further related approximation methods have
been studied in
Milstein \& Tretjakov~\cite{mt05}.
In \cite{HutzenthalerJentzen2014Memoires}
(see, e.g., Section~3.6.3 in
\cite{HutzenthalerJentzen2014Memoires})
several types of appropriately tamed schemes
have been investigated.
The taming often constitutes by
dividing the increment of an Euler-Maruyama
step through a possibly large number and
thereby decreasing the increment of the
tamed scheme
(cf., e.g., (3.140), (3.141) and (3.145)
in \cite{HutzenthalerJentzen2014Memoires}).
The larger the number by which we divide the
original increment of the Euler-Maruyama step
the stronger is the a priori bound that we
can expect for the tamed scheme.
In particular, if the increment of the
Euler-Maruyama step is tamed by an appropriate
exponential term, then we might obtain
a scheme that admits exponential integrability
properties. For instance, consider
stochastic processes
$
  Z^N \colon \{ 0, 1, \dots, N \} \times \Omega
  \to \R^d
$,
$ N \in \N $,
which satisfy for all 
$ N \in \N $,
$ n \in \{ 0, 1, \dots, N - 1 \} $
that
\begin{equation}
\label{eq:scheme_exp_1}
  Z_{ n + 1 }^N
  =
  Z^N_n
  +
  \frac{
    \mu( Z^N_n ) \frac{ T }{ N }
    +
    \sigma( Z^N_n )
    (
      W_{ ( n + 1 ) T / N }
      -
      W_{ n T / N }
    )
  }{
    \exp\!\left(
      \left\|
    \mu( Z^N_n ) \frac{ T }{ N }
    +
    \sigma( Z^N_n )
    (
      W_{ ( n + 1 ) T / N }
      -
      W_{ n T / N }
    )
      \right\|^2
    \right)
  }
\end{equation}
or, more generally, consider
stochastic processes
$
  Z^N \colon \{ 0, 1, \dots, N \} \times \Omega
  \to \R^d
$,
$ N \in \N $,
which satisfy that there exist
(appropriate) $ \alpha, \beta, \gamma \in \R $
such that
for all $ N \in \N $,
$ n \in \{ 0, 1, \dots, N - 1 \} $
it holds that
\begin{equation}
\label{eq:scheme_exp_2}
  Z_{ n + 1 }^N
  =
  Z^N_n
  +
  \frac{
    \mu( Z^N_n ) \frac{ T }{ N }
    +
    \sigma( Z^N_n )
    (
      W_{ ( n + 1 ) T / N }
      -
      W_{ n T / N }
    )
  }{
    \max\!\left\{
    1,
    \frac{
      T^{ \alpha }
    }{
      N^{ \alpha }
    }
    \exp\!\left(
    \frac{
      T^{ \beta }
    }{
      N^{ \beta }
    }
      \left\|
    \mu( Z^N_n ) \frac{ T }{ N }
    +
    \sigma( Z^N_n )
    (
      W_{ ( n + 1 ) T / N }
      -
      W_{ n T / N }
    )
      \right\|^{ \gamma }
    \right)
    \right\}
  }
  .
\end{equation}
Under suitable assumptions,
it might be the case that
schemes of the form
\eqref{eq:scheme_exp_1} and
\eqref{eq:scheme_exp_2}
admit exponential integrability
properties.

%

\section{Consistency and convergence 
of a class of stopped and tamed schemes}
\label{sec:consistency}

In Section~\ref{sec:exponential} 
exponential integrability properties for 
certain numerical approximation processes of SDEs have been established.
In this section we show under suitable assumptions that these approximation processes 
converge in probability and strongly to the exact solution process of the considered SDE;
see 
Corollary~\ref{cor:convergence_increment_tamed}
and 
Corollary~\ref{cor:for_examples}
in
Subsection~\ref{sec:convergence_stopped}.
For this we extend the notions and the convergence results in Sections~3.2--3.4 in 
\cite{HutzenthalerJentzen2014Memoires}.
More specifically,
in Theorem~3.3 in \cite{HutzenthalerJentzen2014Memoires} convergence in
probability has, under suitable assumptions, 
been established for numerical approximations
that are $ ( \mu, \sigma ) $-consistent in the sense
of 
Definition~3.1 in \cite{HutzenthalerJentzen2014Memoires}.
In this article we slightly generalize this notion 
(Definition~3.1 in \cite{HutzenthalerJentzen2014Memoires})
and the corresponding 
convergence in probability result 
(Theorem~3.3 in \cite{HutzenthalerJentzen2014Memoires})
in
Definition~\ref{def:consistent},
Proposition~\ref{prop:convergence_probab},
and Proposition~\ref{prop:convergence.interpolation}
below.
In addition, 
we establish several auxiliary results that 
provide sufficient conditions to ensure that
a considered approximation scheme 
is $ ( \mu , \sigma ) $-consistent 
in the sense of Definition~\ref{def:consistent}
below;
see Lemma~\ref{lem:consistent_stopped}
for consistency for a class of stopped schemes,
see Lemma~\ref{lem:consistent_incrementtamed}
for consistency for a class of increment-tamed Euler-Maruyama schemes,
and see Corollary~\ref{lem:consistent_stoppedtamed}
(which is an immediate consequence from Lemma~\ref{lem:consistent_stopped}
and Lemma~\ref{lem:consistent_incrementtamed})
for consistency for a class of stopped increment-tamed
Euler-Maruyama schemes.
As a consequence of 
Corollary~\ref{lem:consistent_stoppedtamed}
and Proposition~\ref{prop:convergence.interpolation}
we then obtain 
convergence in probability (and, under additional assumptions, 
also strong convergence) of the stopped increment-tamed
Euler-Maruyama schemes;
see Corollary~\ref{cor:convergence_increment_tamed}.
Combining Corollary~\ref{cor:convergence_increment_tamed}, in turn,
with the exponential integrability result
in Corollary~\ref{cor:stopped.Euler.bounded.increments}
in Section~\ref{sec:exponential}
will then allow us to derive 
Corollary~\ref{cor:for_examples}
(the main result of this article).

\begin{definition}[Consistency]
\label{def:consistent}
We say that
$ \phi $ is $ ( \mu, \sigma ) $-consistent
with respect to Brownian motion
if and only if
there exist real numbers 
$ T \in (0,\infty) $, $ d, m \in \N $,
an open set $ D \subseteq \R^d $,
a probability space 
$ ( \Omega, \mathcal{F}, \P ) $,
and a standard Brownian
motion
$
  W \colon [0,T] \times \Omega \to \R^m
$
such 
\begin{enumerate}[(i)]
\item 
that
$
  \phi 
  \in
  \mathbbm{M}(
    \R^d \times (0,T] \times \R^m
    ,
    \R^d
  )
$,
\item  
that
$ \mu \in \mathbbm{M}( D , \R^d ) $,
\item
that
$ \sigma \in \mathbbm{M}( D , \R^{ d \times m } ) $,
\item  
that 
$
  \forall \, t \in (0,T] 
  \colon
  (
    \R^d \times \R^m
    \ni (x,y) \mapsto
    \phi( x, t, y ) \in \R^d
  )
  \in
  \mathcal{M}(
    \mathcal{B}( \R^d \times \R^m ) ,
    \mathcal{B}( \R^d ) 
  )
$,
and 
\item
that
for all non-empty compact
sets $ K \subseteq D $ it holds that
\begin{equation}
  \limsup_{ t \searrow 0 }
  \left(
    \tfrac{ 1 }{ \sqrt{ t } }
    \cdot
    \sup_{ x \in K }
    \E\big[
      \|
        \sigma( x ) W_t
        -
        \phi( x, t, W_t )
      \|
    \big]
  \right)
  = 0
  =
  \limsup_{ t \searrow 0 }
  \left(
    \sup_{ x \in K }
    \left\|
      \mu( x )
      -
      \tfrac{ 1 }{ t } \cdot
      \E\!\left[
        \phi( x, t, W_t )
      \right]
    \right\|
  \right)
  .
\end{equation}
\end{enumerate}
\end{definition}

In Definition~3.1 in
\cite{HutzenthalerJentzen2014Memoires},
the increment function $ \phi $ is assumed
to be Borel measurable while
in Definition~\ref{def:consistent}
above the increment function $ \phi $ does not
need to be Borel measurable in all three arguments
$ (x,t,y) \in \R^d \times (0,T] \times \R^m $
(see Definition~\ref{def:consistent} for details).
In Proposition~\ref{prop:convergence_probab}
below it is shown under suitable
assumptions that if a numerical one-step
scheme is
consistent in the sense of
Definition~\ref{def:consistent},
then it converges in probability
to the exact solution of the considered SDE
(cf.\ also
Corollaries~3.11--3.13
in \cite{HutzenthalerJentzen2014Memoires}
for strong convergence results
based on consistency).

\subsection{Consistency of stopped schemes}

The next lemma establishes consistency for 
appropriately stopped numerical approximation schemes.

\begin{lemma}
\label{lem:consistent_stopped}
  Let $ T \in (0,\infty) $,
  $ d, m \in \mathbb{N} $,
  let $ D \subseteq \mathbb{R}^d $
  be an open set,
  let
  $
    \mu \colon
    D \rightarrow \mathbb{R}^d
  $
  and
  $
    \sigma \colon
    D \rightarrow
    \mathbb{R}^{ d \times m }
  $
  be functions,
  let
  $
    \phi
    \colon
    \R^d \times (0,T] \times \R^m
    \to \R^d
  $
  be
  $ ( \mu, \sigma ) $-consistent
  with respect to Brownian motion,
  and
  let
  $ D_t \in \mathcal{B}( \R^d ) $,
  $ t \in (0,T] $,
  be a non-increasing family of sets satisfying
  $
    D
    \subseteq
    \cup_{ t \in (0,T] }
    \mathring{ D }_t
  $.
  Then the function
  $
    \R^d \times (0,T] \times \R^m
    \ni ( x, t, y ) \mapsto
    \mathbbm{1}_{ D_t }( x )
    \cdot
    \phi( x, t, y )
    \in \R^d
  $
  is
  $
    (\mu, \sigma)
  $-consistent with
  respect to Brownian motion.
\end{lemma}

\begin{proof}[Proof
of Lemma~\ref{lem:consistent_stopped}]
Throughout this proof assume w.l.o.g.\ that
$ D \neq \emptyset $,
let $ K \subseteq D $ be an arbitrary 
non-empty compact subset of $ D $,
let $ ( \Omega, \mathcal{F}, \P ) $ 
be a probability space, 
and let 
$ W \colon [0,T] \times \Omega \to \R^m $
be a standard Brownian motion.
The fact that $ K $ is a compact set and
the assumption that
$
  D \subseteq \cup_{ t \in (0,T] } \mathring{ D }_t
$
together 
with 
the assumption that
$
  \phi
$
is 
$ ( \mu, \sigma ) $-consistent with respect to Brownian motion
ensures that there exists a real number
$ t_K \in (0,T] $
such that
$
  K \subseteq \mathring{ D }_{ t_K }
$
and 
\begin{equation}
  \sup_{ t \in (0,t_K] }
  \left(
    \tfrac{ 1 }{ \sqrt{ t } }
    \cdot 
    \sup_{ x \in K }
    \E\big[ 
      \|
        \sigma( x ) W_t
        -
        \phi( x, t, W_t )
      \|
    \big]
  \right)
  < \infty
  .
\end{equation}
The fact that the family
$ D_t $, $ t \in (0,T] $,
is non-increasing 
hence shows that
for all $ t \in (0, t_K] $ 
it holds that
  \begin{equation} \begin{split}
      \tfrac{1}{\sqrt{t}}
      \cdot
      \sup_{
        x \in K
      }
      \mathbb{E}\Big[
      \big\|
        \sigma(x)
        W_t
        -
          \1_{D_t }(x)
          \,
          \phi(
            x,
            t,
            W_t
          )
      \big\|
      \Big]
    =
    \tfrac{1}{\sqrt{t}}
      \cdot
      \sup_{
        x \in K
      }
      \mathbb{E}\Big[
      \big\|
        \sigma(x)
        W_t
        -
           \phi(
            x,
            t,
            W_t
          )
      \big\|
      \Big]
  \end{split} \end{equation}
  and
  \begin{equation} \begin{split}
      \sup_{
        x \in K
      }
      \big\|
        \mu(x)
        -
        \tfrac{ 1 }{ t }
        \cdot
        \E\big[
          \1_{D_t }(x)
          \,
          \phi(
            x,
            t,
            W_t
          )
        \big]
      \big\|
    =
      \sup_{
        x \in K
      }
      \big\|
        \mu(x)
        -
        \tfrac{ 1 }{ t }
        \cdot
        \E\big[
          \phi(
            x,
            t,
            W_t
          )
        \big]
      \big\|
      .
  \end{split} \end{equation}
  Combining this with the assumption that
  $
    \phi
  $
  is
  $ ( \mu, \sigma ) $-consistent
  with respect to Brownian motion implies
  \begin{equation}
  \label{eq:stopped.consistency_1}
    \limsup_{ t \searrow 0 }
    \left(
    \tfrac{1}{\sqrt{t}}
      \cdot
      \sup_{
        x \in K
      }
      \mathbb{E}\Big[
      \big\|
        \sigma(x)
        W_t
        -
          \1_{D_t }(x) \phi(
            x,
            t,
            W_t
          )
      \big\|
      \Big]
    \right) = 0
  \end{equation}
  and
  \begin{equation} \begin{split}
    \label{eq:stopped.consistency_2}
    \limsup_{ t \searrow 0 }
    \left(
    \sup_{
      x \in K
    }
    \left\|
      \mu(x)
      -
      \tfrac{1}{t}
      \cdot
       \mathbb{E}\big[
        \1_{D_t }(x) \phi(
          x,
          t,
          W_t
        )
      \big]
    \right\|
    \right)
   & = 0.
  \end{split}\end{equation}
Combining
\eqref{eq:stopped.consistency_1}
and
\eqref{eq:stopped.consistency_2}
with
Definition~\ref{def:consistent}
completes the proof of
Lemma~\ref{lem:consistent_stopped}.
\end{proof}

\subsection{Consistency of a class of incremented-tamed Euler-Maruyama schemes}

The following lemma proves consistency for a class of increment-tamed Euler-Maruyama
approximation schemes for SDEs.

\begin{lemma}
\label{lem:consistent_incrementtamed}
  Let $ T \in (0,\infty) $, $ q \in (1,\infty) $,
  $ d, m \in \mathbb{N} $,
  let $ D \subseteq \mathbb{R}^d $
  be an open set, and let
  $
    \mu \in
    \mathcal{M}( \mathcal{B}( D ) , \mathcal{B}( \mathbb{R}^d ) )
  $
  and
  $
    \sigma \in
    \mathcal{M}( 
      \mathcal{B}( D ) ,
      \mathcal{B}( \mathbb{R}^{ d \times m } )
    )
  $
  be locally bounded.
Then it holds that the function
\begin{equation}
    \R^d \times (0,T] \times \R^m
    \ni ( x, t, y ) \mapsto
  \left\{
  \begin{array}{ll}
    \frac{
      \mu( x ) t + \sigma( x ) y
    }{
      1 +
      \left\| \mu( x ) t + \sigma( x ) y \right\|^q
    }
  &
  \colon
    x \in D 
  \\[1ex]
    0
  &
  \colon
    x \in \R^d \backslash D
  \end{array}
  \right\}
    \in \R^d
\end{equation}
  is
  $
    (\mu, \sigma)
  $-consistent with
  respect to Brownian motion.
\end{lemma}

\begin{proof}[Proof
of Lemma~\ref{lem:consistent_incrementtamed}]
Throughout this proof 
let $ ( \Omega, \mathcal{F}, \P ) $ 
be a probability space
and let 
$ W \colon [0,T] \times \Omega \to \R^m $
be a standard Brownian motion.
Observe that for all non-empty compact sets
$ K \subseteq D $ 
it holds that
  \begin{equation} \begin{split}
  \label{eq:tamed.consistency_1}
    &
    \limsup_{ t \searrow 0 }
    \left(
      \tfrac{1}{\sqrt{t}}
      \cdot
      \sup_{
        x \in K
      }
      \mathbb{E}\bigg[
      \Big\|
        \sigma(x)
        W_t
        -
          \frac{
      \mu( x ) t + \sigma( x ) W_t
    }{
      1 +
      \left\| \mu( x ) t + \sigma( x ) W_t \right\|^q
    }
      \Big\|
      \bigg]
    \right)
  \\ & =
  \limsup_{ t \searrow 0 }
    \left(
      \tfrac{1}{\sqrt{t}}
      \cdot
      \sup_{
        x \in K
      }
      \mathbb{E}\bigg[
      \Big\|
          \frac{
      \sigma( x ) W_t \left\| \mu( x ) t + \sigma( x ) W_t \right\|^q  - \mu( x ) t
    }{
      1 +
      \left\| \mu( x ) t + \sigma( x ) W_t \right\|^q
    }
      \Big\|
      \bigg]
    \right)
  \\ & \leq
  \limsup_{ t \searrow 0 }
    \left(
      \tfrac{1}{\sqrt{t}}
      \cdot
      \sup_{
        x \in K
      }
      \mathbb{E}\Big[
        \big\|
          \sigma( x ) W_t \left\| \mu( x ) t + \sigma( x ) W_t \right\|^q  - \mu( x ) t
        \big\|
      \Big]
    \right)
    \\ & \leq
  \limsup_{ t \searrow 0 }
    \left(
      \tfrac{1}{\sqrt{t}}
      \cdot
      \sup_{
        x \in K
      }
      \mathbb{E}\big[
        \left\| \sigma( x ) W_t \right\|
        \left\| \mu( x ) t + \sigma( x ) W_t \right\|^q
      \big]
    \right)
    +
  \limsup_{ t \searrow 0 }
    \left(
        \sqrt{t}
        \cdot
      \sup_{
        x \in K
      }
        \| \mu( x ) \|
    \right)
  \\ & \leq
  2^{ (q-1) } \cdot
  \limsup_{ t \searrow 0 }
    \left(
      t^{ \left( q - \frac{ 1 }{ 2 } \right) }
      \left[
      \sup_{
        x \in K
      }
        \left\| \sigma( x ) \right\|_{ L( \R^m, \R^d ) }
        \left\| \mu( x ) \right\|^q
      \right]
      \mathbb{E}\big[
        \| W_t \|
      \big]
    \right)
\\ &
    +
  2^{ (q-1) } \cdot
  \limsup_{ t \searrow 0 }
    \left(
      \tfrac{1}{\sqrt{t}}
      \left[
      \sup_{
        x \in K
      }
        \left\|
          \sigma( x )
        \right\|_{
          L( \R^m, \R^d )
        }^{ ( 1 + q ) }
      \right]
      \mathbb{E}\!\left[
        \left\| W_t \right\|^{ ( 1 + q ) }
      \right]
    \right)
    = 0 .
  \end{split} \end{equation}
In addition, note that for all non-empty compact sets 
$ K \subseteq D $ it holds that
  \begin{equation} \begin{split}
  \label{eq:tamed.consistency_2}
    &
    \limsup_{ t \searrow 0 }
    \left(
    \sup_{
      x \in K
    }
    \left\|
      \mu(x)
      -
      \tfrac{1}{t}
      \cdot
       \mathbb{E}\!\left[
        \frac{
      \mu( x ) t + \sigma( x ) W_t
    }{
      1 +
      \left\| \mu( x ) t + \sigma( x ) W_t \right\|^q
    }
      \right]
    \right\|
    \right)
  \\ & \leq
   \limsup_{ t \searrow 0 }
    \left( \sup_{
      x \in K
    }
    \left\| \mathbb{E}\!\left[
        \frac{
        \mu(x)
      \left\| \mu( x ) t + \sigma( x ) W_t \right\|^q
    }{
      1 +
      \left\| \mu( x ) t + \sigma( x ) W_t \right\|^q
    }
    \right] \right\| \right)
    +
    \limsup_{ t \searrow 0 }
    \left(
      \tfrac{ 1 }{ t } \cdot
      \sup_{
        x \in K
      }
      \left\|
        \mathbb{E}\!\left[
          \frac{
            \sigma( x ) W_t
          }{
            1 +
            \left\| \mu( x ) t + \sigma( x ) W_t \right\|^q
          }
        \right]
      \right\|
    \right)
  \\ & \leq
   \limsup_{ t \searrow 0 }
    \left(
      \sup_{
        x \in K
      }
      \mathbb{E}\big[
        \left\| \mu(x) \right\|
        \left\| \mu( x ) t + \sigma( x ) W_t \right\|^q
      \big]
    \right)
  \\ & +
    \limsup_{ t \searrow 0 }
    \left(
      \tfrac{ 1 }{ t } \cdot
      \sup_{
        x \in K
      }
      \left\|
        \mathbb{E}\!\left[
          \frac{
            \sigma( x ) W_t
          }{
            1 +
            \left\| \mu( x ) t + \sigma( x ) W_t \right\|^q
          }
          -
          \sigma( x ) W_t
        \right]
      \right\|
    \right)
\\ & =
    \limsup_{ t \searrow 0 }
    \left(
      \tfrac{ 1 }{ t } \cdot
      \sup_{
        x \in K
      }
      \left\|
        \mathbb{E}\!\left[
          \frac{
            \sigma( x ) W_t
            \left\| \mu( x ) t + \sigma( x ) W_t \right\|^q
          }{
            1 +
            \left\| \mu( x ) t + \sigma( x ) W_t \right\|^q
          }
        \right]
      \right\|
    \right)
\\ & \leq
    \limsup_{ t \searrow 0 }
    \left(
      \tfrac{ 1 }{ t } \cdot
      \sup_{
        x \in K
      }
        \mathbb{E}\big[
          \left\| \sigma( x ) \right\|_{ L( \R^m, \R^d ) }
          \left\| W_t \right\|
          \left\| \mu( x ) t + \sigma( x ) W_t \right\|^q
        \big]
    \right)
    = 0
    .
  \end{split} \end{equation}
Combining \eqref{eq:tamed.consistency_1} and \eqref{eq:tamed.consistency_2}
with Definition~\ref{def:consistent}
completes the proof of Lemma~\ref{lem:consistent_incrementtamed}.
\end{proof}

\subsection{Convergence of stopped increment-tamed Euler-Maruyama schemes}
\label{sec:convergence_increment_tamed}

This subsection establishes consistency, convergence in probability, 
strong convergence, and numerically weak convergence
of a class of stopped increment-tamed Euler-Maruyama schemes.

\subsubsection{Setting}
\label{subsec:converg.prob.setting}

Throughout Subsection~\ref{sec:convergence_increment_tamed} the following
setting is frequently used.
Let $ T \in ( 0, \infty) $,
$ d, m \in \N $,
let $ D \subseteq \mathbb{R}^d $
be an open set,
let 
$
  (
    \Omega, \mathcal{F}, \P
  )
$ 
be a probability space with a normal filtration 
$
  (
    \mathcal{F}_t
  )_{
    t \in [0, T ]
  }
$,
let $ W \colon [0,T] \times \Omega \to \R^m $
be a standard $ ( \mathcal{F}_t )_{ t \in [0,T] } $-Brownian
motion with continuous sample paths,
let $ \mu \colon D \to \R^d $
and $ \sigma \colon D \to \R^{ d \times m } $
be locally Lipschitz continuous functions,
and let $ X \colon [0,T] \times \Omega \to D $
be an 
$ ( \mathcal{F}_t )_{ t \in [0,T] } $-adapted
stochastic process with continuous sample paths
which satisfies that for all $ t \in [0,T] $
it holds $ \P $-a.s.\ that
\begin{equation}
  X_t = X_0 + \int_0^t \mu( X_s ) \, ds
  +
  \int_0^t \sigma( X_s ) \, dW_s
  .
\end{equation}

\subsubsection{Convergence in probability of appropriate 
time-continuous interpolations}

The next proposition,
Proposition~\ref{prop:convergence_probab},
is a slight generalization
of Theorem~3.3
in \cite{HutzenthalerJentzen2014Memoires}.
The proof of
Proposition~\ref{prop:convergence_probab}
is entirely analogous to
the proof of Theorem~3.3
in \cite{HutzenthalerJentzen2014Memoires}
and therefore omitted.

\begin{prop}
\label{prop:convergence_probab}
Assume the setting in Subsection~\ref{subsec:converg.prob.setting},
let
$
  \phi \colon
  \R^d \times (0,T] \times \R^m
  \to \R^d
$
be
$ ( \mu, \sigma ) $-consistent,
and
let
$ Y^N \colon [0,T] \times \Omega \to \R^d $,
$ N \in \N $,
be mappings satisfying
for all
$ N \in \N $,
$
  n \in \{ 0, 1, \ldots, N - 1 \}
$,
$
  t \in
  ( \frac{ n T }{ N }, \frac{ ( n + 1 ) T }{ N } ]
$
that
$ Y^N_0 = X_0 $
and
$
  Y^N_t
  =
  Y^N_{ \frac{n T}{N} }
  +
  \left(
    \tfrac{ t N }{ T } - n
  \right)
  \cdot
  \phi\big(
    Y^N_{ \frac{n T}{N} }
    ,
    \tfrac{T}{N}
    ,
    W_{
      \frac{ ( n + 1 ) T }{ N }
    } - W_{ \frac{n T}{N} }
  \big)
$.
Then
it holds for all $ \varepsilon \in ( 0, \infty ) $
that
$
  \limsup_{ N \to \infty }
  \P\big[
    \sup_{ t \in [0,T] }
    \|
      X_t - Y_t^N
    \|
    \geq
    \varepsilon
  \big]
  = 0
$.
\end{prop}

The next proposition is an extension of
Proposition~\ref{prop:convergence_probab}
and
proves convergence in probability of suitable
time-continuous interpolations of numerical approximation processes
of consistent schemes.

\begin{prop}
\label{prop:convergence.interpolation}
Assume the setting in Subsection~\ref{subsec:converg.prob.setting},
let
$
  \Psi
  \colon
  \R^d \times (0,T]^2 \times \R^m
  \to \R^d
$
be a function
satisfying that
$
  \R^d \times (0,T] \times \R^m
  \ni
  ( x, t, y )
  \mapsto
  \Psi( x, t, t, y ) \in \R^d
$
is
$ ( \mu, \sigma ) $-consistent,
let
$ Y^N \colon [0,T] \times \Omega \to \R^d $,
$ N \in \N $,
be stochastic processes with continuous
sample paths
satisfying
for all
$ N \in \N $,
$ n \in \{ 0, 1, \ldots, N - 1 \} $,
$ t \in ( \frac{n T}{N}, \frac{(n + 1) T}{N} ] $
that
$ Y^N_0 = X_0 $
and
$
  Y^N_t
  =
  Y^N_{ \frac{n T}{N} }
  +
  \Psi\big(
    Y^N_{ \frac{n T}{N} }
    ,
    \tfrac{T}{N}
    ,
    t - \tfrac{n T}{N}
    ,
    W_t - W_{ \frac{n T}{N} }
  \big)
$,
assume
for all non-empty compact sets $ K \subseteq D $
that there exists an $ h_K \in (0,T] $
such that for all $ h \in ( 0, h_K ] $
it holds that
$
    \sup_{ x \in K }
    \sup_{ t \in [0,T] }
    \|
      \Psi(
        x, h, t -\lfloor t \rfloor_h,
        W_t - W_{\lfloor t \rfloor_h}
      )
    \|
$
is
$ \mathcal{F} $/$ \mathcal{B}( \R )
$-measurable,
and
assume
for
all non-empty compact sets $ K \subseteq D $
that
\begin{equation}
\label{eq:convergence.expectation.zero}
  \limsup_{ h \searrow 0 }
  \E\!\left[
    \sup_{ x \in K }
    \sup_{ t \in [0,T] }
    \left\|
      \Psi(
        x, h, t -\lfloor t \rfloor_h,
        W_t - W_{\lfloor t \rfloor_h}
      )
    \right\|
  \right]
  = 0 .
\end{equation}
Then
it holds for all $ \varepsilon \in ( 0, \infty ) $
that
$
  \limsup_{ N \to \infty }
  \P\big[
    \sup_{ t \in [0,T] }
    \|
      X_t - Y_t^N
    \|
    \geq
    \varepsilon
  \big]
  = 0
$.
\end{prop}

\begin{proof}[Proof
of Proposition~\ref{prop:convergence.interpolation}]
The triangle inequality implies for all $ N \in \N $ that
\begin{equation}
\sup_{ t \in [0,T] }
    \left\|
      X_t - Y_t^N
    \right\|
  \leq
    \sup_{ t \in [0,T] }
    \left\|
      X_t
      -
      X_{\lfloor t \rfloor_{ T / N } }
    \right\|
    +
    \sup_{ t \in [0,T] }
    \left\|
      X_{\lfloor t \rfloor_{ T / N } }
      -
      Y_{\lfloor t \rfloor_{ T / N } }^N
    \right\|
    +
    \sup_{ t \in [0,T] }
    \left\|
      Y_t^N -Y_{\lfloor t \rfloor_{ T / N } }^N
    \right\|
    .
\end{equation}
Combining this, Proposition~\ref{prop:convergence_probab},
and sample paths continuity of
$ X_t $, $ t \in [0, T] $,
implies for all $ \varepsilon \in ( 0, \infty ) $ that
\begin{equation} \begin{split}
\label{eq:convergence.split}
&
  \limsup_{ N \to \infty }
  \P\!\left[
    \sup_{ t \in [0,T] }
    \left\|
      X_t - Y_t^N
    \right\|
    \geq
    \varepsilon
  \right]
\\  & \leq
  \limsup_{ N \to \infty }
  \P\!\left[
    \sup_{ t \in [0,T] }
    \left\|
      X_t
      -
      X_{\lfloor t \rfloor_{ T / N } }
    \right\|
    \geq
    \tfrac{ \varepsilon }{ 3 }
  \right]
  +
  \limsup_{ N \to \infty }
  \P\!\left[
    \sup_{ t \in [0,T] }
    \left\|
      X_{\lfloor t \rfloor_{ T / N } }
      -
      Y_{\lfloor t \rfloor_{ T / N } }^N
    \right\|
    \geq
    \tfrac{ \varepsilon }{ 3 }
  \right]
\\ &
  +
  \limsup_{ N \to \infty }
  \P\!\left[
    \sup_{ t \in [0,T] }
    \left\|
      Y_t^N -Y_{ \lfloor t \rfloor_{ T / N } }^N
    \right\|
    \geq
    \tfrac{ \varepsilon }{ 3 }
  \right]
\\ & =
  \limsup_{ N \to \infty }
  \P\!\left[
    \sup_{ t \in [0,T] }
    \left\|
      Y_t^N -Y_{\lfloor t \rfloor_{ T / N } }^N
    \right\|
    \geq
    \tfrac{ \varepsilon }{ 3 }
  \right]
  .
\end{split} \end{equation}
It thus remains to prove that
$
    \sup_{ t \in [0,T] }
    \|
      Y_t^N -Y_{\lfloor t \rfloor_{ T / N } }^N
    \|
$
converges to zero
in probability
as $ N \to \infty $.
To prove this let
$ D_v \subseteq D $, $v \in \N$, 
be open sets 
with the property that 
for all $ v \in \N $
it holds that
$
  D_v = \{ x \in D\colon \|x\| < v \text{ and } \text{dist}(x, D^c) > \tfrac{1}{v} \}
$.
Then \eqref{eq:convergence.expectation.zero} 
and Markov's inequality show for all
$ \varepsilon \in ( 0, \infty ) $, $ v \in \N $
that
\begin{equation}
\label{eq:Dv}
\begin{split}
  &
  \limsup_{ N \to \infty }
  \P\!\left[
    \left\{
      \sup\nolimits_{ t \in [0,T] }
      \big\|
        Y_t^N -Y_{\lfloor t \rfloor_{ T/N } }^N
      \big\|
      \geq
      \varepsilon
    \right\}
    \cap
    \left\{
      \forall \, t \in [0,T]
      \colon
      Y^N_{ \lfloor t \rfloor_{ T / N } } \in D_v
    \right\}
  \right]
  \\ & =
  \limsup_{ N \to \infty }
  \P\!\left[
      \big(
        \sup\nolimits_{ t \in [0,T] }
        \big\|
          Y_t^N -Y_{\lfloor t \rfloor_{ T/N } }^N
        \big\|
        \,
      \big)
      \,
      \mathbbm{1}_{
        \{
          \forall \, t \in [0,T]
          \colon
          Y^N_{ \lfloor t \rfloor_{ T / N } } \in D_v
        \}
      }
      \geq
      \varepsilon
  \right]
  \\ &
  \leq
  \frac{ 1 }{ \varepsilon }
  \,
  \limsup_{ N \to \infty }
  \E\!\left[
      \big(
        \sup\nolimits_{ t \in [0,T] }
        \big\|
          Y_t^N - Y_{\lfloor t \rfloor_{ T/N } }^N
        \big\|
        \,
      \big)
      \,
      \mathbbm{1}_{
        \{
          \forall \, t \in [0,T]
          \colon
          Y^N_{ \lfloor t \rfloor_{ T / N } } \in D_v
        \}
      }
  \right]
  \\ &
  =
  \frac{ 1 }{ \varepsilon }
  \,
  \limsup_{ N \to \infty }
  \E\!\left[
      \big(
        \sup\nolimits_{ t \in [0,T] }
        \big\|
          \Psi\big(
            Y_{ \lfloor t \rfloor_{ T / N } }^N
            ,
            \tfrac{ T }{ N }
            ,
            t - \lfloor t \rfloor_{ T / N }
            ,
            W_t - W_{ \lfloor t \rfloor_{ T/N } }
          \big)
        \big\|
        \,
      \big)
      \,
      \mathbbm{1}_{
        \{
          \forall \, t \in [0,T]
          \colon
          Y^N_{ \lfloor t \rfloor_{ T / N } } \in D_v
        \}
      }
  \right]
  \\ & \leq
  \frac{ 1 }{ \varepsilon }
  \,
  \limsup_{ N \to \infty }
  \E\!\left[
        \sup_{ x \in \overline{ D_v } }
        \sup_{ t \in [0,T] }
        \big\|
          \Psi\big(
            x
            ,
            \tfrac{ T }{ N }
            ,
            t - \lfloor t \rfloor_{ T / N }
            ,
            W_t - W_{ \lfloor t \rfloor_{ T/N } }
          \big)
        \big\|
  \right]
  = 0
  .
  \end{split} \end{equation}
  In addition,
  Proposition~\ref{prop:convergence_probab}
  and the continuity of the sample paths
  of $ X $
  imply
  for all $ \varepsilon \in ( 0, \infty ) $
  that
  \begin{equation}
  \label{eq:Dv_complement}
  \begin{split}
  &
  \limsup_{ v \to \infty }
  \limsup_{ N \to \infty }
  \P\!\left[
    \left\{
      \sup\nolimits_{ t \in [0,T] }
      \big\|
        Y_t^N -Y_{\lfloor t \rfloor_{ T/N } }^N
      \big\|
      \geq
      \varepsilon
    \right\}
    \cap
    \left\{
      \exists \, t \in [0,T]
      \colon
      Y^N_{ \lfloor t \rfloor_{ T / N } } \notin D_{ 2 v }
    \right\}
  \right]
  \\ & \leq
  \limsup_{ v \to \infty }
  \limsup_{ N \to \infty }
  \P\!\left[
      \exists \, t \in [0,T]
      \colon
      Y^N_{ \lfloor t \rfloor_{ T / N } } \notin D_{ 2 v }
  \right]
  \\ & \leq
  \limsup_{ v \to \infty }
  \limsup_{ N \to \infty }
  \P\!\left[
    \left\{
      \sup\nolimits_{ t \in [0,T] }
      \big\|
        X_{ \lfloor t \rfloor_{ T/N } }
        -
        Y_{ \lfloor t \rfloor_{ T/N } }^N
      \big\|
      <
      \tfrac{ 1 }{ 2 v }
    \right\}
    \cap
    \left\{
      \exists \, t \in [0,T]
      \colon
      Y^N_{ \lfloor t \rfloor_{ T / N } } \notin D_{ 2 v }
    \right\}
  \right]
  \\ & +
  \limsup_{ v \to \infty }
  \limsup_{ N \to \infty }
  \P\!\left[
      \sup\nolimits_{ t \in [0,T] }
      \big\|
        X_{ \lfloor t \rfloor_{ T/N } }
        -
        Y_{ \lfloor t \rfloor_{ T/N } }^N
      \big\|
      \geq
      \tfrac{ 1 }{ 2 v }
  \right]
  \\ & =
  \limsup_{ v \to \infty }
  \limsup_{ N \to \infty }
  \P\!\left[
    \left\{
      \sup\nolimits_{ t \in [0,T] }
      \big\|
        X_{ \lfloor t \rfloor_{ T/N } }
        -
        Y_{ \lfloor t \rfloor_{ T/N } }^N
      \big\|
      <
      \tfrac{ 1 }{ 2 v }
    \right\}
    \cap
    \left\{
      \exists \, t \in [0,T]
      \colon
      Y^N_{ \lfloor t \rfloor_{ T / N } } \notin D_{ 2 v }
    \right\}
  \right]
  \\ & \leq
  \limsup_{ v \to \infty }
  \limsup_{ N \to \infty }
  \P\!\left[
    \exists \, t \in [0,T]
    \colon
    X_{ \lfloor t \rfloor_{ T / N } } \notin D_v
  \right]
  \leq
  \limsup_{ v \to \infty }
  \P\big[
    \exists \, t \in [0,T]
    \colon
    X_t \notin D_v
  \big]
  = 0
  .
  \end{split}
  \end{equation}
  Combining \eqref{eq:Dv} and
  \eqref{eq:Dv_complement} proves that
  $
    \sup_{ t \in [0,T] }
    \|
      Y^N_t -
      Y^N_{
        \lfloor t \rfloor_{ T / N }
      }
    \|
  $
  converges in probability to zero
  as $ N $ tends to infinity.
  The proof of Proposition~\ref{prop:convergence.interpolation}
  is thus completed.
\end{proof}

\subsubsection{Convergence of stopped increment-tamed Euler-Maruyama schemes}
\label{sec:convergence_stopped}

Combining Lemma~\ref{lem:consistent_stopped} and
Lemma~\ref{lem:consistent_incrementtamed}
immediately proves the following consistency result.

\begin{corollary}
\label{lem:consistent_stoppedtamed}
  Let $ T \in (0,\infty) $,
  $ q \in (1,\infty) $,
  $ d, m \in \mathbb{N} $,
  let $ D \subseteq \R^d $
  be an open set,
  let
  $ D_t \in \mathcal{B}( \R^d ) $,
  $ t \in (0,T] $,
  be a non-increasing family of sets satisfying
  $
    D
    \subseteq
    \cup_{ t \in (0,T] }
    \mathring{ D }_t
  $,
  and let
  $
    \mu \in \mathcal{M}( \mathcal{B}( D ), \mathcal{B}( \R^d ) )
  $
  and
  $
    \sigma \in \mathcal{M}( \mathcal{B}( D ), \mathcal{B}( \R^{ d \times m } ) )
  $
  be locally bounded.
  Then it holds that the function
\begin{equation}
    \R^d \times (0,T] \times \R^m
    \ni ( x, t, y ) \mapsto
  \left\{
  \begin{array}{ll}
      \frac{
    \mathbbm{1}_{ D_t }( x )
    \,
    \left[
        \mu( x ) t + \sigma( x ) y
    \right]
      }{
        1 +
        \left\| \mu( x ) t + \sigma( x ) y \right\|^q
      }
  &
  \colon
    x \in D 
  \\[1ex]
    0
  &
  \colon
    x \in \R^d \backslash D
  \end{array}
  \right\}
    \in \R^d
\end{equation}
  is
  $
    (\mu, \sigma)
  $-consistent with
  respect to Brownian motion.
\end{corollary}

Combining Corollary~\ref{lem:consistent_stoppedtamed}
with
Proposition~\ref{prop:convergence.interpolation}
shows that the stopped increment-tamed Euler-Maruyama schemes
converge in probability.
This is the subject of the next result.

\begin{corollary}
\label{cor:convergence_increment_tamed}
Assume the setting in Subsection~\ref{subsec:converg.prob.setting},
let
  $ q \in (1,\infty) $,
  let
  $ D_t \in \mathcal{B}( \R^d ) $,
  $ t \in (0,T] $,
  be a non-increasing family of sets satisfying
  $
    D
    \subseteq
    \cup_{ t \in (0,T] }
    \mathring{ D }_t
  $,
  and
  let
  $
    Y^N \colon [0,T] \times \Omega \to \R^d
  $,
  $ N \in \N $,
  be mappings satisfying
  for all
  $ N \in \N $,
  $ n \in \{ 0, 1, \dots, N - 1 \} $,
  $
    t \in \big[ \frac{ n T }{ N } , \frac{ ( n + 1 ) T }{ N } \big]
  $,
  $ \omega \in \Omega $
  that
  $ Y^N_0( \omega ) = X_0( \omega ) $
  and
  \begin{equation}
    Y^N_t( \omega )
  =
    Y^N_{
      \frac{ n T }{ N }
    }( \omega )
    +
  \begin{cases}
    \frac{
      \mu\big(
        Y_{
          \frac{ n T }{ N }
        }^N( \omega )
      \big)
      \big(
        t
        -
        \frac{ n T }{ N }
      \big)
      +
      \sigma\big(
        Y_{
          \frac{ n T }{ N }
        }^N( \omega )
      \big)
      \big(
        W_{
          t
        }( \omega )
        -
        W_{
          \frac{ n T }{ N }
        }( \omega )
      \big)
    }{
      1 +
      \big\|
      \mu\big(
        Y_{
          \frac{ n T }{ N }
        }^N( \omega )
      \big)
      \big(
        t
        -
        \frac{ n T }{ N }
      \big)
      +
      \sigma\big(
        Y_{
          \frac{ n T }{ N }
        }^N( \omega )
      \big)
      \big(
        W_{
          t
        }( \omega )
        -
        W_{
          \frac{ n T }{ N }
        }( \omega )
      \big)
      \big\|^q
    }
  &
    \colon 
    Y^N_{ \frac{ n T }{ N } }( \omega )
    \in D_{ \frac{ T }{ N } }
  \\
    0
  &
    \colon 
    Y^N_{ \frac{ n T }{ N } }( \omega )
    \in \R^d \backslash D_{ \frac{ T }{ N } }
  \end{cases}
    .
  \end{equation}
  Then 
\begin{enumerate}[(i)]
\item 
\label{item:cor_convergence_probability}
  it holds for all
  $ \varepsilon \in ( 0, \infty ) $
  that
  $
    \limsup_{ N \to \infty }
    \P\big[
      \sup_{ t \in [0,T] }
      \|
        X_t
        -
        Y^N_t
      \|
      \geq \varepsilon
    \big]
    = 0
  $,
\item 
\label{item:cor_convergence_Lp}
it holds for all
$ p \in (0,\infty) $,
$ r \in (0,p) $
with
$
  \limsup_{ N \to \infty }
  \sup_{ t \in [0,T] }
  \E\big[
    \| Y^N_t \|^p
  \big]
  < \infty
$
that
$
  \sup_{ t \in [0,T] }
  \E\big[
    \| X_t \|^p
  \big]
  < \infty
$
and
$
  \limsup_{ N \to \infty }
  \big(
  \sup_{ t \in [0,T] }
  \E\big[
    \| X_t - Y^N_t \|^r
  \big]
  \big)
  = 0
$,
and
\item 
\label{item:cor_convergence_f}
it holds for all continuous 
$ f \colon C( [0,T], \R^d ) \to \R $
with 
$
  \limsup_{ p \searrow 1 }
  \limsup_{ N \to \infty }
  \E\big[
    | f( Y^N ) |^p
  \big]
  < \infty
$
that
$
  \E\big[
    | f( X ) |
  \big]
  < \infty
$
and
$
  \limsup_{ N \to \infty }
  \big|
  \E\big[
    f( Y^N )
  \big]
  -
  \E\big[
    f( X )
  \big]
  \big| = 0
$.
\end{enumerate}
\end{corollary}

\begin{proof}[Proof
of Corollary~\ref{cor:convergence_increment_tamed}]
Throughout this proof let
$
  \Psi
  \colon
  \R^d \times (0,T]^2 \times \R^m
  \to \R^d
$
be the function 
which satisfies
for all
$
  ( x, t, s, y ) \in
  \R^d \times (0,T]^2 \times \R^m
$
that
\begin{equation}
  \Psi( x, t, s, y )
=
\begin{cases}
  \mathbbm{1}_{
    D_t
  }( x )
  \left[
    \frac{
      \mu( x ) s
      +
      \sigma( x ) y
    }{
      1 +
      \left\|
        \mu( x ) s
        +
        \sigma( x ) y
      \right\|^q
    }
  \right]
&
\colon
  x \in D
\\
  0
&
\colon
  x \in \R^d \backslash D
\end{cases}
  .
\end{equation}
Next observe that
Corollary~\ref{lem:consistent_stoppedtamed}
implies that
$
  \R^d \times (0,T] \times \R^m
  \ni
  ( x, t, y )
  \mapsto
  \Psi( x, t, t, y ) \in \R^d
$
is
$ ( \mu, \sigma ) $-consistent with respect to Brownian motion.
In addition, observe that for all non-empty compact sets
$ K \subseteq D $ it holds that
\begin{equation}
\begin{split}
&
  \limsup_{ h \searrow 0 }
  \E\!\left[
    \sup_{ x \in K }
    \sup_{ t \in [0,T] }
    \left\|
      \Psi(
        x, h, t -\lfloor t \rfloor_h,
        W_t - W_{\lfloor t \rfloor_h}
      )
    \right\|
  \right]
\\ & =
  \limsup_{ h \searrow 0 }
  \E\!\left[
    \sup_{ x \in K }
    \sup_{ t \in [0,T] }
    \left\|
  \mathbbm{1}_{
    D_h
  }( x )
  \left[
    \frac{
      \mu( x )
      ( t - \lfloor t \rfloor_h )
      +
      \sigma( x )
      ( W_t - W_{ \lfloor t \rfloor_h } )
    }{
      1 +
      \left\|
        \mu( x )
        ( t - \lfloor t \rfloor_h )
        +
        \sigma( x )
        ( W_t - W_{ \lfloor t \rfloor_h } )
      \right\|^q
    }
  \right]
    \right\|
  \right]
\\ & \leq
  \limsup_{ h \searrow 0 }
  \E\!\left[
    \sup_{ x \in K }
    \sup_{ t \in [0,T] }
    \left(
    \frac{
      \left\|
        \mu( x )
        ( t - \lfloor t \rfloor_h )
        +
        \sigma( x )
        ( W_t - W_{ \lfloor t \rfloor_h } )
      \right\|
    }{
      1 +
      \left\|
        \mu( x )
        ( t - \lfloor t \rfloor_h )
        +
        \sigma( x )
        ( W_t - W_{ \lfloor t \rfloor_h } )
      \right\|^q
    }
    \right)
  \right]
\\ & \leq
  \limsup_{ h \searrow 0 }
  \E\!\left[
    \sup_{ x \in K }
    \sup_{ t \in [0,T] }
    \left\|
        \mu( x )
        ( t - \lfloor t \rfloor_h )
      +
      \sigma( x )
      ( W_t - W_{ \lfloor t \rfloor_h } )
    \right\|
  \right]
\\ & \leq
  \left(
    \limsup_{ h \searrow 0 }
    h
  \right)
  \left(
    \sup_{ x \in K }
    \left\|
      \mu( x )
    \right\|
  \right)
  +
  \left(
    \limsup_{ h \searrow 0 }
    \E\!\left[
    \sup_{ t \in [0,T] }
    \left\|
      W_t - W_{ \lfloor t \rfloor_h }
    \right\|
    \right]
  \right)
  \left(
    \sup_{ x \in K }
    \left\|
      \sigma( x )
    \right\|_{
      L( \R^m, \R^d )
    }
  \right)
  = 0 .
\end{split}
\end{equation}
Proposition~\ref{prop:convergence.interpolation}
hence shows
for all $ \varepsilon \in ( 0, \infty ) $
that
$
  \limsup_{ N \to \infty }
  \P\big[
    \sup_{ t \in [0,T] }
    \|
      X_t - Y_t^N
    \|
    \geq
    \varepsilon
  \big]
  = 0
$.
The proof of the strong convergence
statement in
Corollary~\ref{cor:convergence_increment_tamed}
is entirely analogous to
the proof of Corollary~3.12
in \cite{HutzenthalerJentzen2014Memoires}
and thus omitted.
It thus remains to prove the weak
convergence statement in
Corollary~\ref{cor:convergence_increment_tamed}.
For this
assume that
$ p \in (1,\infty) $
is a real number,
that
$ N_0 \in \N $
is a natural number,
and that
$ f \colon C( [0,T], \R^d ) \to \R $
is a continuous function
with
$
  \sup_{ N \in \{ N_0, N_0 + 1, \dots \} }
  \E\big[
    | f( Y^N ) |^p
  \big]
  < \infty
$.
The fact that
$
  \sup_{ t \in [0,T] }
  \| X_t - Y^N_t \|
$
converges in probability
to zero as $ N \to \infty $
together with, e.g., Lemma~3.10
in \cite{HutzenthalerJentzen2014Memoires}
proves then that
\begin{equation}
\label{eq:convergence_fYN}
  \E\big[
    | f( X ) |^p
  \big]
  < \infty
  \qquad
\text{and}
  \qquad
  \forall \, \varepsilon \in (0,\infty)
  \colon
  \qquad
  \limsup_{ N \to \infty }
  \P\!\left[
    | f( X ) - f( Y^N ) |
    \geq
    \varepsilon
  \right]
  = 0 .
\end{equation}
This shows that the family
$
  | f( X ) - f( Y^N ) |
$,
$ N \in \{ N_0, N_0 + 1, \dots \} $,
of random variables is uniformly
integrable.
Combining this
and \eqref{eq:convergence_fYN}
with, e.g.,
Theorem~6.25 in
Klenke~\cite{Klenke2008}
proves that
$
  \limsup_{ N \to \infty }
  \E\big[
    |
      f( X ) - f( Y^N )
    |
  \big]
  = 0
$.
The proof
of
Corollary~\ref{cor:convergence_increment_tamed}
is thus completed.
\end{proof}

Combining Corollary~\ref{cor:convergence_increment_tamed}
with Corollary~\ref{cor:stopped.Euler.bounded.increments}
and Fatou's lemma
results in Corollary~\ref{cor:for_examples}.
Corollary~\ref{cor:for_examples} establishes
both exponential integrability properties
and
for any $ r \in [0,\infty) $
strong $ L^r $-convergence.

\begin{corollary}
\label{cor:for_examples}
  Let
  $ d, m \in \N $,
  $
    \rho
    \in [0,\infty)
  $,
  $
    T 
    \in (0,\infty)
  $,
  $ c, q \in (1,\infty) $,
  $
    U \in \cup_{ p \in [1,\infty) } C_{ p, c }^3( \mathbb{R}^d, [0,\infty) )
  $,
  $
    \bar{U} \in C( \R^d, [ - c, \infty ) )
  $,
  $
    \mu \in \mathcal{M}( \mathcal{B}( \R^d ) , \mathcal{B}( \R^d ) )
  $,
  $
    \sigma
    \in
    \mathcal{M}( \mathcal{B}( \R^d ), \mathcal{B}( \R^{ d \times m } ) )
  $,
  let
  $
    D \subseteq \R^d
  $
  be an open set,
  let
  $
    (
      \Omega, \mathcal{F}, \P
    )
  $ 
  be a probability space with a normal filtration 
  $
    (
      \mathcal{F}_t
    )_{
      t \in [0, T ]
    }
  $,
  let
  $
    W \colon [0,T]\times\Omega\to\R^m
  $
  be a standard
  $
    (\mathcal{F}_t)_{t\in[0,T]}
  $-Brownian motion with continuous sample paths,
  assume that $ \mu|_D \colon D \to \R^d $
  and $ \sigma|_D \colon D \to \R^{ d \times m } $
  are locally Lipschitz continuous, 
let $ X \colon [0,T] \times \Omega \to D $
be an 
$ ( \mathcal{F}_t )_{ t \in [0,T] } $-adapted
stochastic process with continuous sample paths
which satisfies
$
  \E\big[
    e^{ U( X_0 ) }
  \big]
  < \infty
$
and
which satisfies that
for all $ t \in [0,T] $
it holds $ \P $-a.s.\ that
$
  X_t = X_0 + \int_0^t \mu( X_s ) \, ds
  +
  \int_0^t \sigma( X_s ) \, dW_s
$,
let
$ Y^N \colon [0,T] \times \Omega \to \R^d $,
$ N \in \N $,
and 
$ \tau_N \colon \Omega \to [0,T] $,
$ N \in \N $,
be mappings satisfying
for all
$ N \in \N $,
$ n \in \{ 0, 1, \dots, N - 1 \} $,
$
  t \in \big[ \frac{ n T }{ N } , \frac{ ( n + 1 ) T }{ N } \big]
$
that
$
  \tau_N =
  \inf\big(\big\{
    s \in 
    \{ 0, \frac{ T }{ N }, \frac{ 2 T }{ N }, \dots, T \}
    \colon
    Y^N_s
    \notin D
    \text{ or }
    \| Y^N_s \|
    >
    \exp(
      | \ln( T / N ) |^{ 1 / 2 }
    )
  \big\}
  \cup\{ T \}
  \big)
$,
$ Y^N_0 = X_0 $,
and
  \begin{equation}
    Y^N_t
  =
    Y^N_{
      \frac{ n T }{ N }
    }
    +
    \mathbbm{1}_{
      \left\{
        Y^N_{ n T / N } \in D
        \text{ and }
        \| Y^N_{ n T / N } \|
        \leq
        \exp(
          | \ln( T / N ) |^{ 1 / 2 }
        )
      \right\}
    }
    \left[
    \frac{
      \mu(
        Y_{
          \frac{ n T }{ N }
        }^N
      )
      \big(
        t
        -
        \frac{ n T }{ N }
      \big)
      +
      \sigma(
        Y_{
          \frac{ n T }{ N }
        }^N
      )
      \big(
        W_{
          t
        }
        -
        W_{
          \frac{ n T }{ N }
        }
      \big)
    }{
      1 +
      \big\|
      \mu(
        Y_{
          \frac{ n T }{ N }
        }^N
      )
      \big(
        t
        -
        \frac{ n T }{ N }
      \big)
      +
      \sigma(
        Y_{
          \frac{ n T }{ N }
        }^N
      )
      \big(
        W_{
          t
        }
        -
        W_{
          \frac{ n T }{ N }
        }
      \big)
      \big\|^q
    }
    \right]
    ,
  \end{equation}
  and 
  assume 
  for all
  $ x, y \in \R^d $
  that
  \begin{align}
    &
      \| \mu(x) \|
      +
      \| \sigma(x) \|_{\HS(\R^m,\R^d)}
  \leq
    c
    \left(
      1 +
      \| x \|^{ c }
    \right)
    ,
    \qquad
      | \bar{U}(x) - \bar{U}(y) |
    \leq
    c
    \left(
      1 + \| x \|^{ c } + \| y \|^{ c }
    \right)
      \| x - y \|
    ,
  \\
  \label{eq:important_for_examples}
    &
    (\mathcal{G}_{\mu,\sigma}U)(x)
     +\tfrac{1}{2}\left\|\sigma(x)^{*}(\nabla U)(x)\right\|^2
     +\bar{U}(x)
  \leq
    \rho \cdot U(x) ,
    \qquad
    \| x \|^{ 1 / c }
    \leq
    c
    \left(
      1 +
      U( x )
    \right)
    .
  \end{align}
  Then it holds for all
$ r \in (0,\infty) $ that
$
  \limsup_{ N \to \infty }
  \big(
  \sup_{ t \in [0,T] }
  \E\big[
    \| X_t - Y^N_t \|^r
  \big]
  \big)
  = 0
$,
that
\begin{equation}
\label{eq:Y_bound_1}
  \sup_{ t \in [0,T] }
  \E\!\left[
       \exp\!\left(
         \tfrac{
           U( X_t )
         }{
           e^{ \rho t }
         }
         +
         \smallint_0^t
           \tfrac{
             \bar{U}( X_s )
           }{
             e^{ \rho s }
           }
         \, ds
       \right)
     \right]
  \leq
  \sup_{
    N \in \N
  }
  \sup_{ t \in [0,T] }
  \E\!\left[
       \exp\!\left(
         \tfrac{
           U( Y^N_t )
         }{
           e^{ \rho t }
         }
         +
         \smallint_0^{ t \wedge \tau_N }
           \tfrac{
             \bar{U}( Y^N_s )
           }{
             e^{ \rho s }
           }
         \, ds
       \right)
     \right]
  <
  \infty
  ,
  \quad
  \text{and that}
\end{equation}
\begin{equation}
\label{eq:Y_bound_2}
  \sup_{ t \in [0,T] }
  \E\!\left[
       \exp\!\left(
         \tfrac{
           U( X_t )
         }{
           e^{ \rho t }
         }
         +
         \smallint_0^t
           \tfrac{
             \bar{U}( X_s )
           }{
             e^{ \rho s }
           }
         \, ds
       \right)
     \right]
  \leq
  \limsup_{ N \to \infty }
  \sup_{ t \in [0,T] }
  \E\!\left[
       \exp\!\left(
         \tfrac{
           U( Y^N_t )
         }{
           e^{ \rho t }
         }
         +
         \smallint_0^{ t \wedge \tau_N }
           \tfrac{
             \bar{U}( Y^N_s )
           }{
             e^{ \rho s }
           }
         \, ds
       \right)
     \right]
  \leq
  \E\!\left[
        e^{
          U( X_0 )
        }
  \right]
  .
\end{equation}
\end{corollary}

\begin{proof}[Proof 
of Corollary~\ref{cor:for_examples}]
Throughout this proof 
let $ \theta_N \in \mathcal{P}_T $,
$ N \in \N $,
be the sets which satisfy for all 
$ N \in \N $ that
$
  \theta_N = 
  \big\{ 
    0, \frac{ T }{ N }, \frac{ 2 T }{ N } ,
    \dots , \frac{ ( N - 1 ) T }{ N }, T 
  \big\}
$,
let 
$ D_h \subseteq \R^d $,
$ h \in (0,T] $,
be the sets which satisfy 
for all $ h \in (0,T] $
that
$
  D_h = 
  \big\{ 
    x \in D \colon
    \| x \| \leq 
    \exp\!\big(
      |
        \ln(
          \min\{ 1, h \}
        )
      |^{ 1 / 2 }
    \big)
  \big\}
$,
let 
$ Z^{ \theta } \colon [0,T] \times \Omega \to \R^d $,
$ \theta \in \mathcal{P}_T $,
be $ ( \mathcal{F}_t )_{ t \in [0,T] } $-adapted
stochastic processes with continuous sample paths 
which satisfy
for all
$ \theta \in \mathcal{P}_T $,
$ t \in [0,T] $
that
$ Z^{ \theta }_0 = X_0 $
and
  \begin{equation}
  \begin{split}
  &
    Z^{ \theta }_t
  =
    Z^{ \theta }_{
      \lfloor t \rfloor_{ \theta }
    }
    +
    \mathbbm{1}_{
      D_{ \left| \theta \right|_T }
    }\big(
      Z^{ \theta }_{
        \lfloor t \rfloor_{ \theta }
      }      
    \big)
    \left[
    \frac{
      \mu(
        Z_{
          \lfloor t \rfloor_{ \theta }
        }^{ \theta }
      )
      \big(
        t
        -
        \lfloor t \rfloor_{ \theta }
      \big)
      +
      \sigma(
        Z_{
          \lfloor t \rfloor_{ \theta }
        }^{ \theta }
      )
      \big(
        W_{
          t
        }
        -
        W_{
          \lfloor t \rfloor_{ \theta }
        }
      \big)
    }{
      1 +
      \big\|
      \mu(
        Z_{
          \lfloor t \rfloor_{ \theta }
        }^{ \theta }
      )
      \big(
        t
        -
        \lfloor t \rfloor_{ \theta }
      \big)
      +
      \sigma(
        Z_{
          \lfloor t \rfloor_{ \theta }
        }^{ \theta }
      )
      \big(
        W_{
          t
        }
        -
        W_{
          \lfloor t \rfloor_{ \theta }
        }
      \big)
      \big\|^q
    }
    \right]
    ,
  \end{split}
  \end{equation}
and 
let 
$ \varrho_N \colon \Omega \to [0,T] $,
$ N \in \N $,
be the functions which satisfy
for all $ N \in \N $ that
$
  \varrho_N
  =
  \inf\!\big(
    \big\{
      s \in \theta_N
      \colon
      Z^{ \theta_N }_s
      \notin D_{ T / N }
    \big\}
    \cup
    \{ T \}
  \big)
$.
Observe that the assumption that
$
  U \in \cup_{ p \in [1,\infty) } C^3_{ p, c }( \R^d, [0,\infty) )
$
together with 
Lemma~\ref{l:function.space.estimate}
shows that
$
  \limsup_{ p \to \infty }
  \sup_{ x \in \R^d }
  \frac{ U(x) }{ [ 1 + \| x \| ]^p } 
  < \infty
$.
This implies that there exists a real number $ \hat{c} \in [c,\infty) $
such that for all 
$ h \in (0,T] $,
$ 
  x \in 
  \big\{ 
    y \in \R^d \colon \| y \| 
    \leq 
    \exp\!\big( | \ln( \min\{ 1, h \} ) |^{ 1 / 2 } \big)
  \big\} 
$
it holds that
$
  U(x) 
  \leq 
  \hat{c} \exp\!\big( \hat{c} \, | \ln( \min\{ 1, h \} ) |^{ 1 / 2 } \big)
$.
This and the fact that
$
  \forall \, h \in (0,T] \colon
  \hat{c} \exp\!\big( \hat{c} \, | \ln( \min\{ 1, h \} ) |^{ 1 / 2 } \big)
  \leq
  \hat{c} \exp\!\big( \hat{c} \, | \ln( h ) |^{ 1 / 2 } \big)
$
assure that 
$ D_h \subseteq \R^d $,
$ h \in (0,T] $,
is a non-increasing family of sets
which satisfies
for all $ h \in (0,T] $
that
  $
    D_{ h } \in 
    \mathcal{B}(
      \{
        x \in D \colon
        U( x ) \leq
        \hat{c} 
        \exp(
          \hat{c} \,
          |
            \ln( h )
          |^{ 1 / 2 }
        )
      \}
    )
  $,
  $
    \mu|_{ D_h } \in C( D_h, \R^d )
  $,
  and 
  $
    \sigma|_{ D_h } \in C( D_h, \R^{ d \times m } )
  $.
We can hence apply 
Corollary~\ref{cor:stopped.Euler.bounded.increments}
to obtain that
\begin{equation}
\label{eq:Z_bound_1}
  \sup_{
    \theta \in \mathcal{P}_T
  }
  \sup_{ t \in [0,T] }
  \E\!\left[
       \exp\!\left(
         e^{ - \rho t }
         \,
         U( Z^{ \theta }_t )
         +
         \smallint_0^t
           e^{ - \rho s }
           \,
           \1_{ D_{ | \theta |_T } }(
             Z^{ \theta }_{\lfloor s \rfloor_{ \theta } }
           )
           \,
           \bar{U}( Z^{ \theta }_s )
         \, ds
       \right)
     \right]
  <
  \infty
\end{equation}
and 
\begin{equation}
\label{eq:Z_bound_2}
    \limsup_{
      | \theta |_T \searrow 0
    }
    \sup_{ t \in [0,T] }
    \E\!\left[
      \exp\!\left(
        e^{ - \rho t }
        \,
        U( Z_t^{ \theta } )
        +
        \smallint_0^t
            e^{ - \rho s }
            \,
            \1_{ D_{ | \theta |_T } }(
              Z_{ \lfloor s \rfloor_{ \theta } }^{ \theta }
            )
            \,
            \bar{U}( Z_s^{ \theta } )
        \, ds
      \right)
    \right]
    \leq
    \limsup_{
      | \theta |_T \searrow 0
    }
      \E\big[
        e^{
          U( Z_0^{ \theta } )
        }
      \big]
  =
      \E\big[
        e^{
          U( X_0 )
        }
      \big]
  .
\end{equation}
Inequalities~\eqref{eq:Z_bound_1}--\eqref{eq:Z_bound_2}
and the assumption that 
$
      \E\big[
        e^{
          U( X_0 )
        }
      \big]
    < \infty
$,
in particular, ensure that
\begin{equation}
\label{eq:Z_bound_1b}
  \sup_{
    N \in \N 
  }
  \sup_{ t \in [0,T] }
  \E\!\left[
       \exp\!\left(
         e^{ - \rho t }
         \,
         U( Z^{ \theta_N }_t )
         +
         \smallint_0^t
           e^{ - \rho s }
           \,
           \1_{ D_{ T / N } }(
             Z^{ \theta_N }_{\lfloor s \rfloor_{ \theta_N } }
           )
           \,
           \bar{U}( Z^{ \theta_N }_s )
         \, ds
       \right)
     \right]
  <
  \infty
\end{equation}
and 
\begin{equation}
\label{eq:Z_bound_2b}
    \limsup_{
      N \to \infty
    }
    \sup_{ t \in [0,T] }
    \E\!\left[
      \exp\!\left(
        e^{ - \rho t }
        \,
        U( Z_t^{ \theta_N } )
        +
        \smallint_0^t
            e^{ - \rho s }
            \,
            \1_{ D_{ T / N } }(
              Z_{ \lfloor s \rfloor_{ \theta_N } }^{ \theta_N }
            )
            \,
            \bar{U}( Z_s^{ \theta_N } )
        \, ds
      \right)
    \right]
    \leq
      \E\big[
        e^{
          U( X_0 )
        }
      \big]
    < \infty
  .
\end{equation}
The definition of $ \varrho_N $, $ N \in \N $,
hence proves that
\begin{equation}
\label{eq:Z_bound_1c}
  \sup_{
    N \in \N 
  }
  \sup_{ t \in [0,T] }
  \E\!\left[
       \exp\!\left(
         e^{ - \rho t }
         \,
         U( Z^{ \theta_N }_t )
         +
         \smallint_0^{ t \wedge \varrho_N }
           e^{ - \rho s }
           \,
           \bar{U}( Z^{ \theta_N }_s )
         \, ds
       \right)
     \right]
  <
  \infty
\end{equation}
and 
\begin{equation}
\label{eq:Z_bound_2c}
    \limsup_{
      N \to \infty
    }
    \sup_{ t \in [0,T] }
    \E\!\left[
      \exp\!\left(
        e^{ - \rho t }
        \,
        U( Z_t^{ \theta_N } )
        +
         \smallint_0^{ t \wedge \varrho_N }
            e^{ - \rho s }
            \,
            \bar{U}( Z_s^{ \theta_N } )
        \, ds
      \right)
    \right]
    \leq
      \E\big[
        e^{
          U( X_0 )
        }
      \big]
    < \infty
  .
\end{equation}
In the next step we observe that
inequality~\eqref{eq:Z_bound_1c}
and the assumption that
$
  \inf_{ x \in \R^d }
  \bar{U}( x ) \geq - c
$
prove that
$
  \sup_{
    N \in \N 
  }
  \sup_{ t \in [0,T] }
  \E\big[
     \exp\!\big(
       e^{ - \rho T }
       \big\{
         1 +
         U( Z^{ \theta_N }_t )
       \big\}
     \big)
  \big]
  < \infty
$.
This and the assumption that
$
  \forall \, x \in \R^d \colon
  \| x \|^{ 1 / c } \leq c \, ( 1 + U(x) )
$
ensure that
$
  \sup_{
    N \in \N 
  }
  \sup_{ t \in [0,T] }
  \E\big[
     \exp\!\big(
       c^{ - 1 }
       e^{ - \rho T }
       \| Z^{ \theta_N }_t \|^{ 1 / c }
     \big)
  \big]
  < \infty
$.
Hence, we obtain for all $ p \in (0,\infty) $ that
$
  \sup_{ N \in \N }
  \sup_{ t \in [0,T] }
  \E\big[ 
    \| Z^{ \theta_N }_t \|^p
  \big]
  < \infty
$.
Combining this 
with 
Items~\eqref{item:cor_convergence_probability}--\eqref{item:cor_convergence_Lp} 
in Corollary~\ref{cor:convergence_increment_tamed}
assures for all $ r \in (0,\infty) $ that
\begin{equation}
\label{eq:convergence_in_probability_Z}
  \limsup\nolimits_{ N \to \infty }
  \P\big[ 
    \sup\nolimits_{ t \in [0,T] }
    \| X_t - Z^{ \theta_N }_t \| \geq r
  \big] 
  +
  \limsup\nolimits_{ N \to \infty }
  \sup\nolimits_{ t \in [0,T] }
  \E\big[
    \| X_t - Z^{ \theta_N }_t \|^r
  \big]
  = 0
  .
\end{equation}
In addition, observe 
for all $ \varepsilon \in (0,\infty) $,
$ N \in \N $
that
\begin{equation}
\begin{split}
&
  \P\big[  
    \mathbbm{1}_{
      \{ \varrho_N < T \}
    }
    \geq \varepsilon
  \big]
\leq
  \P\big[  
    \mathbbm{1}_{
      \{ \varrho_N < T \}
    }
    \geq 1
  \big]
=
  \P\big[  
    \varrho_N < T 
  \big]
\\ & 
\leq
  \P\!\left[ 
      \exists \, s \in \theta_N \colon
      Z^{ \theta_N }_s \notin D_{ T / N }
  \right]
\leq
  \P\!\left[ 
      \exists \, s \in [0,T] \colon
      Z^{ \theta_N }_s \notin D_{ T / N }
  \right]
\\ & \leq
  \P\!\left[ 
      \exists \, s \in [0,T] \colon
      Z^{ \theta_N }_s \notin D
  \right]
  +
  \P\big[ 
      \exists \, s \in [0,T] \colon
      \| Z^{ \theta_N }_s \|
      >
      \exp\!\big( 
        | 
          \ln(
            \min\{ 1, \nicefrac{ T }{ N } \}
          )
        |^{ 1 / 2 }
      \big)
  \big]
\\ & \leq
  \P\!\left[ 
    \left\{
      \exists \, t \in [0,T] \colon
      Z^{ \theta_N }_t \notin D
    \right\}
    \cap 
    \left\{ 
      \sup_{ t \in [0,T] } 
      \| X_t - Z^{ \theta_N }_t \|
      < 
      \inf\!\big(
        \{ \infty \} 
        \cup
        \{ 
          \| X_t - v \|
          \in \R
          \colon
          t \in [0,T] , 
          v \in ( \R^d \backslash D )
        \}
      \big)
    \right\}
  \right]
\\ & \quad
  +
  \P\!\left[ 
      \sup_{ t \in [0,T] } 
      \| X_t - Z^{ \theta_N }_t \|
      \geq
      \inf\!\big(
        \{ \infty \} 
        \cup
        \big\{ 
          \| X_t - v \|
          \in \R
          \colon
          t \in [0,T] , 
          v \in ( \R^d \backslash D )
        \big\}
      \big)
  \right]
\\ & \quad
  +
  \P\!\left[ 
      \sup_{ t \in [0,T] }
      \| Z^{ \theta_N }_t \|
      >
      \exp\!\big( 
        | 
          \ln(
            \min\{ 1, \nicefrac{ T }{ N } \}
          )
        |^{ 1 / 2 }
      \big)
  \right]
\\ & \leq
  \P\!\left[ 
      \sup\nolimits_{ t \in [0,T] } 
      \| X_t - Z^{ \theta_N }_t \|
      \geq
      \inf\!\big(
        \{ \infty \} 
        \cup
        \big\{ 
          \| X_t - v \|
          \in \R
          \colon
          t \in [0,T] , 
          v \in ( \R^d \backslash D )
        \big\}
      \big)
  \right]
\\ & \quad
  +
  \P\Big[ 
      \sup\nolimits_{ t \in [0,T] }
      \| 
        X_t 
        -
        Z^{ \theta_N }_t
      \|
      +
      \sup\nolimits_{ t \in [0,T] }
      \| 
        X_t 
      \|
      >
      \exp\!\big( 
        | 
          \ln(
            \min\{ 1, \nicefrac{ T }{ N } \}
          )
        |^{ 1 / 2 }
      \big)
  \Big]
  .
\end{split}
\end{equation}
This,
\eqref{eq:convergence_in_probability_Z},
and the assumption that $ X \colon [0,T] \times \Omega \to D $ has continuous sample paths
prove
for all $ \varepsilon \in (0,\infty) $
that
\begin{equation}
\label{eq:varrho_convergence}
\begin{split}
&
  \limsup_{ N \to \infty }
  \P\big[  
    \mathbbm{1}_{
      \{ \varrho_N < T \}
    }
    \geq \varepsilon
  \big]
\\ & \leq
  \limsup_{ N \to \infty }
  \P\!\left[ 
      \sup\nolimits_{ t \in [0,T] } 
      \| X_t - Z^{ \theta_N }_t \|
      \geq
      \inf\!\big(
        \{ \infty \} 
        \cup
        \big\{ 
          \| X_t - v \|
          \in \R
          \colon
          t \in [0,T] , 
          v \in ( \R^d \backslash D )
        \big\}
      \big)
  \right]
\\ & \quad
  +
  \limsup_{ N \to \infty }
  \P\Big[ 
      \sup\nolimits_{ t \in [0,T] }
      \| 
        X_t 
        -
        Z^{ \theta_N }_t
      \|
    \geq 1
  \Big]
\\ & \quad
  +
  \limsup_{ N \to \infty }
  \P\Big[ 
      \sup\nolimits_{ t \in [0,T] }
      \| 
        X_t 
      \|
      >
      \exp\!\big( 
        | 
          \ln(
            \min\{ 1, \nicefrac{ T }{ N } \}
          )
        |^{ 1 / 2 }
      \big)
      - 1
  \Big]
  = 0 .
\end{split}
\end{equation}
Furthermore, we note that 
the assumption that 
$
  \forall \, x, y \in \R^d \colon
  | \bar{U}( x ) - \bar{U}( y ) |
  \leq 
  c \, 
  \big( 1 + \| x \|^c + \| y \|^c \big)
  \,
  \| x - y \|
$ 
implies that
$
  \forall \, x, y \in \R^d \colon
  | \bar{U}( x ) - \bar{U}( y ) |
  \leq 
  c \, 2^c 
  \big( 
    1 
    +
    \| x - y \|^c 
    + 
    \| y \|^c 
  \big)
  \,
  \| x - y \|
$.
This together with the triangle 
inequality ensures for all $ t \in [0,T] $
that
\begin{equation}
\label{eq:U_estimate_convergence_probab}
\begin{split}
&
    \left|
      \tfrac{
        U( Z_t^{ \theta_N } )
      }{
        e^{ \rho t }
      }
      +
      \smallint_0^{ t \wedge \varrho_N }
        \tfrac{
          \bar{U}( Z_s^{ \theta_N } )
        }{
          e^{ \rho s }
        }
        \, ds
      -
      \tfrac{
        U( X_t )
      }{
        e^{ \rho t }
      }
      -
      \smallint_0^t
        \tfrac{
          \bar{U}( X_s )
        }{
          e^{ \rho s }
        }
        \, ds
    \right|
\\ & \leq
    \big|
        U( Z_t^{ \theta_N } )
      -
        U( X_t )
    \big|
    +
    \left|
      \smallint_0^{ t \wedge \varrho_N }
        \tfrac{
          \bar{U}( Z_s^{ \theta_N } )
        }{
          e^{ \rho s }
        }
        \, ds
      -
      \smallint_0^{ t \wedge \varrho_N }
        \tfrac{
          \bar{U}( X_s )
        }{
          e^{ \rho s }
        }
        \, ds
    \right|
    +
    \left|
      \smallint_0^{ t \wedge \varrho_N }
        \tfrac{
          \bar{U}( X_s )
        }{
          e^{ \rho s }
        }
        \, ds
        -
      \smallint_0^t
        \tfrac{
          \bar{U}( X_s )
        }{
          e^{ \rho s }
        }
        \, ds
    \right|
\\ & \leq
    \big|
        U( Z_t^{ \theta_N } )
      -
        U( X_t )
    \big|
    +
    \smallint_0^{ t \wedge \varrho_N }
      \left|
        \tfrac{
          \bar{U}( Z_s^{ \theta_N } )
        }{
          e^{ \rho s }
        }
        -
        \tfrac{
          \bar{U}( X_s )
        }{
          e^{ \rho s }
        }
      \right|
    ds
    +
    \smallint_{ t \wedge \varrho_N }^t
      \left|
        \tfrac{
          \bar{U}( X_s )
        }{
          e^{ \rho s }
        }
      \right|
    ds
\\ & \leq
    \big|
        U( Z_t^{ \theta_N } )
      -
        U( X_t )
    \big|
    +
    T
    \left[ 
      \sup_{ s \in [0,T] }
      \left|
          \bar{U}( Z_s^{ \theta_N } )
        -
          \bar{U}( X_s )
      \right|
    \right]
    +
    \left(
      t - \min\{ t , \varrho_N \}
    \right)
    \left[
      \sup_{ s \in [0,T] }
      \left|
        \bar{U}( X_s )
      \right|
    \right] 
\\ & \leq
    T c \, 2^c
    \left[ 
      1 +
        \sup_{ s \in [0,T] }
        \left\|
          Z_s^{ \theta_N } 
          -
          X_s
        \right\|^c
        +
        \sup_{ s \in [0,T] }
        \left\|
          X_s
        \right\|^c
    \right]
    \left[ 
      \sup_{ s \in [0,T] }
      \left\|
        Z_s^{ \theta_N } 
        -
        X_s 
      \right\|
    \right]
\\ & \quad
    +
    \big|
        U( Z_t^{ \theta_N } )
      -
        U( X_t )
    \big|
    +
    \mathbbm{1}_{
      \{ \varrho_N < T \}
    }
    \,
    T
    \left[
      \sup_{ s \in [0,T] }
      \left|
        \bar{U}( X_s )
      \right|
    \right] 
  .
\end{split}
\end{equation}
Next note that 
\eqref{eq:convergence_in_probability_Z}
and the assumption that $ U $ 
is continuous establish that
for all $ \varepsilon \in (0,\infty) $,
$ t \in [0,T] $ 
it holds that
$
  \limsup_{ N \to \infty }
  \P\big[ 
    |
        U( Z_t^{ \theta_N } )
      -
        U( X_t )
    |
    \geq 
    \varepsilon
  \big]
  = 0
$.
This, \eqref{eq:convergence_in_probability_Z},
\eqref{eq:U_estimate_convergence_probab}, 
and \eqref{eq:varrho_convergence} prove
for all 
$ \varepsilon \in (0,\infty) $, $ t \in [0,T] $ 
that
\begin{equation}
  \limsup_{ N \to \infty }
  \P\Bigg[ 
    \bigg|
      \tfrac{
        U( Z_t^{ \theta_N } )
      }{
        e^{ \rho t }
      }
      +
      \smallint_0^{ t \wedge \varrho_N }
        \tfrac{
          \bar{U}( Z_s^{ \theta_N } )
        }{
          e^{ \rho s }
        }
        \, ds
      -
      \tfrac{
        U( X_t )
      }{
        e^{ \rho t }
      }
      -
      \smallint_0^t
        \tfrac{
          \bar{U}( X_s )
        }{
          e^{ \rho s }
        }
        \, ds
    \bigg|
    \geq \varepsilon
  \Bigg] = 0
  .
\end{equation}
Combining this with
a well-known modification of 
Fatou's lemma
(see, e.g., Lemma~3.10
in \cite{HutzenthalerJentzen2014Memoires})
proves 
for all $ t \in [0,T] $
that
\begin{equation}
  \E\!\left[
       \exp\!\left(
         \tfrac{
           U( X_t )
         }{
           e^{ \rho t }
         }
         +
         \smallint_0^t
           \tfrac{
             \bar{U}( X_s )
           }{
             e^{ \rho s }
           }
         \, ds
       \right)
     \right]
  \leq
  \liminf_{ N \to \infty }
  \E\!\left[
       \exp\!\left(
         \tfrac{
           U( Z^{ \theta_N }_t )
         }{
           e^{ \rho t }
         }
         +
         \smallint_0^{ t \wedge \varrho_N }
           \tfrac{  
             \bar{U}( Z^{ \theta_N }_s )
           }{
             e^{ \rho s }
           }
         \, ds
       \right)
     \right]
  .
\end{equation}
Hence, we obtain that
\begin{equation}
  \sup_{ t \in [0,T] }
  \E\!\left[
       \exp\!\left(
         \tfrac{
           U( X_t )
         }{
           e^{ \rho t }
         }
         +
         \smallint_0^t
           \tfrac{
             \bar{U}( X_s )
           }{
             e^{ \rho s }
           }
         \, ds
       \right)
     \right]
  \leq
  \liminf_{ N \to \infty }
  \sup_{ t \in [0,T] }
  \E\!\left[
       \exp\!\left(
         \tfrac{
           U( Z^{ \theta_N }_t )
         }{
           e^{ \rho t }
         }
         +
         \smallint_0^{ t \wedge \varrho_N }
           \tfrac{  
             \bar{U}( Z^{ \theta_N }_s )
           }{
             e^{ \rho s }
           }
         \, ds
       \right)
     \right]
  .
\end{equation}
This together with 
\eqref{eq:Z_bound_2c}
ensures that
\begin{equation}
  \sup_{ t \in [0,T] }
  \E\!\left[
       \exp\!\left(
         \tfrac{
           U( X_t )
         }{
           e^{ \rho t }
         }
         +
         \smallint_0^t
           \tfrac{
             \bar{U}( X_s )
           }{
             e^{ \rho s }
           }
         \, ds
       \right)
     \right]
  \leq
  \limsup_{ N \to \infty }
  \sup_{ t \in [0,T] }
  \E\!\left[
       \exp\!\left(
         \tfrac{
           U( Z^{ \theta_N }_t )
         }{
           e^{ \rho t }
         }
         +
         \smallint_0^{ t \wedge \varrho_N }
           \tfrac{  
             \bar{U}( Z^{ \theta_N }_s )
           }{
             e^{ \rho s }
           }
         \, ds
       \right)
     \right]
  \leq
  \E\big[
    e^{ U( X_0 ) }
  \big]
  < \infty
  .
\end{equation}
Combining this with 
the fact that for all 
$ N \in \N \cap [ T, \infty ) $ it holds that
$
  Y^N
  =
  Z^{ \theta_N } 
$
and 
$
  \tau_N
  =
  \varrho_N
$
proves \eqref{eq:Y_bound_2}.
It thus remains to prove 
\eqref{eq:Y_bound_1}.
For this observe that 
\eqref{eq:Z_bound_1b}
together with the fact that 
for all 
$ N \in \N \cap [T,\infty) $ it holds that
$
  Y^N
  =
  Z^{ \theta_N } 
$
and 
$
  \tau_N
  =
  \varrho_N
$
assures that
\begin{equation}
\label{eq:bound_Y_largeN}
  \sup_{
    N \in \N 
    \cap [ T, \infty )
  }
  \sup_{ t \in [0,T] }
  \E\!\left[
       \exp\!\left(
         e^{ - \rho t }
         \,
         U( Y^N_t )
         +
         \smallint_0^{ t \wedge \tau_N }
           e^{ - \rho s }
           \,
           \bar{U}( Y^N_s )
         \, ds
       \right)
     \right]
  <
  \infty
  .
\end{equation}
In addition, we observe that the fact that
$
  \forall \, N \in \N , t \in [0,T] 
  \colon
  Y^N_t = Y^N_{ t \wedge \tau_N }
$
proves that
for all $ N \in \N $,
$ t \in [0,T] $
it holds that
\begin{equation}
\label{eq:U_YN_t}
\begin{split}
&
     \exp\!\left(
       e^{ - \rho t }
       \,
       U( Y^N_t )
       +
       \smallint_0^{ t \wedge \tau_N }
         e^{ - \rho s }
         \,
         \bar{U}( Y^N_s )
       \, ds
     \right)
\\ & =
     \exp\!\left(
       e^{ - \rho t }
       \,
       U( Y^N_{ t \wedge \tau_N } )
       +
       \smallint_0^{ t \wedge \tau_N }
         e^{ - \rho s }
         \,
         \bar{U}( Y^N_{ s \wedge \tau_N } )
       \, ds
     \right)
\\ & =
  \mathbbm{1}_{
    \{ \tau_N = 0 \}
  }
     \exp\!\left(
       e^{ - \rho t }
       \,
       U( Y^N_{ t \wedge \tau_N } )
       +
       \smallint_0^{ t \wedge \tau_N }
         e^{ - \rho s }
         \,
         \bar{U}( Y^N_{ s \wedge \tau_N } )
       \, ds
     \right)
\\ &
  +
  \mathbbm{1}_{
    \{ \tau_N > 0 \}
  }
     \exp\!\left(
       e^{ - \rho t }
       \,
       U( Y^N_{ t \wedge \tau_N } )
       +
       \smallint_0^{ t \wedge \tau_N }
         e^{ - \rho s }
         \,
         \bar{U}( Y^N_{ s \wedge \tau_N } )
       \, ds
     \right)
\\ & =
  \mathbbm{1}_{
    \{ \tau_N = 0 \}
  }
     \exp\!\left(
       e^{ - \rho t }
       \,
       U( X_0 )
     \right)
  +
  \mathbbm{1}_{
    \{ \tau_N > 0 \}
  }
     \exp\!\left(
       e^{ - \rho t }
       \,
       U( Y^N_{ t \wedge \tau_N } )
       +
       \smallint_0^{ t \wedge \tau_N }
         e^{ - \rho s }
         \,
         \bar{U}( Y^N_{ s \wedge \tau_N } )
       \, ds
     \right)
     .
\end{split}
\end{equation}
Moreover, we note that the fact that
$
  \forall \, x \in \R \colon 
  | x | \leq 1 + | x |^q
$
ensures for all 
$ N \in \N $, $ t \in [0,T] $ 
that
\begin{equation}
  \left\|
    Y^N_t 
  \right\|
\leq
  \big\|
    Y^N_{ \llcorner t \lrcorner_{ \theta_N } }
  \big\|
  +
    \frac{
      \big\|
      \mu(
        Y_{
          \llcorner t \lrcorner_{ \theta_N }
        }^{ \theta_N }
      )
      \,
      (
        t
        -
        \llcorner t \lrcorner_{ \theta_N }
      )
      +
      \sigma(
        Y^{ \theta_N }_{
          \llcorner t \lrcorner_{ \theta_N }
        }
      )
      \,
      (
        W_{
          t
        }
        -
        W_{
          \llcorner t \lrcorner_{ \theta_N }
        }
      )
      \big\|
    }{
      1 +
      \big\|
      \mu(
        Y_{
          \llcorner t \lrcorner_{ \theta_N }
        }^N
      )
      \,
      (
        t
        -
        \llcorner t \lrcorner_{ \theta_N }
      )
      +
      \sigma(
        Y_{
          \llcorner t \lrcorner_{ \theta_N }
        }^N
      )
      \,
      (
        W_{
          t
        }
        -
        W_{
          \llcorner t \lrcorner_{ \theta_N }
        }
      )
      \big\|^q
    }
\leq
  \big\|
    Y^N_{ \llcorner t \lrcorner_{ \theta_N } }
  \big\|
  +
  1
  .
\end{equation}
This implies for all $ N \in \N $, $ t \in [0,T] $
that
\begin{equation}
  \left\|
    Y^N_{ t \wedge \tau_N }
    \mathbbm{1}_{
      \{ \tau_N > 0 \}
    }
  \right\|
\leq
  \big\|
    Y^N_{ \llcorner t \wedge \tau_N \lrcorner_{ \theta_N } }
    \mathbbm{1}_{
      \{ \tau_N > 0 \}
    }
  \big\|
  +
  1
\leq 
  \exp\!\left(
    \left|
      \ln( T / N )
    \right|^{ 1 / 2 }
  \right)
  + 1
  .
\end{equation}
Combining this with \eqref{eq:U_YN_t} establishes
for all $ N \in \N $, $ t \in [0,T] $
that
\begin{equation}
\begin{split}
&
     \exp\!\left(
       e^{ - \rho t }
       \,
       U( Y^N_t )
       +
       \smallint_0^{ t \wedge \tau_N }
         e^{ - \rho s }
         \,
         \bar{U}( Y^N_s )
       \, ds
     \right)
\\ & \leq
  \mathbbm{1}_{
    \{ \tau_N = 0 \}
  }
     \exp\!\left(
       U( X_0 )
     \right)
  +
  \mathbbm{1}_{
    \{ \tau_N > 0 \}
  }
     \exp\!\left(
       U( Y^N_{ t \wedge \tau_N } )
       +
       \smallint_0^{ t \wedge \tau_N }
         |
           \bar{U}( Y^N_{ s \wedge \tau_N } )
         |
       \, ds
     \right)
\\ & \leq
  e^{
    U( X_0 )
  }
  +
  \mathbbm{1}_{
    \{ \tau_N > 0 \}
  }
     \exp\!\left(
       U\big( 
         Y^N_{ t \wedge \tau_N } 
         \mathbbm{1}_{
           \{ \tau_N > 0 \}
         }
       \big)
       +
       \smallint_0^T
         \big|
           \bar{U}\big( 
             Y^N_{ s \wedge \tau_N } 
             \mathbbm{1}_{
               \{ \tau_N > 0 \}
             }
           \big)
         \big|
       \, ds
     \right)
\\ & \leq
  e^{
    U( X_0 )
  }
  +
  \mathbbm{1}_{
    \{ \tau_N > 0 \}
  }
     \exp\!\left(
     \left[
       \sup_{ 
         v \in \R^d , \| v \| \leq \exp( | \ln( T / N ) |^{ 1 / 2 } ) + 1
       }
       U( v )
     \right]
       +
       T
       \left[
       \sup_{ 
         v \in \R^d , \| v \| \leq \exp( | \ln( T / N ) |^{ 1 / 2 } ) + 1
       }
         |
           \bar{U}( v )
         |
       \right]
     \right)
     .
\end{split}
\end{equation}
Hence, we obtain for all $ N \in \N $ that
\begin{equation}
\begin{split}
&
   \sup_{ t \in [0,T] }
   \E\!\left[
     \exp\!\left(
       e^{ - \rho t }
       \,
       U( Y^N_t )
       +
       \smallint_0^{ t \wedge \tau_N }
         e^{ - \rho s }
         \,
         \bar{U}( Y^N_s )
       \, ds
     \right)
   \right]
\\ & \leq
  \E\!\left[
    e^{
      U( X_0 )
    }
  \right]
  +
     \exp\!\left(
     \left[
       \sup_{ 
         v \in \R^d , \| v \| \leq \exp( | \ln( T / N ) |^{ 1 / 2 } ) + 1
       }
       U( v )
     \right]
       +
       T
       \left[
       \sup_{ 
         v \in \R^d , \| v \| \leq \exp( | \ln( T / N ) |^{ 1 / 2 } ) + 1
       }
         |
           \bar{U}( v )
         |
       \right]
     \right)
     .
\end{split}
\end{equation}
Combining this with the assumption that
$
  \E\big[ e^{ U( X_0 ) } \big] < \infty
$
and the assumption that $ U $ and $ \bar{U} $ are continuous
ensures that
\begin{equation}
\label{eq:bound_Y_smallN}
\begin{split}
&
  \sup_{ 
    N \in \N \cap [ 0, \max\{ T, 1 \} ] 
  }
   \sup_{ t \in [0,T] }
   \E\!\left[
     \exp\!\left(
       e^{ - \rho t }
       \,
       U( Y^N_t )
       +
       \smallint_0^{ t \wedge \tau_N }
         e^{ - \rho s }
         \,
         \bar{U}( Y^N_s )
       \, ds
     \right)
   \right]
\\ & \leq
  \E\!\left[
    e^{
      U( X_0 )
    }
  \right]
  +
     \exp\!\left(
     \left[
       \sup_{ 
         v \in \R^d , \| v \| \leq \exp( | \ln( T ) |^{ 1 / 2 } ) + 1
       }
       U( v )
     \right]
       +
       T
       \left[
       \sup_{ 
         v \in \R^d , \| v \| \leq \exp( | \ln( T ) |^{ 1 / 2 } ) + 1
       }
         |
           \bar{U}( v )
         |
       \right]
     \right)
   < \infty
     .
\end{split}
\end{equation}
Inequality~\eqref{eq:bound_Y_smallN} together with inequality~\eqref{eq:bound_Y_largeN} 
establishes inequality~\eqref{eq:Y_bound_1}.
The proof of Corollary~\ref{cor:for_examples}
is thus completed.
\end{proof}

Observe, in the setting of Corollary~\ref{cor:for_examples},
that the assumption that $ X = ( X_t )_{ t \in [0,T] } \colon [0,T] \times \Omega \to D $
is an $ ( \mathcal{F}_t )_{ t \in [0,T] } $-adapted stochastic process, in particular,
ensures that the initial random variable
$
  X_0 \colon \Omega \to D
$
is an $ \mathcal{F}_0 $/$ \mathcal{B}( D ) $-measurable mapping.

\section{Examples of SDEs with exponential moments}
\label{sec:examples}

In this section Corollary~\ref{cor:for_examples}
is applied to a number of example SDEs from the literature.
To keep this article at a reasonable length, we present
the example SDEs here in a very brief way
and refer to \cite{HutzenthalerJentzen2014Memoires,CoxHutzenthalerJentzen2014}
for references and further details for these example SDEs.

\subsection{Setting}
\label{sec:ex_setting}

Throughout Section~\ref{sec:examples}
the following setting is used.
Let
$ T \in ( 0, \infty) $,
$ d, m \in \N $,
$ \mu \in \mathcal{M}( \mathcal{B}( \R^d ), \mathcal{B}( \R^d ) ) $,
$ \sigma \in \mathcal{M}( \mathcal{B}( \R^d ), \mathcal{B}( \R^{ d \times m } ) ) $,
let $ D \subseteq \R^d $
be an open set,
let 
$
  (
    \Omega, \mathcal{F}, \P
  )
$ 
be a probability space with a normal filtration 
$
  (
    \mathcal{F}_t
  )_{
    t \in [0, T ]
  }
$,
assume that $ \mu|_D \colon D \to \R^d $
and $ \sigma|_D \colon D \to \R^{ d \times m } $
are locally Lipschitz continuous,
let $ W \colon [0,T] \times \Omega \to \R^m $
be a standard $ ( \mathcal{F}_t )_{ t \in [0,T] } $-Brownian
motion with continuous sample paths,
let $ X = ( X^1, \dots, X^d ) \colon [0,T] \times \Omega \to D $
be an $ ( \mathcal{F}_t )_{ t \in [0,T] } $-adapted stochastic process 
with continuous sample paths which satisfies that
for all $ t \in [0,T] $ it holds $ \P $-a.s.\ that
\begin{equation}
  X_t = X_0 + \int_0^t \mu( X_s ) \, ds
  +
  \int_0^t \sigma( X_s ) \, dW_s
  ,
\end{equation}
and let
$ Y^N = ( Y^{ 1, N }, \dots, Y^{ d, N } ) \colon [0,T] \times \Omega \to \R^d $,
$ N \in \N $,
and
$
  \tau_N \colon \Omega \to [0,T]
$,
$ N \in \N $,
be functions
satisfying
for all $ N \in \N $,
$ n \in \{ 0, 1, \dots, N - 1 \} $,
$ t \in [ \frac{ n T }{ N } , \frac{ ( n + 1 ) T }{ N } ] $
that
$ Y^N_0 = X_0 $
and
\begin{equation}
  Y_t^N
  =
  Y_{ \frac{ n T }{ N } }^N
  +
  \1_{
    \left\{
      Y^N_{ n T / N } \in D
    \right\}
    \cap
    \left\{
      \| Y_{ n T / N }^N \|
      \leq
      \exp\left(
        | \ln( T / N ) |^{ 1 / 2 }
      \right)
    \right\}
  }
  \left[
    \tfrac{
      \mu( Y_{ n T / N }^N )
      ( 
        t - 
        \frac{ n T }{ N }
      )
      +
      \sigma( Y_{ n T / N }^N )
      (
        W_t
        -
        W_{ n T / N }
      )
    }{
      1 +
      \|
        \mu( Y_{ n T / N }^N)
        ( 
          t - 
          \frac{ n T }{ N }
        )
        +
        \sigma( Y_{ n T / N }^N )
        (
          W_t
          -
          W_{ n T / N }
        )
      \|^2
    }
   \right]
\end{equation}
and
$
  \tau_N =
  \inf\!\big(
  \big\{
    s \in 
    \{ 
      0, \frac{ T }{ N }, \frac{ 2 T }{ N }, \dots, T 
    \}
    \colon
    Y^N_s \notin D
    \text{ or }
    \| Y^N_s \|
    >
    \exp(
      | \ln( T / N ) |^{ 1 / 2 }
    )
  \big\}
  \cup 
  \{ T \}
  \big)
$.
Then
Corollary~\ref{cor:convergence_increment_tamed}
ensures for all
$ \varepsilon \in ( 0, \infty ) $
that
$
    \limsup_{ N \to \infty }
    \P\big[
      \sup_{ t \in [0,T] }
      \|
        X_t
        -
        Y^N_t
      \|
      \geq \varepsilon
    \big]
    = 0
$.

\subsection{Stochastic Ginzburg-Landau equation}
\label{ssec:Stochastic Ginzburg-Landau equation}

In this subsection assume the setting in
Subsection~\ref{sec:ex_setting},
let $ \alpha \in [0,\infty) $, $ \beta, \delta \in (0,\infty) $,
$ \varepsilon \in (0, \frac{ \delta }{ \beta^2 } ] $,
$ U, \bar{U} \in C( \R, \R ) $,
and
assume for all $ x \in \R $
that
$ d = m = 1 $,
$ D = \R $,
$
  \mu( x ) = \alpha x - \delta x^3
$,
$
  \sigma( x ) = \beta x
$,
$
  U( x ) = \varepsilon x^2
$,
$
  \bar{U}( x )
  =
  2
  \varepsilon
  \,
  [
    \delta
    -
    \beta^2 \varepsilon
  ]
  \,
  x^4
$,
and 
$
  \E\big[ e^{ U( X_0 ) } \big] < \infty
$.
Then it holds for all
$ x \in \R $ that
\begin{equation}
\begin{split}
&
  ( \mathcal{G}_{ \mu, \sigma } U)( x )
  +
  \tfrac{ 1 }{ 2 }
  \|
    \sigma( x )^* ( \nabla U )( x )
  \|^2
  +
  \bar{ U }( x )
=
  \varepsilon
  \left[
    2 x
    \left[
      \alpha x - \delta x^3
    \right]
    +
    \beta^2 x^2
  \right]
  +
  2
  \left( \beta \varepsilon \right)^2
  x^4
  +
  \bar{U}( x )
\\ & =
  \varepsilon
  \left[
  2
  \alpha
  +
  \beta^2
  \right]
  x^2
  +
  2
    \varepsilon
  \left[
    \beta^2 \varepsilon
    -
    \delta
  \right]
  x^4
  +
  \bar{ U }( x )
=
  \left[
  2
  \alpha
  +
  \beta^2
  \right]
  U( x )
\end{split}
\end{equation}
and Corollary~\ref{cor:for_examples}
hence shows
for all $ r \in ( 0, \infty ) $
that
$
  \limsup_{ N \to \infty }
  \big(
  \sup_{ t \in [0,T] }
  \E\big[
    \| X_t - Y^N_t \|^r
  \big]
  \big)
  = 0
$
and
\begin{gather}
  \sup_{ N \in \N }
  \sup_{
    t \in [0,T]
  }
  \E\!\left[
       \exp\!\left(
         \tfrac{
           \varepsilon \, ( Y^N_t )^2
         }{
           e^{ [ 2 \alpha + \beta^2 ] t }
         }
         +
         \smallint_0^{ t \wedge \tau_N }
           \tfrac{
             2 \varepsilon \, [ \delta - \beta^2 \varepsilon ]
             \, ( Y^N_s )^4
           }{
             e^{ [ 2 \alpha + \beta^2 ] s }
           }
         \, ds
       \right)
     \right]
  <
  \infty ,
\\
  \limsup_{ N \to \infty }
  \sup_{ t \in [0,T] }
  \E\!\left[
       \exp\!\left(
         \tfrac{
           \varepsilon \, ( Y^N_t )^2
         }{
           e^{ [ 2 \alpha + \beta^2 ] t }
         }
         +
         \smallint_0^{ t \wedge \tau_N }
           \tfrac{
             2 \varepsilon \, [ \delta - \beta^2 \varepsilon ]
             \, ( Y^N_s )^4
           }{
             e^{ [ 2 \alpha + \beta^2 ] s }
           }
         \, ds
       \right)
     \right]
  \leq
  \E\!\left[
    e^{
      \varepsilon | X_0 |^2
    }
  \right]
  < \infty
  .
\end{gather}

\subsection{Stochastic Lorenz equation with additive noise}
\label{ssec:stochastic.Lorenz.equation}

In this subsection assume the setting in
Subsection~\ref{sec:ex_setting},
let
$
  \alpha_1, \alpha_2, \alpha_3, \beta \in [0,\infty)
$,
$ \varepsilon \in (0,\infty) $,
$ 
  U, \bar{U} \in C( \R^3, \R ) 
$,
$
  \vartheta 
  =
    \min_{ r \in (0,\infty) }
    \max\{
      ( \alpha_1 + \alpha_2 )^2
      /
      r
      -
      2 \alpha_1
      ,
        r - 1
      ,
      0
    \}
  \in [0,\infty)
$,
and 
assume 
for all $ x = ( x_1, x_2, x_3 ), u = ( u_1, u_2, u_3 ) \in \R^3 $ that
$ d = m = 3 $,
$ D = \R^3 $,
$
  \sigma( x ) u = \sqrt{ \beta } u
$,
$
  \mu( x )
  =
  \big( 
    \alpha_1 \left( x_2 - x_1 \right) 
    ,
    \alpha_2 x_1 - x_2 - x_1 x_3
    ,
    x_1 x_2 - \alpha_3 x_3
  \big)
$,
$
  U( x ) = \varepsilon \| x \|^2
$,
$
  \bar{U}( x )
  =
  - 3 \varepsilon \beta
$,
and
$
  \E\big[ e^{ U( X_0 ) } \big] < \infty
$.
Then 
it holds for all
$ x \in \R^3 $ 
that
$
  ( \mathcal{G}_{ \mu, \sigma } U)( x )
  +
  \tfrac{ 1 }{ 2 }
  \| \sigma(x)^* ( \nabla U )(x) \|^2
  +
  \bar{ U }( x )
\leq
  \left[ 2 \varepsilon \beta + \vartheta \right] U(x)
$
(cf.\ Subsection~4.4 in Cox et al.~\cite{CoxHutzenthalerJentzen2014})
and Corollary~\ref{cor:for_examples}
hence shows
for all $ r \in ( 0, \infty ) $
that
$
  \limsup_{ N \to \infty }
  \big(
  \sup_{ t \in [0,T] }
  \E\big[
    \| X_t - Y^N_t \|^r
  \big]
  \big)
  = 0
$,
$
  \sup_{ N \in \N }
  \sup_{
    t \in [0,T]
  }
  \E\!\left[
       \exp\!\left(
         \varepsilon \, \| Y^N_t \|^2
         \,
         e^{ - [ 2 \varepsilon \beta + \vartheta ] t }
       \right)
     \right]
  <
  \infty
$,
and
\begin{equation}
  \limsup_{ N \to \infty }
  \sup_{ t \in [0,T] }
  \E\!\left[
       \exp\!\left(
         \tfrac{
           \varepsilon \, \| Y^N_t \|^2
         }{
           e^{ [ 2 \varepsilon \beta + \vartheta ] t }
         }
       \right)
     \right]
  \leq
  \exp\!\left(
         \smallint\nolimits_0^T
           \tfrac{
             3 \varepsilon \beta
           }{
             e^{ [ 2 \varepsilon \beta + \vartheta ] s }
           }
         \, ds
  \right)
  \E\!\left[
    e^{
      \varepsilon \| X_0 \|^2
    }
  \right]
  < \infty
  .
\end{equation}

\subsection{Stochastic van der Pol oscillator}
\label{ssec:stochastic.van.der.Pol.oscillator}

In this subsection assume the setting in
Subsection~\ref{sec:ex_setting},
let
$
  \alpha, \varepsilon \in ( 0, \infty )
$,
$
  \gamma, \delta, \eta_0, \eta_1
  \in [0,\infty)
$,
$ U, \bar{U} \in C( \R^2 , \R ) $,
$
  \vartheta
=
  \min_{ r \in ( 0, \infty ) }
  \max\{
    | \delta - 1 | /
    r
    +
    \eta_1
    ,
    r \,
    | \delta - 1 |
    +
    2 \gamma
    +
    4 \eta_0 \varepsilon
  \}
  \in [0,\infty)
$,
let
$ g \colon \R \to \R^{ 1 \times m } $
be a globally Lipschitz continuous function
which satisfies for all $ y \in \R $ that
$
  \| g(y) \|^2
  \leq
  \eta_0 +
  \eta_1 | y |^2
$,
and assume 
for all
$
  x = (x_1, x_2)
  \in \mathbb{R}^2
$,
$ u \in \R^m $
that
$ d = 2 $,
$ D = \R^2 $,
$
  \mu( x )
=
  \left(
    x_2 ,
    \left( \gamma - \alpha ( x_1 )^2 \right)
    x_2
    - \delta x_1
  \right)
$,
$
  \sigma( x ) u
=
  \left(
    0 ,
    g( x_1 ) u
  \right)
$,
$
  \varepsilon \eta_1 \leq \alpha
$,
$
  U( x ) = \varepsilon \| x \|^2
$,
$
  \bar{U}( x )
  =
  2 \varepsilon
  \left[
    \alpha
    -
    \varepsilon \eta_1
  \right]
  ( x_1 x_2 )^2
  -
  \varepsilon \eta_0
$,
and
$
  \E\big[
    e^{ U( X_0 ) }
  \big]
  < \infty
$.
Then it holds for all
$ x \in \R^2 $
that
$
  ( \mathcal{G}_{ \mu, \sigma } U)( x )
  +
  \tfrac{ 1 }{ 2 }
  \|
    \sigma( x )^* ( \nabla U )(x)
  \|^2
  +
  \overline{U}( x )
\leq
  \vartheta
  \,
  U(x)
$
(cf.\ Subsection~4.2 in Cox et al.~\cite{CoxHutzenthalerJentzen2014})
and 
Corollary~\ref{cor:for_examples}
hence shows
for all $ r \in ( 0, \infty ) $
that
$
  \limsup_{ N \to \infty }
  \big(
  \sup_{ t \in [0,T] }
  \E\big[
    \| X_t - Y^N_t \|^r
  \big]
  \big)
  = 0
$
and
\begin{gather}
  \sup_{ N \in \N }
  \sup_{
    t \in [0,T]
  }
  \E\!\left[
       \exp\!\left(
         \tfrac{
           \varepsilon \, \| Y^N_t \|^2
         }{
           e^{ \vartheta t }
         }
         +
         \smallint_0^{ t \wedge \tau_N }
           \tfrac{
             2 \varepsilon \,
             [ \alpha - \varepsilon \eta_1 ]
             \,
             | Y^{ 1, N }_s Y^{ 2, N }_s |^2
           }{
             e^{ \vartheta s }
           }
         \, ds
       \right)
     \right]
  <
  \infty ,
\\
  \limsup_{ N \to \infty }
  \sup_{ t \in [0,T] }
  \E\!\left[
       \exp\!\left(
         \tfrac{
           \varepsilon \, \| Y^N_t \|^2
         }{
           e^{ \vartheta t }
         }
         +
         \smallint_0^{ t \wedge \tau_N }
           \tfrac{
             2 \varepsilon \,
             [ \alpha - \varepsilon \eta_1 ]
             \,
             | Y^{ 1, N }_s Y^{ 2, N }_s |^2
           }{
             e^{ \vartheta s }
           }
         \, ds
       \right)
     \right]
  \leq
  \exp\!\left(
    \smallint_0^T
    \tfrac{
      \varepsilon \eta_0
    }{
      e^{ \vartheta s }
    }
    \, ds
  \right)
  \E\!\left[
    e^{
      \varepsilon \| X_0 \|^2
    }
  \right]
  < \infty
  .
\end{gather}

\subsection{Stochastic Duffing-van
der Pol oscillator}
\label{ssec:stochastic.Duffing.van.der.Pol.oscillator}

In this subsection assume the setting in
Subsection~\ref{sec:ex_setting},
let
$
  \eta_0, \eta_1, \alpha_1 \in [0,\infty)
$,
$
  \alpha_2, \alpha_3, \varepsilon
  \in (0,\infty)
$,
$
  U, \bar{U} \in C( \R^2, \R )
$,
let $ g \colon \R \to \R^{ 1 \times m } $
be a globally Lipschitz continuous function
which satisfies for all $ y \in \R $
that
$
  \| g(y) \|^2
  \leq
  \eta_0 +
  \eta_1 | y |^2
$,
and assume 
for all
$
  x = (x_1, x_2)
  \in \mathbb{R}^2
$,
$ u \in \R^m $
that
$ d = 2 $,
$ D = \R^2 $,
$
  \sigma( x ) u
=
  \left(
    0 ,
    g( x_1 ) u
  \right)
$,
$
  \mu( x )
=
  \left(
    x_2 ,
    \alpha_2 x_2 - \alpha_1 x_1
    - \alpha_3 ( x_1 )^2 x_2
    - ( x_1 )^3
  \right)
$,
$
  \varepsilon \eta_1 \leq \alpha_3
$,
$
  U(x_1,x_2)
=
  \varepsilon
  \big[
    \frac{ 1 }{ 2 }
    \left( x_1 \right)^4 
    +
    \alpha_1
    \left( x_1 \right)^2
    +
    \left( x_2 \right)^2
  \big]
$,
$
  \bar{U}( x )
  =
  2 \varepsilon
  \left[
    \alpha_3
    -
    \varepsilon
    \eta_1
  \right]
  ( x_1 x_2 )^2
  -
  \varepsilon \eta_0
  -
    \tfrac{
      \varepsilon
    \left|
      0 \vee
      (
        \eta_1
        -
        2 \alpha_1
        (
          \varepsilon
          \eta_0
          +
          \alpha_2
        )
      )
    \right|^2
    }{
      4
      \left(
        \varepsilon
        \eta_0
        +
        \alpha_2
      \right)
    }
$,
and 
$
  \E\big[
    e^{ U( X_0 ) }
  \big] < \infty
$.
Then it holds for all
$ x \in \R^2 $
that
$
  ( \mathcal{G}_{ \mu, \sigma } U)( x )
  +
  \tfrac{ 1 }{ 2 }
  \|
    \sigma(x)^*
    ( \nabla U)( x )
  \|^2
  +
  \bar{ U }( x )
\leq
      2
      \left(
        \varepsilon \eta_0
        +
        \alpha_2
      \right)
      U(x)
$
(cf.\ Subsection~4.3 in Cox et al.~\cite{CoxHutzenthalerJentzen2014})
and Corollary~\ref{cor:for_examples}
hence shows
for all $ r \in ( 0, \infty ) $
that
$
  \limsup_{ N \to \infty }
  \big(
  \sup_{ t \in [0,T] }
  \E\big[
    \| X_t - Y^N_t \|^r
  \big]
  \big)
  = 0
$
and
\begin{equation}
  \sup_{ N \in \N }
  \sup_{
    t \in [0,T]
  }
  \E\!\left[
       \exp\!\left(
         \tfrac{
           \frac{ \varepsilon }{ 2 }
           \left| Y^{ 1, N }_t \right|^4
           +
           \varepsilon \alpha_1
           \left| Y^{ 1, N }_t \right|^2
           +
           \varepsilon
           \left| Y^{ 2, N }_t \right|^2
         }{
           e^{ 2 t [ \varepsilon \eta_0 + \alpha_2 ] }
         }
         +
         \smallint_0^{ t \wedge \tau_N }
           \tfrac{
             2 \varepsilon \,
             \left[
               \alpha_3 - \varepsilon \eta_1
             \right]
             \,
             \left|
               Y^{ 1, N }_s Y^{ 2, N }_s
             \right|^2
           }{
             e^{ 2 s [ \varepsilon \eta_0 + \alpha_2 ] }
           }
         \, ds
       \right)
     \right]
  <
  \infty ,
\end{equation}
\begin{equation}
\begin{split}
&
  \limsup_{ N \to \infty }
  \sup_{ t \in [0,T] }
  \E\!\left[
       \exp\!\left(
         \tfrac{
           \frac{ \varepsilon }{ 2 }
           \left| Y^{ 1, N }_t \right|^4
           +
           \varepsilon \alpha_1
           \left| Y^{ 1, N }_t \right|^2
           +
           \varepsilon
           \left| Y^{ 2, N }_t \right|^2
         }{
           e^{ 2 t [ \varepsilon \eta_0 + \alpha_2 ] }
         }
         +
         \smallint_0^{ t \wedge \tau_N }
           \tfrac{
             2 \varepsilon \,
             \left[
               \alpha_3 - \varepsilon \eta_1
             \right]
             \,
             \left|
               Y^{ 1, N }_s Y^{ 2, N }_s
             \right|^2
           }{
             e^{ 2 s [ \varepsilon \eta_0 + \alpha_2 ] }
           }
         \, ds
       \right)
     \right]
\\ & \leq
    \exp\!\left(
      \smallint_0^T
      \tfrac{
        \varepsilon \eta_0
      }{
        e^{
          2 s
          \left[
            \varepsilon \eta_0
            +
            \alpha_2
          \right]
        }
      }
      +
      \tfrac{
      \varepsilon
      \,
    \left|
      0 \vee
      (
        \eta_1
        -
        2 \alpha_1
        \left[
          \varepsilon \eta_0
          +
          \alpha_2
        \right]
      )
    \right|^2
  }{
      4
      \,
      \left[
        \varepsilon \eta_0
        +
        \alpha_2
      \right]
      \,
      e^{
        2 s
        \left[
          \varepsilon \eta_0
          +
          \alpha_2
        \right]
      }
  }
      \, ds
    \right)
  \E\!\left[
    e^{
      \tfrac{ \varepsilon }{ 2 }
      | X^1_0 |^4
      +
      \varepsilon \alpha_1
      | X^1_0 |^2
      +
      \varepsilon
      | X^2_0 |^2
    }
  \right]
  < \infty
  .
\end{split}
\end{equation}

\subsection{Experimental psychology model}
\label{ssec:experimental.psychology}

In this subsection assume the setting in
Subsection~\ref{sec:ex_setting},
let
$ \alpha, \delta, \varepsilon \in (0,\infty) $,
$ \beta \in \R $,
$ q \in [3,\infty) $,
$ U \in C( \R^2, \R ) $,
and
assume for all
$ x = ( x_1 , x_2 ) \in \R^2 $
that
$ d = 2 $, 
$ m = 1 $,
$ D = \R^2 $,
$
  \mu( x_1, x_2 )
=
  \big(
      ( x_2 )^2
      ( \delta + 4 \alpha x_1 )
      - \frac{ 1 }{ 2 } \beta^2 x_1
    ,
      - x_1 x_2
      ( \delta + 4 \alpha x_1 )
      - \frac{ 1 }{ 2 } \beta^2 x_2
  \big)
$,
$
  \sigma( x_1, x_2 )
  =
  ( - \beta x_2 , \beta x_1 )
$,
$
  U( x )
  =
  \varepsilon 
  \| x \|^q
$,
and
$
  \E\big[
    e^{ U( X_0 ) }
  \big]
  < \infty
$.
Then it holds for all
$ x \in \R^2 $ that
$
  ( \mathcal{G}_{ \mu, \sigma } U )( x )
  +
  \tfrac{ 1 }{ 2 }
  \|
    \sigma( x )^* \,
    ( \nabla U )( x )
  \|^2
=
  0
$
(cf.\ Subsection~4.8 in Cox et al.~\cite{CoxHutzenthalerJentzen2014})
and Corollary~\ref{cor:for_examples}
hence shows
for all $ r \in ( 0, \infty ) $
that
$
  \limsup_{ N \to \infty }
  \big(
  \sup_{ t \in [0,T] }
  \E\big[
    \| X_t - Y^N_t \|^r
  \big]
  \big)
  = 0
$,
$
  \sup_{ N \in \N }
  \sup_{
    t \in [0,T]
  }
  \E\!\left[
       \exp\!\left(
           \varepsilon
           \| Y^N_t \|^q
       \right)
     \right]
  <
  \infty
$,
and
$
  \limsup_{ N \to \infty }
  \sup_{ t \in [0,T] }
  \E\!\left[
       \exp\!\left(
           \varepsilon
           \| Y^N_t \|^q
       \right)
     \right]
  \leq
  \E\!\left[
       \exp\!\left(
           \varepsilon
           \| X_0 \|^q
       \right)
  \right]
$.

\subsection{Stochastic SIR model}
\label{ssec:stochastic.SIR.model}

In this subsection assume the setting in
Subsection~\ref{sec:ex_setting},
let
$
  \alpha, \beta, \gamma, \delta,
  \varepsilon 
  \in (0,\infty)
$,
$ U, \bar{U} \in C( \R^3, \R ) $,
$
  \hat{ \varepsilon } \in
  ( 0, \frac{ 4 \varepsilon \delta }{ \gamma } ]
$,
assume that
$ d = 3 $, 
$ m = 1 $,
and
$ D = (0,\infty)^3 $,
assume 
for all
$
  x = (x_1, x_2, x_3)
  \in D
$
that
$
  \mu\!\left(
      x_1 ,
      x_2 ,
      x_3
  \right)
=
  (
      - \alpha x_1 x_2
      - \delta x_1
      + \delta
      ,
      \alpha x_1 x_2
      -
      ( \gamma + \delta ) x_2
      ,
      \gamma x_2 - \delta x_3
  )
$,
$
  \sigma\!\left(
      x_1
    ,
      x_2
    ,
      x_3
  \right)
=
  (
      - \beta x_1 x_2
    ,
      \beta x_1 x_2
    ,
      0
  )
$,
assume 
for all $ x \in \R^3 \backslash D $
that
$
  \mu( x ) = \sigma( x ) = 0
$,
let $ \phi \colon \R \to [0,1] $
and $ \psi \colon \R^2 \to [0,1] $
be infinitely often differentiable functions
which satisfy
for all $ x \in ( - \infty, 0 ] $
that
$ \phi( x ) = 0 $,
which satisfy
for all $ x \in [1,\infty) $
that
$ \phi( x ) = 1 $,
and 
which satisfy
for all $ x = ( x_1, x_2 ) \in \R^2 $
that
$
  \psi( x_1 , x_2 )
  =
  \phi( x_1 ) \cdot \phi( - x_2 )
  +
  \phi( - x_1 ) \cdot \phi( x_2 )
$,
and assume 
for all $ x = ( x_1, x_2, x_3 ) \in \R^3 $
that
$
  U( x )
=
  \varepsilon
  \left[
    \tfrac{ 5 }{ 2 }
    +
    ( x_1 + x_2 )^2
    -
    2 \cdot x_1 \cdot x_2 \cdot \psi( x_1, x_2 )
  \right]
  +
  \hat{ \varepsilon }
  \left[ x_3 \right]^2
$,
$
  \bar{ U }( x ) = - 2 \varepsilon \delta
$,
and 
$ \E\big[ e^{ U( X_0 ) } \big] < \infty $.
Then note
for all $ x = ( x_1, x_2, x_3 ) \in D $
that
\begin{equation}
\begin{split}
&
  ( \mathcal{G}_{ \mu, \sigma } U )( x )
  +
  \tfrac{ 1 }{ 2 }
  \|
    \sigma( x )^* ( \nabla U )( x )
  \|^2
  +
  \bar{ U }( x )
=
  ( \mathcal{G}_{ \mu, \sigma } U )( x )
  +
  \tfrac{ 1 }{ 2 }
  \left|
    \left<
      \sigma( x ),
      ( \nabla U )( x )
    \right>
  \right|^2
  +
  \bar{ U }( x )
\\ & =
  ( \mathcal{G}_{ \mu, \sigma } U )( x )
  +
  \bar{ U }( x )
=
  2 \varepsilon
  \left[
    x_1 + x_2
  \right]
  \left[
    - \delta x_1 + \delta
    - ( \gamma + \delta ) x_2
  \right]
  +
  2 \hat{ \varepsilon } x_3
  \left[
    \gamma x_2 - \delta x_3
  \right]
  +
  \bar{ U }( x )
\\ & =
  - 2 \varepsilon \delta
  \left[
    x_1 + x_2
  \right]
  \left[
    x_1 + x_2 - 1
  \right]
  -
  2 \varepsilon \gamma
  \left[
    x_1 + x_2
  \right]
  x_2
  +
  2 \hat{ \varepsilon } x_3
  \left[
    \gamma x_2 - \delta x_3
  \right]
  +
  \bar{ U }( x )
\\ & =
  - 2 \varepsilon \delta
  \left[
    x_1 + x_2 - 1
  \right]^2
  -
  2 \varepsilon \delta
  \left[
    x_1 + x_2
  \right]
  -
  2 \varepsilon \gamma
  \left[ x_1 + x_2 \right]
  x_2
  -
  2 \hat{ \varepsilon } \delta
  \left[ x_3 \right]^2
  +
  2 \varepsilon \delta
  +
  2 \hat{ \varepsilon } \gamma x_2 x_3
  +
  \bar{ U }( 0 )
\\ & \leq
  \bar{ U }( 0 )
  +
  2 \varepsilon \delta
  -
  2 \varepsilon \gamma
  \left[ x_2 \right]^2
  -
  2 \hat{ \varepsilon } \delta
  \left[ x_3 \right]^2
  +
  \left[
    2 \sqrt{ \varepsilon \gamma } x_2
  \right]
  \left[
    \tfrac{
      \hat{ \varepsilon }
      \sqrt{ \gamma }
      x_3
    }{
      \sqrt{ \varepsilon }
    }
  \right]
\leq
  \bar{ U }( 0 )
  +
  2 \varepsilon \delta
  +
  \hat{ \varepsilon }
  \left[
    \tfrac{
      \hat{ \varepsilon } \gamma
    }{ 2 \varepsilon }
    -
    2 \delta
  \right]
  \left[
    x_3
  \right]^2
\leq
  0
  .
\end{split}
\end{equation}
Combining 
(4.34)--(4.35) in 
Section~4.6 in \cite{HutzenthalerJentzen2014Memoires}
with Corollary~\ref{cor:for_examples}
therefore implies that 
for all $ r \in ( 0, \infty ) $
it holds that
$
  \limsup_{ N \to \infty }
  \big(
  \sup_{ t \in [0,T] }
  \E\big[
    \| X_t - Y^N_t \|^r
  \big]
  \big)
  = 0
$,
$
  \sup_{ N \in \N }
  \sup_{ t \in [0,T] }
  \E\big[
    e^{
      \frac{ \varepsilon }{ 2 }
        |
          Y^{ 1, N }_t
        |^2
        +
      \frac{ \varepsilon }{ 2 }
        |
          Y^{ 2, N }_t
        |^2
      +
      \hat{ \varepsilon }
      |
        Y^{ 3, N }_t
      |^2
    }
  \big]
  < \infty
$,
and
$
  \limsup_{ N \to \infty }
  \sup_{ t \in [0,T] }
  \E\big[
    e^{
      U( Y^N_t )
      -
      2 \varepsilon \delta
      \left(
        t \wedge \tau_N
      \right)
    }
  \big]
\leq
  \E\!\left[
    e^{
      U( X_0 )
    }
  \right]
  < \infty
$.

\subsection{Langevin dynamics}
\label{ssec:Langevin.dynamics}

In this subsection assume the setting in
Subsection~\ref{sec:ex_setting},
let
$ \beta, \gamma \in (0,\infty) $,
$
  \varepsilon \in
  ( 0, \frac{ 2 \gamma }{ \beta } ]
$,
$
  U, \bar{U} \in C( \R^{ 2 m }, \R )
$,
$ 
  V \in 
  C^3( \R^m, [0,\infty) ) \cap 
  ( 
    \cup_{ p, c \in (0,\infty) } C^3_{ p, c }( \R^m, [0,\infty) ) 
  )
$,
and
assume 
for all $ x = ( x_1, x_2 ) \in \R^{ 2 m } $, $ u \in \R^m $
that
$
  \limsup_{ r \searrow 0 }
  \sup_{ z \in \R^m }
  \frac{
    \| z \|^r
  }{
    1 + V( z )
  }
  < \infty
$,
$ d = 2 m $,
$ D = \R^d $,
$
  \mu( x )
  =
  ( x_2, - ( \nabla V )( x_1 ) - \gamma x_2 )
$,
$
  \sigma( x ) u
  = ( 0, \sqrt{ \beta } u )
$,
$
  U( x ) = \varepsilon \, V( x_1 ) + \tfrac{ \varepsilon }{ 2 } \, \| x_2 \|^2
$,
$
  \bar{U}( x )
  =
  \varepsilon
  \big[
    \gamma
    -
    \frac{ \varepsilon \beta }{ 2 }
  \big]
  \| x_2 \|^2
  -
  \frac{ \varepsilon \beta m }{ 2 }
$,
and 
$
  \E\big[
    e^{ U( X_0 ) }
  \big]
  < \infty
$.
Then it holds for all
$ x \in \R^{ 2 m } $
that
$
  ( \mathcal{G}_{ \mu, \sigma } U )( x )
  +
  \tfrac{
    1
  }{ 2 }
    \|
      \sigma( x )^*
      ( \nabla U )( x )	
    \|^2
  + \bar{ U }( x )
  = 0
$
(cf.\ Subsection~4.5 in Cox et al.~\cite{CoxHutzenthalerJentzen2014})
and 
Corollary~\ref{cor:for_examples}
hence shows
for all $ r \in ( 0, \infty ) $
that
$
  \limsup_{ N \to \infty }
  \big(
  \sup_{ t \in [0,T] }
  \E\big[
    \| X_t - Y^N_t \|^r
  \big]
  \big)
  = 0
$
and
\begin{gather}
  \sup_{ N \in \N }
  \sup_{ t \in [0,T] }
  \E\!\left[
    \exp\!\left(
      \varepsilon \, V( Y^{ 1, N }_t )
      +
      \tfrac{
        \varepsilon
      }{ 2 }
      \,
        \|
          Y^{ 2, N }_t
        \|^2
      +
      \smallint_0^{ t \wedge \tau_N }
        \varepsilon
        \left[
          \gamma
          -
          \tfrac{ \varepsilon \beta }{ 2 }
        \right]
        \|
          Y^{ 2, N }_s
        \|^2
        \,
      ds
    \right)
  \right]
  < \infty
  ,
\\
\nonumber
  \limsup_{ N \to \infty }
  \sup_{ t \in [0,T] }
  \E\!\left[
    \exp\!\left(
      \varepsilon \, V( Y^{ 1, N }_t )
      +
      \tfrac{
        \varepsilon
      }{ 2 }
      \,
        \|
          Y^{ 2, N }_t
        \|^2
      +
      \smallint_0^{ t \wedge \tau_N }
        \varepsilon
        \left[
          \gamma
          -
          \tfrac{ \varepsilon \beta }{ 2 }
        \right]
        \|
          Y^{ 2, N }_s
        \|^2
        \,
      ds
    \right)
  \right]
\leq
  \E\!\left[
    e^{
      \frac{
        \varepsilon \beta m T
      }{ 2 }
      +
      \varepsilon
      V( X_0^1 )
      +
      \frac{ \varepsilon }{ 2 }
      \| X^2_0 \|^2
    }
  \right]
  .
\end{gather}

\subsection{Brownian dynamics (Overdamped Langevin dynamics)}
\label{sec:overdamped_Langevin}
\label{ssec:overdamped.Langevin.dynamics}

In this subsection assume the setting in
Subsection~\ref{sec:ex_setting},
let
$ \beta \in (0,\infty) $,
$ \eta_0, \eta_1 \in [0,\infty) $, 
$ \eta_2 \in [ 0, \frac{ 2 }{ \beta } ) $,
$
  V \in 
  \cup_{ p, c \in (0,\infty) } C^3_{ p, c }( \R^d, [0,\infty) )
$,
$ \varepsilon \in (0, \frac{ 2 }{ \beta } - \eta_2 ] $,
$ U, \bar{U} \in C( \R^d, \R ) $,
assume 
for all $ x, u \in \R^d $
that
$ d = m $,
$ D = \R^d $,
$
  \limsup_{ r \searrow 0 }
  \sup_{ z \in \R^d }
  \frac{ \| z \|^r }{ 1 + V(z) }
  < \infty
$,
$
  \mu( x ) = - ( \nabla V )( x )
$,
$
  \sigma( x ) u =
  \sqrt{ \beta } u
$,
$
  ( \triangle V)( x )
\leq
  \eta_0
  +
  2\eta_1
  V(x)
  +
  \eta_2
  \left\|
    ( \nabla V )( x )
  \right\|^2
$,
$
  U(x)
  =
  \varepsilon V(x)
$,
$
  \bar{U}(x)
  =
  \varepsilon
  \,
  (
    1
    -
    \frac{ \beta }{ 2 }
    ( \eta_2 + \varepsilon )
  )
  \,
  \|
    ( \nabla V )( x )
  \|^2
  -
  \frac{ \varepsilon \beta \eta_0 }{ 2 }
$,
and
$
  \E\big[ 
    e^{ U( X_0 ) }
  \big]
  < \infty
$.
Then it holds for all
$ x \in \R^d $ that
$
  ( \mathcal{G}_{ \mu, \sigma } U )( x )
  +
  \tfrac{ 1 }{ 2 }
  \left\|
    \sigma( x )^*
    ( \nabla U )( x )
  \right\|^2
  +
  \bar{U}(x)
\leq
  \beta
  \eta_1
  U(x)
$
(cf.\ Subsection~4.6 in Cox et al.~\cite{CoxHutzenthalerJentzen2014})
and Corollary~\ref{cor:for_examples}
hence shows
for all $ r \in ( 0, \infty ) $
that
$
  \limsup_{ N \to \infty }
  \big(
  \sup_{ t \in [0,T] }
  \E\big[
    \| X_t - Y^N_t \|^r
  \big]
  \big)
  = 0
$
and
\begin{gather}
  \sup_{ N \in \N }
  \sup_{ t \in [0,T] }
  \E\!\left[
    \exp\!\left(
      \tfrac{
        \varepsilon V( Y^N_t )
      }{
        e^{ \beta \eta_1 t }
      }
      +
      \smallint_0^{ t \wedge \tau_N }
      \tfrac{
        \varepsilon
        \left[
          1
          -
          \frac{ \beta }{ 2 } ( \eta_2 + \varepsilon )
        \right]
      }{
        e^{ \beta \eta_1 s }
      }
      \|
        ( \nabla V )( Y^N_s )
      \|^2
      \,
      ds
    \right)
  \right]
  < \infty
  ,
\\
  \limsup_{ N \to \infty }
  \sup_{ t \in [0,T] }
  \E\!\left[
    \exp\!\left(
      \tfrac{
        \varepsilon V( Y^N_t )
      }{
        e^{ \beta \eta_1 t }
      }
      +
      \smallint_0^{ t \wedge \tau_N }
      \tfrac{
        \varepsilon
        \left[
          1
          -
          \frac{ \beta }{ 2 } ( \eta_2 + \varepsilon )
        \right]
      }{
        e^{ \beta \eta_1 s }
      }
      \|
        ( \nabla V )( Y^N_s )
      \|^2
      -
      \tfrac{
        \varepsilon \beta \eta_0
      }{
        2
        e^{
          \beta \eta_1 s
        }
      }
      \,
      ds
    \right)
  \right]
\leq
  \E\!\left[
    e^{
      \varepsilon V( X_0 )
    }
  \right]
  < \infty
  .
\end{gather}

\section{Counterexamples to exponential integrability properties}

Corollary~\ref{cor:for_examples} above establishes, under suitable assumptions, 
that stopped increment-tamed Euler-Maruyama approximations
converge strongly to the exact solution process
of the considered SDE 
and also inherit suitable exponential integrability properties of the exact solution process 
of the SDE.
In this section we illustrate in the case of one simple example SDE 
that several other approximation schemes, which converge strongly to the exact 
solution process of this example SDE, fail to preserve appropriate exponential integrability 
properties of the exact solution process of the SDE.

\subsection{An example SDE with finite exponential moments}
\label{ssec:exampleSDE}
Let $ T \in (0,\infty) $,
let
  $
    (
      \Omega, \mathcal{F}, \P
    )
  $ 
  be a probability space with a normal filtration 
  $
    (
      \mathcal{F}_t
    )_{
      t \in [0, T ]
    }
  $,
  let
  $
    W \colon [0,T] \times \Omega
    \rightarrow \mathbb{R}
  $
  be a standard 
  $ ( \calF_t )_{ t \in [0,T] } $-Brownian motion
  with continuous sample paths,
  let $ X \colon [0,T] \times \Omega \to \R $
  be an
  $ ( \mathcal{F}_t )_{ t \in [0,T] } $-adapted
  stochastic process
  with continuous sample paths which satisfies 
  that for all $ t \in [0,T] $
  it holds $ \P $-a.s.\ that
  \begin{equation}  
  \label{eq:SDE.minus.x3}
    X_t=X_0-\int_0^t (X_s)^3\,ds
           +\int_0^t dW_s
           ,
  \end{equation}
  let
  $
    \mu, \sigma \colon \R \to \R
  $
  be the functions with the property that
  for all $ x \in \R $
  it holds that
  $
    \mu(x) = - x^3
  $
  and
  $
    \sigma(x) = 1
  $,
  let
  $
    \eps \in ( 0, \tfrac{ 1 }{ 2 } ]
  $
  satisfy
  $
    \E\big[
      \exp( \eps | X_0 |^4 )
    \big]
    < \infty
  $,
  and
  let
  $
    U_{ \delta } \colon \R \to [0,\infty)
  $,
  $
    \delta \in [0,\infty)
  $,
  be the functions with the property that
  for all
  $ x \in \R $, 
  $ \delta \in [0,\infty) $
  it holds that
  $
    U_{ \delta }( x ) = \delta x^4
  $.
  Then observe for all
  $ \delta \in [0,\infty) $,
  $ x \in \R $
  that
  \begin{equation}  \begin{split}
  &
      \left(
        \mathcal{ G }_{ \mu , \sigma } U_{ \delta }
      \right)( x )
      +
      \tfrac{ 1 }{ 2 }
      \left|
        \sigma( x )^* ( \nabla U_{ \delta } )( x )
      \right|^2
  \\ &
  =
    \langle
      ( \nabla U_{ \delta } )( x ) , \mu(x)
    \rangle
    +
    \tfrac{ 1 }{ 2 }
    \tr\!\big(
      \sigma( x ) [ \sigma(x) ]^*
      ( \operatorname{Hess} U_{ \delta } )( x )
    \big)
    +
    \tfrac{ 1 }{ 2 }
    \left|
      \sigma( x )^* ( \nabla U_{ \delta } )( x )
    \right|^2
  \\ & =
    \delta 4 x^3 \cdot ( - x^3 )
    +
    \tfrac{ 1 }{ 2 }
    \delta 12 x^2
    +
    \tfrac{ 1 }{ 2 }
    \left| \delta 4 x^3 \right|^2
  =
    6 \delta x^2 -
    ( 4 \delta - 8 \delta^2 ) x^6
  =
    6 \delta x^2
    - 4 \delta
    \left( 1 - 2 \delta \right) x^6
    .
  \end{split}     
  \end{equation}
  Corollary~2.4 in Cox et al.~\cite{CoxHutzenthalerJentzen2014}
  (with
  $
    \bar{U} =
    [0,T] \times \R \ni (t,x) \mapsto 
    4 \delta
    \left( 1 - 2 \delta \right) x^6
    -
    6 \delta x^2
    \in \R
  $
  in the notation of Corollary~2.4 in Cox et al.~\cite{CoxHutzenthalerJentzen2014}; 
  see also Corollary~\ref{cor:for_examples} above)
  hence implies for all $ \delta \in [0,\varepsilon] $ that
  \begin{equation}  \begin{split} \label{eq:exampleSDE.Lyapunov}
  &
    \sup_{ t \in [0,T] }
    \E\!\left[
      \exp\!\left(
        \delta \left| X_t \right|^4
        +
        \int_0^t
        4 \delta
        \left( 1 - 2 \delta \right)
        \left| X_s \right|^6
        -
        6 \delta \left| X_s \right|^2
        ds
      \right)
    \right]
  \leq
    \E\!\left[
      e^{
        \delta |X_0|^4
      }
    \right]
  < \infty
    .
  \end{split}     \end{equation}
This shows, in particular,
that
for all $ \delta \in ( - \infty , \varepsilon] \cap ( - \infty , \nicefrac{ 1 }{ 2 } ) $
it holds that
$
  \sup_{ t \in [0,T] }
    \E\!\left[
      \exp\!\left(
        \delta \, | X_t |^4
      \right)
    \right]
  < \infty
$.

\subsection{Infinite exponential moments for (stopped) Euler approximation schemes}

The Euler scheme stopped after leaving certain sets is not suitable
for approximating the exponential moments on the left-hand side of~\eqref{eq:exampleSDE.Lyapunov}
as there is at least one Euler step and this results in tails of a normal
distribution.
Note that
in the special case $D_t=\R$, $t\in(0,T]$,
the numerical scheme~\eqref{eq:stopped.Euler}
is the Euler scheme for the SDE~\eqref{eq:SDE.minus.x3}.
We also note that Liu and Mao consider in \cite{LiuMao2013} a stopped Euler scheme
with $D_t=[0,\infty)$, $t\in[0,T]$,
for SDEs on the domain $[0,\infty)$.
\begin{lemma}
\label{l:stopped.Euler.no.exp}
  Assume the setting in Subsection~\ref{ssec:exampleSDE},
  let
  $
    D_t \in \mathcal{B}(\R)
  $,
  $
    t \in (0,T]
  $,
  be a non-increasing family of sets
  satisfying
  $
    \lambda_{ \R }( D_T )
    \cdot
    \P\big[
      X_0 \in D_T
    \big]
    > 0
  $
  and
  $
    \cup_{ t \in (0,T] } \mathring{ D }_t = \R
  $,
  and
  let 
  $
    Y^N \colon [0,T] \times \Omega \to \R 
  $,
  $ N \in \N $, 
  be the mappings which satisfy 
  for all
  $ N \in \N $,
  $ n \in \{ 0, 1, \ldots , N - 1 \} $,
  $
    t \in \big[ \frac{ n T }{ N } , \frac{ (n+1) T }{ N } \big]
  $
  that
  $ Y^N_0 = X_0 $ 
  and
  \begin{equation}  \label{eq:stopped.Euler}
    Y_{t}^{N}=Y_{\frac{nT}{N}}^N
    +\1_{D_{\frac{T}{N}}}(Y_{\frac{nT}{N}}^{N})
    \left(
           W_{t}-W_{\frac{nT}{N}}
    -\big(Y_{\frac{nT}{N}}^N\big)^3\left(t-\tfrac{nT}{N}\right)
     \right)
     .
  \end{equation}
  Then it holds
  for all 
  $ t \in (0,T] $, 
  $ N \in \N $, 
  $ p \in (0,\infty) $,
  $ q \in (2,\infty) $
  that
  $
    \limsup_{ M \to \infty }
    \P\big[
      \sup_{s\in[0,T]}
      |
        X_s - Y_s^M 
      |
      > p 
    \big] 
    = 0
  $
  and
  $
    \E\!\left[ \exp\!\left( p \, |Y_t^N|^q\right)\right]
    =\infty
  $.
\end{lemma}
\begin{proof}[Proof of Lemma~\ref{l:stopped.Euler.no.exp}]
  Throughout this proof
  let  
  $ N \in \N $, $ n \in \{ 0, 1, \ldots, N - 1 \} $,
  $ t \in \big( \frac{ n T }{ N } , \frac{ ( n + 1 ) T }{ N } \big]
  $,
  $
    p \in (0,\infty)
  $,
  and
  $
    q \in (2,\infty)
  $
  be numbers.
  Lemma~\ref{lem:consistent_stopped}
  implies that
  the function
  $
    \R \times (0,T] \times \R \ni
    (x,h,y) \mapsto \1_{ D_h }( x ) \left( y - x^3 h \right) \in \R
  $
  is $(\mu,\sigma)$-consistent with respect to Brownian motion.
  Then Proposition~\ref{prop:convergence.interpolation}
  applied to the function
  $
    \R \times (0,T]^2 \times \R
    \ni (x, h, s, y ) \mapsto
    \1_{ D_h }( x ) \big( y - x^3 s \big) 
    \in \R
  $
  shows that
  $
    \limsup_{M\to\infty}
   \P\big[
   \sup_{s\in[0,T]}
   \big|X_{s}-Y_s^{M}\big|>p\big]=0
 $.
  Moreover,
  the fact that
  $
    \P\big[
      Y_0^N \in D_{ T / N }
    \big] 
    =
    \P\big[
      X_0 \in D_{ T / N }
    \big] 
    > 0
  $
  and the fact that
  $
    \lambda_{ \R }\big( D_{ T / N } \big) 
    \geq 
    \lambda_{ \R }( D_T ) 
    > 0
  $
  prove that
  $
    \P\big[
      Y_{ n T / N }^N \in D_{ T / N }
    \big] > 0
  $.
  Combining this
  with 
  the fact that
  $ W_t - W_{ \frac{ n T }{ N } } $ 
  is independent from $ Y_{ \frac{ n T }{ N } }^N $
  yields that
  \begin{equation}
  \begin{split}
    &
    \E\!\left[
      \exp\!\left( p \, |Y_t^N|^q \right)
    \right]
  \geq
    \E\!\left[
      \mathbbm{1}_{
        \{
          Y^N_{ n T / N } \in D_{ T / N }
        \}
      }
      \exp\!\left( p \, |Y_t^N|^q \right)
    \right]
  \\ & =
    \int_{
      D_{ \frac{ T }{ N } }
    }
    \E\!\left[
      \exp\!\left( p \, | Y_t^N |^q \right)
      \big|
      \,
      Y_{ \frac{ n T }{ N }}^N = x
    \right]
    \P\!\left[
      Y_{ \frac{ n T }{ N } }^N \in dx
    \right]
  \\ & =
    \int_{
      D_{ \frac{ T }{ N } }
    }
    \E\!\left[
      \exp\!\left(
        p \,
        \big|
          W_t - W_{ \frac{ n T }{ N } }
          + x - x^3 \, ( t - \tfrac{ n T }{ N } )
        \big|^q
      \right)
    \right]
    \P\!\left[
      Y_{ \frac{ n T }{ N } }^N \in dx
    \right]
    .
  \end{split}
  \end{equation}
  The fact that
  $
    \forall \, x, y \in \R 
    \colon
    | y + x |^q \geq \frac{ |y|^q }{ 2^q } - |x|^q
  $
  and 
  the fact that 
  $
    \lim_{ y \to \infty }
    \left(
      |y|^q / y^2 
    \right) = \infty
  $
  hence prove that
  \begin{equation}
  \begin{split}
  &
    \E\!\left[
      \exp\!\left( p \, |Y_t^N|^q \right)
    \right]
  \\ & \geq
    \int_{
      D_{ T / N }
    }
    \E\!\left[
      \exp\!\left(
        \tfrac{ p }{ 2^q }
        \,
        \big|
          W_t - W_{ \frac{ n T }{ N } }
        \big|^q
      - p
      \left| x - x^3 ( t - \tfrac{ n T }{ N } )
      \right|^q
    \right)
    \right]
    \P\!\left[
      Y_{ \frac{ n T }{ N } }^N \in dx
    \right]
  \\ & =
    \E\!\left[
      \exp\!\left(
        \tfrac{ p }{ 2^q }
        \,
        \big|
          W_{ t - \frac{ n T }{ N } }
        \big|^q
      \right)
    \right]
    \left[
    \int_{
      D_{ \frac{ T }{ N } }
    }
    \exp\!\left(
      - p
      \left| x - x^3 ( t - \tfrac{ n T }{ N } )
      \right|^q
    \right)
    \P\!\left[
      Y_{ \frac{ n T }{ N } }^N \in dx
    \right]
    \right]
  \\ & =
    \E\!\left[
      \exp\!\left(
        \tfrac{ p }{ 2^q }
        \,
        \big|
          W_{ t - \frac{ n T }{ N } }
        \big|^q
      \right)
    \right]
    \E\!\left[
      \mathbbm{1}_{
        \{
          Y^N_{ n T / N } \in D_{ T / N }
        \}
      }
    \exp\!\left(
      - p \,
      \big|
        Y^N_{ n T / N } -
        \big[ Y^N_{ n T / N } \big]^3 [ t - \tfrac{ n T }{ N } ]
      \big|^q
    \right)
    \right]
  \\ & =
    \int_{-\infty}^\infty\tfrac{1}{\sqrt{2\pi}}
      \exp\!\left((\tfrac{Nt-nT}{4N})^{\frac{q}{2}}p\,|y|^{q}-\tfrac{y^2}{2}\right)
    dy\;
    \E\!\left[
      \mathbbm{1}_{
        \{
          Y^N_{ n T / N } \in D_{ T / N }
        \}
      }
    \exp\!\left(
      - p \,
      \big|
        Y^N_{ n T / N } -
        \big[ Y^N_{ n T / N } \big]^3 [ t - \tfrac{ n T }{ N } ]
      \big|^q
    \right)
    \right]
  \\ &
    = \infty
    .
  \end{split}
  \end{equation}
  This
  finishes the proof
  of Lemma~\ref{l:stopped.Euler.no.exp}.
\end{proof}

\subsection{Infinite exponential moments for a (stopped) linear-implicit Euler approximation scheme}

The following lemma shows that the stopped
linear-implicit Euler scheme~\eqref{eq:stopped.linear.implicit.Euler}
is not suitable
for approximating the
exponential moments on the left-hand side of~\eqref{eq:exampleSDE.Lyapunov}.
Display~\eqref{eq:stopped.linear.implicit.Euler}
shows that the linear-implicit Euler scheme~\eqref{eq:stopped.linear.implicit.Euler}
with $ \forall \, t \in (0,T] \colon D_t = \R $ 
belongs to the class of balanced implicit methods (choose
$ c^0 = [0,\infty) \times \R \ni x \mapsto x^2 \in \R $ and 
$ c^1 = [0,\infty) \times \R \ni x \mapsto 0 \in \R $
in the notation of (3.3) in \cite{mps98})
introduced
in Milstein, Platen \& Schurz~\cite{mps98}.

\begin{lemma}
\label{l:linear.implicit.Euler.no.exp}
  Assume the setting in Subsection~\ref{ssec:exampleSDE},
  let
  $
    D_t \in \mathcal{B}(\R)
  $,
  $
    t \in (0,T]
  $,
  be a non-increasing family of sets
  satisfying
  $
    \lambda_{ \R }( D_T )
    \cdot
    \P[ X_0 \in D_T ] > 0
  $
  and
  $
    \cup_{ t \in (0,T] }
    \mathring{ D }_t
    = \R
  $,
  and
  let
  $
    Y^N \colon [0,T] \times \Omega \to \R
  $,
  $ N \in \N
  $,
  be the mappings which satisfy 
  for all
  $ N \in \N $,
  $ n \in \{ 0, 1, \ldots , N - 1 \} $,
  $
    t \in \big[ \frac{ n T }{ N } , \frac{ (n+1) T }{ N } \big]
  $
  that
  $
    Y^N_0 = X_0
  $
  and
  \begin{align}  \label{eq:stopped.linear.implicit.Euler}
    Y_{t}^{N}
    &=Y_{\frac{nT}{N}}^N
    +\1_{D_{\frac{T}{N}}}(Y_{\frac{nT}{N}}^{N})
    \left(
           W_{t}-W_{\frac{nT}{N}}
    -Y_{t}^{N}\big(Y_{\frac{nT}{N}}^N\big)^2\left(t-\tfrac{nT}{N}\right)
     \right)
    \\&
    =Y_{\frac{nT}{N}}^N
    +\1_{D_{\frac{T}{N}}}(Y_{\frac{nT}{N}}^{N})
    \left(
           W_{t}-W_{\frac{nT}{N}}
    -\big(Y_{\frac{nT}{N}}^N\big)^3\left(t-\tfrac{nT}{N}\right)
     \right)
    +
    \1_{
      D_{ \frac{ T }{ N } }
    }\!(
      Y_{ \frac{ n T }{ N } }^N
    )
    \,
    \big(
      Y_{ \frac{ n T }{ N } }^N
    \big)^2
    \big(
      Y_{ \frac{ n T }{ N } }^N
      -
      Y_t^N
    \big)
    \left(
      t - \tfrac{ n T }{ N }
    \right)
  \nonumber
    .
  \end{align}
  Then it holds 
  for all $ t \in (0,T] $, $ M \in \N $, $ p \in (0,\infty) $, $ q \in (2, \infty ) $
  that
  $
    \limsup_{N\to\infty}
    \E\big[
      \sup_{
        n \in \{ 0, 1, \ldots, N \}
      }
      |
        X_{ \frac{ n T }{ N } }
        -
        Y_{ \frac{ n T }{ N }
        }^N
      |^p
    \big] = 0
  $
  and
  $
    \E\!\left[ \exp\!\left(p\,|Y_t^M|^q\right)\right]
    =\infty
  $.
\end{lemma}
\begin{proof}[Proof of Lemma~\ref{l:linear.implicit.Euler.no.exp}]
Throughout this proof let 
$ N \in \N $,
$ n \in \{ 0, 1, \ldots , N - 1 \} $, 
$ t \in \big( \frac{ n T }{ N } , \frac{ ( n + 1 ) T }{ N } \big] $, 
and $ q \in (2,\infty) $ 
be numbers. 
Lemma~3.30 in~\cite{HutzenthalerJentzen2014Memoires}
  and
  Lemma~\ref{lem:consistent_stopped}
  imply that
the function
$
  \R \times (0,T] \times \R
  \ni (x,h,y) \mapsto \1_{ D_h }( x ) 
  \,
  \big( 
    \tfrac{ x + y }{ 1 + x^2 h } - x 
  \big)
  \in \R
$
is $ (\mu, \sigma ) $-consistent with respect to Brownian motion.
Then Proposition~\ref{prop:convergence.interpolation}
applied to the function
$
  \R \times (0,T]^2 \times \R
  \ni (x, h, s, y) \mapsto
  \1_{ D_h }( x )
  \big(
    \tfrac{ x + y }{ 1 + x^2 s } - x
  \big)
  \in \R
$
shows
for all
$
  p \in (0, \infty )
$
that
$
  \limsup_{ M \to \infty }
  \P\big[
    \sup_{ s \in [0,T] }
    |
      X_s - Y_s^M
    | > p
  \big]
  = 0
$.
In addition, Lemma~2.28 in~\cite{HutzenthalerJentzen2014Memoires}
yields
for all $ p \in (0,\infty) $
that
$
  \sup_{ M \in \N }
  \big\|
    \sup_{
      m \in \{ 0, 1, \ldots, M \}
    }
    | Y_{ \frac{ m T }{ M } }^M |
  \big\|_{
    L^p( \Omega; \R )
  }
  < \infty
$
and
this shows for all $ p \in (0, \infty) $
that
the family
of random variables
  $
    \sup_{
      m \in \{ 0, 1, \ldots, M \}
    }
    |
      X_{ \frac{ m T }{ M } } - Y_{ \frac{ m T }{ M } }^M
    |^p
  $,
  $
    M \in \N
  $,
  is uniformly integrable.
  Combining this with convergence in probability
  and, e.g.,
  Theorem~6.25 in Klenke~\cite{Klenke2008}
  proves
  for all $ p \in (0, \infty) $
  that
  $
    \limsup_{ M \to \infty }
    \E\big[
      \sup_{
        m \in \{ 0, 1, \ldots, M \}
      }
      |
        X_{ \frac{ m T }{ M } }
        -
        Y_{ \frac{ m T }{ M } }^M
      |^p
    \big] = 0
$.
Moreover, the fact that
$
  \P\big[
    Y_0^N \in D_{ T / N }
  \big] 
  =
  \P\big[
    X_0 \in D_{ T / N }
  \big] 
  > 0
$
and the fact that
$
  \lambda_{ \R }\big( D_{ T / N } \big)
  \geq
  \lambda_{ \R }\big( D_T \big)
  > 0
$
prove that
$
    \P\big[
      Y_{ \frac{ n T }{ N } }^N
      \in D_{ \frac{ T }{ N } }
    \big] > 0
$.
This
and the fact that 
$ 
  \lim_{ y \to \infty }
  \big(
    |y|^q / y^2 
  \big)
  = \infty
$
imply that
  \begin{equation}  \begin{split}
    \E\!\left[ \exp\!\left( p \, |Y_t^N|^q \right) \right]
  & \geq
    \E\!\left[
      \mathbbm{1}_{
        \{
          Y^N_{ n T / N } \in D_{ T / N }
        \}
      }
      \exp\!\left( p \, |Y_t^N|^q \right)
    \right]
  \\ & =
    \int_{
      D_{ \frac{ T }{ N } }
    }
    \E\!\left[
      \exp\!\left(
        p \, | Y_t^N |^q
      \right)
      \big| \,
      Y_{ \frac{ n T }{ N } }^N = x
    \right]
    \P\!\left[
      Y_{ \frac{ n T }{ N } }^N
      \in dx
    \right]
  \\ & =
    \int_{
      D_{ \frac{ T }{ N } }
    }
    \E\!\left[
      \exp\!\left(
        p
        \,
        \Big|
          \tfrac{
            x + W_t - W_{ \frac{ n T }{ N } }
          }{
            1 + x^2
            \left( t - \frac{ n T }{ N }
            \right)
          }
        \Big|^q
      \right)
    \right]
    \P\!\left[
      Y_{
        \frac{ n T }{ N }
      }^N
      \in dx
    \right]
    \\&
    =
    \int_{
      D_{ \frac{ T }{ N } }
    }
    \int_{
      - \infty
    }^{ \infty }
    \tfrac{ 1 }{ \sqrt{ 2 \pi } }
      \exp\!\left(
        p
        \left|
          \tfrac{ x + \sqrt{ t - \frac{ n T }{ N } } y
          }{
            1 + x^2 \left( t - \frac{ n T }{ N } \right)
          }
        \right|^q
        -
        \tfrac{ y^2 }{ 2 }
      \right)
    dy
    \;
    \P\!\left[Y_{\frac{nT}{N}}^N\in dx\right]
    = \infty
    .
  \end{split}     \end{equation}
  This
  finishes the proof of Lemma~\ref{l:linear.implicit.Euler.no.exp}.
\end{proof}

\subsection{Unbounded exponential moments for a (stopped) increment-tamed Euler approximation scheme}

\begin{lemma}
\label{l:unbounded.expmoments}
Let
$
  T, q, \delta, \alpha, \beta \in (0,\infty)
$
satisfy
$
  q \beta > 2 \alpha + 1
$,
  let
  $
    f \colon \R \to \R
  $
  be a $ \mathcal{B}( \R ) $/$ \mathcal{B}( \R ) $-measurable and locally bounded function,
  let
  $
    D_{ h, t } \in \mathcal{B}( \R )
  $,
  $
    h, t \in (0,T]
  $,
  be
  sets
  satisfying
  $
    \forall \, n \in \N \colon
    \exists \, r \in (0,T] \colon
    [ - n , n ]
    \subseteq     
    \cap_{ h \in (0,r] }
    \cap_{ t \in (0,h] }
    D_{ h, t } 
  $,
  let
  $
    (
      \Omega, \mathcal{F}, \P
    )
  $ 
  be a probability space with a normal filtration 
  $
    (
      \mathcal{F}_t
    )_{
      t \in [0, T ]
    }
  $,
  let $W\colon[0,T]\times\Omega\to\R$ be a standard
  $(\mathcal{F}_t)_{t\in[0,T]}$-Brownian motion
  with continuous sample paths,
  let
  $
    Y^N
    \colon
    [0,T] \times \Omega \to \R
  $,
  $ N \in \N $,
  be 
  $ ( \mathcal{F}_t )_{ t \in [0,T] } $-adapted
  stochastic processes which satisfy
  that
  for all 
  $ N \in \N $,
  $
    n \in \{ 0, 1, \ldots, N - 1 \}
  $,
  $
    t \in
    \big(
      \frac{ n T }{ N } ,
      \frac{ ( n + 1 ) T }{ N }
    \big]
  $
  it holds $ \P $-a.s.\ that
  \begin{equation}
  \label{eq:condition.unbounded.expmoments}
    | Y_t^N |
  \geq
    \big[
      \tfrac{
        2 \delta
      }{
        (
          t - \frac{ n T }{ N }
        )^{
          \beta
        }
      }
      -
      f(
        Y_{ \frac{ n T }{ N } }^N
      )
    \big]
    \,
    \1_{
      \big\{
        1
      \leq
        \left( t - \frac{ n T }{ N } \right)^{ \alpha }
        (
          W_t - W_{ \frac{ n T }{ N } }
        )
      \leq
        2
      \big\}
      \cap
      \big\{
        Y_{
          \frac{ n T }{ N }
        }^N
        \in
        D_{ \frac{ T }{ N } , t - \frac{ n T }{ N } }
      \big\}
    }
    ,
  \end{equation}
  and let
  $X\colon[0,T]\times\Omega\to\R$
  be an 
  $ ( \mathcal{F}_t )_{ t \in [0,T] } $-adapted
  stochastic process
  with continuous sample paths 
  which satisfies
  $
    \limsup_{ N \to \infty }
    \sup_{ n \in \{ 0, 1, \ldots, N \} }
    \P\big[
      |
        X_{ n T / N }
        -
        Y_{ n T / N }^N
      |
      > 1
    \big] = 0
  $.
  Then
  $
    \liminf_{ N \to \infty }
    \inf_{ t \in (0,T] }
    \E\!\left[
      \exp\!\left(
        \delta
        \left| Y_t^N \right|^q
      \right)
    \right]
    = \infty
  $.
\end{lemma}
\begin{proof}
[Proof of Lemma~\ref{l:unbounded.expmoments}]
  Assumption~\eqref{eq:condition.unbounded.expmoments}
  implies that
  for all
  $ N \in \N $,
  $ n \in \{ 0, 1, \ldots, N - 1 \} $,
  $
    t \in 
    \big( \frac{ n T }{ N } , \frac{ ( n + 1 ) T }{ N } 
    \big]
  $
  it holds $ \P $-a.s.\ that
  \begin{align}
  &
    \exp\!\left(\delta\left|Y_t^N\right|^q\right)
  \nonumber
  \\ & \geq
  \nonumber
    \exp\!\left(
      \delta
      \left|
      \max\!\left\{
        0 ,
        \tfrac{
          2 \delta
        }{
          (
            t - \frac{ n T }{ N }
          )^{
            \beta
          }
        }
        -
        f(
          Y_{ \frac{ n T }{ N } }^N
        )
      \right\}
      \right|^q
    \right)
    \1_{
      \left\{
        1
      \leq
        \left( t - \frac{ n T }{ N } \right)^{ \alpha }
        (
          W_t - W_{ \frac{ n T }{ N } }
        )
      \leq
        2
      \right\}
      \cap
      \left\{
        Y_{
          \frac{ n T }{ N }
        }^N
        \in
        D_{ \frac{ T }{ N } , t - \frac{ n T }{ N } }
      \right\}
    }
  \\ & \geq
    \exp\!\left(
      \delta
      \left|
      \max\!\left\{
        0 ,
        \tfrac{
          2 \delta
        }{
          (
            t - \frac{ n T }{ N }
          )^{
            \beta
          }
        }
        -
        f(
          Y_{ \frac{ n T }{ N } }^N
        )
      \right\}
      \right|^q
    \right)
    \1_{
      \left\{
        1
      \leq
        \left( t - \frac{ n T }{ N } \right)^{ \alpha }
        (
          W_t - W_{ \frac{ n T }{ N } }
        )
      \leq
        2
      \right\}
      \cap
      \left\{
        Y_{
          \frac{ n T }{ N }
        }^N
        \in
        D_{ \frac{ T }{ N } , t - \frac{ n T }{ N } }
      \right\}
      \cap
      \left\{
        f\big(
          Y_{
            \frac{ n T }{ N }
          }^N
        \big)
        \leq
        \frac{ \delta }{
          ( t - \frac{ n T }{ N } )^{ \beta }
        }
      \right\}
    }
  \nonumber
  \\ & \geq
  \nonumber
    \exp\!\left(
      \delta
      \left|
        \tfrac{
          \delta
        }{
          (
            t - \frac{ n T }{ N }
          )^{
            \beta
          }
        }
      \right|^q
    \right)
    \1_{
      \left\{
        1
      \leq
        \left( t - \frac{ n T }{ N } \right)^{ \alpha }
        (
          W_t - W_{ \frac{ n T }{ N } }
        )
      \leq
        2
      \right\}
      \cap
      \left\{
        Y_{
          \frac{ n T }{ N }
        }^N
        \in
        D_{ \frac{ T }{ N } , t - \frac{ n T }{ N } }
      \right\}
      \cap
      \left\{
        f\big(
          Y_{
            \frac{ n T }{ N }
          }^N
        \big)
        \leq
        \frac{ \delta }{
          ( t - \frac{ n T }{ N } )^{ \beta }
        }
      \right\}
    }
  \\ & =
    \exp\!\left(
        \tfrac{
          \delta^{ ( q + 1 ) }
        }{
          (
            t - \frac{ n T }{ N }
          )^{
            q \beta
          }
        }
    \right)
    \cdot
    \1_{
      \left\{
        1
      \leq
        \left( t - \frac{ n T }{ N } \right)^{ \alpha }
        (
          W_t - W_{ \frac{ n T }{ N } }
        )
      \leq
        2
      \right\}
    }
    \cdot
    \1_{
      \left\{
        Y_{
          \frac{ n T }{ N }
        }^N
        \in
        D_{ \frac{ T }{ N } , t - \frac{ n T }{ N } }
      \right\}
    }
    \cdot
    \1_{
      \left\{
        f\big(
          Y_{
            \frac{ n T }{ N }
          }^N
        \big)
        \leq
        \frac{ \delta }{
          ( t - \frac{ n T }{ N } )^{ \beta }
        }
      \right\}
    }
    .
  \end{align}
  The fact that
  for all
  $ N \in \N $,
  $ n \in \{ 0, 1, \ldots, N - 1 \} $,
  $
    t \in \big( \frac{ n T }{ N } , \frac{ ( n + 1 ) T }{ N } \big] 
  $
  it holds that
  $
    W_t - W_{ \frac{ n T }{ N } }
  $ 
  and $ Y_{ \frac{ n T }{ N } }^N $
  are independent 
  hence implies
  for all
  $
    N \in \N
  $,
  $
    n \in \{ 0, 1, \ldots, N - 1 \}
  $,
  $
    t \in \big( \frac{ n T }{ N } , \frac{ ( n + 1 ) T }{ N } \big]
  $
  that
  \begin{align}
  \label{eq:unbounded.expmoments1}
  &
    \E\!\left[
      \exp\!\left(
        \delta
        \left|
          Y_t^N
        \right|^q
      \right)
    \right]
  \nonumber
  \\ & \geq
  \nonumber
    \exp\!\left(
        \tfrac{
          \delta^{ ( q + 1 ) }
        }{
          (
            t - \frac{ n T }{ N }
          )^{
            q \beta
          }
        }
    \right)
    \cdot
    \P\!\left[
      1
      \leq
      \left(
        t - \tfrac{ n T }{ N }
      \right)^{ \alpha }
      \big(
        W_t - W_{ \frac{ n T }{ N } }
      \big)
      \leq
      2
    \right]
    \cdot
    \P\!\left[
      |
        f( Y_{ \frac{ n T }{ N } }^N )
      |
      \left(
        t - \tfrac{ n T }{ N }
      \right)^{ \beta }
      \leq \delta
      ,
      \;
      Y_{ \frac{ n T }{ N } }^N
      \in
      D_{
        \frac{ T }{ N } , t - \frac{ n T }{ N }
      }
    \right]
  \\ &
    =
    \exp\!\left(
        \tfrac{
          \delta^{ ( q + 1 ) }
        }{
          (
            t - \frac{ n T }{ N }
          )^{
            q \beta
          }
        }
    \right)
    \cdot
    \int_1^2
    \tfrac{
      1
    }{
      \sqrt{
        2 \pi
        ( t - \frac{ n T }{ N } )^{ ( 2 \alpha + 1 ) }
      }
    }
    \exp\!\left(
      -
      \tfrac{
        y^2
      }{
        2 \, ( t - \frac{ n T }{ N } )^{ ( 2 \alpha + 1 ) }
      }
    \right)
    dy
    \cdot
    \P\!\left[
      |
        f( Y_{ \frac{ n T }{ N } }^N )
      |
      \left(
        t - \tfrac{ n T }{ N }
      \right)^{ \beta }
      \leq \delta
      ,
      \;
      Y_{ \frac{ n T }{ N } }^N
      \in
      D_{
        \frac{ T }{ N } , t - \frac{ n T }{ N }
      }
    \right]
  \nonumber
  \\ & \geq
    \tfrac{ 1 }{
      \sqrt{ 2 \pi T^{ ( 2 \alpha + 1 ) } }
    }
    \cdot
    \exp\!\left(
      \tfrac{
        \delta^{ (q + 1) }
      }{
        ( t - \frac{ n T }{ N }
        )^{ q \beta }
      }
      -
      \tfrac{ 2 }{
        ( t - \frac{ n T }{ N } )^{ ( 2 \alpha + 1 ) }
      }
    \right)
    \cdot
    \P\!\left[
      | f( Y_{ \frac{ n T }{ N } }^N ) |
      \left(
        t - \tfrac{ n T }{ N }
      \right)^{
        \beta
      }
      \leq \delta ,
      \;
      Y_{
        \frac{ n T }{ N }
      }^N
      \in D_{ \frac{ T }{ N } , t - \frac{ n T }{ N } }
    \right]
  \\ & \geq
  \nonumber
    \tfrac{ 1 }{
      \sqrt{ 2 \pi T^{ ( 2 \alpha + 1 ) } }
    }
    \cdot
    \exp\!\left(
      \inf_{ 
        z \in ( 0, \frac{ T }{ N } ]
      }
      \left[
        \tfrac{
          \delta^{ (q+1) }
        }{
          z^{ q \beta }
        }
        -
        \tfrac{ 2 }{ 
          z^{ ( 2 \alpha + 1 ) }
        }
      \right]
    \right)
    \cdot
    \P\!\left[
      | f( Y_{ \frac{ n T }{ N } }^N ) |
      \left(
        t - \tfrac{ n T }{ N }
      \right)^{
        \beta
      }
      \leq \delta ,
      \;
      Y_{
        \frac{ n T }{ N }
      }^N
      \in D_{ \frac{ T }{ N } , t - \frac{ n T }{ N } }
    \right]
  \\ & =
\nonumber
    \tfrac{ 1 }{
      \sqrt{ 2 \pi T^{ ( 2 \alpha + 1 ) } }
    }
    \cdot
    \exp\!\left(
      \inf_{
        z \in [ \frac{ N }{ T } , \infty ) 
      }
      \left[
        \delta^{ ( q + 1 ) } 
        z^{ q \beta }
        -
        2 z^{ ( 2 \alpha + 1 ) }
      \right]
    \right)
    \cdot
    \P\!\left[
      | f( Y_{ \frac{ n T }{ N } }^N ) |
      \left(
        t - \tfrac{ n T }{ N }
      \right)^{
        \beta
      }
      \leq \delta ,
      \;
      Y_{
        \frac{ n T }{ N }
      }^N
      \in D_{ \frac{ T }{ N } , t - \frac{ n T }{ N } }
    \right]
    .
  \end{align}
  In the next step we note that the fact the sample paths of $ X $ are continuous 
  ensures that there exists a natural number
  $ k \in \N $
  such that
  \begin{equation}
  \label{eq:Pgeq12}
    \P\big[
      \sup\nolimits_{ s \in [0,t] }
      | X_s | \leq k - 1
    \big]
    \geq
    \tfrac{ 1 }{ 2 }
    .
  \end{equation}
  The assumption that
  $
    \forall \, n \in \N \colon
    \exists \, r \in (0,T] \colon
    [-n,n]
    \subseteq
    \cap_{ h \in (0,r] }
    \cap_{ t \in (0,h] }
    \,
    D_{ h, t }
  $
  and the fact that
  $
    \sup_{
      x \in [ - k, k ]
    }
    | f(x) | < \infty
  $
  yield that there exists a natural number
  $ N_0 \in \N $
  such that
  $
    N_0
  \geq
    T
    \,
    \big[
      \frac{ 1 }{ \delta }
      \cdot
      \sup_{ x \in [ - k , k ] } |f(x)|
    \big]^{ 1 / \beta }
  $
  and 
  $
    [ - k, k ]
    \subseteq
    \big(
      \cap_{ h \in (0, \frac{ T }{ N_0 } ] }
      \cap_{ t \in (0,h] }
      \, D_{ h, t }
    \big)
    =
    \big(
      \cap_{ N = N_0 }^{ \infty }
      \cap_{ h \in (0, \frac{ T }{ N } ] }
      \cap_{ t \in (0,h] }
      \, D_{ h, t }
    \big)
  $.
  This shows
  for all $ N \in \N \cap [ N_0, \infty ) $ that
  \begin{equation}  \begin{split}
  [ - k, k ]
  & \subseteq
    \bigcap_{ h \in (0, \frac{ T }{ N } ] }
    \bigcap_{t\in(0,h]}
    \left\{
      x \in D_{ h, t } \cap [ - k, k ]
      \colon
        N
        \geq
        T
        \big[
          \tfrac{ 1 }{ \delta }
          \cdot
          \sup\nolimits_{ y \in [ - k, k ]  }
          |f(y)|
        \big]^{ 1 / \beta }
    \right\}
  \\
  & \subseteq
    \bigcap_{ h \in (0, \frac{ T }{ N } ] }
    \bigcap_{t\in(0,h]}
    \left\{
      x \in D_{ h, t } \cap [ - k, k ]
      \colon
        \left[
          \tfrac{ N }{ T }
        \right]^{ \beta }
        \geq
          \tfrac{ 1 }{ \delta }
          \cdot
          | f(x) |
    \right\}
  \\ &
  \subseteq
    \bigcap_{
      h \in ( 0, \frac{ T }{ N } ]
    }
    \bigcap_{
      t \in (0, h ]
    }
    \Big\{
      x \in D_{ h, t }
      \colon
      | f(x) | \, t^{ \beta }
      \leq \delta
    \Big\}
  \subseteq
    \bigcap_{
      t \in (0, \frac{ T }{ N } ]
    }
    \Big\{
      x \in D_{ \frac{ T }{ N }, t }
      \colon
      | f(x) | \, t^{ \beta }
      \leq \delta
    \Big\}
  \\&
  =
    \bigcap_{ n = 0 }^{ N - 1 }
    \bigcap_{
      t \in
      ( \frac{ n T }{ N } , \frac{ ( n + 1 ) T }{ N } ]
    }
    \bigg\{
      x \in
      D_{ \frac{ T }{ N } , t - \frac{ n T }{ N }
      }
      \colon
      | f(x) |
      \left[
        t - \tfrac{ n T }{ N }
      \right]^{ \beta }
      \leq \delta
    \bigg\}
    .
  \end{split}     
  \end{equation}
  Hence, we obtain that
  for all
  $ N \in \N \cap [ N_0, \infty ) $,
  $ n \in \{ 0, 1, \ldots, N - 1 \} $,
  $
    t \in \big( \frac{ n T }{ N } , \frac{ ( n + 1 ) T }{ N } \big]
  $
  it holds that
\begin{equation}
    \Big\{
      x \in
      D_{ \frac{ T }{ N } , t - \frac{ n T }{ N }
      }
      \colon
      | f(x) |
      \left[
        t - \tfrac{ n T }{ N }
      \right]^{ \beta }
      \leq \delta
    \Big\}
  \supseteq
  \Big\{ 
    x \in \R \colon
    | x | \leq k
  \Big\}
  .
\end{equation}
  Combining this 
  with 
  \eqref{eq:Pgeq12} 
  and
  the monotonicity of $ \P $ 
  yields that
  for all
  $ N \in \N \cap [ N_0, \infty ) $,
  $ n \in \{ 0, 1, \ldots, N - 1 \} $,
  $
    t \in \big( \frac{ n T }{ N } , \frac{ ( n + 1 ) T }{ N } \big]
  $
  it holds that
  \begin{equation}
  \label{eq:unbounded.expmoments2}
  \begin{split}
  &
    \P\!\left[
      |
        f( Y_{ \frac{ n T }{ N } }^N )
      |
      \left(
        t - \tfrac{ n T }{ N }
      \right)^{ \beta }
      \leq \delta
      ,
      \;
      Y_{
        \frac{ n T }{ N }
      }^N
      \in D_{ \frac{ T }{ N } , t - \frac{ n T }{ N } }
    \right]
  \geq
    \P\!\left[
      |
        Y_{ \frac{ n T }{ N } }^N
      |
      \leq k
    \right]
  \\ & \geq
    \P\!\left[
      |
        X_{ \frac{ n T }{ N } }
      |
      \leq k - 1
      ,
      \;
      |
        X_{ \frac{ n T }{ N } }
        -
        Y_{ \frac{ n T }{ N } }^N
      |
      \leq 1
    \right]
  \geq
     \frac{ 1 }{ 2 }
     -
     \sup_{
       m \in \{ 0, 1, \ldots, N \}
     }
     \P\!\left[
       |
         X_{ \frac{ m T }{ N } }
         -
         Y_{ \frac{ m T }{ N } }^N
       | > 1
     \right]
     .
  \end{split}     \end{equation}
  In the next step we combine 
  inequalities~\eqref{eq:unbounded.expmoments1}
  and~\eqref{eq:unbounded.expmoments2} 
  with the assumption that
  $
    q \beta > 2 \alpha + 1
  $
  and with the assumption that
  $
    \limsup_{
      N \to \infty
    }
    \sup_{
      n \in \{ 0, 1, \ldots, N \}
    }
    \P\big[
      |
        X_{ \frac{ n T }{ N } }
        -
        Y_{ \frac{ n T }{ N } }^N
      | > 1
    \big] = 0
  $
  to obtain that
  \begin{equation}  \begin{split}
    &
    \liminf_{ N \to \infty }
    \inf_{ t \in (0,T] }
    \E\!\left[
      \exp\!\left(
        \delta
        \left| Y_t^N \right|^q
      \right)
    \right]
    \\&
    \geq
    \liminf_{N\to\infty}
    \left[
    \tfrac{1}{\sqrt{2\pi T^{2\alpha+1}}}
    \cdot
    \exp\!\left(
      \inf_{
        z \in [ \frac{ N }{ T } , \infty ) 
      }
      \left[
        \delta^{ ( q + 1 ) } 
        z^{ q \beta }
        -
        2 z^{ ( 2 \alpha + 1 ) }
      \right]
    \right)
    \left(
      \frac{ 1 }{ 2 }
      -
      \sup_{
        m \in \{ 0, 1, \ldots, N \}
      }
      \P\!\left[
        |
          X_{ \frac{ m T }{ N } }
          -
          Y_{ \frac{ m T }{ N } }^N
        | > 1
      \right]
    \right)
    \right]
  \\ &
    =\infty.
  \end{split}     \end{equation}
  The proof of Lemma~\ref{l:unbounded.expmoments}
  is thus completed.
\end{proof}

\begin{corollary}
\label{c:increment.tamed.no.exp1}
  Assume the setting in Subsection~\ref{ssec:exampleSDE},
  let
  $
    D_t \in \mathcal{B}(\R)
  $,
  $ t \in (0,T] $,
  be a non-increasing family of sets
  satisfying
  $
    \forall \, n \in \N \colon
    \exists \, t \in (0,T] \colon
    [ -n , n ] \subseteq D_t
  $,
  and
  let
  $
    Y^N \colon [0,T] \times \Omega \to \R
  $,
  $ N \in \N $,
  be the mappings 
  which satisfy
  for all
  $ N \in \N $, 
  $ n \in \{ 0, 1, \ldots, N - 1 \} $,
  $ 
    t \in \big( \frac{ n T }{ N } , \frac{ ( n + 1 ) T }{ N } \big]
  $
  that
  $
    Y^N_0 = X_0
  $
  and
  \begin{equation}  
  \label{eq:increment-tamed.Euler1}
    Y_t^N
  =
    Y_{
      \frac{ n T }{ N }
    }^N
    +
    \1_{
      D_{ \frac{ T }{ N } }
    }\!(
      Y_{ \frac{ n T }{ N } }^N
    )
    \left[
      \frac{
        W_t - W_{ \frac{ n T }{ N } }
        -
        \big(
          Y_{ \frac{ n T }{ N } }^N
        \big)^3
        (
          t - \frac{ n T }{ N }
        )
      }{
        \max\!\big\{
          1,
          \left( t - \frac{ n T }{ N } \right)
          \big|
            W_t - W_{ \frac{ n T }{ N } }
            -
            \big(
              Y_{ \frac{ n T }{ N } }^N
            \big)^3
            (
              t - \frac{ n T }{ N }
            )
          \big|
        \big\}
      }
    \right]
    .
  \end{equation}
  Then 
  it holds
  for all
  $ p \in (0, \infty ) $,
  $ q \in ( 3, \infty ) $
  that
  $
    \limsup_{ N \to \infty }
    \P\big[
      \sup_{
        n \in \{ 0, 1, \ldots, N \}
      }
      |
        X_{ \nicefrac{ n T }{ N } }
        -
        Y_{ \nicefrac{ n T }{ N }
        }^N
      |
      > p
    \big] = 0
  $
  and
  $
    \liminf_{ N \to \infty }
    \inf_{ t \in (0,T] }
    \E\!\left[
      \exp\!\left(
        p \left| Y_t^N \right|^q
      \right)
    \right]
    = \infty
  $.
\end{corollary}

\begin{proof}[Proof of Corollary~\ref{c:increment.tamed.no.exp1}]
Throughout this proof let
$ p \in (0,\infty) $
and
$ q \in (3,\infty) $
be real numbers
and let
$
  \psi \colon \R \times (0,T]^2 \times \R \to \R
$
be the mapping with the property that
for all 
$
  (x,h,t,y) \in \R \times (0,T]^2 \times \R
$
it holds that
$
    \psi( x, h, t, y)
    =
    \1_{ D_h }\!( x )
    \tfrac{
      y - x^3 t
    }{
      \max\{ 1, t | y - x^3 t | \}
    }
$.
In the next step we apply 
Lemma 3.28
in~\cite{HutzenthalerJentzen2014Memoires}
and
Lemma~\ref{lem:consistent_stopped}
to obtain that
the function $\R\times(0,T]\times\R\ni(x,t,y)\mapsto \psi(x,t,t,y)\in\R$
  is $(\mu,\sigma)$-consistent with respect to Brownian motion.
  Proposition~\ref{prop:convergence_probab}
  hence implies for all $ r \in ( 0, \infty ) $ that
  $
    \limsup_{ N \to \infty }
    \P\big[
      \sup_{
        n \in \{ 0, 1, \ldots, N \}
      }
      |
        X_{ \frac{ n T }{ N } }
        -
        Y_{ \frac{ n T }{ N } }^N
      | > r
    \big] = 0
  $.
  For proving the divergence statement
  in Corollary~\ref{c:increment.tamed.no.exp1},
  we intend to apply Lemma~\ref{l:unbounded.expmoments} above.
To this end we first prove 
inequality~\eqref{eq:condition.unbounded.expmoments}.
  For this let
  $
    D_{ h, t } \in \mathcal{B}( \R )
  $,
  $ h, t \in (0,T] $,
  be the sets 
which satisfy
for all $ h, t \in (0,T] $
that
  $
    D_{ h, t }
    =
      D_h
      \cap
      [ - t^{ - 2 / 3 } , t^{ - 2 / 3 } ]
  $.
  Next note that for all $ r \in (0,T] $
  it holds that
  \begin{equation}
  \begin{split}
    \cap_{ h \in (0,r] }
    \cap_{ t \in (0,h] }
    D_{ h, t }
  &
    =
    \cap_{ h \in (0,r] }
    \cap_{ t \in (0,h] }
    \big(
      D_h
      \cap
      \big[
        - t^{ - 2 / 3 }
        ,
        t^{ - 2 / 3 }
      \big]
    \big)
   =
    \cap_{ h \in (0,r] }
    \big(
      D_h
      \cap
      \big[
        - h^{ - 2 / 3 }
        ,
        h^{ - 2 / 3 }
      \big]
    \big)
  \\ & =
      D_r
      \cap
      \big[
        - r^{ - 2 / 3 }
        ,
        r^{ - 2 / 3 }
      \big]
    .
  \end{split}
  \end{equation}
  This and the assumption that
  $
    \forall \, n \in \N \colon \exists \, t \in (0,T]
    \colon
    [ -n , n ] \subseteq D_t 
  $
  assure that for all $ n \in \N $
  there exists a real number $ r \in (0,T] $
  such that
  \begin{equation}
  \label{eq:sets_property}
    [-n,n] \subseteq
    \cap_{ h \in (0,r] }
    \cap_{ t \in (0,h] }
    D_{ h, t }
    .
  \end{equation}
  Moreover, observe that
  for all
  $
    ( x, h, t, y)
    \in \R \times (0,T]^2 \times \R
  $
  it holds that
  \begin{equation}  \begin{split}
    \left|
      x + \psi( x, h, t, y )
    \right|
  & \geq
    \left(
      \tfrac{
        |y - x^3 t |
      }{
        \max( 1, t | y - x^3 t | )
      }
      - |x|
    \right)
    \cdot
    \1_{
      D_{ h, t }
    }(x)
    \cdot
    \1_{ [1, 2] }\big( t y \big)
  \\
  & \geq
    \left(
      \tfrac{
        |y |
      }{
        \max( 1, t | y - x^3 t | )
      }
      - |x|
      - | x |^3 t
    \right)
    \cdot
    \1_{
      D_{ h, t }
    }(x)
    \cdot
    \1_{ [1, 2] }\big( t y \big)
  \\ & \geq
    \left(
      \tfrac{
        |y|
      }{
        1 + t^2 |x|^3 + t |y|
      }
      - |x| - |x|^3 t
    \right)
    \cdot
    \1_{
      D_{ h, t }
    }(x)
    \cdot
    \1_{ [1, 2] }\big( t y \big)
  \\ & \geq
    \left(
      \tfrac{ |y| }{ 2 + t |y| } - |x| - |x|^3 T
    \right)
    \cdot
    \1_{
      D_{ h, t }
    }(x)
    \cdot
    \1_{ [1, 2] }\big( t y \big)
  \\ & \geq
    \left(
      \tfrac{ 1 }{ 3 t }
      - |x| - |x|^3 T
    \right)
    \cdot
    \1_{
      D_{ h, t }
    }(x)
    \cdot
    \1_{ [1, 2] }\big( t y \big)
    .
  \end{split}     \end{equation}
  This together with the fact that
  for all
  $
    N \in \N
  $,
  $
    n \in \{ 0, 1, \ldots, N - 1 \}
  $,
  $
    t \in \big( \frac{ n T }{ N } , \frac{ (n + 1) T }{ N } \big]
  $
  it holds that
  $
    Y_t^N =
    Y_{ \frac{ n T }{ N } }^N
    +
    \psi\big(
      Y_{ \frac{ n T }{ N } }^N , \tfrac{ T }{ N } , t - \tfrac{ n T }{ N }, W_t - W_{ \frac{ n T }{ N } }
    \big)
  $
  shows that
  for all
  $
    N \in \N
  $,
  $
    n \in \{ 0, 1, \ldots, N - 1 \}
  $,
  $
    t \in \big( \frac{ n T }{ N } , \frac{ (n + 1) T }{ N } \big]
  $
  it holds that
  \begin{equation}
    | Y_t^N |
  \geq
    \bigg[
      \tfrac{
        2 \min\{ \frac{ 1 }{ 6 } , p \}
      }{
        (
          t - \frac{ n T }{ N }
        )
      }
      -
      \left[
        \big|
          Y_{ \frac{ n T }{ N } }^N
        \big|
        +
        \big|
          Y_{ \frac{ n T }{ N } }^N
        \big|^3
        T
      \right]
    \bigg]
    \1_{
      \left\{
        1
      \leq
        \left( t - \frac{ n T }{ N } \right)
        (
          W_t - W_{ \frac{ n T }{ N } }
        )
      \leq
        2
      \right\}
      \cap
      \left\{
        Y_{
          \frac{ n T }{ N }
        }^N
        \in
        D_{ \frac{ T }{ N } , t - \frac{ n T }{ N } }
      \right\}
    }
    .
  \end{equation}
  This and \eqref{eq:sets_property} allows us to apply 
  Lemma~\ref{l:unbounded.expmoments}
  (with $\alpha=\beta=1$ in the notation of
  Lemma~\ref{l:unbounded.expmoments})
  to obtain
  the divergence statement in
  Corollary~\ref{c:increment.tamed.no.exp1}.
  The proof of
  Corollary~\ref{c:increment.tamed.no.exp1}
  is thus completed.
\end{proof}

The proof of the following corollary, Corollary~\ref{c:increment.tamed.no.exp2},
is analogous
to the proof of Corollary~\ref{c:increment.tamed.no.exp1} 
and therefore omitted.

\begin{corollary}
\label{c:increment.tamed.no.exp2}
  Assume the setting in Subsection~\ref{ssec:exampleSDE},
  let
  $
    D_t \in \mathcal{B}(\R)
  $,
  $
    t \in (0,T]
  $,
  be a non-increasing family of sets
  satisfying
  $
    \forall \, n \in \N 
    \colon
    \exists \, t \in (0,T]
    \colon
    [-n, n]
    \subseteq
    D_t
  $,
  and
  let
  $
    Y^N \colon [0,T] \times \Omega \to \R
  $,
  $ N \in \N $,
  be mappings which satisfy
  for all
  $ N \in \N $,
  $
    n \in \{ 0, 1, \ldots, N - 1 \}
  $,
  $
    t \in
    \big(
      \frac{ n T }{ N },
      \frac{ ( n + 1 ) T }{ N }
    \big]
  $
  that
  $
    Y^N_0 = X_0
  $
  and
  \begin{equation}  \label{eq:increment-tamed.Euler2}
    Y_t^N
  =
    Y_{ \frac{ n T }{ N } }^N
    +
    \1_{
      D_{ \frac{ T }{ N } }
    }\!(
      Y_{ \frac{ n T }{ N } }^N
    )
    \left[
    \frac{
      W_t - W_{ \frac{ n T }{ N } }
      -
      (
        Y_{ \frac{ n T }{ N } }^N
      )^3
      \left(
        t - \frac{ n T }{ N }
      \right)
    }{
      1 +
      \left( t - \frac{ n T }{ N } \right)
      \big|
        W_t -
        W_{
          \frac{ n T }{ N }
        }
        -
        (
          Y_{ \frac{ n T }{ N } }^N
        )^3
        \left(
          t - \frac{ n T }{ N }
        \right)
      \big|
    }
    \right]
    .
  \end{equation}
  Then
  it holds  
  for all $ p \in (0,\infty) $, $ q \in ( 3, \infty) $
  that
  $
    \limsup_{ N \to \infty }
    \P\big[
      \sup_{ n \in \{0, 1, \ldots, N \} }
      |
        X_{ \nicefrac{ n T }{ N } }
        -
        Y_{ \nicefrac{ n T }{ N } }^N
      |
      > p
    \big]
    = 0
  $
  and
  $
    \liminf_{ N \to \infty }
    \inf_{ t \in (0,T] }
    \E\!\left[
      \exp\!\left(
        p \,
        | Y_t^N |^q
      \right)
    \right]
    = \infty
  $.
\end{corollary}

%
%

\subsection*{Acknowledgements}

Special thanks are due to Andreas Herzwurm, Primoz Pusnik, 
and Klaus Ritter for fruitful discussions on exponentially growing test functions
and some auxiliary lemmas.
XW is grateful to the Institute for Mathematical Research (FIM) at ETH Zurich, 
which provided the office space for him and partially organized his short 
visit to AJ in 2013.
This project has been partially supported by the
research project
``Numerical approximation of stochastic differential equations with non-globally
Lipschitz continuous coefficients'' (HU1889/2-1)
funded by the German Research Foundation and by a research funding from NSF of China (11671405, 11301550).

%
%

\bibliographystyle{acm}
 \bibliography{bibfile}

\end{document}